\documentclass[10pt]{preprint}

\usepackage[margin=3.5cm]{geometry}
\usepackage[utf8x]{inputenc}
\usepackage[english]{babel}
\usepackage[full]{textcomp}
\usepackage[osf]{newtxtext}
\usepackage{amssymb}
\usepackage{amsmath}
\usepackage{amsthm}
\usepackage{mhequ}
\usepackage{booktabs}
\usepackage{tikz}
\usepackage{mathrsfs}
\usepackage[noadjust]{cite}
\usepackage{microtype}
\usepackage{margincomment}
\usepackage{slashed}
\usepackage{mathtools}
\usepackage{centernot}
\usepackage{footnote}
\usepackage{enumerate}
\usepackage[shortlabels]{enumitem}
\usepackage{stackrel}
\usepackage{longtable}
\usepackage{cprotect}
\usepackage{xstring}
\usepackage{calc}
\usepackage{dsfont}
\usepackage{bbm}
\usepackage[colorlinks=true, pdfstartview=FitV, linkcolor=colorLink, citecolor=colorCite, urlcolor=colorLink, linktocpage=true]{hyperref}
\usepackage{upgreek}
\usepackage{parskip}
\usepackage{graphicx}
\usepackage{cleveref}
\usepackage{orcidlink}
\usepackage{caption}

\makeatletter
\renewcommand{\paragraph}{%
\@startsection{paragraph}{4}%
{\z@}{1.5ex \@plus 1.5ex \@minus .2ex}{-0.7em}%
{\normalfont\normalsize\bfseries}%
}
\makeatother

\makeatletter
\def\thm@space@setup{%
  \thm@preskip=\parskip \thm@postskip=0pt
}
\makeatother

\setlist[itemize]{leftmargin=5mm}

\crefformat{footnote}{#2\footnotemark[#1]#3}
\DeclareSymbolFont{timesoperators}{T1}{ptm}{m}{n}
\SetSymbolFont{timesoperators}{bold}{T1}{ptm}{b}{n}
\DeclareMathAlphabet{\mathbb}{U}{jkpsyb}{m}{n}
\SetMathAlphabet{\mathbb}{bold}{U}{jkpsyb}{bx}{n}
\DeclarePairedDelimiter\abs\lvert\rvert

\allowdisplaybreaks
\definecolor{colorLink}{RGB}{0,100,162}
\definecolor{colorCite}{RGB}{8,124,100}
\def\M{\mathbb{M}}
\def\D{\mathbb{D}}
\def\B{\mathbb{B}}

\def\N{\mathbb{N}}
\def\R{\mathbb{R}}

\def\Z{\mathbb{Z}}

\def\E{\mathbb{E}}
\def\P{\mathbb{P}}

\def\S{\mathbb{S}}
\def\A{\mathbb{A}}
\def\G{\mathbb{G}}

\def\X{\mathbb{X}}
\def\H{\mathbb{H}}
\def\L{\mathbb{L}}

\def\CC{\mathcal{C}}

\def\CF{\mathcal{F}}
\def\CG{\mathcal{G}}
\def\CH{\mathcal{H}}

\def\CM{\mathcal{M}}
\def\CN{\mathcal{N}}
\def\CO{\mathcal{O}}
\def\CP{\mathcal{P}}

\def\CR{\mathcal{R}}

\def\CX{\mathcal{X}}
\def\CY{\mathcal{Y}}

\def\rmT{\mathscr{T}}

\def\rmA{\mathrm{A}}
\def\rmB{\mathrm{B}}
\def\rmC{\mathrm{C}}
\def\rmD{\mathrm{D}}
\def\rmE{\mathrm{E}}
\def\rmF{\mathrm{F}}
\def\rmG{\mathrm{G}}

\def\rmK{\mathrm{K}}
\def\rmL{\mathrm{L}}

\def\rmR{\mathrm{R}}
\def\rmS{\mathrm{S}}
\def\rmT{\mathrm{T}}
\def\rmU{\mathrm{U}}

\def\rmW{\mathrm{W}}
\def\rmX{\mathrm{X}}
\def\rmY{\mathrm{Y}}
\def\rmZ{\mathrm{Z}}

\def\bfF{\mathbf{F}}
\def\bfG{\mathbf{G}}
\def\bfC{\mathbf{C}}
\def\bfZ{\mathbf{Z}}
\def\bfW{\mathbf{W}}

\def\frkI{\mathfrak{I}}
\def\frkJ{\mathfrak{J}}

\def\frkK{\mathfrak{K}}
\def\frkF{\mathfrak{F}}

\def\frkd{\mathfrak{d}}

\def\one{\mathds{1}}

\let\eps\upvarepsilon
\let\iota\upiota
\let\alpha\upalpha
\let\beta\upbeta
\let\delta\updelta
\let\gamma\upgamma
\let\mu\upmu
\let\eta\upeta
\let\nu\upnu
\let\rho\uprho
\let\chi\upchi
\let\xi\upxi
\let\phi\upphi
\let\psi\uppsi
\let\zeta\upzeta
\let\tau\uptau
\let\varphi\upvarphi
\let\lambda\uplambda
\let\theta\uptheta
\let\pi\uppi
\let\Upsilon\Upupsilon
\let\Theta\Uptheta
\let\Psi\Uppsi
\let\Xi\Upxi

\let\f\frac
\let\d\partial

\DeclareMathOperator{\Leb}{Leb}

\DeclareMathOperator{\argmax}{argmax}

\DeclareMathOperator{\Osc}{Osc}

\DeclareMathOperator{\Law}{Law}
\DeclareMathOperator{\gammac}{\gamma_{c}}
\DeclareMathOperator{\dd}{d}
\DeclareMathOperator{\new}{new}
\def\loc{\mathrm{loc}}

\def\scal#1{\langle#1\rangle}

\newcommand{\eqdef}{\stackrel{\mbox{\tiny\rm def}}{=}}
\newcommand{\eqlaw}{\stackrel{\mbox{\tiny\rm law}}{=}}

\def\dash{\leavevmode\unskip\kern0.18em--\penalty\exhyphenpenalty\kern0.18em}
\def\slash{\leavevmode\unskip\kern0.15em/\penalty\exhyphenpenalty\kern0.15em}

\makeatletter
\renewcommand{\operator@font}{\mathgroup\symtimesoperators}
\makeatother

\makeatletter
\DeclareRobustCommand{\TitleEquation}[2]{\texorpdfstring{\StrLeft{\f@series}{1}[\@firstchar]$\if%
b\@firstchar\boldsymbol{#1}\else#1\fi$}{#2}}
\makeatother

\makeatletter
\newcommand{\pushright}[1]{\ifmeasuring@#1\else\omit\hfill$\displaystyle#1$\fi\ignorespaces}
\newcommand{\pushleft}[1]{\ifmeasuring@#1\else\omit$\displaystyle#1$\hfill\fi\ignorespaces}
\makeatother

\renewcommand{\bar}{\overline}
\renewcommand{\hat}{\widehat}
\renewcommand{\tilde}{\widetilde}
\newcommand{\ceps}{{(\eps)}}

\makeatletter
\newcommand{\oset}[3][0ex]{%
  \mathrel{\mathop{#3}\limits^{
    \vbox to#1{\kern-2\ex@
    \hbox{$\scriptstyle#2$}\vss}}}}
\makeatother

\newcommand{\frka}{\mathfrak{a}}
\newcommand{\frkg}{\mathfrak{g}}
\newcommand{\frkh}{\mathfrak{h}}
\newcommand{\frkm}{\mathfrak{m}}

\usetikzlibrary{shapes.misc}
\usetikzlibrary{shapes.symbols}
\usetikzlibrary{shapes.geometric}
\usetikzlibrary{decorations}
\usetikzlibrary{decorations.markings}
\usetikzlibrary{calc}
\usetikzlibrary{external}
\usetikzlibrary{arrows}
\usetikzlibrary{patterns}
\theoremstyle{plain}
\newtheorem{theorem}{Theorem}[section]
\newtheorem{corollary}[theorem]{Corollary}
\newtheorem{lemma}[theorem]{Lemma}
\newtheorem{proposition}[theorem]{Proposition}

\newtheorem{theoremA}{Theorem}[section]
\newtheorem{corollaryA}[theoremA]{Corollary}

\theoremstyle{definition}
\newtheorem{definition}[theorem]{Definition}

\newtheorem*{acknowledgements}{Acknowledgements}
\newtheorem{assumption}[theorem]{Assumption}
\newtheorem{remark}[theorem]{Remark}

\numberwithin{equation}{section}

\colorlet{darkblue}{blue!90!black}
\colorlet{darkgreen}{green!50!black}

\begin{document}

\title{How does the supercritical GMC converge?}

\author{Federico Bertacco$^1$\orcidlink{0000-0002-6363-1294}, Martin Hairer$^{2}$\orcidlink{0000-0002-2141-6561}}

\institute{Imperial College London, Email: \href{mailto:f.bertacco20@imperial.ac.uk}{\color{black}\texttt{f.bertacco20@imperial.ac.uk}} \and EPFL \& Imperial College London, Email: \href{mailto:martin.hairer@epfl.ch}{\color{black}\texttt{martin.hairer@epfl.ch}}}

\maketitle

\begin{abstract}
In the spirit of \cite{BiskupLouidor}, we study the local structure of $\star$-scale invariant fields \dash a class of log-correlated Gaussian fields \dash around their extremal points by characterising the law of the ``shape'' of the field's configuration near such points. As a consequence, we obtain
a refined understanding of the freezing phenomenon in supercritical Gaussian multiplicative chaos.
\end{abstract}

\setcounter{tocdepth}{2}
\tableofcontents

%%%%%%%%%%%%%%%%%%%%%%%%%%%%%%%%%%%%%%%%%%%%%%
%%%%%%%%%%%%%%%%%%%%%%%%%%%%%%%%%%%%%%%%%%%%%%
\section{Introduction}
\label{sec:intro}
The theory of Gaussian Multiplicative Chaos (GMC) involves the study of random measures that can be formally expressed as 
\begin{equation}
\label{e:GMCformal}
\mu_{\gamma}(dx) \mathrel{\text{``$=$''}} e^{\gamma \rmX(x)} dx \;,
\end{equation}
where $\gamma > 0$ is a real positive parameter representing the \emph{inverse temperature} of the model, $\rmX$ is a \emph{log-correlated Gaussian field} on a domain $\rmD \subseteq \R^d$, and $dx$ denotes Lebesgue measure on $\rmD$. Since $\rmX$ only exists as a random Schwartz distribution, regularisation and renormalisation are necessary to show the existence of the measure $\mu_{\gamma}$ as defined above \cite{Kahane, DS11, RV_Review, Shamov, Berestycki_Elementary}.

It is by now well known that the behaviour of the random measure \eqref{e:GMCformal} exhibits a phase transition at 
\begin{equation*}
\gammac \eqdef \sqrt{\smash[b]{2d}} \;.
\end{equation*}
Following standard convention, we call the regime $\gamma < \gammac$ \emph{subcritical}, the borderline case $\gamma = \gammac$ \emph{critical}, and the range $\gamma > \gammac$ \emph{supercritical}. These three different regimes differ in the normalisation needed to obtain a non-trivial limiting measure, as well as in the qualitative features of the limiting measure (see also \cite{Lacoin} for a more detailed phase diagram 
including complex values of $\gamma$). 
Notably, in the supercritical regime, the limiting random measure is \emph{not} measurable with respect to the underlying field $\rmX$ and is 
purely \emph{atomic}. Before delving into more details and stating our main results, we introduce the family of log-correlated fields we will be working with.

\paragraph{The class of $\star$-scale invariant fields.} 
We consider log-correlated Gaussian fields $\rmX$ on $\R^d$ with short-range correlations, which naturally admit an approximation by a martingale $(\rmX_t)_{t \geq 0}$. Here, each $\rmX_t$ is a smooth 
Gaussian field, and for every $x \in \R^d$, the process $(\rmX_t(x))_{t \geq 0}$ is a standard Brownian motion. 
Moreover, these fields satisfy a certain type of scale-invariance called \emph{$\star$-scale invariance} \cite[Section~2.3]{RV_Review}. In a nutshell, it states that for any $s$, $t>0$, the fields $\rmX_t$ and $\rmX_{t+s} - \rmX_t$ are independent, and that the latter is equal in law to the field $\rmX_s$, spatially rescaled by a factor $e^t$.

The key ingredient in constructing a $\star$-scale invariant field is the so-called \emph{seed covariance function} $\frkK:\R^d \to \R$, which we assume satisfies the following properties:
\begin{enumerate}[start=1,label={{{(K\arabic*})}}]
\item \label{hp_K1} $\frkK$ is positive definite, radial, and $\frkK(0) = 1$.
\item \label{hp_K2} $\frkK \in \CC^{\infty}(\R^d)$ and it is supported in $B(0,1)$.
\end{enumerate}

\begin{remark}
The unit ball appearing in \ref{hp_K2} can of course be replaced by any compact subset of $\R^d$.
\end{remark}

We write $\bar\frkK:\R^d \to \R$ for the (unique) positive definite function such that the convolution of $\bar \frkK$ with itself equals $\frkK$. 

\begin{remark}\label{rem:strict}
As a consequence of \ref{hp_K1} \dash \ref{hp_K2}, the (inverse) Fourier transform of $\frkK$ given by 
$\smash{\hat \frkK(\omega) \eqdef (2 \pi)^{-d} \int_{\R^d} \frkK(x) e^{i \omega \cdot x} dx}$ is a 
probability measure with a smooth density and admitting moments of all orders. 
In particular, the Hessian $D^2 \frkK(0)$ at the
origin is strictly negative definite in the sense that there exists 
$\delta > 0$ such that, for any $v \in \R^d$, one has 
$\scal{v,D^2 \frkK(0) v} \le -\delta |v|^2$.
\end{remark}

\begin{definition}
\label{def:fields}
Let $\frkK : \R^d \to \R$ be a function satisfying assumptions \ref{hp_K1} \dash \ref{hp_K2}, and let $\bar\frkK:\R^d \to \R$ denote the (unique) positive definite function whose convolution with itself equals $\frkK$.
For $\xi$ a space-time white noise on $\R^d \times (0, \infty)$, we define the \emph{$\star$-scale invariant field with seed covariance $\frkK$} by  
\begin{equation}
\label{eq:field}
\rmX(\cdot) \eqdef \int_{\R^d} \int_0^{\infty} \bar\frkK\bigl(e^{r}(y - \cdot)\bigr) e^{\f{rd}{2}} \xi(dy, dr) \;.
\end{equation}
Furthermore, for $0 \leq s < t$, we let $\rmX_{s, t}$ be the field on $\R^d$ given by
\begin{equation}
\label{eq:fieldST}
\rmX_{s, t}(\cdot) \eqdef \int_{\R^d} \int_{s}^{t} \bar\frkK\bigl(e^{r}(y - \cdot)\bigr) e^{\f{rd}{2}} \xi(dy, dr) \;,	
\end{equation}
with the notational convention that $\rmX_{0, t} = \rmX_t$. 
\end{definition}
For $t \geq 0$, the fields $\rmX$ and $\rmX_t$ have the following covariance structures, for all $x$, $y \in \R^d$,
\begin{equation}
\label{eq:covs}
\E\bigl[\rmX(x) \rmX(y)\bigr] = \int_0^{\infty} \frkK\bigl(e^{r}(x-y)\bigr) dr \;, \qquad \E\bigl[\rmX_t(x) \rmX_t(y)\bigr] = \int_0^{t} \frkK\bigl(e^{r}(x-y)\bigr) dr \;.
\end{equation}
Clearly $\E[\rmX(x)^2] = \infty$, so $\rmX$ can only be realised as a random Schwartz distribution.
The collection of fields $(\rmX_t)_{t \geq 0}$ is called the \emph{martingale approximation} of $\rmX$. Indeed, 
by construction, $\smash{(\rmX_t)_{t \geq 0}}$ is a martingale for the filtration $\smash{(\CF_t)_{t \geq 0}}$ given by 
\begin{equation}
\label{eq:defSigmaFieldT}
\CF_t \eqdef \sigma\bigl(\rmX_{s} \, : \, s \in [0, t)\bigr) \;.
\end{equation}
Moreover, as $t \to \infty$ the field $\rmX_t$ converges almost surely to a limit $\rmX$ in the Sobolev space $\CH_{\loc}^{-\kappa}(\R^d)$ for any $\kappa > 0$. We refer to \cite[Proposition~4.1~(iv)]{Junnila_deco} for a proof of this fact.

\paragraph{The three phases of GMC measures.} 
As previously mentioned, GMC measures exhibit three distinct phases depending on the value of the parameter $\gamma > 0$ in \eqref{e:GMCformal}. Each phase is characterised by the specific form of renormalisation required to obtain a nontrivial limiting measure.

In the \emph{subcritical regime}, i.e., when $\gamma \in (0, \gammac)$, the sequence of random measures
\begin{equation}
\label{e:def_GMC}
\mu_{\gamma, t}(dx) = e^{\gamma \rmX_t(x) - \frac{\gamma^2}{2}t} dx
\end{equation}
converges weakly in probability to a limiting positive random measure $\mu_{\gamma}$ as $t \to \infty$ \cite{Kahane, RV_Review, Shamov, Berestycki_Elementary}, which is almost surely nontrivial. It is well known that $\mu_{\gamma}$ is almost surely non-atomic, but singular with respect to Lebesgue measure. Many further properties of these measures concerning, among others, moments and multifractal behaviour are known \cite{RV_Review, BerMulti}. An important feature of the measure $\mu_{\gamma}$ is that it is carried by the set of $\gamma$-thick points. Intuitively, a thick point is a point where the field takes an unusually large value: one where it is of the order of its variance instead of the order of its standard deviation. 

When $\gamma \geq \gammac$ a phase transition occurs and if one considers the sequence of measures defined in \eqref{e:def_GMC}, then for any compact subset $\rmA \subset \R^d$, it holds that $\smash{\mu_{\gamma, t}(\rmA) \to 0}$ in probability as $t \to \infty$. Therefore, in order to define a nontrivial limiting measure at the critical threshold $\gammac$, one needs to give the sequence of approximating measures an extra ``push'' in the right direction.
More precisely, in the \emph{critical regime}, i.e., when $\gamma = \gammac$, the sequence of random measures
\begin{equation}
\label{e:def_crtical_GMC_Der}
\mu_{\gammac, t}(dx) = \bigl(-\rmX_{t}(x) + \gammac t \bigr)e^{\gammac \rmX_t(x) - \frac{\gammac^2}{2} t} dx
\end{equation}
converges weakly in probability as $t \to \infty$ to a limiting positive random measure $\smash{\mu_{\gammac}}$ \cite{Critical_der, Criticial_KPZ, Powell_Critical}, which is non-atomic and has full support. The normalisation used in \eqref{e:def_crtical_GMC_Der} is known as the ``derivative normalisation'' since it can be obtained by evaluating at $\gamma = \gammac$ the derivative with respect to $\gamma$ of the expression on the right-hand side of \eqref{e:def_GMC}. 

\begin{remark}
There is also an equivalent deterministic normalisation, called ``Seneta--Heyde normalisation'', which produces the same limiting measure $\smash{\mu_{\gammac}}$ up to a deterministic multiplicative constant,
namely
\begin{equation}
\label{e:def_crtical_GMC_SH}
\mu_{\gammac, t}^{\mathrm{SH}}(dx) = \sqrt{t} e^{\gammac \rmX_t(x) - \frac{\gammac^2}{2}t} dx
\end{equation}
converges weakly in probability to $\smash{\sqrt{\smash[b]{2/\pi}} \, \mu_{\gammac}}$ as $t \to \infty$ \cite{Junnila_deco,Powell_Critical}. In the present article, unless otherwise stated, we always refer to the critical GMC as the one obtained using the derivative normalisation \eqref{e:def_crtical_GMC_Der}.
\end{remark}

 In the low temperature or \emph{supercritical regime}, i.e.\ $\gamma > \gammac$, the GMC exhibits atomic 
behaviour under suitable renormalisation, with the locations and masses of the atoms dictated by the extremal 
statistics or near maximum values of the fields $\rmX_t$. In the continuum, the only available mathematical 
result we are aware of is \cite{Glassy}, where Madaule, Rhodes, and Vargas show that for $\gamma > \gammac$, the sequence of random measures
\begin{equation}
\label{e:norm_super}
\mu_{\gamma, t}(dx) = t^{\f{3\gamma}{2 \sqrt{\smash[b]{2d}}}} e^{t(\gamma/\sqrt{\smash[b]{2}} - \sqrt{\smash[b]{d}})^2} e^{\gamma \rmX_t(x) - \frac{\gamma^2}{2}t} dx\;, 
\end{equation}
converges weakly in law as $t \to \infty$ to a nontrivial purely atomic limiting measure $\mu_\gamma$ whose law was previously conjectured in \cite{Critical_der} and can be characterised explicitly in terms of the law of the critical GMC $\smash{\mu_{\gammac}}$. In order to describe this limiting measure, it is convenient to introduce the following notation which will be used throughout the paper.

\begin{definition}
\label{def:PPP}
For $\gamma > \gammac$ and a non-negative, locally finite Borel measure $\nu$ on $\R^d$, we let $\eta_{\gamma}[\nu]$ be the Poisson point measure on $\R^d \times (0, \infty)$ with intensity measure given by $\smash{\nu(dx) \otimes z^{-(1+\gammac/\gamma)} dz}$. We also define the integrated atomic random measure with parameter $\gamma$ and spatial intensity $\nu$ as the random purely atomic measure $\CP_{\gamma}[\nu]$ on $\R^d$ given by
\begin{equation*}
\CP_{\gamma}[\nu](dx) \eqdef \int_{0}^{\infty} z \, \eta_{\gamma}[\nu](dx, dz) \;. 
\end{equation*}
\end{definition}

With this notation in place, we can state the main result of \cite{Glassy} more precisely.
\begin{theorem}[{\cite[Theorem~2.2]{Glassy}}]
\label{th:GlassyEx}
For any $\smash{\gamma > \gammac}$ there exists a constant $c_{\gamma} > 0$ such that the sequence of random measures $\smash{(\mu_{\gamma, t})_{t > 0}}$ defined in \eqref{e:norm_super} converges in law, with respect to the vague topology, to $\smash{c_{\gamma}\CP_{\gamma}[\mu_{\gammac}]}$ as $t \to \infty$, where $\mu_{\gammac}$ is the critical GMC.	
\end{theorem}

%%%%%%%%%%%%%%%%%%%%%%%%%%%%%%%%%%%%%%%%%%%%%%
\subsection{Overview of the main results}
\label{sub:ov_main_results}
Building on Theorem~\ref{th:GlassyEx}, the main goal of the present article is to gain a deeper understanding of the convergence behaviour of supercritical GMC measures as $t \to \infty$. A natural approach is as follows: instead of simply taking the limit as $t \to\infty$ of $\mu_{\gamma, t}$ defined in \eqref{e:norm_super}, we consider the measure-valued stochastic processes $(\mu_{\gamma, t + s})_{s \geq 0}$ as $t \to \infty$. As we will see below, this procedure yields a limiting stochastic process $(\nu_{\gamma, s})_{s \geq 0}$, and our goal is to investigate its nature\footnote{If the convergence in Theorem~\ref{th:GlassyEx} was in probability, then the process $\nu_{\gamma, s}$ would necessarily be constant in $s$.}. 

\begin{remark}
A helpful way to interpret the role of the process $(\nu_{\gamma, s})_{s \geq 0}$ is through an analogy with the CLT.
Consider a collection $(\rmX_n)_{n \in \N}$ of i.i.d.\ centred random variables with unit variance, and let $\smash{\rmS_n \eqdef \sum_{k = 1}^{n} \rmX_k}$. The CLT tells us that $n^{-1/2} \rmS_n$ converges in law to a standard normal random variable as $n \to \infty$.
 On the linear time scale, the normalised sums quickly settle into a ``stable'' distribution. However, if we switch to a logarithmic time scale, we can capture how the marginals at different times evolve together. More precisely, for every $t \geq 0$, define $\smash{\rmY_t \eqdef \lfloor e^t \rfloor^{-1/2} \rmS_{\lfloor e^t \rfloor}}$. Then, the process $\smash{(\rmY_{t+s})_{s \geq 0}}$ converges in the finite-dimensional sense to a stationary Ornstein--Uhlenbeck process $\smash{(\rmU_s)_{s \geq 0}}$ as $t \to \infty$. Intuitively,  this follows since, for fixed $s \geq 0$, it holds that 
\begin{equation*}
\rmY_{t+s} = \frac{\rmS_{\lfloor e^t \rfloor}}{\sqrt{\lfloor e^{t+s} \rfloor}} + \frac{\rmS_{\lfloor e^{t+s} \rfloor} - \rmS_{\lfloor e^t \rfloor}}{\sqrt{\lfloor e^{t+s} \rfloor}}	\approx e^{-s/2} \rmY_t + (1-e^{-s})^{1/2} \rmY'_t \;,
\end{equation*}
where $\rmY'_t$ is an independent copy of $\rmY_t$. In particular, by taking the limit as $t \to \infty$ in the display above, we obtain that for all $s \geq 0$,
\begin{equation*}
\rmU_s = e^{-s/2} \rmU_0 + (1-e^{-s})^{1/2} \rmZ \;,
\end{equation*}
where $\rmU_0$ and $\rmZ$ are independent standard normal random variables. This is precisely the defining property of a stationary Ornstein--Uhlenbeck process.
\end{remark}

Returning to our setting, we show in Corollary~\ref{co:processMeasures} that there exists an $\R$-valued process $(\bfW_{\gamma, s})_{s \geq 0}$ with $\bfW_{\gamma, 0}=0$, which we refer to as the \emph{weight process}, such that, for any $s \geq 0$, the measure $\nu_{\gamma, s}$ can be expressed as follows
\begin{equation*}
	\nu_{\gamma, s} = \sum_{j \in \N} e^{\bfW_{\gamma, s, j}} w_j \delta_{x_j} \;,
\end{equation*}
where $(x_j, w_j)_{j \in \N}$ enumerates (in an arbitrary manner) the atoms of the Poisson point measure $\eta_{\gamma}[\mu_{\gammac}]$ as introduced in Definition~\ref{def:PPP}, and the collection $(\bfW_{\gamma, s, j})_{j \in \N}$ consists of i.i.d.\ copies of $\bfW_{\gamma, s}$. We emphasise that the spatial locations of the point masses are fixed once and for all and do not change as $s$ varies. The only aspect that evolves in the process $\nu_{\gamma, s}$ are the weights of the point masses, whose dynamics are governed by the \emph{weight process} as described above. Moreover, the mass of each atom evolves independently 
of all the others.

To prove this result, we adopt a general framework where we consider the \emph{joint} limit as $t \to \infty$ of the family of measures $(\mu_{\gamma, t, i})_{i \in [n]}$ defined as follows
\begin{equation*}
\mu_{\gamma, t, i}(dx) \eqdef t^{\f{3\gamma}{2 \sqrt{\smash[b]{2d}}}} e^{t(\gamma/\sqrt{\smash[b]{2}} - \sqrt{\smash[b]{d}})^2} e^{\gamma (\rmX_t(x) + \rmW_{i, t}(x)) - \frac{\gamma^2}{2}t} dx = e^{\gamma \rmW_{i, t}(x)} \mu_{\gamma, t}(dx)\;, 
\end{equation*}
where, the collection of processes $(\rmW_{i, \cdot})_{i \in [n]}$\footnote{Here, and in what follows, we write $[n] = \{1, \ldots, n\}$ and $[n]_0 = \{0, 1, \ldots, n\}$.} satisfies some suitable assumptions (see Assumption~\ref{as:FieldsW} for details). Specifically, in Theorem~\ref{th:stableConv}, we analyse the joint convergence of these types of measures.

\begin{remark}
It is worth noting that for $s \geq 0$, the convergence of the collection of measures $(\mu_{\gamma, t+s})_{t \geq 0}$ naturally fits within this general framework. This follows directly from the decomposition $\rmX_{t+s} = \rmX_t + \rmX_{t, t+s}$, together with the fact that the field $\rmX_{t, t+s}$ satisfies our assumptions.
\end{remark}

\begin{remark}
\label{rm:companionPaper}
One motivation for working within this general framework is that it facilitates our companion paper \cite{BHUniq}, where we establish the uniqueness of the supercritical GMC measure. Specifically, we show that if $\rmX_{\ceps}$ denotes the convolution approximation of a $\star$-scale invariant field at level $\eps$, then, roughly speaking, $\rmX_{(e^{-t})} \approx \rmX_{t} + \rmW(e^t \cdot)$ for a smooth Gaussian field $\rmW$ that is independent of $\rmX$ (see \cite[Proposition~B]{BHUniq} for details).
\end{remark}

The proof of convergence of the process $(\mu_{\gamma, t+s})_{s \geq 0}$ as $t \to \infty$, in addition to relying on technical results from \cite{Madaule_Max}, requires a fine understanding of the structure of $\star$-scale invariant fields around their extremal points. 
Specifically, for $t \gg 0$, let $\rmX_t$ be the martingale approximation of a $\star$-scale invariant field $\rmX$, and condition it on two events: first, that its value at the origin is comparable to the maximum of $\rmX_t$ within an order-one region; second, that the origin is a ``mesoscopic maximum'' within this region, i.e., a local maximum that is global within a ball of radius much smaller than $1$, but much larger 
than $e^{-t}$.

Then, after shifting the coordinate system at this maximum (so that the value at the origin becomes $0$) and after a suitable rescaling, we want to study the ``shape of the field'' in that region. This is the content of Theorem~\ref{th:cluster}, which describes the law of the field governing the shape of the field $\rmX_t$ around a mesoscopic maximum.

We emphasise that in the case of dimension $d = 2$, a similar investigation has been conducted for the discrete Gaussian free field (DGFF) in \cite{BiskupLouidor}. However, their results rely heavily on the fact that the DGFF is defined on the discrete grid $\Z^2$. 
There, one can condition a field on the origin being the maximum in a bounded region and then take the limit as the size of the region tends to infinity in order to obtain
the ``cluster process'' which describes the shape of the field around a mesoscopic
maximum. In the continuum however, conditioning the origin to be a maximum even within a bounded region becomes problematic, as such a conditioning is already degenerate. 

To address this issue, we introduce a ``softer'' form of conditioning: we fix an arbitrary 
threshold $\lambda > 0$ and, instead of conditioning on the origin being a mesoscopic 
maximum, we condition on the value at the origin being at least as large as the nearest 
mesoscopic maximum minus $\lambda$. This conditioning is non-degenerate within a bounded region as long as $\lambda > 0$. Theorem~\ref{th:cluster} proves that we can then take the limit as the region size increases to infinity, thus yielding a limiting field $\smash{\tilde\Upsilon_{\! \lambda}}$.
Finally, we show the existence of a unique random field $\Psi$ on $\R^d$ which is $0$ at the origin, takes only negative values, is independent of the arbitrary threshold $\lambda$, and such that $\smash{\tilde \Upsilon_{\! \lambda}}$ can be expressed as a randomly shifted version of $\Psi$ (under a suitably tilted measure, see \eqref{e:tilt3} for details). In this sense, we can consider $\Psi$ as the canonical field describing the shape of the field $\rmX$ from the perspective of a mesoscopic maximum.

\begin{acknowledgements}
{\small This work grew out of discussions between MH and Christophe Garban at the 2023 SwissMAP Workshop in Mathematical Physics. 
We are also grateful to Michael Aizenman and R\'emi Rhodes for interesting discussions on this topic.
Both authors were supported by the Royal Society through MH's Research Professorship RP\textbackslash R1\textbackslash 191065.}
\end{acknowledgements}

%%%%%%%%%%%%%%%%%%%%%%%%%%%%%%%%%%%%%%%%%%%%%%
%%%%%%%%%%%%%%%%%%%%%%%%%%%%%%%%%%%%%%%%%%%%%%
\section{Main results}
\label{sec:main_results}

We now provide a description of our main results. First, in Section~\ref{sub:main_local}, we discuss the shape of a $\star$-scale invariant field around a mesoscopic maximum. Following that, in Section~\ref{sub:main_conv}, we state the result concerning the convergence of the supercritical GMC. In Section~\ref{sub:measureValuedPr}, we present the result regarding the convergence of the measure-valued processes $(\mu_{\gamma, t + s})_{s \geq 0}$ as $t \to \infty$. 

%%%%%%%%%%%%%%%%%%%%%%%%%%%%%%%%%%%%%%%%%%%%%%
\subsection{Local structure of extremal points}
\label{sub:main_local}
%\federico{In what follows, I've introduced two scales, $0 \ll b \ll t$, so that hopefully the discussion is clearer.}
We aim to investigate the local structure around points within the domain where the field attains unusually large values, comparable to its maximum. The strong correlation with nearby points suggests that each peak in the field comes with a cluster of high values. These clusters of high-value points are generally well-separated from each other. By selecting one of these clusters and identifying as reference point the maximum of the field inside the cluster, our goal is to describe the ``shape'' of the field in the vicinity of this reference point.

\begin{remark}
We will see that the behaviour of $\rmX_t$ near a mesoscopic maximum is dominated by its suprema
over spherical shells. These in turn behave like a Brownian motion, with time given by the 
logarithm of the radius of the shell. It is useful to keep this analogy in mind when parsing the 
results in this section.
\end{remark}

In order to make this heuristic precise, we need to introduce some notation. 
\begin{definition}
\label{def:harmFcts}
For $t \in (0, \infty)$, we introduce the \emph{recentering constant} $\frkm_t$ by letting
\begin{equation}
\label{eq:recentering}
\frkm_t \eqdef \sqrt{\smash[b]{2d}} t - \frac{3}{2 \sqrt{\smash[b]{2d}}} \log t \;.
\end{equation}
Furthermore, we define functions $\frkh_t :\R^d \to \R$ and $\frka_{t}: \R^d \to \R$ by
\begin{equation}
\label{eq:frkgb}
\frkh_t(x) \eqdef \frac{1}{t} \int_0^t \frkK(e^{-s} x) ds \;, \qquad \frka_{t}(x) \eqdef \int_0^{t} \bigl(1-\frkK(e^{-s}x)\bigr) ds \;,
\end{equation}
with the definition of $\frka_t$ extended also to the case $t = \infty$. 
\end{definition}

In \cite{Madaule_Max}, Madaule proved that there exists a constant $c > 0$ such that the following convergence in law holds as $t \to \infty$,
\begin{equation*}
\sup_{x \in B(0, 1)} \rmX_t(x) - \frkm_t - c \Rightarrow \rmG + \log \mu_{\gammac}\bigl(B(0, 1)\bigr) \;,
\end{equation*}
where $\rmG$ is an independent random variable with standard Gumbel distribution, and $\mu_{\gammac}$ is the critical GMC.
Hence, if we aim to describe the shape of the field around a point where the value of the field is comparable to its maximum, it seems natural to condition the field $\rmX_t$ on achieving the value $\frkm_t + z$ at the origin, for some fixed $z \in \R$, while simultaneously requiring that the origin is a mesoscopic maximum. We first zoom in around the origin by introducing the rescaled field
\begin{equation*}
\bar\rmX_t(\cdot) \eqdef \rmX_t(e^{-t} \cdot) \;. 	
\end{equation*}
We now introduce several Gaussian fields that will play an important role in our analysis, and whose significance will become clear in the following paragraphs.
\begin{definition}
\label{def:Fields}
For $t \in \R^{+} \cup \{\infty\}$, we let $\Phi_t$ denote the centred Gaussian field on $\R^d$ with the same law as the field $\bar\rmX_t - \bar\rmX_t(0)$. In particular, for all $x$, $y \in \R^d$,
\begin{equation*}
	\E\bigl[\Phi_t(x) \Phi_t(y)\bigr] = \frka_{t}(x) + \frka_{t}(y) - \frka_{t}(x-y) \;.
\end{equation*}
Moreover, for $t \in \R^{+}  \cup \{\infty\}$, we let $\Upsilon_{\! t}$ be the Gaussian field on $\R^d$ given by
\begin{equation*}
\Upsilon_{\! t}(\cdot)  \eqdef \Phi_{t}(\cdot) - \sqrt{\smash[b]{2d}} \frka_t(\cdot) \;.
\end{equation*}
\end{definition}

\begin{remark}
As one can easily check, the covariance of the field $\Phi_{\infty}$ resembles very much the covariance structure of the DGFF on $\Z^2$ pinned to zero at zero (see \cite[Equation~(2.7)]{BiskupLouidor}). Indeed, in this setting, the covariance takes the same form as for the field $\Phi_{\infty}$ with $\frka_{\infty} : \Z^2 \to \R$ given by the potential kernel of the simple symmetric random walk started from zero (see \cite[Equation~(2.8)]{BiskupLouidor}). 
\end{remark}

Now, going back to our previous discussion, a straightforward calculation shows that the field $\bar \rmX_t$ conditioned to take the value $\frkm_t + z$ at the origin has the same law as the shifted field
\begin{equation*}
\bar \rmX_t(\cdot) - \frkh_t(\cdot) \bigl(\bar \rmX_t(0) - (\frkm_t + z)\bigr)\;.
\end{equation*}
In particular, for $0 \ll b \ll t$, and for every $x$, $y \in B(0, e^b)$, we note that, for $t \gg 0$ large enough, it holds that
\begin{equation*}
\E\bigl[\bigl(\bar \rmX_t(x) - \frkh_t(x) \bar \rmX_t(0)\bigr) \bigl(\bar \rmX_t(y) - \frkh_t(y) \bar \rmX_t(0)\bigr)\bigr] \approx \frka_t(x) + \frka_t(y) - \frka_t(x-y) \;,
\end{equation*}
and also
\begin{equation*}
(\frkm_t + z) - (\frkm_t + z) \frkh_t(x) \approx \sqrt{\smash[b]{2d}} \frka_t(x) \;.
\end{equation*}
Therefore, the preceding computations imply that, as $t \to \infty$, the following convergence in law holds when restricted to the ball $B(0, e^b)$:
\begin{equation*}
 \bigl(\bar \rmX_t(\cdot) - \frkh_t(\cdot) \bigl(\bar \rmX_t(0) - (\frkm_t + z)\bigr)\bigr) - (\frkm_t + z) \Rightarrow  \Phi_{\infty}(\cdot) - \sqrt{\smash[b]{2d}}\frka_{\infty}(\cdot) = \Upsilon_{\! \infty}(\cdot) \;.
\end{equation*}

These considerations suggest that the shape of the field $\rmX$ near a mesoscopic maximum is given by the limit in law, as $b \to \infty$, of the field $\Upsilon_{\! \infty}$ conditioned to be non-positive on the ball $B(0, e^b)$.

Before proving that this limit indeed exists, we introduce some additional notation. For $k \geq 0$, we set
\begin{equation}
\label{eq:defCRk}
\CR_k \eqdef \bigl\{\bfF : \CC(\R^d) \to \R \, : \, \bfF(\phi) = \bfF(\psi) \;\text{whenever}\; \phi|_{B(0, k)} = \psi|_{B(0, k)}\bigr\} \;.
\end{equation}
In other words, $\CR_k$ is the set of (measurable) mappings from $\CC(\R^d)$ to $\R$ that depend on the values of the input function only inside $B(0, k)$. Furthermore, with a slight abuse of the usual notation, we define
\begin{equation}
\label{eq:defCCbloc}
\CC^b_{\loc}(\CC(\R^d)) \eqdef \bigcup_{k \geq 0} \CR_k \cap \CC^b(\CC(\R^d)) \;,
\end{equation}
where $\CC^b$ denotes continuous bounded functions.
We can now state our first main result, where we write $\M_{0, b}(f)$ as a shorthand for $\smash{\sup_{\abs{x} \leq e^b} f(x)}$.
\begin{theoremA}
\label{th:cluster}
For each $\lambda > 0$, there exists a continuous random field $\tilde{\Upsilon}_{\!\lambda}$ on $\R^d$ such that, for any function $\bfF \in \CC^b_{\loc}(\CC(\R^d))$, one has
\begin{equation}
\label{eq:weakLimit}
\E\bigl[\bfF\bigl(\tilde{\Upsilon}_{\!\lambda} \bigr)\bigr] 
=
\lim_{b \to \infty} \E\big[\bfF\bigl(\Upsilon_{\! b}\bigr) \, \big| \, \M_{0, b}(\Upsilon_{\! b}) \leq \lambda \bigr] 
=
\lim_{b \to \infty} \E\big[\bfF\bigl(\Upsilon_{\! \infty}\bigr) \, \big| \, \M_{0, b}(\Upsilon_{\! \infty}) \leq \lambda \bigr] 
 \;.
\end{equation}
\end{theoremA}
We emphasise that the existence of the weak limits in \eqref{eq:weakLimit} is part of the statement.
We also observe that the conditioning on the right-hand side of \eqref{eq:weakLimit} is singular as $b \to \infty$. More precisely, by letting
\begin{equation*}
\alpha \eqdef \sqrt{\smash[b]{2/\pi}} \;,
\end{equation*}
we have the following result.
\begin{theoremA}
\label{th:clusterProb}
For each $\lambda > 0$, there exists a constant $c_{\star, \lambda} > 0 $, such that
\begin{equation*}
\lim_{b\to \infty} \sqrt{b} \, \P\bigl(\M_{0, b}(\Upsilon_{\! b})  \leq \lambda \bigr)
= \lim_{b\to \infty} \sqrt{b} \, \P\bigl(\M_{0, b}(\Upsilon_{\! \infty})  \leq \lambda \bigr)  
= \alpha \, c_{\star, \lambda} \;. 
\end{equation*}
\end{theoremA}

\begin{remark}\label{rem:decomposition}
In Section~\ref{sec:cluster}, we introduce a characterisation of the field $\Phi_b$ in terms of a stochastic integral driven by a one-dimensional Brownian motion $(B_t)_{t \geq 0}$ (as a cartoon, think of this as being a smoothened out version of the field $x \mapsto -B_{\log |x|}$), plus an independent centred Gaussian field, see \eqref{eq:Phi}. In particular, we will prove a stronger version of Theorems~\ref{th:cluster}~and~\ref{th:clusterProb}, where we also allow conditioning on the value of $B$ at time $b$. We refer to Propositions~\ref{pr:AsyConv}~and~\ref{pr:AsyProb} for the precise statements.
\end{remark}

\begin{remark}
Although we don't have an explicit representation for the constant $c_{\star, \lambda}$, we will see in \eqref{eq:defcstar} below that it is given by
\begin{equation}
\label{eq:defCStar}
c_{\star, \lambda} = \lim_{k\to \infty} \E \bigl[B_k \one_{\{B_k \in [k^{1/6}, k^{5/6}]\}} \one_{\{\M_{0, k}(\Upsilon_{\! k})  \leq \lambda\}}\bigr]\;,
\end{equation}
with $B$ and $\Upsilon_{\! k}$ related as in Remark~\ref{rem:decomposition}.
The exponents $1/6$ and $5/6$ appearing here are of course unimportant and could probably be replaced by any
values in $(0,1/2)$ and $(1/2,\infty)$ respectively. We refer to Lemma~\ref{lm:boundXione} for a proof of the fact that $c_{\star, \lambda} \in (0, \infty)$.
\end{remark}

We recall that the introduction of a threshold $\lambda > 0$ in Theorems~\ref{th:cluster} and~\ref{th:clusterProb} is necessary due to the continuous setting in which we are working. In such a context, conditioning on the event that a field, which is zero at the origin, remains negative is ill-posed. However, it is desirable to define a ``canonical'' field that captures the local structure of $\rmX$ around an extremal point, without being arbitrarily dependent on $\lambda$.

To achieve this, we introduce below a field $\smash{\Psi_{\lambda}}$, which is essentially just $\smash{\tilde{\Upsilon}_{\! \lambda}}$ shifted to move its maximum to the origin, but under a slightly tilted law. At first glance, it may seem contradictory to define a field that we claim is independent of $\lambda$ while still denoting it as $\Psi_{\lambda}$. This notation arises because, from its definition, it is not immediately evident that $\Psi_{\lambda}$ is indeed independent of the threshold $\lambda$. However, this independence (albeit in a slightly weaker sense) will be established a posteriori (see Proposition~\ref{pr:PsiIndepLambda}).

Before defining the field $\Psi_{\lambda}$, for all $x \in \R^d$, we introduce the shift operator $\tau_x: \CC(\R^d) \to \CC(\R^d)$ by
\begin{equation}
\label{eq:shiftOperator}
\tau_x f(\cdot) = f(\cdot + x) - f(x)\;, \qquad \forall \, f \in \CC(\R^d) \;.
\end{equation}
Furthermore, we let $\Lambda \subseteq \R^{+}$ denote the uncountable set introduced in Lemma~\ref{lm:Countable} below. Roughly speaking, $\Lambda$ consists of the ``good thresholds'' $\lambda$ for which the law of the field $\smash{\tilde{\Upsilon}_{\! \lambda}}$ exhibits some desirable properties.  
More precisely, we need to exclude the ``bad'' values of $\lambda$ such that 
$\smash{\P(\sup_{x \in \R^d} \tilde \Upsilon_{\!\lambda}(x) = \lambda) > 0}$ or $\smash{\P(\abs{\{y \in \R^d \, : \, \tilde \Upsilon_{\! \lambda}(y) = \sup_{x \in \R^d}\tilde{\Upsilon}_{\! \lambda}(x) -\lambda\}} > 0) > 0}$ where $\abs{\cdot}$ denotes Lebesgue measure.
We emphasise that we don't expect any such bad values to exist, i.e., we expect that $\Lambda = \R^{+}$. However, since the field $\smash{\tilde{\Upsilon}_{\! \lambda}}$ is itself defined by a singular conditioning, proving this fact would require additional effort. Since our main result does not require ruling out the existence of bad values, we will not investigate this fact further. 
In any case, we have to  exclude at most countably many points, so $\Lambda$ is dense.

\begin{definition}
\label{def:recoverPhi}
For $\lambda \in \Lambda$, we let $\Psi_{\lambda}$ be the field on $\R^d$ uniquely characterised by the fact that, for all $\bfF \in \CC^b(\CC(\R^d))$,
\begin{equation}
\label{e:recoverPhi}
\E\bigl[\bfF(\Psi_{\lambda})\bigr] \propto \E\left[\frac{\bfF\bigl(\tau_{x_{\star}} \! \tilde \Upsilon_{\! \lambda}\bigr) e^{\sqrt{\smash[b]{2d}}\tilde\Upsilon_{\! \lambda}(x_{\star})}}{\int_{\R^d} e^{\sqrt{\smash[b]{2d}}\tilde\Upsilon_{\! \lambda}(x)} \one_{\{\tilde \Upsilon_{\! \lambda}(x) \geq \tilde \Upsilon_{\! \lambda}(x_{\star}) -\lambda\}} dx}\right] \;,
\end{equation} 
where $\smash{x_{\star} = \argmax\{\tilde{\Upsilon}_{\! \lambda}(x) \, : \, x \in \R^d\}}$ and the proportionality constant is chosen in such a way that $\E[1] = 1$. 
\end{definition}

\begin{remark}
The fact that the proportionality constant in the previous definition lies in $(0, \infty)$ is proved in Lemma~\ref{lm:NorContPsiOK}.
\end{remark}

We now state the following key ``resampling property'' of the field $\smash{\tilde{\Upsilon}_{\! \lambda}}$, whose proof is given in Section~\ref{sub:resamplingProp}.
\begin{proposition}
\label{pr:resempPropUps}
For each $\lambda \in \Lambda$ and for all $\bfF \in \CC^b_{\loc}(\CC(\R^d))$, it holds that
\begin{equation}
\label{e:wantedUpsilon}
\E\bigr[\bfF(\tilde \Upsilon_{\! \lambda})\bigr] = \E\Biggl[\frac{\int_{\R^d} \bfF(\tau_x\tilde\Upsilon_{\! \lambda}) e^{\sqrt{\smash[b]{2d}}\tilde\Upsilon_{\! \lambda}(x)}\one_{\{\tilde\Upsilon_{\! \lambda}(x)\geq \tilde \Upsilon_{\! \lambda}(x_{\star}) - \lambda\}}dx}{\int_{\R^d} e^{\sqrt{\smash[b]{2d}}\tilde\Upsilon_{\! \lambda}(x)}\one_{\{\tilde\Upsilon_{\! \lambda}(x)\ge \tilde \Upsilon_{\! \lambda}(x_{\star}) - \lambda\}}dx}\Biggr] \;,
\end{equation}
where we recall that $\smash{x_{\star} = \argmax\{\tilde{\Upsilon}_{\! \lambda}(x) \, : \, x \in \R^d\}}$.
\end{proposition}

\begin{remark}
In fact, Proposition~\ref{pr:resempPropUps} is quite general and applies to a large class of Gaussian fields
of the form $\Phi - \eta \frka$ where $\Phi$ is centred with $\E[(\Phi(x) - \Phi(y))^2] = \frka(x-y)$. 
In particular, one can replace the 
field $\tilde \Upsilon_{\! \lambda}$ in \eqref{e:wantedUpsilon} by a drifted, conditioned Brownian motion.
Specifically, for $\lambda > 0$ and $\eta \geq 0$, let $\smash{(X^{\lambda, \eta}_t)_{t \in \R}}$ be a two-sided Brownian motion with drift $t \mapsto -\eta |t|$ for $t \in \R$, conditioned to remain below $\lambda$ at all times. 
When $\eta > 0$ this is a condition that happens with positive probability, while the case $\eta = 0$ can be 
covered by a limiting procedure, yielding a two-sided three-dimensional Bessel process. 
Furthermore, consider the set
\begin{equation*}
\rmA_{\lambda, \eta} \eqdef \Biggl\{t \in \R \, : \,  X^{\lambda, \eta}_t \geq \sup_{s \in \R} X^{\lambda, \eta}_s - \lambda\Biggr\} \;,
\end{equation*}
and let $\rho_{\lambda, \eta}$  be the (random) probability measure on $\rmA_{\lambda, \eta}$ defined as follows 
\begin{equation*}
\rho_{\lambda, \eta}(dt) \eqdef \frac{e^{\eta X^{\lambda, \eta}_t}}{\int_{\rmA_{\lambda, \eta}} e^{\eta X^{\lambda, \eta}_s}  ds}	\one_{\{t \in \rmA_{\lambda, \eta}\}}dt \;.
\end{equation*}
Then, if $t_{\star}$ is a point sampled from the probability measure $\rho_{\lambda, \eta}$, we have the identity in law
\begin{equation*}
\bigl(X^{\lambda, \eta}_{t + t_{\star}} - X^{\lambda, \eta}_{t_{\star}}\bigr)_{t \in \R} \eqlaw \bigl(X^{\lambda, \eta}_t\bigr)_{t \in \R} \;,
\end{equation*}
which we were unable to find in the existing literature.
Note that, in particular, when $\eta = 0$, the probability measure $\rho_{\lambda, 0}$ is the uniform measure on $\rmA_{\lambda, 0}$.
\end{remark}

Now, returning to our setting, we observe that, thanks to Proposition~\ref{pr:resempPropUps}, an alternative way to characterise $\Psi_{\lambda}$ is by inverting \eqref{e:recoverPhi}. More precisely, we have the following result, whose proof is given in Section~\ref{sub:resamplingProp}.
\begin{proposition}
\label{pr:inversionPsi}
For any $\lambda \in \Lambda$ and for all $\bfF \in \CC^b_{\loc}(\CC(\R^d))$, it holds that
\begin{equation}
\label{e:tilt3}
\E\bigl[\bfF(\tilde \Upsilon_{\!\lambda})\bigr] \propto \E\Biggl[\int_{\R^d} \bfF(\tau_x \Psi_{\lambda})e^{\sqrt{\smash[b]{2d}}\Psi_{\lambda}(x)}\one_{\{\Psi_{\lambda}(x)\geq - \lambda\}} dx\Biggr]\;,
\end{equation}
where the proportionality constant is chosen in such a way that $\E[1] = 1$.
\end{proposition}

\begin{remark}
We observe that the fact that the proportionality constant in the previous proposition lies in $(0, \infty)$ follows directly from Definition~\ref{def:recoverPhi} of the field $\Psi_{\lambda}$, together with the fact that the proportionality constant appearing in that definition lies in $(0, \infty)$.
\end{remark}

As claimed above, we now state a result that confirms that the field $\Psi_{\lambda}$, introduced in Definition~\ref{def:recoverPhi}, is indeed canonical. Specifically, we have the following result, whose proof is given in Section~\ref{sub:indthreshold}.
\begin{proposition}
\label{pr:PsiIndepLambda}
For any $\lambda_1$, $\lambda_2 \in \Lambda$, one has $\Psi_{\lambda_1} \eqlaw \Psi_{\lambda_2}$.
\end{proposition}

Since the law of $\Psi_{\lambda}$ does not depend on $\lambda$, except possibly for a countable
number of ``bad'' values, we just write $\Psi$ from now on.

\begin{remark}
Since we have just established that \eqref{e:recoverPhi} is (essentially) independent of $\lambda$, it is tempting to try to 
take the limit $\lambda \to 0$. If we believe that $\Psi$ is smooth and has a unique non-degenerate global maximum 
at the origin then, for small $\lambda$, one expects to have $|x_\star| = \CO(\sqrt{\lambda})$ and 
\begin{equ}
\Leb \{x\,:\,\tilde \Upsilon_{\! \lambda}(x) \geq \tilde \Upsilon_{\! \lambda}(x_{\star}) -\lambda\}
\approx \frac{(\pi \lambda)^{d/2}}{\Gamma(\frac{d}{2}+1)} \lvert\det D^2 \tilde\Upsilon_{\! \lambda}(0)\rvert^{-1/2}\;.
\end{equ}
This strongly suggests that one could alternatively obtain $\Psi$ by
\begin{equ}[e:limitBias]
\E \bfF(\Psi) \propto \lim_{\lambda \to 0} \E \Bigl(\bfF(\tilde \Upsilon_{\! \lambda}) \sqrt{\lvert\det D^2 \tilde\Upsilon_{\! \lambda}(0)\rvert}\Bigr)\;.
\end{equ}
We do however not have sufficient control on $\tilde \Upsilon_{\! \lambda}$ for small $\lambda$ to 
be able to justify this identity rigorously.
The factor $\lvert\det D^2 \tilde\Upsilon_{\! \lambda}(0)\rvert^{1/2}$ appearing in \eqref{e:limitBias} can be 
interpreted as compensating for the fact that if we condition on the origin being an ``almost maximum'', then
we introduce a bias towards fields that are nearly flat there.
\end{remark}

%%%%%%%%%%%%%%%%%%%%%%%%%%%%%%%%%%%%%%%%%%%%%%
\subsection{Stable convergence of supercritical GMC}
\label{sub:main_conv}
We now state our main result regarding the convergence of GMC measures.  Before proceeding, we fix for the remainder of this section the set $\Lambda \subseteq \R^{+}$ introduced in Lemma~\ref{lm:Countable} below. We also let $\Psi$ denote the field introduced in Definition~\ref{def:recoverPhi}, which, as noted in Proposition~\ref{pr:PsiIndepLambda}, has a law that does not depend on $\lambda \in \Lambda$.

We begin by introducing the assumptions considered in our next main theorem. We recall that we write $[n]$ as a shortcut for the set $\{1,\ldots,n\}$.
\begin{assumption}
\label{as:FieldsW}
For $n \in \N$, consider a collection of fields $(\rmW_{i, t})_{i \in [n], t \geq 0}$ on $\R^d$ such that: 
\begin{enumerate}[start=1,label={{{(W\arabic*})}}]
	\item \label{hp:W1} For any $t \geq 0$, the collection of fields $(\rmW_{i, t})_{i \in [n]}$ is independent of the $\sigma$-field $\CF_t$ defined in \eqref{eq:defSigmaFieldT}.
	\item \label{hp:W2} There exist stationary fields $(\rmW_i)_{i \in [n]}$ on $\R^d$ such that, for any fixed $t \geq 0$,
\begin{equation*}
\bigl(\rmW_{i, t}(\cdot)\bigr)_{i \in [n]} \eqlaw \bigl(\rmW_{i}(e^t \cdot)\bigr)_{i \in [n]} \;.
\end{equation*} 
\item \label{hp:W3} For $t > 0$ and for all $x$, $y \in \R^d$ such that $\abs{x - y} > e^{-t}$, it holds that 
\begin{equation*}
\bigl(\rmW_{i, t}(x)\bigr)_{i \in [n]} \perp \bigl(\rmW_{i, t}(y)\bigr)_{i \in [n]} \;. 
\end{equation*}
\item \label{hp:W4} For all $\gamma > \sqrt{\smash[b]{2d}}$, it holds that $\sup_{x \in \R^d} \sum_{i = 1}^n \E[e^{\gamma \rmW_i(x)}] < \infty$.
\end{enumerate}
\end{assumption}

\begin{remark}
 We emphasise that the collection of fields $(\rmW_{i, t})_{i \in [n], t \geq 0}$ satisfying Assumption~\ref{as:FieldsW} is quite arbitrary. 
As briefly mentioned in Remark~\ref{rm:companionPaper}, the reason for working at this level of generality is that it allows us to study both the convergence, as $t \to \infty$, of the process $(\mu_{\gamma, t+s})_{s \geq 0}$, and the uniqueness of the supercritical GMC measure established in the companion paper~\cite{BHUniq}.
\end{remark}

For $\lambda \in \Lambda$, recalling \eqref{eq:defCStar}, we define the constant
\begin{equation}
\label{eq:defAStar}
a_{\star} \eqdef \frac{\alpha \, c_{\star, \lambda}}{\gamma \, \E\bigl[\int_{\R^d} e^{\sqrt{\smash[b]{2d}}\Psi(x)}\one_{\{\Psi(x)\geq - \lambda\}}dx\bigr]} \in (0, \infty) \;.
\end{equation}
The subscript $\lambda$ is not included in the notation $a_{\star}$ since it turns out that the 
right-hand side of \eqref{eq:defAStar} does not actually depend on $\lambda \in \Lambda$, as 
shown in Lemma~\ref{lm:AStarIndepLambda} below.

Before stating our next main theorem, we recall the definition of stable convergence,
which interpolates to some extent between convergence in law and convergence in probability. 
We refer to the monographs \cite{JacodLimit, StableBook} and references therein for more details 
on stable convergence in a more general setting.

Let $\CX$ be a locally compact Polish space and let $\CM^{+}(\CX)$ be the space of non-negative, locally finite measures on $\CX$ endowed with the topology of vague convergence. We equip the space of probability distributions on $\CM^{+}(\CX)$ with the topology 
of weak convergence.
\begin{definition}
\label{def:stabelConv}
Let $(\nu_{t})_{t \geq 0}$ be a collection of $\CM^{+}(\CX)$-valued random variables defined on
a common probability space $(\Omega,\P)$, and let $\nu$ be a $\CM^{+}(\CX)$-valued random variables defined
on a possibly larger probability space. Consider a $\sigma$-algebra $\Sigma$ over $\Omega$.
We say that $\nu_{t}$ converges $\Sigma$-stably to $\nu \in \CM^{+}(\CX)$ as $t \to \infty$, if $(Z, \nu_{t}) \Rightarrow (Z, \nu)$\footnote{Here and below, we write $\Rightarrow$ to denote convergence in law.} for all $\Sigma$-measurable random variables $Z$. 
\end{definition}

\begin{remark}
If $\Sigma$ is the trivial $\sigma$-algebra, then this coincides with convergence in law. Conversely, if $\Sigma$ is the full $\sigma$-algebra of the probability space $\Omega$ and the limiting random variable is defined on $(\Omega, \P)$, then this corresponds to convergence in probability.
\end{remark}

We are now ready to state our next main theorem.
\begin{theoremA} 
\label{th:stableConv}
Let $\gamma > \sqrt{\smash[b]{2d}}$ and consider the sequence of measures $(\mu_{\gamma, t})_{t \geq 0}$ defined in \eqref{e:norm_super}. 
For $n \in \N$, consider a collection of fields $(\rmW_{i, t})_{i \in [n], t \geq 0}$ 
satisfying  \ref{hp:W1} \dash \ref{hp:W4} and independent of the field $\Psi$ introduced in Definition~\ref{def:recoverPhi}.
For each $t \geq 0$ and each $i \in [n]$, define the measure $\mu_{\gamma, t, i}$ by
\begin{equation*}
\mu_{\gamma, t, i}(dx) \eqdef e^{\gamma \rmW_{i, t}(x)} \mu_{\gamma, t}(dx) \;.
\end{equation*}
Consider the $\R^n$-valued random variable $\bfZ_{\gamma}$ such that, for each $i \in [n]$, the $i$-th component $\bfZ_{\gamma, i}$ is given by
\begin{equation}
\label{e:defbZ}
\bfZ_{\gamma, i} \eqdef \int_{\R^d} \exp\bigl(\gamma\bigl(\Psi(y) + \rmW_i(y)\bigr)\bigr) dy \;.
\end{equation}
Let $(x_j, w_j)_{j \in \N} \subseteq \R^d \times [0, \infty]$ be an arbitrary enumeration of the atoms of the Poisson point measure $\eta_{\gamma}[\mu_{\gammac}]$ as introduced in Definition~\ref{def:PPP}. Consider the collection of measures $(\mu_{\gamma, i})_{i \in [n]}$ constructed as follows
\begin{equation*}
\mu_{\gamma, i}(dx) \eqdef a_{\star}^{\frac{\gamma}{\sqrt{\smash[b]{2d}}}} \sum_{j \in \N} \bfZ_{\gamma, i, j} w_j \delta_{x_j} \;, \qquad \forall \, i \in [n]\;,
\end{equation*}
where $\smash{((\bfZ_{\gamma, i, j})_{i \in [n]})_{j \in \N}}$ is a collection of i.i.d.\ copies of $\smash{(\bfZ_{\gamma, i})_{i \in [n]}}$, independent of $\eta_{\gamma}[\mu_{\gammac}]$. Then, the sequence of measures $(\mu_{\gamma, t, i})_{i\in [n]}$ converges $\sigma(\rmX)$-stably to the collection of measures $(\mu_{\gamma, i})_{i\in [n]}$ as $t \to \infty$.
\end{theoremA}

\begin{remark}
Let $\bfZ_{\gamma}$ be the random variable defined by 
\begin{equation*}
\bfZ_{\gamma} \eqdef \int_{\R^d} \exp\bigl(\gamma\bigl(\Psi(y)\bigr)\bigr) dy \;.
\end{equation*}
 Then, as follows from the proof of Proposition~\ref{pr:laplaceJoint} below, we have $\smash{\E[\bfZ_{\gamma}^{\sqrt{\smash[b]{2d}}/\gamma}] < \infty}$. In particular, this implies that $\bfZ_{\gamma}$ is almost surely finite.
\end{remark}

\begin{remark}
Theorem~\ref{th:stableConv} is significantly more general than \cite[Theorem~2.2]{Glassy}.
 For instance, taking $n = 1$ and $\rmW_{1, \cdot} = 0$, we not only 
recover Theorem~\ref{th:GlassyEx} but also obtain a relatively explicit representation for the multiplicative constant appearing in front of the limiting measure. Indeed, we deduce that the constant $c$ appearing in \cite[page~646]{Glassy} (and for which
only an existence statement is provided there) is given by
\begin{equation*}
c = a_\star^{\frac{\gamma}{\sqrt{\smash[b]{2d}}}} \E\bigl[\bfZ_{\gamma}^{\sqrt{\smash[b]{2d}}/\gamma}\bigr]^{\frac{\gamma}{\sqrt{\smash[b]{2d}}}} \;.
\end{equation*}
\end{remark}

For $\gamma > \sqrt{\smash[b]{2d}}$, we define the constant $\beta(d, \gamma)$ by letting
\begin{equation}
\label{eq:defBetaDG}
\beta(d, \gamma) \eqdef \frac{\Gamma(1-\sqrt{\smash[b]{2d}}/\gamma)}{\sqrt{\smash[b]{2d}}/\gamma} \;.
\end{equation}

Theorem~\ref{th:stableConv} is a direct consequence of the following result, the proof of 
which is given in Section~\ref{sub:JointStable}, where we compute the joint Laplace transform 
of the collection of measures $(\mu_{\gamma, t, i})_{i \in [n]}$. In what follows, for a measure 
$\nu$ on $\R^d$ and a function $f:\R^d \to \R$, we write $\nu(f)$ to denote the integral of 
$f$ against $\nu$.

\begin{proposition}
\label{pr:laplaceJoint}
Consider the same setting described in Theorem~\ref{th:stableConv}. Consider the mapping $\rmT_{\gamma}: (\CC^{+}_{c}(\R^d))^n \to \CC^{+}_{c}(\R^d)$ defined by
\begin{equation}
\label{eq:defSTGamma}
\rmT_{\gamma}\bigl[f_1, \ldots, f_n\bigr](\cdot) \eqdef \E\Biggl[\Biggl(\sum_{i = 1}^{n} f_i(\cdot) \int_{\R^d} \exp\bigl(\gamma\bigl(\Psi(y) + \rmW_i(y)\bigr)\bigr) dy\Biggr)^{\!\!\! \f{\sqrt{\smash[b]{2d}}}{\gamma}}\Biggr] \;.
\end{equation}
Then, for all $(\varphi, (f_i)_{i \in [n]}) \in \CC^{\infty}_c(\R^d) \times (\CC_c^{+}(\R^d))^n$, the following limit holds
\begin{equation*}
\lim_{t \to \infty} \E\Biggl[\exp\bigl(i \langle \rmX, \varphi \rangle \bigr) \prod_{i = 1}^{n} \exp\bigl(-\mu_{\gamma, t, i}(f_i)\bigr)\Biggr] = \E\Bigl[\exp\bigl(i \langle \rmX, \varphi \rangle \bigr) \exp\bigl(- \tilde a_{\star} \mu_{\gammac}\bigl(\rmT_{\gamma}[f_1, \ldots, f_n]\bigr)\bigr)\Bigr] \;,
\end{equation*}		
where $\tilde a_{\star} = \beta(d, \gamma) a_{\star} > 0$ with $a_{\star}$ as defined in \eqref{eq:defAStar}.
\end{proposition}

%%%%%%%%%%%%%%%%%%%%%%%%%%%%%%%%%%%%%%%%%%%%%%
\subsection{A measure-valued process}
\label{sub:measureValuedPr}
As we have already briefly mentioned, an interesting situation where Theorem~\ref{th:stableConv} can be applied is when the admissible fields are given by the increments of the martingale approximation of the $\star$-scale invariant field $\rmX$. More precisely, we note that for all $t$, $s \geq 0$ it holds that
\begin{equation*}
	\rmX_{t+s}(x) = \rmX_t(x) + \rmW_{s, t}(x) \;, \qquad \forall \, x \in \R^d  \;,
\end{equation*}
where $\rmW_{s, t}(\cdot) = \rmW_{s}(e^{t} \cdot)$ with $\rmW_s(\cdot)$ a centred Gaussian field independent of $\CF_t$ and such that 
\begin{equation*}
	\E\bigl[\rmW_s(x)\rmW_s(y)\bigr] = \int_0^{s}\frkK\bigl(e^{u}(x-y)\bigr) du \;, \qquad \forall \, x, y \in \R^d \;.
\end{equation*} 
For $\gamma > \sqrt{\smash[b]{2d}}$, the measure $\mu_{\gamma, t+s}$ can then be rewritten as 
\begin{equation}
\label{eq:decoMeasureSum}
\mu_{\gamma, t+s}(dx) = e^{\gamma (\rmW_{s, t}(x) - \sqrt{\smash[b]{2d}} s) + ds} \bigl((t+s)/t\bigr)^{\f{3\gamma}{2 \sqrt{\smash[b]{2d}}}}  \mu_{\gamma, t}(dx) \;.
\end{equation}

Furthermore, for any finite collection of non-negative numbers $(s_1, \ldots, s_n) \subseteq \R^{+}_0$, the family of fields $\smash{(\rmW_{s_i, t})_{i \in [n], t \geq 0}}$ satisfies \ref{hp:W1} \dash \ref{hp:W4}.
Therefore, by Theorem~\ref{th:stableConv} and Kolmogorov's extension theorem, there exists a process $(\nu_{\gamma, s})_{s \geq 0}$ taking values in the space of non-negative, locally finite measures on $\R^d$ such that the collection of measures $(\nu_{\gamma, s_1}, \ldots, \nu_{\gamma, s_n})$ is the $\sigma(\rmX)$-stable limit of $(\mu_{\gamma, t + s_1}, \ldots, \mu_{\gamma, t + s_n})$ as $t \to \infty$.

Let $\Psi$ denote the field appearing in \eqref{e:defbZ}. For all $s \geq 0$, we define the field $\Psi_s$ on $\R^d$ by letting 
\begin{equation}
\label{e:Psis}
\Psi_s(\cdot) = \Psi(e^{-s}\cdot) + \bigl(\rmW_s(e^{-s} \cdot) - \sqrt{\smash[b]{2d}} s\bigr)\;.
\end{equation}
As a consequence of \eqref{eq:decoMeasureSum}, we then have the following corollary of Theorem~\ref{th:stableConv}.

\begin{corollaryA}
\label{co:processMeasures}
For $\gamma > \sqrt{\smash[b]{2d}}$, consider the $\R$-valued process $(\bfW_{\gamma, s})_{s \geq 0}$ given by
\begin{equation}
\label{eq:weightProcess}
\bfW_{\gamma, s} \eqdef  \log\Biggl(\int_{\R^d} \exp\bigl(\gamma \Psi_s(x)\bigr) dx\Biggr) \;.
\end{equation}
Let $(x_j, w_j)_{j \in \N} \subseteq \R^d \times [0, \infty]$ be an arbitrary enumeration of the atoms of the Poisson point measure $\eta_{\gamma}[\mu_{\gammac}]$ as introduced in Definition~\ref{def:PPP}. Consider the collection of measures $(\nu_{\gamma, s})_{s \geq 0}$ defined as follows
\begin{equation}
\label{e:reweigh}
\nu_{\gamma, s} \eqdef a_{\star}^{\frac{\gamma}{\sqrt{\smash[b]{2d}}}} \sum_{j \in \N} e^{\bfW_{\gamma, s, j}} w_j \delta_{x_j} \;, \qquad \forall \, s \geq 0 \;,
\end{equation}
where $\smash{((\bfW_{\gamma, s, j})_{s \geq 0})_{j \in \N}}$ is a collection of i.i.d.\ copies of $\smash{(\bfW_{\gamma, s})_{s \geq 0}}$, independent of $\eta_{\gamma}[\mu_{\gammac}]$. Then, the collection of measures $(\mu_{\gamma, t + s})_{s \geq 0}$ converges $\sigma(\rmX)$-stably in the finite dimensional sense to the collection of measures $(\nu_{\gamma, s})_{s \geq 0}$ as $t \to \infty$.
\end{corollaryA}

We conclude this section with a couple of remarks and conjectures.

\begin{remark}[Stationarity modulo tilt]
For $s \geq 0$, write $\smash{x_{s} = \argmax\{\Psi_s(x) \, : \, x \in \R^d\}}$, $\rmZ_s = \Psi_s(x_s)$, and $\tilde \Psi_s = \Psi_s(x_s + \cdot) - \Psi_s(x_s)$. Fix a finite collection of non-negative numbers $(s_0, \ldots, s_n) \subseteq \R^{+}_0$ such that  $s_0 \leq \ldots\leq s_n$, and consider the joint limit of $(\mu_{\gamma, t + s_i})_{i \in [n]}$ as $t \to \infty$. For all $(f_i)_{i \in [n]} \in (\CC_c^{+}(\R^d))^n$, their joint Laplace transform is given by the expression in Proposition~\ref{pr:laplaceJoint} with
\begin{align*}
\rmT_{\gamma}\bigl[f_1, \ldots, f_n\bigr](\cdot) &= \E\Biggl[\Biggl(\sum_{i = 1}^{n} f_i(\cdot) \int_{\R^d} \exp\bigl(\gamma \Psi_{s_i}(y)\bigr) dy\Biggr)^{\!\!\! \f{\sqrt{\smash[b]{2d}}}{\gamma}}\Biggr] \\
&= \E\Biggl[e^{\sqrt{\smash[b]{2d}} \rmZ_{s_0}}\Biggl(\sum_{i = 1}^{n} f_i(\cdot)e^{\gamma(\rmZ_{s_i}-\rmZ_{s_0})} \int_{\R^d} \exp\bigl(\gamma \tilde \Psi_{s_i}(y)\bigr) dy\Biggr)^{\!\!\! \f{\sqrt{\smash[b]{2d}}}{\gamma}}\Biggr]\;.
\end{align*}
This strongly suggests that the process $(\tilde \Psi_s)_{s \geq 0}$ is ``stationary modulo tilt''
in the sense that, for any non-decreasing sequence of non-negative times $(s_0, \ldots, s_n) \subseteq \R^{+}_0$, any $t \geq 0$, and any function $\bfF \in \CC^b((\CC(\R^d))^n)$, 
one has
\begin{equation}
\label{e:conjecture}
\E\bigl[e^{\sqrt{\smash[b]{2d}} \rmZ_{s_0}} \bfF(\tilde \Psi_{s_1},\ldots, \tilde \Psi_{s_n})\bigr]
=
\E\bigl[e^{\sqrt{\smash[b]{2d}} \rmZ_{s_0+t}} \bfF(\tilde \Psi_{s_1+t},\ldots, \tilde \Psi_{s_n+t})\bigr]\;.
\end{equation}
Now, consider the process $(\rho_t)_{t \geq 0}$ given by
\begin{equation}
\label{eq:defProcessRho}
	\rho_t \eqdef \sum_{j \in \N} \delta_{x_j} \otimes \delta_{w_j + \rmZ_{j, t}} \otimes \delta_{\tilde\Psi_{j, t}} \;, \qquad \forall \, t \geq 0 \;,
\end{equation}
where $\smash{(\rmZ_j, \tilde{\Psi}_j)_{j \in \N}}$ is a collection of i.i.d.\ copies of $\smash{(\rmZ, \tilde{\Psi})}$, and $(x_j, w_j)_{j \in \N}$ is an arbitrary enumeration of the Poisson point process with intensity $a_{\star} \gamma \mu_{\gamma_c}(dx) \otimes \exp(-\sqrt{\smash[b]{2d}} w) dw$, where $a_{\star}$ is the constant introduced in \eqref{eq:defAStar}.
Then, similarly to \cite[Proposition~3.1]{Stationary}, one can show that \eqref{e:conjecture} implies that the process $(\rho_t)_{t \geq 0}$ is stationary. 
\end{remark}

\begin{remark}[Full process convergence]
For $0 \ll b \ll t$, recalling \eqref{eq:recentering}, we define the measure $\rho_{t, b}$ on $[0, 1]^d \times \R \times \CC(\R^d)$ as follows 
\begin{equation*}
\rho_{t, b} \eqdef \sum_{x \in \R^d \, : \, \rmX_t(x) = \sup_{\abs{y-x} \leq e^{b-t}} \rmX_t(y)} \delta_x \otimes \delta_{\rmX_t(x) - \frkm_t} \otimes \delta_{\rmX_t(x + e^{-t} \cdot) - \rmX_t(x)} \;,
\end{equation*}
which roughly speaking keeps track of the locations of the maxima, their heights, and the shape of the field around them. Then, combining the results of the present article with the techniques of \cite{BiskupLouidor}, it should not be too hard to show, analogously to \cite[Theorem~2.1]{BiskupLouidor}, the convergence in law
\begin{equation*}
\lim_{b \to \infty}	\lim_{t \to \infty} \rho_{t, b} = \mathrm{PPP} \bigl(a_{\star} \gamma \mu_{\gammac}(dx) \otimes e^{-\sqrt{\smash[b]{2d}} w} dw \otimes \nu(d \phi)\bigr) \;,
\end{equation*}
where $\nu$ denotes the law of the field $\Psi$ on $\CC(\R^d)$, and $a_{\star}$ is the constant introduced in \eqref{eq:defAStar}. 
Furthermore, the discussion in the preceding paragraph strongly suggests that 
\begin{equation*}
\lim_{b \to \infty}	\lim_{t \to \infty} \rho_{t+\cdot, b} = \rho_\cdot\;,
\end{equation*}
where $\rho$ is the process introduced in \eqref{eq:defProcessRho}.
\end{remark}

%%%%%%%%%%%%%%%%%%%%%%%%%%%%%%%%%%%%%%%%%%%%%%
\subsection{Outline}
 The remainder of the paper is organised as follows. 
 In Section~\ref{sec:Setup}, we gather some background material that will be used throughout the paper. 
 Section~\ref{sec:cluster} focuses on the local structure of extremal points, and contains the proofs of Theorems~\ref{th:cluster} and~\ref{th:clusterProb}.
In Section~\ref{sec:resampling}, we establish key properties of the shape field, including a suitable resampling property, and prove Propositions~\ref{pr:inversionPsi}~and~\ref{pr:PsiIndepLambda}.
Section~\ref{sec:proofMainTh1} is dedicated to the proof of Theorem~\ref{th:stableConv}, which is based on Proposition~\ref{pr:laplaceJoint}. The proof of this proposition, in turn, relies on a key technical result, Proposition~\ref{pr:joint}, whose proof is given in Section~\ref{sec:Joint}.
We conclude the paper with three appendices. In Appendix~\ref{ap:BBEstimates}, we establish several estimates concerning the probability of a Brownian bridge remaining above a slowly growing positive or negative curve. Appendix~\ref{ap:proofbReg} contains the proof of a technical lemma used in Section~\ref{sec:Joint}. Finally, in Appendix~\ref{sec:GaussianTool}, we collect some standard results on Gaussian fields.

%%%%%%%%%%%%%%%%%%%%%%%%%%%%%%%%%%%%%%%%%%%%%%
%%%%%%%%%%%%%%%%%%%%%%%%%%%%%%%%%%%%%%%%%%%%%%
\section{Background and preliminaries}
\label{sec:Setup}
In this section, we collect some preliminary results needed for the proof of our main theorems. In particular, in  Section~\ref{sub:basic}, we collect some basic and recurrent notation that will be used throughout the paper. In Section~\ref{sub:topPrel}, we recall some standard results related to convergence in distribution of random measures, and in particular we briefly introduce the concept of stable convergence. Finally, in Section~\ref{sub:propStarScale}, we record some properties of $\star$-scale invariant fields and their martingale approximations.

%%%%%%%%%%%%%%%%%%%%%%%%%%%%%%%%%%%%%%%%%%%%%%
\subsection{Basic and recurrent notation}
\label{sub:basic}
\paragraph{Numbers.} We write $\N = \{1,2, \ldots\}$ and $\N_0 = \{0, 1, 2, \ldots\}$. We let $\R^{+} = (0, \infty)$ and $\R^{+}_0 = [0, \infty)$. Without specific mention, the logarithm will be taken with respect to the natural base $e$. For $a \in \R$, we use $\lfloor a \rfloor$ to represent the largest integer not greater than $a$. Given $n \in \N$, we write $[n] = \{1, \ldots, n\}$ and $[n]_0 = \{0, 1, \ldots, n\}$.

\paragraph{Subsets of Euclidean space.} We consider the space $\R^d$ where $d \geq 1$ is a fixed dimension. We let $(e_1, \ldots, e_d)$ be the orthonormal basis of $\R^d$. For $x \in \R^d$, we write $x = (x_1, \ldots, x_d)$ for its coordinates. For $r \in \R$ and $x\in\R^d$, we write $B(x, r)$ for the ball centred at $x$ and with radius $r$. Furthermore, we write 
\begin{equation}
\label{eq:expBalls}
\B_r(x) \eqdef B(x, e^r)	 \;, 
\end{equation}
and we simply write $\B_r$ for $\B_r(0)$. For every Lebesgue measurable set $\rmD \subseteq \R^d$, we denote its Lebesgue measure by $\abs{\rmD}$. 

\paragraph{Functions and measure spaces.} We write $\CC(\R^d)$ (resp.\ $\CC_c(\R^d)$, $\CC^b(\R^d)$) for the space of continuous (resp.\ continuous with compact support, continuous and bounded) functions from $\R^d$ to $\R$. We write $\CC^{+}_c(\R^d)$ for the space of positive continuous functions from $\R^d$ to $\R$ with compact support. Given a measure $\nu$ and a function $f$, we write $\nu(f)$ to denote the integral of $f$ against $\nu$.

\paragraph{Maxima and related sets.}For a subset $\rmD \subseteq \R^d$, a function $f:\rmD\to\R$, and $\lambda > 0$ we let
\begin{equation}
\label{eq:maximal}
\M_{\rmD}(f) \eqdef \sup_{x \in \rmD} f(x)\;, \qquad \D^{\lambda}_{\rmD}(f) \eqdef \bigl\{x \in \rmD \, : \, f(x) \geq \M_{\rmD}(f) - \lambda\bigr\} \;.
\end{equation}
For $\rmD = \R^d$, we simply write $\M(f)$ and $\D^{\lambda}(f)$. Additionally, if $\rmD = [0, \rmR]^d$ for some $\rmR \geq 0$, then we write $\M_{\rmR}(f)$ (resp.\ $\D^{\lambda}_{\rmR}(f)$) in place of $\M_{[0,\rmR]^d}(f)$ (resp.\ $\D^{\lambda}_{[0, \rmR]^d}(f)$). 
Furthermore, when it will be convenient to do so, we will use the following shorthands,
 \begin{equation}
 \label{eq:maximalExpBalls}
 \M_{x, r}(f) = \M_{\B_r(x)}(f)\;, \qquad \D^{\lambda}_{x, r}(f) = \D^{\lambda}_{\B_r(x)}(f) \;,
 \end{equation}
 and also, for $r_2 > r_1$, we set 
 \begin{equation}
 \label{eq:maximalExpAnnulus}
 \M_{x, r_2, r_1}(f) = \M_{\B_{r_2}(x) \setminus \B_{r_1}(x)}(f)\;, \qquad \D^{\lambda}_{x, r_2, r_1}(f) = \D^{\lambda}_{\B_{r_2}(x) \setminus \B_{r_1}(x)}(f) \;.
\end{equation} 

%%%%%%%%%%%%%%%%%%%%%%%%%%%%%%%%%%%%%%%%%%%%%%
\subsection{Topological preliminaries}
\label{sub:topPrel}
In this subsection, we collect some results on the convergence of random measures and introduce the concept of stable convergence, which plays an important role in our main result regarding the convergence of supercritical GMC measures.

\paragraph{Laplace functionals.}
It is well known that if $\eta$ is a random point measure on $\R^d$, then its law is uniquely characterised by its Laplace functional on the set $\CC_c^{+}(\R^d)$:
\begin{equation*}
\CC_c^{+}(\R^d) \ni \varphi \mapsto \E\bigl[\exp\bigl(-\eta(\varphi)\bigr)\bigr] \;.
\end{equation*}
We also recall that, if $\eta_{\nu}$ is a Poisson point measure with intensity measure $\nu$, then
\begin{equation*}
\E\exp\bigl[\bigl(-\eta_{\nu}(\varphi)\bigr)\bigr]  = \exp\left(-\int_{\R^d} \bigl(1-e^{-\varphi(x)}\bigr) \nu(dx)\right) \;.
\end{equation*}

\begin{remark}
\label{rem:comp_Laplace}
For $\gamma > \sqrt{\smash[b]{2d}}$ and a Radon measure $\nu$ on $\R^d$, let $\CP_{\gamma}[\nu]$ be the integrated atomic random measure with parameter $\gamma$ and spatial intensity $\nu$ as specified in Defnition~\ref{def:PPP}.
Then, in this case, for every $\varphi \in \CC_c^{+}(\R^d)$, it holds that
\begin{align}
\E\bigl[\exp\bigl(-\CP_{\gamma}[\nu](\varphi)\bigr)\bigr]
%& = \E\Biggl[\exp\Biggl(- \int_{\R^d \times \R^{+}}  \varphi(x) z \, \eta_{\gamma}[\nu](dx, dz) \Biggr)\Biggr] \\
%& = \exp\Biggl(- \int_{\R^d \times \R^{+}} \f{1-e^{-\varphi(x) z}}{z^{1+\sqrt{\smash[b]{2d}}/\gamma}} \nu(dx) dz \Biggr) \\
& = \exp\Biggl(- \beta(d, \gamma) \int_{\R^d} \varphi(x)^{\f{\sqrt{\smash[b]{2d}}}\gamma}  \nu(dx) \Biggr) \;, \label{eq:LaplaceCompInt} 
\end{align}
where we recall the definition \eqref{eq:defBetaDG} of the constant $\beta(d, \gamma)$.
\end{remark}

\paragraph{Convergence in distributions of random measures.}
For a locally compact Polish space $\CX$, we let $\CM^{+}(\CX)$ be the space of non-negative, locally finite measures on $\CX$ endowed with the topology of vague convergence.
We equip the space of probability distributions on $\CM^{+}(\CX)$ with the topology 
of weak convergence.
For a collection $(\nu_{t})_{t \geq 0}$ of $\CM^{+}(\CX)$-valued random variables, we write $\nu_t \Rightarrow \nu$ to indicate \emph{vague convergence in distribution}
as $t \to \infty$. 

We record here a useful criterion to establish the vague convergence in distributions of $\CM^{+}(\CX)$-valued random variables.
\begin{lemma}
\label{lem:conv_Laplace}
Let $(\nu_{t})_{t \geq 0}$ be a collection of $\CM^{+}(\CX)$-valued random variables and let $\nu \in \CM^{+}(\CX)$. If for all $f \in \CC_c^{+}(\CX)$, the following limit holds 
\begin{equation}
\label{eq:hpJointLap}
	\lim_{t \to \infty} \E\bigl[\exp\bigl(- \nu_{t}(f)\bigr)\bigr] = \E\bigl[\exp\bigl(- \nu(f)\bigr)\bigr] \;,
\end{equation}
then it holds that $\smash{\nu_{t} \Rightarrow \nu}$ as $t \to \infty$.
\end{lemma}
\begin{proof}
	This is an immediate consequence of the continuity theorem for Laplace transforms, see e.g.\ \cite[Theorem~4.11]{Kallenberg}.
\end{proof}

\begin{remark}
In what follows, we will be interested in the case $\CX = [n]\times \R^d$ which coincides (including the topology) with the space $(\CM^{+}(\R^d))^n$. 
\end{remark}

Given a locally compact Polish space $\CX$ and a Hilbert space $\CY$ equipped with a dense subspace $\CY_0$, we consider probability distributions on $\smash{\CY \times \CM^{+}(\CX)}$. Similarly to before, for a sequence $\smash{(\rmY, \nu_{t})_{t \geq 0}}$ of $\smash{\CY \times \CM^{+}(\CX)}$-valued random variables, we write $(\rmY, \nu_{t}) \Rightarrow (\rmY, \nu)$ to indicate \emph{vague convergence in distribution}. We now state the following result, which provides sufficient conditions for convergence in distribution on the space $\smash{\CY \times \CM^{+}(\CX)}$.

\begin{lemma}
\label{lm:jointConv}
Let $\smash{(\rmY, \nu_{t})_{t \geq 0}}$ be a sequence of $\smash{\CY \times \CM^{+}(\CX)}$-valued random variables, and let $\nu \in \CM^{+}(\CX)$. If for all $(\varphi, f) \in \CY_0 \times \CC_c^{+}(\CX)$ the following limit holds
\begin{equation}
\label{eq:criterionJoint}
	\lim_{t \to \infty} \E\Bigl[\exp\bigl(i \langle \rmY, \varphi \rangle\bigr) \exp\bigl( -\nu_{t}(f)\bigr) \Bigr] =  \E\Bigl[\exp\bigl(i \langle \rmY, \varphi \rangle\bigr) \exp\bigl( - \nu(f)\bigr) \Bigr] \;, 
\end{equation}
then it holds that $(\rmY, \nu_{t}) \Rightarrow (\rmY, \nu)$ as $t \to \infty$.
\end{lemma}
\begin{proof}
Taking $\varphi = 0$ in \eqref{eq:criterionJoint} and using Lemma~\ref{lem:conv_Laplace}, we deduce that $\nu_{t} \Rightarrow \nu$. Consequently, the joint distribution $(\rmY, \nu_{t})$ is tight in $\smash{\CY \times \CM^{+}(\CX)}$. Thus, the joint convergence in distribution follows if we can show the convergence of the finite-dimensional distributions. As both $\rmY$ and $\nu_{t}$ are linear forms, the convergence of the finite-dimensional distributions can be inferred from that of the one-dimensional distributions. Therefore, it suffices to verify that, for all $\smash{(\varphi, f) \in \CY_0 \times \CC_c^{+}(\CX)}$, the following convergence holds
\begin{equation}
\label{eq:jointFinDim}
\bigl(\langle \rmY, \varphi \rangle, \nu_{t}(f)\bigr) \Rightarrow \bigl(\langle \rmY, \varphi \rangle, \nu(f)\bigr)  \;.
\end{equation}
Since the random variable $\smash{\nu(f)}$ is almost surely non-negative, it can be readily observed that the joint convergence in distribution \eqref{eq:jointFinDim} holds if the corresponding joint Fourier--Laplace transform converges, i.e.\ if \eqref{eq:criterionJoint} holds. 
\end{proof}

\begin{remark}
In what follows, we will be interested in the case $\CX = [n]\times \R^d$ and $\CY = \CH_{\loc}^{-\kappa}(\R^d)$, for some $\kappa > 0$, in which case we can take $\CY_0 = \CC^{\infty}_{c}(\R^d)$. 
\end{remark}

We conclude this section with a simple fact about stable convergence of random measures.
\begin{lemma}
\label{lm:stableRV}
Consider the same setting described in Definition~\ref{def:stabelConv}. Furthermore, let $\CY$ be a Polish space and $\rmY$ a $\CY$-valued random variable. Then $\nu_{t}$ converges $\sigma(\rmY)$-stably to $\nu$ if and only if $\smash{(\rmY, \nu_{t}) \Rightarrow (\rmY, \nu)}$ as $t \to \infty$.
\end{lemma}
\begin{proof}
See for instance \cite[Exercise~3.11]{StableBook}.
\end{proof}

%%%%%%%%%%%%%%%%%%%%%%%%%%%%%%%%%%%%%%%%%%%%%%
\subsection{Some properties of \texorpdfstring{$\star$}{*}-scale invariant fields}
\label{sub:propStarScale}
Let $\rmX$ be a $\star$-scale invariant field as defined in \eqref{eq:field} and with seed covariance function $\frkK$ satisfying assumptions \ref{hp_K1} \dash \ref{hp_K2}.
Recalling the covariance structure \eqref{eq:covs} of $\rmX$, we point out that there exists a smooth function $g:\R^d\times\R^d \to \R$ such that, for all $x$, $y \in \R^d$,
\begin{equation*}
\E\bigl[\rmX(x) \rmX(y)\bigr] = -\log |x-y| + g(x, y) \;,
\end{equation*}
or in other words, $\rmX$ is a log-correlated Gaussian field. Thanks to \cite{Junnila_deco}, it is known that a partial converse is true, i.e., given a log-correlated Gaussian field $\rmX$ with covariance of the form $-\log\abs{x-y} + g(x, y)$ for some function $g: \R^d \times \R^d \to \R$ satisfying certain (weak) regularity assumptions, then $\rmX$ can be decomposed as $\rmX = \rmX^{\star} + \rmL$, where $\rmX^{\star}$ is a $\star$-scale invariant field, $\rmL$ is a centred Gaussian field with H\"older regularity, and $\rmX^{\star}$ and $\rmL$ are jointly Gaussian.

We recall that $(\rmX_t)_{t \geq 0}$ denotes the martingale approximation of $\rmX$ as defined in \eqref{eq:fieldST}. For every $t \geq 0$, we recall the definition \eqref{eq:defSigmaFieldT} of the $\sigma$-field $\CF_t$.
The following properties are straightforward to check:
\begin{enumerate}
	\item For $0 \leq s < t$, the random field $\rmX_{s, t}$ is independent from the $\sigma$-field $\CF_{s}$. 
	\item For any fixed $x \in \R^d$, the process $(\rmX_t(x))_{t \geq 0}$ has the law of a standard Brownian motion.
	\item  For $0 \leq s < t$, the following scaling relation holds 
\begin{equation}
\label{e:scaling_rel}
\rmX_{s, t}(\cdot) \eqlaw \rmX_{t-s}(e^s \cdot)\;.
\end{equation}
\end{enumerate}

We now introduce a field that will play an import role in what follows. We recall that $\bar \frkK$ is the (unique) positive definite function such that the convolution of $\bar \frkK$ with itself equals $\frkK$.
\begin{definition}
\label{def:fieldsZ}
For $\xi'$ a space-time white noise on $\R^d \times \R^{+}$, we define the field $\rmZ_{\infty}$ on $\R^d$ by letting,
\begin{equation}
\label{eq:fieldZ}
\rmZ_{\infty}(\cdot) \eqdef \int_{0}^{\infty} \int_{\R^d} \bigl(\bar \frkK(e^{-r}(\cdot - y)) - \frkK(e^{-r}  \cdot)\bar\frkK(e^{-r} y)\bigr) e^{-\f{d r}{2}} \xi'(dy,dr) \;.
\end{equation}
Furthermore, for $0 \leq s < t$, we let $\rmZ_{s, t}$ be the field on $\R^d$ given by
\begin{equation}
\label{eq:fieldZST}
\rmZ_{s, t}(\cdot) \eqdef \int_{s}^{t} \int_{\R^d} \bigl(\bar \frkK(e^{-r}(\cdot - y)) - \frkK(e^{-r} \cdot)\bar\frkK(e^{-r} y)\bigr) e^{-\f{d r}{2}} \xi'(dy,dr) \;, 
\end{equation}
with the notational convention that $\rmZ_{0, t} = \rmZ_t$.  
\end{definition}

We observe that, for any $x$, $y \in \R^d$ and $s$, $t\geq 0$, it holds that 
\begin{align}
 \E\bigl[\rmZ_{\infty}(x) \rmZ_{\infty}(y)\bigr] &= \int_0^{\infty} \bigl(\frkK(e^{-r}(x - y)) - \frkK(e^{-r}x) \frkK(e^{-r} y)\bigr) dr \;, \label{e:covZinfty}\\
 \E\bigl[\rmZ_s(x) \rmZ_t(y)\bigr] &= \int_0^{s \wedge t} \bigl(\frkK(e^{-r}(x - y)) - \frkK(e^{-r}x) \frkK(e^{-r}y)\bigr) dr \;.	\label{e:covZb}
\end{align}
It is straightforward to check that the field $\rmZ_t(\cdot)$ converges weakly in law with respect to the local uniform topology in $\CC(\R^d)$ to $\rmZ_{\infty}(\cdot)$ as $t \to \infty$ (see e.g.\ \cite[Proposition~2.4]{Madaule_Max}). 

We also record here a decomposition result for $\star$-scale invariant fields originally stated in \cite{Critical_der} and which can be proved by standard computation of covariances.
\begin{lemma}[{\cite[Lemma~16]{Critical_der}}]
\label{lm:decoPointsStar}
For $z \in \R^d$, the field $(\rmX_t(x))_{t \geq 0, x \in \R^d}$ admits the following decomposition
\begin{equation*}
	\rmX_t(x) = \int_0^t \frkK\bigl(e^r(x-z)\bigr) d\rmX_r(z) + \rmZ^z_t(x) \;, \qquad \forall \, t \geq  0\;, \forall \, x \in \R^d \;,
\end{equation*}
where $\smash{(\rmZ_t^z(x))_{t \geq 0, x \in \R^d}}$ is a centred Gaussian field independent of $(\rmX_t(z))_{t \geq 0}$ and with the following covariance structure, 
\begin{equation*}
\E\bigl[\rmZ^z_s(x) \rmZ^z_t(y)\bigr] = \int_0^{s \wedge t} \bigl(\frkK(e^{r}(x - y)) - \frkK(e^{r}(x - z)) \frkK(e^{r}(y - z))\bigr) dr \;, \quad \forall \, s,t \geq 0, \; \forall \, x, y \in \R^d \;.
\end{equation*}
\end{lemma}
We emphasise that for all $t \geq 0$, it holds that $\smash{\rmZ^0_t(e^{-t}\cdot) \eqlaw \rmZ_t(\cdot)}$, where $\rmZ_t$ is the field introduced in Definition~\ref{def:fieldsZ}. 

%%%%%%%%%%%%%%%%%%%%%%%%%%%%%%%%%%%%%%%%%%%%%%
%%%%%%%%%%%%%%%%%%%%%%%%%%%%%%%%%%%%%%%%%%%%%%
\section{Local structure of extremal points}
\label{sec:cluster}
The main goal of this section is to prove Theorems~\ref{th:cluster} and~\ref{th:clusterProb}. We emphasise that the arguments for the proofs of these two theorems are inspired by and follow similar lines to the proofs of \cite[Theorems~2.3~and~2.4]{BiskupLouidor}. 

This section is structured as follows. In Section~\ref{subsec:setupCluster}, we introduce the precise setup and state a slightly stronger version of Theorems~\ref{th:cluster} and~\ref{th:clusterProb}. In Section~\ref{subsec:redBMCluster}, we explain how to convert the statement about the supremum of the field being less than $\lambda$ into a condition on a Brownian motion/bridge to stay above a polylogarithmic curve. In Section~\ref{subsec:techLemmasCluster}, we collect some technical lemmas that are needed for the proof of the main theorems, while their proofs are contained in Section~\ref{subsec:asyFormulaCluster}. 

%%%%%%%%%%%%%%%%%%%%%%%%%%%%%%%%%%%%%%%%%%%%%%
\subsection{Setup and statement of results}
\label{subsec:setupCluster}
Recalling Definition~\ref{def:fieldsZ}, for $b \in \N \cup \{\infty\}$, we introduce the field $\Phi_b$ on $\R^d$ given by 
\begin{equation}
\label{eq:Phi}
\Phi_{b}(\cdot) \eqdef - \int_0^b \bigl(1 - \frkK(e^{-s} \cdot)\bigr) dB_{s} + \rmZ_b(\cdot) \;,
\end{equation}
where $(B_s)_{s \geq 0}$ is a standard Brownian motion independent of the space-time white noise $\xi'$ used in the definition of the field $\rmZ_{b}$ introduced in Definition~\ref{def:fieldsZ}. We observe that the field $\Phi_b$ defined in \eqref{eq:Phi} is equal in law to the field introduced in Definition~\ref{def:Fields} (which justifies using the same notation), and we also recall that $\Upsilon_{\! b}$ is given by
\begin{equation}
\label{eq:defUpsilonBBeg}
\Upsilon_{\! b} = \Phi_{b} - \sqrt{\smash[b]{2d}} \frka_b \;.	
\end{equation}
We will carry our analysis in a slightly more general setting than what we specified in the introduction,
namely we allow the field $\Upsilon_{\! b}$ to be perturbed by a suitable ``well-behaved'' independent field. This will be needed for the proof of Theorem~\ref{th:stableConv}, and in particular for the proof of Proposition~\ref{pr:joint}.

For $b \in \N$, we consider an independent random field $\frkg_b$ on $\R^d$, which we fix for the reminder of this section and assume to satisfy the following properties:
\begin{enumerate}[start=1,label={{{(G\arabic*})}}]
\item \label{as:GG1} One has $\frkg_b(0) = 0$ almost surely. 
\item \label{as:GG2} For all $j \in [b]_0$, there exists a constant $c >0$ such that for all $\eta \geq 0$,
\begin{equation*}
\P\bigl(\M_{0, j}(\frkg_b) \geq \eta\bigr) \lesssim e^{-c e^{2(b-j)} \eta^2} \;.
\end{equation*} 
\item \label{as:GG3} There exist constants $c > 0$ such that for all $\eta \geq 0$,
\begin{equation*}
\P\Biggl(\sup_{i, k \in [d]} \M_{0,b}(\partial^2_{i, k} \frkg_b) \geq \eta\Biggr) \lesssim  e^{-c \eta^2} \;.
\end{equation*}
\end{enumerate}
For $b \in \N$, we then define the field $\Upsilon_{\! b, \frkg}$ on $\R^d$ by 
\begin{equation}
\label{e:defUpsbg}
\Upsilon_{\! b, \frkg}(\cdot)  = \Upsilon_{\! b}(\cdot) + \frkg_b(\cdot) \;.
\end{equation} 
In what follows, given $b \in \N$ and $x$, $y \in \R$, we use the convention that under $\P_{x, y, b}$, the law of $(B_s)_{s \in [0, b]}$ is that of a Brownian bridge from $x$ to $y$ in time $b$. Furthermore, $\E_{x, y, b}$ denotes the expectation with respect to $\P_{x, y, b}$.

The main goal of this section is to prove the following two propositions, which are analogous to Theorems~\ref{th:cluster}~and~\ref{th:clusterProb}, but with the Brownian motion replaced by the Brownian bridge.
\begin{proposition}
\label{pr:AsyConv}
For each $\lambda > 0$, there exists a continuous random field $\tilde \Upsilon_{\! \lambda}$ on $\R^d$ such that, for any function $\bfF \in \CC^b_{\loc}(\CC(\R^d))$, and for all $u \in [b^{1/4}, b^{3/4}]$, one has 
\begin{equation}
\label{eq:limitCondExp}
\E\bigl[\bfF\bigl(\tilde \Upsilon_{\! \lambda} \bigr)\bigr] 
= \lim_{b \to \infty} \E_{0, u, b}\bigl[\bfF\bigl(\Upsilon_{\! b, \frkg}\bigr) \, \big| \M_{0, b}(\Upsilon_{\! b, \frkg}) \leq \lambda \bigr] 
= \lim_{b \to \infty} \E\big[\bfF\bigl(\Upsilon_{\! \infty}\bigr) \, \big| \, \M_{0, b}(\Upsilon_{\! \infty}) \leq \lambda \bigr] \;.
\end{equation}
\end{proposition}

\begin{proposition}
\label{pr:AsyProb}
For $\lambda > 0$, let $c_{\star, \lambda}>0$ be the constant defined as follows
\begin{equation*}
c_{\star, \lambda} = \lim_{k\to \infty} \E \bigl[B_k \one_{\{B_k \in [k^{1/6}, k^{5/6}]\}} \one_{\{\M_{0, k}(\Upsilon_{\! k})  \leq \lambda\}}\bigr] \;,
\end{equation*}
then, for all $u \in [b^{1/4}, b^{3/4}]$, it holds that 
\begin{equation*}
\lim_{b \to \infty} \frac{b}{u} \P_{0, u, b}\bigl(\M_{0, b}(\Upsilon_{\! b, \frkg}) \leq \lambda\bigr) = 2 c_{\star, \lambda} \;.
\end{equation*}
\end{proposition}

%%%%%%%%%%%%%%%%%%%%%%%%%%%%%%%%%%%%%%%%%%%%%%
\subsection{Reduction to a Brownian motion}
\label{subsec:redBMCluster}
The main goal of this subsection is to show that the condition that the supremum of $\Upsilon_{\!b, \frkg}$ over the ball $\B_b$ is bounded above by $\lambda$ can essentially be rewritten as a condition on $(B_s)_{s \in [0, b]}$ appearing in \eqref{eq:Phi}. The idea is to use a suitable decomposition across annuli of the field $\Phi_b$, which will be introduced below. 
The key feature of this decomposition is that, for any $j \in [b-1]$, the supremum of the field in the annulus $\B_{j+1} \setminus \B_j$ is given by the position of the driving Brownian motion at time $j$ (modulo a sign change), plus a remainder term whose tails we have good control over (see Figure~\ref{fig:deco_BM}). Hence, the condition that the supremum of $\Upsilon_{\!b, \frkg}$ over the ball $\B_b$ is bounded above by $\lambda$ can be recast in terms of the requirement that this driving Brownian motion stays above some polylogarithmic curve.

\subsubsection{Decomposition across annuli}
\label{subsub:decAnnuli}
We begin by defining the following sets 
\begin{equation*}
\A_0 \eqdef \B_1 \qquad \text{ and } \qquad \A_j \eqdef \B_{j+1} \setminus \B_{j}, \quad \forall \, j \in \N \;.
\end{equation*}
With this notation in place, for all $b \in \N \cup \{\infty\}$ and $x\in \B_b$, the field $\Phi_b$ in \eqref{eq:Phi} can be conveniently rewritten as follows
\begin{equation}
\label{eq:decoRestrictedShape}
\Phi_b(x) = - \int_0^b \bigl(1 - \frkK(e^{-s} x)\bigr) dB_{s} + \sum_{j = 0}^{b-1} \rmZ_{j}(x) \one_{\{x \in \A_j\}} +  \sum_{j = 0}^{b-1}\rmZ_{j,b}(x) \one_{\{x \in \A_j\}}  \;,
\end{equation}
where the fields $\rmZ_j$ and $\rmZ_{j,b}$ are as in Definition~\ref{def:fieldsZ}.
For every $b \in \N \cup \{\infty\}$ and $j \in [b-1]_0$, we aim to control the tails of the suprema of the fields $\rmZ_{j}$ and $\rmZ_{j, b}$ over $\A_j$. This will be the content of the next two lemmas.

\subsubsection{Controlling the tails}

Recall the definition \eqref{eq:recentering} of the recentering constants $\frkm_j$. 
 
 \begin{lemma}
\label{lm:controlX'}
There exists a constant $c > 0$ such that for all $j \in \N$, it holds that 
\begin{equation*}
\P\Biggl(\Biggl|\sup_{x \in \A_j} \rmZ_{j}(x) - \frkm_{j}\Biggr| \geq \eta \Biggr) \lesssim e^{-c \eta} \;, \qquad \forall \; \eta \geq 0\;.
\end{equation*} 
\end{lemma} 
\begin{proof}
For $j \in \N$ and $\eta \geq 0$, the probability in the lemma statement is equivalent to 
\begin{equation}
\label{eq:probEquivTrivial}
\P\Biggl(\Biggl|\sup_{x \in \B_1 \setminus \B_0} \rmZ_{j}(e^j x) - \frkm_{j}\Biggr| \geq \eta \Biggr)\;.
\end{equation}
On the other hand, since by \ref{hp_K2} the seed covariance function $\frkK$ is supported in $B(0, 1)$, we have that for all $x$, $y \in \B_1 \setminus \B_0$, it holds that 
\begin{equation*}
	\E\bigl[\rmZ_{j}(e^j x) \rmZ_{j}(e^j y)\bigr] = \int_0^j \frkK\bigl(e^{s}(x-y)\bigr) ds \;.
\end{equation*} 
In other words, the field $\rmZ_{j}(e^j \cdot)$ restricted to the annulus $\B_1 \setminus \B_0$ has the same law as the martingale approximation at level $j$ of a $\star$-scale invariant field with seed covariance kernel $\frkK$. 
Hence, if the annulus $\B_1 \setminus \B_0$ in \eqref{eq:probEquivTrivial} is replaced by the $d$-dimensional unit box $[0,1]^d$, then this tightness result follows from \cite[Theorem~1.1]{Acosta}. To deduce the tightness of supremum of $(\rmZ_{ j}(e^{j} \cdot))_{j \in \N}$ over the annulus $\B_1 \setminus \B_0$ from the one over the box $[0,1]^d$, one can simply note that $\B_1 \setminus \B_0$ contains and is contained in a box of order one.
\end{proof}

Regarding the field $\rmZ_{j, b}$, we have the following bound on its supremum over annuli of radii smaller than $j$.
\begin{lemma}
\label{lm:controlZ}
There exists a constant $c > 0$ such that for $b \in \N \cup \{\infty\}$, $j \in [b-1]_0$, and $l \in [j]_0$, it holds that 
\begin{equation*}
\P\Biggl(\sup_{x \in \A_l} \abs{\rmZ_{j, b}(x)} \geq \eta \Biggr) \lesssim e^{-c e^{2(j-l)} \eta^2} \;, \qquad \forall \; \eta \geq 0\;.
\end{equation*} 
\end{lemma}
\begin{proof}
The result follows by a standard application of Fernique's majorizing criterion (Lemma~\ref{lm_Fernique}) and Borell-TIS inequality (Lemma~\ref{lm_Borell}). 
We only detail the case $b=\infty$, as the case $b \in \N$ is completely analogous. 
Fix $j \in \N$ and $l \in [j]$ and note that the probability in the statement
equals
\begin{equation*}
\P\Biggl(\sup_{x \in \B_1 \setminus \B_0} \abs{\rmZ_{j, \infty}(e^l x)} \geq \eta \Biggr)\;.	
\end{equation*}
A simple computation based on \ref{hp_K1} \dash \ref{hp_K2} yields that, for all $x$, $y \in \B_1 \setminus \B_0$, it holds that 
\begin{equation}
\label{eq:boundZblSquared}
\E\bigl[\abs{\rmZ_{j, \infty}(e^l x) -  \rmZ_{j, \infty}(e^l y)}^2\bigr] 
% =  2 \int_{0}^{b-l} \bigl(1 - k(e^{-s}(x-y))^2\bigr) ds - \int_{0}^{b-l} \bigl(k(e^{-s}x) - k(e^{-s}y)\bigr)^2 ds \\
\leq 2 \int_{0}^{\infty} \bigl(1 - \frkK(e^{-(s+j-l)}(x-y))\bigr) ds \lesssim \abs{x-y}^2 \;.
\end{equation}
Therefore, an immediate application of Fernique's majorizing criterion (Lemma~\ref{lm_Fernique}) shows that 
\begin{equation*}
\E\Biggl[\sup_{x \in \B_1 \setminus \B_0} \abs{\rmZ_{j, \infty}(e^l x)}\Biggr] \lesssim 1 \;,
\end{equation*}
for some universal implicit constant. 
The conclusion then follows by Borell-TIS inequality (Lemma~\ref{lm_Borell}) and thanks to the fact that 
\begin{equation*}
	\sup_{x \in \B_1 \setminus B_0} \E\bigl[\rmZ_{j, \infty}(e^l x)^2\bigr] = \int_0^{\infty} \bigl(1 - \frkK(e^{-(s + j - l)} x)^2\bigr) ds \lesssim e^{-2(j-l)}\;,
\end{equation*}
where, once again, we used \ref{hp_K1} \dash \ref{hp_K2}, and the implicit constant is independent of the quantities of interest. 
Finally, we remark that the cases $j = 0$ and $j \in \N$ with $l  = 0$ can be treated similarly. 
\end{proof}

\begin{lemma}
\label{lm:mlA}
There exists a constant $c > 0$ such that for all $b \in \N \cup \{\infty\}$ and $j \in [b-1]$, it holds that
\begin{equation*}
\sup_{x \in \A_j} \abs{\sqrt{\smash[b]{2d}} \frka_b(x) - \frkm_{j}} \leq c + \frac{3}{2 \sqrt{\smash[b]{2d}}} \log\bigl(j\bigr) \;.
\end{equation*}
\end{lemma}
\begin{proof}
Fix $b \in \N$ and $j \in [b-1]$. Then, thanks to \ref{hp_K2}, we have that, uniformly over all $x \in \A_j$, there exists a constant $c > 0$ such that
\begin{equation*}
\bigl|\sqrt{\smash[b]{2d}} \frka_b(x)  - \sqrt{\smash[b]{2d}} j\bigr| =  \Bigl| \int_j^b \bigl(1 - \frkK(e^{-s} x)\bigr) ds\Bigr| \leq c\;.
\end{equation*}
The conclusion follows by recalling the exact expression for $\frkm_{j}$ given in \eqref{eq:recentering}. The case $b = \infty$ and $j \in \N$ can be treated analogously. 
\end{proof}

\subsubsection{Control variables}
In what follows, given $j \in \N_0$ and a function $f:\R^{+}_0 \to \R$, we introduce the following notation,
\begin{equation*}
\Osc_j(f) \eqdef \sup_{s \in [j, j+1]} f(s) -  \inf_{s \in [j, j+1]} f(s) \;.
\end{equation*}
For $j$, $k \in \N_0$, we define 
\begin{equation}
\label{eq:Thetak}
\Theta_k(j) \eqdef \big[\log\bigl(1 + (k \vee j)\bigr)\bigr]^2 \;.
\end{equation}
For $b \in \N \cup \{\infty\}$, we now introduce the control variable $\rmK_b$ which will play an instrumental role in our analysis.

\begin{definition}
\label{def:ControlVar}
For $b \in \N \cup \{\infty\}$, we let $\rmK_{b}$ be the smallest $k \in [b-1]$\footnote{\label{ftn:bInf}With a slight abuse of notation, if $b = \infty$, then $[b-1] = \N$ and $[b -1]_0 = \N_0$.} such that:
\begin{enumerate}[start=1,label={{{(\arabic*})}}]
\item For each $j \in [b-1]_0$\cref{ftn:bInf}, it holds that $\Osc_j(B) \leq \Theta_k(j)$.
\item For each $j \in [b-1]$, it holds that $\abs{\sup_{x \in \A_j} \rmZ_{j}(x) - \frkm_{j}} \leq \Theta_k(j)$.
\item For each $j \in [b-1]_0$ and $l \in [j]_0$, it holds that $\sup_{x \in \A_l}\abs{\rmZ_{j, b}(x)} \leq e^{-(j-l)/2} \Theta_k(j)$.
\item If $b \neq \infty$, for each $j \in [b-1]_0$, it holds that $\sup_{x \in \A_j} \abs{\frkg_{b}(x)} \leq e^{-(b-j)/2}\Theta_k(j)$.
\end{enumerate}
If no such $\rmK_b$ exists, then we set $\rmK_b = b$. 
\end{definition}

\begin{remark}
\label{rm:diffFields}
For each $b \in \N \cup \{\infty\}$, $j \in [b-1]_0$, and $l \in [j]_0$, recalling \eqref{eq:Phi}, we have, for all $x \in \B_{l}$, 
\begin{equation*}
\Upsilon_{\! b, \frkg}(x) - \Upsilon_{\! j, \frkg}(x) = - \int_{j}^{b} \bigl(1 - \frkK(e^{-s} x)\bigr) dB_{s} + \rmZ_{j, b}(x) \;.
\end{equation*}
Now, on the event $\{\rmK_b \leq j\}$, one has
\begin{equation*}
	\sup_{x \in \B_l} \abs{\rmZ_{j, b}(x)} \lesssim e^{-(j-l)/2} \bigl(\log j\bigr)^2 \;,
\end{equation*}
as well as
\begin{equation*}
\sup_{x \in \B_l} \Biggl|\int_{j}^{b} 1 - \frkK(e^{-s} x) dB_{s} \Biggr|	\lesssim e^l\sum_{i=j}^{b-1} e^{-i} \Osc_i(B) \leq  e^{-(j-l)}\sum_{i=0}^{\infty} e^{-i} \Theta_{j}(i+j) \lesssim e^{-(j-l)} \bigl(\log j\bigr)^2 \;.
\end{equation*}

Therefore, combining the previous two bounds, we obtain that for all $l \in [j]_0$, on the event $\{\rmK_b \leq j\}$, it holds that  
\begin{equation}
\label{eq:diffFields}
\sup_{x \in \B_l} \abs{\Upsilon_{\! b, \frkg}(x) - \Upsilon_{\! j, \frkg}(x)}  \lesssim  e^{-(j-l)/2} (\log j)^2 \;.
\end{equation} 
\end{remark}

Thanks to Lemmas~\ref{lm:controlX'}~and~\ref{lm:controlZ}, we have the following result concerning the tail behaviour of the control variables. 
\begin{lemma}
\label{lm:Kk}
There exists a constant $c>0$ and $k_0 \in \N$ such that for $b \in \N$ with $b > k_0$, and $u \in [b^{1/4}, b^{3/4}]$,
\begin{equation}
\label{eq:boundKFinite}
	\P_{0, u, b}\bigl(\rmK_b = k\bigr) \leq e^{-c (\log k)^2} \;, \qquad \forall \, k \in \{k_0, \ldots, b\} \;. 
\end{equation}
Similarly, there exists a constant $\tilde c >0$ and $\tilde k_0 \in \N$ such that
\begin{equation}
\label{eq:boundKInfinite}
\P\bigl(\rmK_{\infty} =  k\bigr) \leq e^{-\tilde c (\log k)^2}\;, \qquad \forall \, k \geq \tilde k_0 \;.
\end{equation}
\end{lemma}
\begin{proof}
We start with \eqref{eq:boundKFinite}. Fix $b \in \N$, $u \in [b^{1/4}, b^{3/4}]$, and let $k \in [b]$. By Definition~\ref{def:ControlVar}, the event $\{\rmK_b = k\}$ is contained in the union of the following events,
\begin{gather*}
\bigcup_{j = 0}^{b-1} \bigl\{\Osc_j(B) > \Theta_{k-1}(j) \bigr\} \;, \qquad \bigcup_{j = 0}^{b-1}  \Biggl\{\sup_{x \in \A_j} \abs{\frkg_b(x)} > e^{-(b-j)/2} \Theta_{k-1}(j) \Biggr\} \;,\\ 
\bigcup_{j = 1}^{b-1}  \Biggl\{\Biggl\lvert \sup_{x \in \A_j} \rmZ_{j}(x) - \frkm_{j}\Biggr\rvert > \Theta_{k-1}(j) \Biggr\}\;, \qquad  \bigcup_{j = 1}^{b-1} \bigcup_{l = 1}^{j} \Biggl\{\sup_{x \in \A_l} \abs{\rmZ_{j, b}(x)} > e^{-(j-l)/2} \Theta_{k-1}(j) \Biggr\} \;.
\end{gather*}
The conclusion then follows since there exists constants $c_1$, $c_2 > 0$ such that
the probabilities of the events appearing in the unions of the above display 
are bounded either by $c_1\exp(-c_2 \Theta_{k-1}(j)^2)$ (for the ones on the first line)
or by $c_1 \exp(- c_2\Theta_{k-1}(j))$ (for the ones on the second line). Summing over $j$ 
then yields the desired bound.

Indeed, for the events in the union on the top-left, this follows since for all $u \in [b^{1/4}, b^{3/4}]$, the oscillation norms $\Osc_j(B)$ have Gaussian tails uniformly over the probability laws $\P_{0, u, b}$, for all $j \in [b-1]_0$. For the event in the union on the top-right this follows from \ref{as:GG2}.
Finally, regarding the events in the union on the bottom-left and bottom-right, this follows by Lemma~\ref{lm:controlX'} and Lemma~\ref{lm:controlZ}, respectively. To conclude, we note that the proof of \eqref{eq:boundKInfinite} proceeds in the same exact way.
\end{proof}

We are now ready to state and prove the following key lemma. For a diagrammatic representation related to this lemma, we refer to Figure~\ref{fig:deco_BM}. 
\begin{lemma}[Approximation by a Brownian motion]
\label{lm:approxBrownianBridge}
There exists a constant $\rmC > 0$ such that for all $b \in \N \cap \{\infty\}$ and all $j \in [b-1]$, on the event $\{\rmK_b < b\}$,
\begin{equation}
\label{eq:ReductionFinite}
\Biggl\lvert\sup_{x \in \A_j} \Upsilon_{\! b, \frkg}(x) +  B_{j}\Biggr\rvert \leq \rmR_{\rmK_b}(j) \;,
\end{equation}
where $\rmR_{k}(j) \eqdef \rmC\bigl(1 + \Theta_{k}(j)\bigr)$. 
\end{lemma}
\begin{proof}
Let $b \in \N$, $j \in [b-1]$, and $x \in \A_j$. Now, recalling \eqref{eq:decoRestrictedShape}, we can write
\begin{equation*}
\Upsilon_{\! b, \frkg}(x) = - B_{j} - \int_{j}^b \bigl(1 - \frkK(e^{-s} x)\bigr) dB_{s} + \rmZ_{j}(x) + \rmZ_{j, b}(x) - \sqrt{\smash[b]{2d}} \frka_b(x) + \frkg_b(x)\;,
\end{equation*}
where we used the fact that thanks to \ref{hp_K2}, the function $\frkK$ is supported in $B(0, 1)$.
In particular, using the triangle inequality, this implies that
\begin{align*}
& \Biggl\lvert\sup_{x \in \A_j}{\Upsilon_{\! b, \frkg}}(x) +  B_{j} \Biggr\rvert \leq \sup_{x \in \A_j} \Biggl\lvert\int_{j}^b \bigl(1 - \frkK(e^{-s} x)\bigr) dB_{s}\Biggr\rvert + \Biggl\lvert\sup_{x \in \A_j} \rmZ_{j}(x) - \frkm_{j}\Biggr\rvert \\
& \hspace{50mm}+ \sup_{x \in \A_j}\abs{\rmZ_{j, b}(x)} + \sup_{x \in \A_j} \abs{\sqrt{\smash[b]{2d}} \frka_b(x) - \frkm_{j}} + \sup_{x \in \A_j} \abs{\frkg_b(x)} \;.
\end{align*}
Now, on the event $\{\rmK_b< b\}$, by Definition~\ref{def:ControlVar} of the control variable $\rmK_b$ and by Lemma~\ref{lm:mlA}, there exists a constant $\rmC > 0$ such that the last four terms on the right-hand side of the above display are bounded from above by a quantity of the form $\rmC\bigl(1 + \Theta_{\rmK_b}(j)\bigr)$.
Therefore, it remains to check that a similar bound also holds for the first term. Again, by Definition~\ref{def:ControlVar} of $\rmK_b$, this follows by a simple computation. Indeed, we have that
\begin{align*}
\sup_{x \in \A_j} \Biggl|\int_{j}^b \bigl(1 - \frkK(e^{-s} x)\bigr) dB_{s}\Biggr|
\lesssim e^{j+1} \sum_{m = j}^{b - 1} e^{-m} \Osc_m(B)
\leq e^{j+1} \sum_{m=j}^{\infty} e^{-m} \Theta_{\rmK_b}(m) \;,
\end{align*}
where the implicit constant is independent of everything else. The quantity on the right-hand side of the above display can be clearly bounded by a quantity of the form $\rmC(1+\Theta_{\rmK_b}(j ))$, for some constant $\rmC >0$. Hence, the estimate in \eqref{eq:ReductionFinite} follows. Finally, we note that if $b = \infty$, then the proof is completely analogous. 
\end{proof}

We record here some useful inclusions that are immediate consequences of Lemma~\ref{lm:approxBrownianBridge} and will be used several times in the remainder of this section.
In particular, recalling the notation introduced in \eqref{eq:maximalExpBalls} and \eqref{eq:maximalExpAnnulus}, and using the same notation as in the previous lemma, for each $\lambda > 0$, thanks to \eqref{eq:ReductionFinite}, one can see that for all $b \in \N$ and $k \in [b-1]$, the following inclusions hold
\begin{align}
\bigl\{\rmK_b < b \bigr\} \cap \bigl\{\M_{0, b}(\Upsilon_{\! b, \frkg}) \leq \lambda \bigr\} & \subseteq \bigcap_{j = 1}^{b-1}\bigl\{B_{j} \geq - \lambda - \rmR_{\rmK_b}(j) \bigr\} \;, \label{eq:Inc1Fin} \\
\bigl\{\rmK_b < b \bigr\} \cap \bigcap_{j = k}^{b-1} \bigl\{B_{j} \geq \lambda + \rmR_{\rmK_b}(j)\bigr\} & \subseteq \bigl\{\M_{0, b, k}(\Upsilon_{\! b, \frkg}) \leq \lambda\bigr\} \;. \label{eq:Inc2Fin}
\end{align}
Similar inclusions hold also for the field $\Upsilon_{\! \infty}$. Indeed, thanks to \eqref{eq:ReductionFinite}, for all $b \in \N$ and $k \in [b-1]$, we have that 
 \begin{align}
\bigl\{\rmK_{\infty} < \infty \bigr\}\cap \bigl\{\M_{0, b}(\Upsilon_{\! \infty}) \leq \lambda\bigr\} & \subseteq \bigcap_{j = 1}^{b-1}\bigl\{B_{j} \geq - \lambda -  \tilde \rmR_{\tilde \rmK_b}(j) \bigr\} \;, \label{eq:Inc1Inf} \\
\bigl\{\rmK_{\infty} < \infty \bigr\} \cap \bigcap_{j = k}^{b-1} \bigl\{B_{j} \geq \lambda + \tilde \rmR_{\tilde \rmK_b}(j)\bigr\} & \subseteq \bigl\{\M_{0, b, k} \Upsilon_{\!\infty}) \leq \lambda\bigr\} \;. \label{eq:Inc2Inf}
\end{align}

\begin{figure}[ht]
\centering
\includegraphics[width=\textwidth, keepaspectratio]{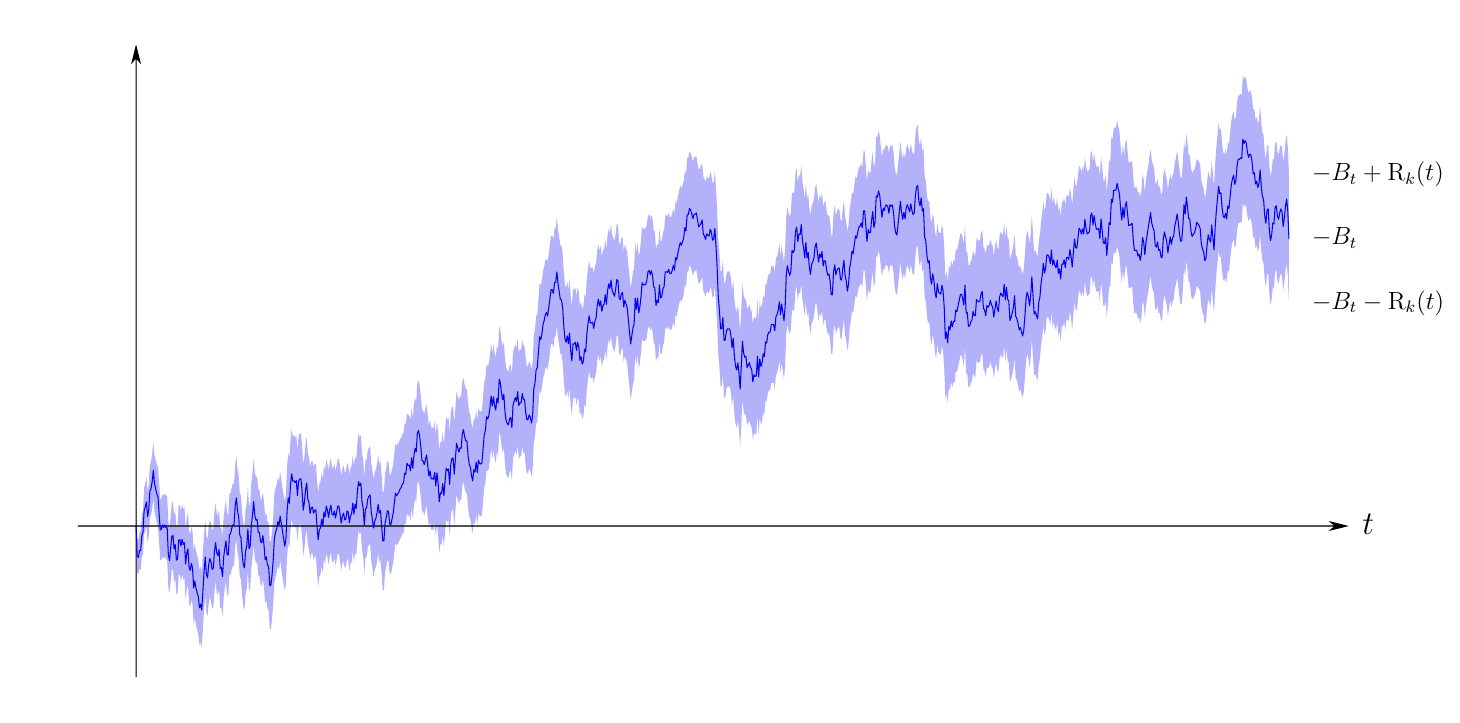}
\caption[short form]{\small{The blue curve represents a standard Brownian motion $-B$ run up to time $b$. The light blue region around the Brownian motion is the area enclosed between the curves $[0, b] \ni t \mapsto -B_t - \rmR_k(t)$ and $[0, b] \ni t \mapsto -B_t + \rmR_k(t)$ for some $0 \ll k \ll b$. Roughly speaking, Lemma~\ref{lm:approxBrownianBridge} states that, with high probability, for any $j \in [b-1]$, the supremum of the field $\Upsilon_{\! b, \frkg}(x)$ in the annulus $\A_j$ lies on the vertical segment at $t = j$ within the light blue region.}}
\label{fig:deco_BM}
\end{figure}

%%%%%%%%%%%%%%%%%%%%%%%%%%%%%%%%%%%%%%%%%%%%%%
\subsection{Some technical lemmas}
\label{subsec:techLemmasCluster}
In this subsection we collect some technical results that are needed for the proofs of Propositions~\ref{pr:AsyConv}~and~\ref{pr:AsyProb}. Before proceeding, we emphasise that all the following lemmas have an analogous counterpart in \cite[Section~4]{BiskupLouidor}.
We begin with the following lemma, which is a slightly augmented version of \cite[Lemma~4.20]{BiskupLouidor}, adapted to our setting. For its statement, given $k \in \N_0$, we use the notation
\begin{equation*}
\CG_k \eqdef \sigma\bigl((B_s)_{s\leq k},\, \rmZ_{\infty}\bigr)	\;.
\end{equation*}

\begin{lemma}
\label{lm:upperBoundKk}
There exists a constant $c > 0$ such that for $\lambda > 0$, for all $b \in \N$ sufficiently large, $u \in [b^{1/4}, b^{3/4}]$, $k \in [b-1]$, and any event $\rmA_k \in \CG_k$,
\begin{equation}
\label{eq:tech1FinAsy}
\P_{0, u, b}\Biggl(\rmK_b = k, \;  \bigcap_{j = k}^{b-1} \bigl\{B_{j} \geq - \lambda - \rmR_{k}(j) \bigr\}, \; \rmA_k \Biggr) \lesssim \f{u}{b} e^{-c (\log k)^2}  \sqrt{\P_{0, u, b}(\rmA_k)}\;.
\end{equation}
Similarly, there exists a constant $\tilde c > 0$ such that for $\lambda > 0$, for all $b \in \N$ sufficiently large, $k \in [b-1]$, and any event $\rmA_k \in \CG_k$,
\begin{equation}
\label{eq:tech1InfAsy}
\P\Biggl(\rmK_{\infty} = k, \; \bigcap_{j=k}^{b-1} \bigl\{B_{j} \geq - \lambda - \tilde \rmR_{k}(j) \bigr\}, \; \rmA_k \Biggr) \lesssim \f{1}{\sqrt{b}} e^{-\tilde c (\log k)^2} \sqrt{\P(\rmA_k)}\;.
\end{equation}
\end{lemma}
\begin{proof}
For simplicity, and without loss of generality, we set $\lambda = 1$.
The proof of this lemma follows a similar approach to the proof of \cite[Lemma~4.20]{BiskupLouidor}. 
We start with the proof of \eqref{eq:tech1FinAsy}. Given $b \in  \N$ and $k \in [b-1]$, we let $\rmE_k$ be the event that the conditions in Definition~\ref{def:ControlVar} hold for all $j \in [k]$ (or $j \in [k]_0$) with $\Theta_k(j)$, but at least one of these conditions is not satisfied if $\Theta_k(j)$ is replaced by $\Theta_{k-1}(j)$. We note that $\rmE_k \in \CG_k$ and that $\{\rmK_b = k\} \subset \rmE_k$ by definition of $\rmK_b$. Hence, the probability on the left-hand side of \eqref{eq:tech1FinAsy} can be bounded from above by
\begin{equation*}
	\E_{0, u, b}\Biggl[\one_{\rmE_k \cap \rmA_k} \P_{0, u, b}\Biggl(\bigcap_{j=k}^{b-1} \bigl\{B_j \geq - 1 - \rmR_k(j)\bigr\} \, \Bigg| \, \CG_k\Biggr)\Biggr]\;.
\end{equation*}
Now, we consider the function $\zeta:\R_0^{+} \to \R_0^{+}$ given by
\begin{equation*}
\zeta(s) \eqdef \rmC\bigl[1 + \log(\rmC+s)\bigr]^2 \;,
\end{equation*}
where $\rmC$ is the constant appearing in the definition of $\rmR_k$ which, without any loss of generality, we can assume to be large enough so that $\zeta$ is increasing  and concave on $\R^{+}_{0}$. Then on the event $\{B_k = z\}$ for some $z \in \R$, again by possibly enlarging the constant $\rmC$ in the definition of $\zeta$, thanks to the Markov property of the Brownian bridge, it holds that
\begin{align*}
\P_{0, u, b}\Biggl(\bigcap_{j=k}^{b-1} \bigl\{B_j \geq -1 - \rmR_k(j)\bigr\} \, \Bigg| \, \CG_k\Biggr) \leq \P_{z, u, b-k}\Biggl(\bigcap_{j = k}^{b-1} \bigl\{B_j \geq - \zeta(j) \bigr\}\Biggr) \;.
\end{align*}
Now, thanks to Lemma~\ref{lm:transferBBRW}, we have that 
\begin{equation}
\label{eq:ProbIncrx}
\P_{z, u, b-k}\Biggl(\bigcap_{j = k}^{b-1} \bigl\{B_j \geq - \zeta(j) \bigr\}\Biggr) \lesssim \P_{z, u, b-k}\Biggl(\inf_{s \in [0, b-k]} \bigl(B_s + 2\zeta(k+s)\bigr) \geq 0\Biggr)\;,
\end{equation}
where the implicit constant is independent of everything else.
Since by definition one has $\Osc_j(B) \leq \Theta_k(j)$ on the event $\rmE_k$, for all $j \in [k]_0$, we can assume that 
\begin{equation*} 
z \in \bigl[-\zeta(k), (k+1)\bigl(\log(k+1)\bigr)^2\bigr]\;.
\end{equation*}
We let $a_k \eqdef \zeta(k) \vee ((k+1) (\log(k+1))^2)$. Since the probability on the right-hand side of \eqref{eq:ProbIncrx} is increasing in $z \in \R$, we can estimate this probability for $z \in [a_k, 3 a_k]$. It follows from Proposition~\ref{pr:BoundBBAboveNeg} that   
\begin{equation}
\label{eq:keyLemmaKeyTech}
\sup_{z \in [a_k, 3 a_k]}\P_{z, u, b-k}\Biggl(\inf_{s \in [0, b-k]} \bigl(B_s + 2\zeta(k+s)\bigr) \geq 0\Biggr) \lesssim a^2_k \frac{u}{b-k}\;.
\end{equation}
We note that the presence of $a_k^2$ instead of $a_k$ in the above expression \dash as one might expect from Proposition~\ref{pr:BoundBBAboveNeg} \dash is due to the first summand in the error term \eqref{eq:ErrorTermNeg}. Furthermore, we observe that a direct application of Lemma~\ref{lm:techRhoAK} shows that the remaining summands in the error term \eqref{eq:ErrorTermNeg} can be bounded uniformly over all $k \in [b-1]$ and $z \in [a_k, 3 a_k]$.

Therefore, combining the previous considerations, so far we have proved that on the event $\rmE_k$, it holds that 
\begin{equation*}
\P_{0, u, b}\Biggl(\bigcap_{j=k}^{b-1} \bigl\{B_j \geq - 1 - \rmR_k(j)\bigr\} \, \Bigg| \, \CG_k\Biggr)  \lesssim a^2_k \frac{u}{b-k} \;. 
\end{equation*}
On the other hand, thanks to Cauchy--Schwarz's inequality and arguing as in the proof of Lemma~\ref{lm:Kk}, we have that $\smash{\P_{0, u, b}(\rmE_k \cap \rmA_k) \leq e^{-c_1 (\log k)^2} \sqrt{\P_{0, u, b}(\rmA_k)}}$, for some constant $c_1 > 0$. Hence, by absorbing the factors $a^2_k$ and $b/(b-k)$ inside the exponential, we conclude that there exists a constant $c_2 > 0$ such that 
\begin{equation*}
\P_{0, u, b}\Biggl(\rmK_b = k, \; \bigcap_{j = k}^{b-1} \bigl\{B_{j} \geq -1 - \rmR_{k}(j) \bigr\}, \; \rmA_k \Biggr) \lesssim \frac{u}{b} e^{-c_2 (\log k)^2} \sqrt{\P_{0, u, b}(\rmA_k)}  \;,
\end{equation*}
from which the conclusion follows. Finally, the proof of \eqref{eq:tech1InfAsy} is obtained in a very similar way 
and hence is omitted.
\end{proof}

In what follows, we also require the following version of Lemma~\ref{lm:upperBoundKk}, where the endpoint of the Brownian bridge is constrained to be less than $b^{1/4}$.
\begin{lemma}
\label{lm:upperBoundKkSmallU}
Let $A > 0$ be fixed. For any $\lambda > 0$, there exists a constant $c > 0$ such that, for all $b \in \N$ sufficiently large, $u \in [-A, b^{1/4}]$, $k \in [b-1]$, and any event $\rmA_k \in \CG_k$,
\begin{equation*}
\P_{0, u, b}\Biggl(\rmK_b = k, \;  \bigcap_{j = k}^{b-1} \bigl\{B_{j} \geq - \lambda - \rmR_{k}(j) \bigr\}, \; \rmA_k \Biggr) \lesssim \f{1}{\sqrt{b}} e^{-c (\log k)^2}  \sqrt{\P_{0, u, b}(\rmA_k)}\;.
\end{equation*}
\end{lemma}
\begin{proof}
The proof is almost identical to that of Lemma~\ref{lm:Kk}. Therefore, we only highlight the necessary changes.
First, for a fixed $A > 0$, we note that \eqref{eq:boundKFinite} in Lemma~\ref{lm:Kk} also holds uniformly for all $u \in [-A, b^{1/4}]$. Then, the only thing that needs to be changed is the bound \eqref{eq:keyLemmaKeyTech}. Indeed, Proposition~\ref{pr:BoundBBAboveNeg} can only be applied for $u \in [b^{\iota}, b^{3/4}]$ for some $\iota \in (0, 1/8)$. However, we can easily overcome this issue by using monotonicity and replacing the endpoint $u$ in the probability on the left-hand side of \eqref{eq:keyLemmaKeyTech} with $u + b^{\iota}$ for some $\iota \in (0, 1/8)$. This allows us to apply Proposition~\ref{pr:BoundBBAboveNeg}, from which we deduce that the probability on the left-hand side of \eqref{eq:keyLemmaKeyTech} is bounded above by a multiple of $a_k^2 (u + b^{\iota})/(b-k)$. The conclusion then follows by proceeding with the remaining part of the argument in the proof of Lemma~\ref{lm:Kk} and recalling that $u < b^{1/4}$.
\end{proof}

The following lemma is analogous to \cite[Lemma~4.21]{BiskupLouidor} in our setting. 
\begin{lemma}
\label{lm:upperBoundKkUpsilon}
For any $\lambda > 0$, there exist constants $c_1$, $c_{2} > 0$ such that, for $b \in \N$ large enough and $u \in [b^{1/4}, b^{3/4}]$,
\begin{equation}
\label{eq:upperBoundKkFinite}
c_1\f{u}{b} \leq \P_{0, u, b}\bigl(\M_{0, b}(\Upsilon_{\! b, \frkg}) \leq \lambda \bigr) \leq c_{2} \frac{u}{b}\;.
\end{equation}
Similarly, there exist constants $\tilde c_1, \tilde c_{2} > 0$ such that for $\lambda > 0$, for $b \in \N$ large enough, 
\begin{equation}
\label{eq:upperBoundKkInfinite}
\tilde c_1 \f{1}{\sqrt{b}} \leq	\P\bigl(\M_{0, b}(\Upsilon_{\! \infty}) \leq \lambda \bigr)\leq \tilde c_{2} \f{1}{\sqrt{b}} \;.
\end{equation}
\end{lemma}
\begin{proof}
We start with the upper bound in \eqref{eq:upperBoundKkFinite}. Thanks to \eqref{eq:Inc1Fin}, it holds that 
\begin{equation*}
\bigl\{\M_{0, b}(\Upsilon_{\! b, \frkg}) \leq \lambda \bigr\} \subseteq \bigl\{\rmK_b = b\bigr\} \cup \Biggl\{\rmK_b < b, \; \bigcap_{j = 1}^{b-1}\bigl\{B_{j} \geq - \lambda - \rmR_{\rmK_b}(j)\bigr\}\Biggr\} \;.
\end{equation*}
Thanks to Lemmas~\ref{lm:Kk}~and~\ref{lm:upperBoundKk}, the probability of the event on the right-hand side of the above display is less than than a multiple of
\begin{equation*}
\f{u}{b} \sum_{k = 1}^{b-1} e^{-c_1 (\log k)^2} + e^{-c_2 (\log b)^2} \;,
\end{equation*}
for some constants $c_1$, $c_2 > 0$. Therefore, the desired bound follows. We observe that the upper bound in \eqref{eq:upperBoundKkInfinite} can be deduced in a similar way. 

We now focus on the lower bounds. In particular, we limit ourself to study the lower bound in \eqref{eq:upperBoundKkFinite}, since the one in \eqref{eq:upperBoundKkInfinite} can be obtained similarly.
For $k$, $b \in \N$ such that $k < b$, let $\rmE_k$ be the event that all conditions in the definition of the control variable $\rmK_b$ are satisfied with $\Theta_k(\cdot)$, except that we do not impose any requirements on the oscillation bounds for the Brownian motion for time indices up to and including time $k$. In particular, the event $\rmE_k$ is conditionally independent from $\sigma((B_s)_{s \leq k})$ given $\sigma(B_k)$. We have the following lower bound
\begin{equation}
\label{eq:lowerCas1}
\P_{0, u, b}\bigl(\M_{0, b}(\Upsilon_{\! b, \frkg}) \leq \lambda\bigr) \geq \P_{0, u, b}\Biggl(\M_{0, k}(\Upsilon_{\! b, \frkg})  \leq \lambda,\; \rmE_k, \; \bigcap_{j = k}^{b-1}\bigl\{B_{j} \geq \lambda + \rmR_{k}(j) \bigr\}\Biggr) \;.
\end{equation}
where we used the fact that, thanks to Lemma~\ref{lm:approxBrownianBridge}, on the event $\rmE_k$\footnote{To be precise, the conclusion of Lemma~\ref{lm:approxBrownianBridge} holds on the event $\{\rmK_b\leq k\}$. However, as one can easily check, on the annulus $\B_b \setminus \B_k$ only the conditions on the oscillations of the Brownian motion after time $k$ are relevant.}, it holds that 
\begin{equation*}
\bigcap_{j = k}^{b-1}\bigl\{B_{j} \geq \lambda + \rmR_{k}(j) \bigr\} \subseteq \bigl\{\M_{0, b, k}(\Upsilon_{\! b, \frkg}) \leq \lambda\bigr\} \;.
\end{equation*}
We also observe that thanks to \eqref{eq:diffFields}, on the event $\rmE_k$, there exists a constant $c >0$ such that
\begin{equation*}
\bigl\{\M_{0, k}(\Upsilon_{\! k, \frkg})  \leq \lambda - c (\log k)^2\bigr\} \subseteq \bigl\{\M_{0, k}(\Upsilon_{\! b, \frkg}) \leq \lambda\bigr\} \;.
\end{equation*}
Therefore, the right-hand side of \eqref{eq:lowerCas1} can be lower bounded by the following probability 
\begin{equation*}
\P_{0, u, b}\Biggl(\M_{0, k}(\Upsilon_{\! k, \frkg}) \leq \lambda - c (\log k)^2,\;  \rmE_k, \; \bigcap_{j = k}^{b-1}\bigl\{B_{j} \geq \lambda + \rmR_{k}(j) \bigr\}\Biggr) \;.
\end{equation*}
We note that the first event inside the above probability is conditionally independent from the second and third events given $\sigma(B_{k})$. In particular, the conditional probability of the first event given $\{B_{k} = z\}$, for some $z \geq \lambda + \rmR_k(k)$, increases as $z$ increases. Therefore, it can bounded from below by a constant $c_1(k) > 0$. 

Since by definition $\{\rmK_b \leq k\} \subseteq \rmE_k$, to conclude it suffices to estimate the following probability 
\begin{equation}
\label{eq:lowerCas2}
\P_{0, u, b}\Biggl(\rmK_b \leq k,\; \bigcap_{j = k}^{b-1}\bigl\{B_{j} \geq \lambda + \rmR_{k}(j) \bigr\}\Biggr)\;.
\end{equation}
To this end, by using Lemma~\ref{lm:upperBoundKk}, we get that there exist constants $c_2$, $c_3 > 0$ such that 
\begin{equation*}
\P_{0, u, b}\Biggl(\rmK_b > k, \; \bigcap_{j = k}^{b-1} \bigl\{B_j \geq \lambda + \rmR_k(j)\bigr\}\Biggr) \leq  c_2 e^{-c_3(\log k)^2} \frac{u}{b} \;,
\end{equation*}
and so, the probability in \eqref{eq:lowerCas2} can be lower bounded by the following sum
\begin{equation*}
\P_{0, u, b}\Biggl(\bigcap_{j = k}^{b-1} \bigl\{B_j \geq \lambda + \rmR_{k}(j)\bigr\}\Biggr) - c_2 e^{-c_3(\log k)^2} \frac{u}{b}\;.
\end{equation*}
Furthermore, by using the lower bound in \eqref{eq:boundsBBstaysPos} and Lemma~\ref{lm:mainEntrRW}, we get that the above probability is lower bounded by
\begin{align*}
& \P_{0, u, b}\Biggl(\bigcap_{j = 1}^{b-1} \bigl\{B_j \geq -\lambda -\rmR_k(j)\bigr\}, \; \bigcap_{j = k}^{b-1} \bigl\{B_j \geq \lambda + \rmR_k(j)\bigr\}\Biggr)\\
& \geq \P_{0, u, b}\Biggl(\bigcap_{j = 1}^{b-1} \bigl\{B_j \geq 0\bigr\}\Biggr) - \P_{0, u, b}\Biggl(\bigcap_{j = 1}^{b-1} \bigl\{B_j \geq -\lambda -\rmR_k(j)\bigr\}, \; \bigcup_{j = k}^{b-1} \bigl\{B_j \leq \lambda + \rmR_k(j)\bigr\}\Biggr) \\
& \geq \P_{0, u, b}\Biggl(\inf_{s \in [1, b]} B_s \geq 0\Biggr) - \P_{0, u, b}\Biggl(\bigcap_{j = 1}^{b-1} \bigl\{B_j \geq -\lambda -\rmR_k(j)\bigr\}, \; \bigcup_{j = k}^{b-1} \bigl\{B_j < \lambda + \rmR_k(j)\bigr\}\Biggr) \\
& \geq c_4 \frac{u}{b}\bigl(1 - k^{-\f{1}{16}}\bigr) \;,
\end{align*}
for some constant $c_4 > 0$. Hence, putting everything together, we showed that there exist constants $c_1(k)$, $c_2$, $c_3$, $c_4 > 0$ such that 
\begin{equation*}
\P_{0, u, b}\bigl(\M_{0, b}(\Upsilon_{\! b, \frkg}) \leq \lambda\bigr)  \geq c_1(k) \frac{u}{b} \bigl(c_4 (1 - k^{-\f{1}{16}}) - c_2 e^{-c_3(\log k)^2}\bigr) \;,
\end{equation*}
from which the claim follows. 
\end{proof}

We finish this subsection with the following lemma.
\begin{lemma}
\label{lm:entroDelta}
For all $k \in \N$, $\eps > 0$, and $\lambda > 0$, there exists $\delta \in (0, \lambda)$ such that for all $b \geq k$ sufficiently large and $\eta \in \{0, \delta\}$, it holds that
\begin{equation}
\label{eq:entr}
\P_{0, u, b}\bigl(\M_{0, k}(\Upsilon_{\! b, \frkg}) \geq \lambda - \delta + \eta, \; \M_{0, b}(\Upsilon_{\! b, \frkg}) \leq \lambda + \eta\bigr) \leq \eps \frac{u}{b} \;.
\end{equation}
Similarly, for all $k \in \N$, $\eps > 0$, and $\lambda > 0$ there exists $\delta \in (0, \lambda)$ such that for all $b \geq k$ sufficiently large and $\eta \in \{0, \delta\}$, it holds that
\begin{equation}
\label{eq:entrBM}
\P\bigl(\M_{0, k}(\Upsilon_{\! \infty}) \geq \lambda - \delta + \eta,\; \M_{0, b}(\Upsilon_{\! \infty}) \leq \lambda + \eta \bigr) \leq \eps \f{1}{\sqrt{b}} \;.
\end{equation} 
\end{lemma}
\begin{proof}
The strategy for the proof of this lemma is quite similar to the strategy developed for the proof of Lemma~\ref{lm:upperBoundKk}. 
We will provide the proof only for the case $\eta = 0$, as the case $\eta = \delta$ is completely analogous. We fix $k \in \N$, $\eps > 0$, and $\lambda > 0$. 
For $\delta \in (0, \lambda/2)$, we begin by observing that, thanks to Lemmas~\ref{lm:approxBrownianBridge},~\ref{lm:upperBoundKk},~and~\ref{lm:mainEntrRW}, it holds that 
\begin{equation*}
	\P_{0, u, b}\bigl(\M_{0, b, k}(\Upsilon_{\! b, \frkg}) \geq \lambda - \delta\;, \M_{0, b}(\Upsilon_{\! b, \frkg}) \leq \lambda\bigr) \lesssim k^{-\f{1}{16}} \frac{u}{b}\;,
\end{equation*}
and so, if $k > b/2$ the claim follows by taking $b > 0$ large enough. Hence, from now on, we can focus on the regime $k < b/2$. Thanks to Lemma~\ref{lm:upperBoundKk}, by choosing $l > l_0$ for some $l_0 = l_0(\eps) > k$, we can assume that we are on the event $\{\rmK_b \leq l\}$ with $\rmK_b$ from Definition~\ref{def:ControlVar}. Now, for $\delta \in (0, \lambda/2)$, we observe that thanks to \eqref{eq:Inc1Fin},
\begin{align}
\label{eq:contRepulsion1Event}
& \bigl\{\M_{0, k}(\Upsilon_{\! b, \frkg}) \geq \lambda - \delta, \; \M_{0, b}(\Upsilon_{\! b, \frkg}) \leq \lambda, \; \rmK_b \leq l\bigr\} \nonumber\\
& \hspace{30mm} \subseteq \Biggl\{\M_{0, k}(\Upsilon_{\! b, \frkg}) \in [\lambda - \delta , \lambda], \; \bigcap_{j = l}^{b-1} \bigl\{B_j \geq - \lambda - \rmR_l(j)\bigr\}, \; \rmK_b \leq l\Biggr\} \;.
\end{align}
Thanks to \eqref{eq:diffFields}, on the event $\{\rmK_b \leq l\}$, by possibly taking $l > l_1$ for some $l_1 = l_1(\eps, \delta) > l_0$, we can assume that the event on the right-hand side in \eqref{eq:contRepulsion1Event} is contained in
\begin{equation*}
\bigl\{\M_{0, k}(\Upsilon_{\! l, \frkg}) \in [\lambda - 2\delta, \lambda + 2\delta]\bigr\} \cap \Bigg\{\bigcap_{j = l}^{b-1} \bigl\{B_j \geq - \lambda - \rmR_l(j)\bigr\}\Biggr\} \;.
\end{equation*}
These two events are conditionally independent given $\sigma(B_{l})$. In particular, on the event $\{B_{l} = z\}$ for some $z \geq -\lambda-\rmR_l(l)$, the first conditional probability is equal to
\begin{equation*}
\P_{0, z, l}\bigl(\M_{0, k}(\Upsilon_{\! l, \frkg}) \in [\lambda - 2\delta, \lambda + 2\delta]\bigr)	\;.
\end{equation*}
Now, thanks e.g.\ to \cite[Theorem~3.1]{SupDens}, we observe that by choosing $\delta = \delta(\eps, k) \in (0, \lambda/4)$ sufficiently small, it holds that
\begin{equation}
\label{eq:boundDenSupEntrop}
\sup_{z \geq -1-\rmR_l(l)}\P_{0, z, l}\bigl(\M_{0, k}(\Upsilon_{\! l, \frkg}) \in [\lambda - 2\delta, \lambda + 2\delta]\bigr) \leq \eps \;.
\end{equation}
To be precise, we cannot directly apply \cite[Theorem~3.1]{SupDens} since $\Upsilon_{\! l, \frkg}(0) = 0$. However, this issue can be easily overcome by noting that, by taking $\zeta = \zeta(\eps) > 0$ sufficiently small, it holds that
\begin{equation*}
\sup_{z \geq -1-\rmR_l(l)} \P_{0, z, l}\Biggl(\sup_{x \in B(0, \zeta)} \Upsilon_{\! l, \frkg}(x) \geq \lambda/2 \Biggr) \leq \eps \;,
\end{equation*}
To obtain \eqref{eq:boundDenSupEntrop}, we note that the variance of the field $\Upsilon_{\! l, \frkg}$ over the annulus $\B_k \setminus B(0, \zeta)$ can be uniformly lower bounded by a quantity depending on $\zeta$. Therefore, by \cite[Theorem~3.1]{SupDens}, the density with respect to the Lebesgue measure of the supremum of the field $\Upsilon_{\! l, \frkg}$ over the annulus $\B_k \setminus B(0, \zeta)$ is bounded above by a constant depending on $k$ and $\zeta$. Hence, \eqref{eq:boundDenSupEntrop} follows by taking $\delta = \delta(\eps, k) \in (0, \lambda/4)$ sufficiently small.

The conclusion then follows since by applying Lemma~\ref{lm:upperBoundKk}, one has that
\begin{equation*}
\P_{0, u, b}\Biggl(\bigcap_{j = l}^{b-1} \bigl\{B_j \geq - \lambda - \rmR_l(j)\bigr\}	\biggr) \lesssim \frac{u}{b}\;.
\end{equation*}
The proof of the bound \eqref{eq:entrBM} follows a similar approach and it is therefore omitted.
\end{proof}

%%%%%%%%%%%%%%%%%%%%%%%%%%%%%%%%%%%%%%%%%%%%%%
\subsection{Asymptotic formulas}
\label{subsec:asyFormulaCluster}
For any $\lambda > 0$, for $b$, $l \in \N$ with $b > l$, for $u \in [b^{1/4}, b^{3/4}]$, and for any function $\bfF \in \CC^b_{\loc}\bigl(\CC(\R^d)\bigr)$, we define the following quantities
\begin{align}
& \H_{l, \lambda}(\bfF) \eqdef \E\Bigl[\bfF(\Upsilon_{\! l}) B_l \one_{\{B_l \in [l^{1/6}, l^{5/6}]\}} \one_{\{\M_{0, l} (\Upsilon_{\! l}) \leq \lambda\}}\Bigr]\;, \label{eq:XikF}\\
& \H^{0, u, b}_{l, \lambda}(\bfF) \eqdef \E_{0, u, b}\Bigl[\bfF(\Upsilon_{\! l}) B_l \one_{\{B_l \in [l^{1/6}, l^{5/6}]\}} \one_{\{\M_{0, l}(\Upsilon_{\! l}) \leq \lambda\}}\Bigr]\;.\label{eq:XikF0ub}
\end{align}
\begin{proposition}
\label{pr:mainAsy}
Let $\bfF \in \CC^b_{\loc}(\CC(\R^d))$. For any $\eps > 0$ there exists $l_0 \in \N$ such that for all $l \geq l_0$, $b \geq l$ sufficiently large, $\lambda > 0$, and $u \in [b^{1/4}, b^{3/4}]$, it holds that
\begin{equation}
\label{eq:mainBoundAsy}
\Biggl\vert\E_{0, u, b}\Bigl[\bfF(\Upsilon_{\! b, \frkg}) \one_{\{\M_{0, b}(\Upsilon_{\! b, \frkg}) \leq \lambda\}}\Bigr] - 2 \H_{l, \lambda}(\bfF) \f{u}{b}\Biggr\rvert \leq \eps \f{u}{b} \;.
\end{equation}
Similarly, for any $\eps > 0$ there exists $l_0 \in \N$ such that for all $l \geq l_0$ and $b \geq l$ sufficiently large, it holds that 
\begin{equation}
\label{eq:mainBoundAsyBM}
\Biggl\vert\E\Bigl[\bfF(\Upsilon_{\! \infty}) \one_{\{\M_{0, b}(\Upsilon_{\! \infty}) \leq \lambda\}}\Bigr] -  \H_{l, \lambda}(\bfF) \f{\alpha}{\sqrt b}\Biggr\rvert \leq \eps \f{1}{\sqrt b} \;,
\end{equation}
where we recall that $\alpha = \sqrt{\smash[b]{2/\pi}}$. Furthermore, the same result holds when $\Upsilon_{\! \infty}$ is replaced by $\Upsilon_{\! b}$.
\end{proposition}

Before proving Proposition~\ref{pr:mainAsy}, we state two auxiliary technical lemmas. 
\begin{lemma}
\label{lm:techLemma1Asy}
For any $\lambda > 0$ and for all $u \in [b^{1/4}, b^{3/4}]$, it holds that 
\begin{equation*}
\lim_{l \to \infty} \limsup_{b \to \infty} \frac{b}{u} \P_{0, u, b}\bigl(B_l \leq l^{1/6}, \; \M_{0, b}(\Upsilon_{\! b, \frkg}) \leq \lambda\bigr) = 0 \;,
\end{equation*}
\end{lemma}
\begin{proof}
The proof follows the exact same argument as in the proof of \cite[Lemma~5.3]{BiskupLouidor} (see also \cite[Lemma~5.5]{BiskupLouidor} for a related statement about Brownian bridges). Specifically, it relies on Lemmas~\ref{lm:upperBoundKk},~\ref{lm:entrBasic},~and~\ref{lm:transferBBRW}. 
%%%%%%%%%%%%%%%%%%%%%%%%%%%%%%%%
%Recalling Definition~\ref{def:ControlVar} of $\rmK_b$, the probability in the statement is trivially bounded from above by 
%\begin{equation*}
%\P_{0, u, b}\bigl(\rmK_b \geq l, \; \M_{0, b}(\Upsilon_{\! b, \frkg}) \leq \lambda\bigr) + \P_{0, u, b}\bigl(\rmK_b < l, \; B_l \leq l^{1/6}, \; \M_{0, b}(\Upsilon_{\! b, \frkg}) \leq \lambda\bigr) \;.
%\end{equation*}
%Regarding the first term, the required bound follows from Lemma~\ref{lm:upperBoundKk}. On the other hand, by \eqref{eq:Inc1Fin}, the second term is bounded from above by 
%\begin{equation*}
%\P_{0, u, b}\Biggr(B_l \leq l^{1/6}, \; \bigcap_{j = 1}^{b-1}\bigl\{B_{j} \geq - 1 - \rmR_{l}(j) \bigr\}\Biggr)\;.
%\end{equation*}
%Therefore, introducing the function $\R^{+}_0 \ni s \mapsto \zeta(s) \eqdef 4\rmC (1 + [\log\bigl(1 + l + s)]^2)$, where here $\rmC > 1$ is the constant introduced in Lemma~\ref{lm:approxBrownianBridge}, and using Lemma~\ref{lm:transferBBRW}, the probability in the above display is bounded from above by
%\begin{equation*}
%\P_{0, u, b}\Biggr(B_l \leq l^{1/6}, \; \inf_{s \in [0, b]}\bigl(B_{s} + \zeta(s)\bigr) \geq 0\Biggr) \prod_{j=1}^{b}\Bigl(1-e^{-\zeta(j)^2}\Bigr)^{-2} \lesssim \frac{u}{b} l^{-\f16} \;.	
%\end{equation*}
%where the bound on the right-hand side follows from Lemma~\ref{lm:entrBasic}.
%%%%%%%%%%%%%%%%%%%%%%%%%%%%%%%%
\end{proof}

\begin{lemma}
\label{lm:techLemma2Asy}
For any $\eps > 0$ there exists $l_0 \in \N$ such that for all $l \geq l_0$, $b \geq l$ sufficiently large, $x \in [l^{1/6}, l^{5/6}]$, and $u \in [b^{1/4}, b^{3/4}]$, it holds that 
\begin{equation*}
\Biggl\lvert \P_{0, u, b}\Biggl(\bigcap_{j = l+1}^{b} \bigl\{B_j \geq 0 \bigr\} \, \Bigg| \,  B_l = x \Biggr) -  \frac{2xu}{b}\Biggr\rvert \leq \eps \frac{xu}{b}\;.
\end{equation*}
\end{lemma}
\begin{proof}
The proof follows the exact same argument as in the proof of \cite[Lemma~5.4]{BiskupLouidor} (see also \cite[Lemma~5.6]{BiskupLouidor} for a related statement about Brownian bridges). In particular, the lower bound follows trivially from Lemma~\ref{lm:basicEstBB}, while the upper bound follows from Lemmas~\ref{pr:BoundBBAboveNeg},~\ref{lm:techRhoAK},~and~\ref{lm:transferBBRW}.
\end{proof}

We now have all the necessary ingredients to prove Proposition~\ref{pr:mainAsy}.
\begin{proof}[Proof of Propositon~\ref{pr:mainAsy}]
We will only prove \eqref{eq:mainBoundAsy}, as the proof of \eqref{eq:mainBoundAsyBM} is completely analogous and, in fact, simpler.

The proof is based on a sequence of replacements that gradually convert one expectation into the other. 
We can and will assume that the function $\bfF \in \CC^b_{\loc}(\CC(\R^d))$ depends only on the value of the field inside the ball $\B_{k_0}$ for some fixed $k_0 > 0$. For $b \in \N$,  $l \in [b-1]_0$, and $k \in [l]_0$, on the event $\{\rmK_b \leq l\}$, arguing as in Remark~\ref{rm:diffFields}, we have that 
\begin{equation}
\label{eq:boundAsymExp}
\sup_{x \in \B_k} \abs{\Upsilon_{\! b, \frkg}(x) - \Upsilon_{\! l}(x)} \lesssim e^{-(l-k)/2} (\log l)^2    \;.
\end{equation} 

\textbf{Step~1:} We start by replacing $\Upsilon_{\!b, \frkg}$ by $\Upsilon_{\! l}$ in the argument of $\bfF$.
For $k \geq k_0$, we recall that on the event $\{\rmK_b \leq k\}$, both $\Upsilon_{\! b, \frkg}$ and $\Upsilon_{\! l}$ are bounded on $\B_{k_0}$ by a quantity depending only on $k$. Therefore, for $k_0 < k < l < b$, using the uniform continuity of $\bfF$ on compacts, the bound in \eqref{eq:boundAsymExp}, and Lemma~\ref{lm:upperBoundKk}, we obtain that for any $\eps > 0$, by taking $b > l$ both large enough, 
\begin{equation}
\label{eq:boundAssy1}
\E_{0, u, b}\Bigl[\abs{\bfF\bigl(\Upsilon_{\! b, \frkg}\bigr) -\bfF\bigl(\Upsilon_{\! l}\bigr)} \one_{\{\M_{0, b}(\Upsilon_{\! b, \frkg}) \leq \lambda\}} \one_{\{\rmK_b \leq k\}}\Bigr] \leq \eps \frac{u}{b} \;.
\end{equation}
Similarly, using the boundedness of $\bfF$ and Lemma~\ref{lm:upperBoundKk}, we also have that for any $\eps > 0$, by taking $k > 0$ large enough, it holds that 
\begin{equation}
\label{eq:boundAssy2}
\E_{0, u, b}\Bigl[\abs{\bfF\bigl(\Upsilon_{\! b, \frkg}\bigr) -\bfF\bigl(\Upsilon_{\! l}\bigr)} \one_{\{\M_{0, b}(\Upsilon_{\! b, \frkg}) \leq \lambda\}} \one_{\{\rmK_b > k\}}\Bigr] \leq \eps \frac{u}{b} \;.	
\end{equation}
Hence, \eqref{eq:boundAssy1} and \eqref{eq:boundAssy2} imply that, for any $\eps >0$, by taking $b > k$ both large enough, it holds that
\begin{equation}
\label{eq:mainBoundAsy0}
\Biggl\lvert\E_{0, u, b}\Bigl[\bfF\bigl(\Upsilon_{\! b, \frkg}\bigr) \one_{\{\M_{0, b}(\Upsilon_{\! b, \frkg}) \leq \lambda\}}\Bigr] - \E_{0, u, b}\Bigl[\bfF\bigl(\Upsilon_{\! l}\bigr) \one_{\{\M_{0, b}(\Upsilon_{\! b, \frkg}) \leq \lambda\}}\Bigr] \Biggr\rvert \leq \eps \frac{u}{b}\;.
\end{equation} 

\textbf{Step 2:} In this step, we show that for any $\eps > 0$, by taking $b > l$ both large enough, it holds that 
\begin{equation}
\label{eq:mainBoundAsy01}
\E_{0, u, b}\Bigl[\bfF(\Upsilon_{\! l}) \one_{\{\M_{0, b}(\Upsilon_{\! b, \frkg}) \leq \lambda\}}\Bigr] \leq  2 \frac{u}{b} \H_{l, \lambda}(\bfF) + \eps \frac{u}{b}  \;.
\end{equation}
To this end, for $\delta \in (0, \lambda)$, and $k_0 < k < l < b$, we define the event $\rmE_{b, l, k, \delta} $ by letting
\begin{equation*}
\rmE_{b, l, k, \delta} \eqdef \bigl\{\rmK_b \leq k\bigr\} \cap \bigl\{\M_{0, k}(\Upsilon_{\! b, \frkg}) < \lambda - \delta\bigr\} \cap \bigl\{B_l \in \bigl[l^{1/6}, l^{5/6}\bigr]\bigr\} \cap \Biggl\{\bigcap_{j = k+1}^{b} \bigl\{B_j \geq \lambda + 2 \rmR_k(j)\bigr\}\Biggr\} \;.
\end{equation*}
For any $\eps >0$, by Lemmas~\ref{lm:upperBoundKk},~\ref{lm:entroDelta},~\ref{lm:techLemma1Asy},~and~\ref{lm:mainEntrRW} for $l > k$ both large enough and $\delta > 0$ small enough,
\begin{equation*}
\limsup_{b \to \infty} \frac{b}{u} \P_{0, u, b}\bigl(\rmE_{b, l ,k, \delta}^c, \; \M_{0, b}(\Upsilon_{\! b, \frkg}) \leq \lambda \bigr) %\lesssim \bigl(e^{-c(\log k)^2} + \eps + l^{-1/6} +  k^{-1/16}\bigr) 
\leq \eps \;.
\end{equation*}
Therefore, so far we proved that, for $b > l$ both large enough, it holds that 
\begin{equation}
\label{eq:ExpAsy}
\E_{0, u, b}\Bigl[\bfF\bigl(\Upsilon_{\! l}\bigr) \one_{\{\M_{0, b}(\Upsilon_{\! b, \frkg})\leq \lambda\}}\Bigr] \leq \eps \frac{u}{b} + \E\Bigl[\bfF\bigl(\Upsilon_{\! l}\bigr) \one_{\{\rmE_{b, l, k, \delta}\}} \Bigr]\;.
\end{equation}

Now, we can choose $l \geq k$ large enough in such a way that the right-hand side of \eqref{eq:boundAsymExp} is less than $\delta$ for $j \in [k]$ and less than $\rmR_k(j)$ for $j \in \{k+1, \ldots, l\}$ (assuming that the constant $\rmC > 0$ in the definition of $\rmR_k$ in Lemma~\ref{lm:approxBrownianBridge} is chosen large enough). In this way, on the event $\{\rmK_b \leq k\}$, we have that 
\begin{equation*}
\bigl\{\M_{0, k}(\Upsilon_{\! b}) < \lambda - \delta\bigr\} \subseteq \bigl\{\M_{0, k}(\Upsilon_{\! l}) < \lambda\bigr\} \;,
\end{equation*}
and also, for $j \in \{k+1, \ldots l\}$,
\begin{equation*}
\bigl\{B_j \geq \lambda + 2 \rmR_k(j)\bigr\}  \subseteq \bigl\{\M_{0, j+1, j}(\Upsilon_{\! b}) \leq \lambda + \rmR_k(j)\bigr\} \subseteq \bigl\{\M_{0, j+1, j}(\Upsilon_{\!l}) \leq \lambda\bigr\} \;.
\end{equation*}
Therefore, putting these facts together, we get that the expectation on the right-hand side of \eqref{eq:ExpAsy} is bounded above by 
\begin{equation*}
\E_{0, u, b}\Bigl[\bfF\bigl(\Upsilon_{\! l}\bigr) \one_{\{B_l \in [l^{1/6}, l^{5/6}]\}} \one_{\{\M_{0, l}(\Upsilon_{\! l}) \leq \lambda\}} \one_{\{\cap_{j = l+1}^{b}\{B_j \geq 0\}\}} \Bigr] \;.
\end{equation*}
Since the field $\Upsilon_{\! l}$ is conditionally independent of $\sigma((B_s)_{s \geq l})$ given $\sigma(B_l)$, it follows from Lemma~\ref{lm:techLemma2Asy}, the boundedness of $\bfF$, the absolute continuity of the law of the Brownian bridge with respect to that of the Brownian motion (see Step~4), and Lemma~\ref{lm:upperBoundKkUpsilon} that
\begin{equation*}
	\lim_{l \to \infty} \limsup_{b \to \infty} \Biggl\lvert\frac{b}{u} \E_{0, u, b}\Bigl[\bfF\bigl(\Upsilon_{\! l}\bigr) \one_{\{B_l \in [l^{1/6}, l^{5/6}]\}} \one_{\{\M_{0, l}(\Upsilon_{\! l}) \leq \lambda\}} \one_{\{\cap_{j = l+1}^{b}\{B_j \geq 0\}\}} \Bigr] - 2\H^{0, u, b}_{l, \lambda}(\bfF)\Biggr\rvert = 0 \;,
\end{equation*}
thus obtaining \eqref{eq:mainBoundAsy01}.

\textbf{Step 3:}  In this step, we show that for any $\eps > 0$, by taking $b > l$ both large enough, it holds that
\begin{equation}
\label{eq:mainBoundAsy1}
\E_{0, u, b}\Bigl[\bfF(\Upsilon_{\!l}) \one_{\{\M_{0, b}(\Upsilon_{\! b, \frkg}) \leq \lambda\}}\Bigr] \geq 2 \frac{u}{b} \H_{l, \lambda}(\bfF) - \eps \frac{u}{b}  \;.
\end{equation}
To get the above inequality, we can proceed similarly to Step~2. For $k_0 < k < l <b$, we define the event $\tilde \rmE_{b, l, k} $ by letting
\begin{equation*}
\tilde \rmE_{b, l, k} \eqdef \bigl\{\rmK_b\leq k\bigr\} \cap \bigl\{\M_{0, b}(\Upsilon_{\! b, \frkg}) \leq \lambda  \bigr\} \cap \bigl\{B_l \in \bigl[l^{1/6}, l^{5/6}\bigr]\bigr\} \;.
\end{equation*}
Now, we note that, for $\delta > 0$, the following trivial lower bound holds true 
\begin{align}
\E_{0, u, b}\Bigl[& \bfF\bigl(\Upsilon_{\! l}\bigr) \one_{\{\M_{0, b}(\Upsilon_{\! b, \frkg}) \leq \lambda\}}\Bigr] \nonumber\\
& \geq \E_{0, u, b}\Bigl[\bfF\bigl(\Upsilon_{\! l}\bigr) \one_{\{\tilde \rmE_{b, l ,k}\}}\one_{\{\M_{0, k}(\Upsilon_{\! b, \frkg}) \leq \lambda + \delta\}} \one_{\{\cap_{j = k +1}^b\{B_j \geq -\lambda-2\rmR_k(j)\}\}}\Bigr] \;. \label{eq:mainBoundAsyAux1}
\end{align}
For any $\eps >0$, by Lemmas~\ref{lm:upperBoundKk},~\ref{lm:entroDelta},~\ref{lm:techLemma1Asy},~and~\ref{lm:mainEntrRW}, for $k < l$ both large enough and $\delta > 0$ small enough, 
\begin{equation}
\label{eq:mainBoundAsyAux2}
\limsup_{b \to \infty} \frac{b}{u} \P_{0, u, b}\Biggl(\tilde \rmE_{b, l ,k}^c,\; \M_{0, k}(\Upsilon_{\! b, \frkg}) \leq \lambda + \delta,\; \bigcap_{j = k +1}^b\bigl\{B_j \geq -\lambda-2\rmR_k(j) \bigr\}\Biggr) \leq \eps \;.
\end{equation}
As before, we can choose $l \geq k$ large enough in such a way that the right-hand side of \eqref{eq:boundAsymExp} is less than $\delta$ for $j \in [k]$ and less than $\rmR_k(j)$ for $j \in \{k+1, \ldots, l\}$ (assuming that the constant $\rmC > 0$ in the definition of $\rmR_k$ in Lemma~\ref{lm:approxBrownianBridge} is chosen large enough). Therefore, on the event $\{\rmK_b\leq k\}$, which is contained in $\tilde \rmE_{b, l, k}$, we have that
\begin{equation*}
\bigl\{\M_{0, k}(\Upsilon_{\! l}) \leq  \lambda \bigr\} \subseteq \bigl\{\M_{0, k}(\Upsilon_{b, \frkg}) \leq \lambda + \delta \bigr\}	
\end{equation*}
and also, for $j \in \{k+1, \ldots l\}$,
\begin{equation*}
\bigl\{\M_{0, j+1, j}(\Upsilon_{\!l}) \leq \lambda\bigr\}  \subseteq \bigl\{\M_{0, j+1, j}(\Upsilon_{\! b, \frkg}) \leq \lambda + \rmR_k(j)\bigr\} \subseteq \bigl\{B_j \geq -\lambda - 2 \rmR_k(j)\bigr\} \;.
\end{equation*}
Therefore, using \eqref{eq:mainBoundAsyAux2}, for $b > l$ both large enough, the right-hand side of \eqref{eq:mainBoundAsyAux1} can be lower bounded by
\begin{align*}
\E_{0, u, b}\Bigl[\bfF\bigl(\Upsilon_{\! l}\bigr) \one_{\{B_l \in [l^{1/6}, l^{5/6}]\}}\one_{\{\M_{0, l}(\Upsilon_{\! l}) \leq \lambda \}} \one_{\{\cap_{j = l+1}^b\{B_j \geq 0\}\}}\Bigr] -\f{u}{b} \eps \;.
\end{align*}
To obtain \eqref{eq:mainBoundAsy1}, we can apply Lemma~\ref{lm:techLemma2Asy} as in Step~2.

\textbf{Step 4:} By combining \eqref{eq:mainBoundAsy0}, \eqref{eq:mainBoundAsy01}, and \eqref{eq:mainBoundAsy1}, we have thus far proven that for any $\eps > 0$, by taking $b > l$ both sufficiently large, 
\begin{equation*}
\Biggl\vert\E_{0, u, b}\Bigl[\bfF(\Upsilon_{\! b, \frkg}) \one_{\{\M_{0, b}(\Upsilon_{\! b, \frkg}) \leq \lambda\}}\Bigr] - 2 \H^{0, u, b}_{l, \lambda}(\bfF) \f{u}{b}\Biggr\rvert \leq \eps \f{u}{b} \;.
\end{equation*}
Therefore, to conclude, it is sufficient to show that we can replace $\H^{0, u, b}_{l, \lambda}(\bfF)$ on the right-hand side of the previous display by $\H_{l, \lambda}(\bfF)$. This is achieved by using the absolute continuity of the law of the Brownian bridge with respect to the law of the Brownian motion. We recall that for any $0 \leq l < b$, it holds that 
\begin{equation*}
	\frac{d \P_{0, u, b}}{d \P} \Big|_{\sigma((B_s)_{s \in [0, l]})} = \sqrt{b/(b-l)} e^{\f{u^2}{2b} - \frac{(B_l - u)^2}{2(b-l)}} \;,
\end{equation*} 
and so, it holds that 
\begin{equation*}
\H_{l, \lambda}^{0, u, b}(\bfF) = \sqrt{b/(b-l)} \E\Biggl[\bfF(\Upsilon_{\! l}) B_l \one_{\{B_l \in [l^{1/6}, l^{5/6}]\}} \one_{\{\M_{0, l}(\Upsilon_{\! l}) \leq 1\}} e^{\f{u^2}{2b} - \f{(B_l-u)^2}{2(b-l)}}\Biggr]\;.
\end{equation*}
Therefore, using the boundedness of $\bfF$, the difference $\abs{\H_{l, \lambda}(\bfF)  - \H_{l, \lambda}^{0, u, b}(\bfF) }$ is bounded from above by a constant times
\begin{equation*}
l^{5/6} \E\Biggl[\Biggl\lvert 1- \sqrt{b/(b-l)} e^{\f{u^2}{2b} - \f{(B_l-u)^2}{2(b-l)}}\Biggr\rvert\Biggr] 
\end{equation*}
which can be made arbitrarily small by taking $b$ large enough, uniformly over $u \in [b^{1/4}, b^{3/4}]$. Hence, the desired result follows. 
\end{proof}

\begin{lemma}
\label{lm:boundXione}
For any $\lambda > 0$, there exist constants $c_1$, $c_2 > 0$ such that, for all $l \in \N$ large enough,
\begin{equation*}
c_1 < \H_{l, \lambda}(1) < c_2 \;.	
\end{equation*}
\end{lemma}
\begin{proof}
By taking $\bfF = 1$ in \eqref{eq:mainBoundAsy}, for each $\eps > 0$, there exists $l_0 \in \N$ sufficiently large such that for all $l > l_0$ and for any $b > l$ sufficiently large and $u \in [b^{1/4}, b^{3/4}]$, it holds that 
\begin{equation*}
\Biggl\vert\P_{0, u, b}\bigl(\M_{0, b}(\Upsilon_{\! b, \frkg}) \leq \lambda\bigr) - \H_{l, \lambda}(1) \f{2u}{b}\Biggr\rvert \leq \eps \f{u}{b}	\;.
\end{equation*}
On the other hand, from Lemma~\ref{lm:upperBoundKkUpsilon}, we know that there exist constants $c_1$, $c_{2} > 0$ such that for $b > 0$ large enough and $u \in [b^{1/4}, b^{3/4}]$,
\begin{equation}
c_1\f{u}{b} \leq \P_{0, u, b}\bigl(\M_{0, b}(\Upsilon_{\! b, \frkg}) \leq \lambda \bigr) \leq c_{2} \frac{u}{b}\;,
\end{equation}
and so the conclusion follows readily. 
\end{proof}

We are now ready to prove Propositions~\ref{pr:AsyProb}~and~\ref{pr:AsyConv} as well as Theorems~\ref{th:cluster}~and~\ref{th:clusterProb}.
\begin{proof}[Proof of Proposition~\ref{pr:AsyProb} and of Theorem~\ref{th:clusterProb}]
For the proof of Proposition~\ref{pr:AsyProb}, by taking $\bfF = 1$ in \eqref{eq:mainBoundAsy}, for any $\eps > 0$, by taking $b > l$ both large enough, it holds that 
\begin{equation*}
\P_{0, u, b}\bigl(\M_{0, b}(\Upsilon_{\! b, \frkg}) \leq \lambda\bigr)  = \bigl(2 \H_{l, \lambda}(1) + \eps_b(l) \bigr) \frac{u}{b} \;,
\end{equation*}
where $\lim_{l \to \infty} \limsup_{b \to \infty} \abs{\eps_b(l)} = 0$. By arguing in the same exact way as in the proof of \cite[Theorem~2.4]{BiskupLouidor}, we get that both limits in
\begin{equation*}
\lim_{b \to \infty} \frac{b}{u} \P_{0, u, b}\bigl(\M_{0, b}(\Upsilon_{\! b, \frkg}) \leq \lambda\bigr)  = \lim_{k \to \infty}2 \H_{k, \lambda}(1)
\end{equation*}
exist and are related as stated. Hence, the conclusion follows by letting 
\begin{equation}
\label{eq:defcstar}
c_{\star, \lambda} \eqdef \lim_{k\to\infty}\H_{k, \lambda}(1)	 \;,
\end{equation}
which is positive and finite thanks to Lemma~\ref{lm:boundXione}. Finally, the proof of Theorem~\ref{th:clusterProb} follows in exactly the same way by taking $\bfF = 1$ in \eqref{eq:mainBoundAsyBM}. 
\end{proof}

\begin{proof}[Proof of Proposition~\ref{pr:AsyConv} and of Theorem~\ref{th:cluster}]
Consider a function $\bfF \in \CC^b_{\loc}(\CC(\R^d))$, then by \eqref{eq:mainBoundAsy} from Proposition~\ref{pr:mainAsy}, the following equality holds
\begin{equation*}
\frac{\E_{0, u, b}\bigl[\bfF(\Upsilon_{\! b, \frkg}) \one_{\{\M_{0, b}(\Upsilon_{\! b, \frkg})\leq \lambda\}}\bigr]}{\P_{0, u, b}\bigl(\M_{0, b}(\Upsilon_{\! b, \frkg}) \leq \lambda\bigr)} = \frac{\H_{l, \lambda}(\bfF) + \tilde \eps_b(l)}{\H_{l, \lambda}(1) + \eps_b(l)} \;,
\end{equation*}
where $\lim_{l \to \infty} \limsup_{b \to \infty} \abs{\eps_b(l)} = 0$ and $\lim_{l \to \infty} \limsup_{b \to \infty} \abs{\tilde \eps_b(l)} = 0$. The same argument as in the proof of \cite[Proposition~5.8]{BiskupLouidor} shows that both limits in
\begin{equation}
\label{eq:IdentLimitShape0}
\lim_{b \to \infty} \frac{\E_{0, u, b}\bigl[\bfF(\Upsilon_{\! b, \frkg}) \one_{\{\M_{0, b}(\Upsilon_{\! b, \frkg})\leq \lambda\}}\bigr]}{\P_{0, u, b}\bigl(\M_{0, b}(\Upsilon_{\! b, \frkg}) \leq \lambda\bigr)} = \lim_{l \to \infty} \frac{\H_{l, \lambda}(\bfF)}{\H_{l, \lambda}(1)} 
\end{equation}
exist and are related as stated. Similarly, using \eqref{eq:mainBoundAsyBM}, we note that the same identity as in \eqref{eq:IdentLimitShape0} holds if we replace the field $\Upsilon_{\! b, \frkg}$ on the left-hand side with $\Upsilon_{\! \infty}$ or $\Upsilon_{\! b}$.

Regarding the limit on the right-hand side of \eqref{eq:weakLimit}, we note that, thanks to \eqref{eq:mainBoundAsyBM}, for all $\bfF \in \CC^b_{\loc}(\CC(\R^d))$, it holds that  
\begin{equation}
\label{eq:IdentLimitShape1}
\lim_{b\to \infty}\frac{\E\bigl[\bfF(\Upsilon_{\! \infty}) \one_{\{\M_{0, b}(\Upsilon_{\! \infty}) \leq \lambda\}}\bigr]}{\P\bigl(\M_{0, b}(\Upsilon_{\! \infty})\leq \lambda\bigr)} = \lim_{l \to \infty} \frac{\H_{l, \lambda}(\bfF)}{\H_{l, \lambda}(1)} \;. 
\end{equation}
Since the right-hand side of the above display is equal to the right-hand side of \eqref{eq:IdentLimitShape0}, the two limits coincide.

Finally, it remains to show that there exists a continuous random field $\tilde{\Upsilon}_{\! \lambda}$ on $\R^d$ such that the right-hand side of \eqref{eq:IdentLimitShape0} equals $\E[\bfF(\tilde{\Upsilon}_{\! \lambda})]$ for any $\bfF \in \CC^b(\CC(\B_k))$. For each $k \in \N$ and for any $\bfF \in \CC^b(\CC(\B_k))$, the limit on the left-hand side of \eqref{eq:IdentLimitShape0} exists. In particular, when $\bfF = 1$ this limit is~$1$
by definition. 
By \cite[Theorem~8.7.1]{BogachevI}, this ensures the existence of a probability measure $\tilde \nu_{k, \lambda}$ on $\CC(\B_k)$ such that the right-hand side of \eqref{eq:IdentLimitShape0} coincides with $\E_{\tilde \nu_{k, \lambda}}[\bfF(\Phi)]$, where under $\P_{\tilde \nu_{k, \lambda}}$ the field $\Phi$ is distributed according to $\tilde \nu_{k, \lambda}$. 
Furthermore, the collection of probability measures $({\tilde \nu_{k, \lambda}})_{k \in \N}$ is consistent, and so, by the Kolmogorov extension theorem, there exists a unique probability measure $\tilde \nu_{\lambda}$ on $\CC(\R^d)$ whose restriction to $\CC(\B_k)$ coincides with $\tilde \nu_{k, \lambda}$. 
Defining $\tilde \Upsilon_{\! \lambda}$ as the continuous random field with law $\tilde \nu_{\lambda}$ on $\CC(\R^d)$, it follows that for all $\bfF \in \CC^b_{\loc}(\CC(\R^d))$, the right-hand side of \eqref{eq:IdentLimitShape0} equals $\E[\bfF(\tilde\Upsilon_{\! \lambda})]$. 
\end{proof}

%%%%%%%%%%%%%%%%%%%%%%%%%%%%%%%%%%%%%%%%%%%%%%
\subsection{Corollaries and applications}
\label{subsec:UsefulCons}
In this section, we gather some results that follow as consequences of those presented in earlier sections and that will be particularly useful later, especially in Section~\ref{sec:Joint}. 

We begin with the following lemma.
\begin{lemma}
\label{lm:useCon1}
There exists a constant $c > 0$ such that for $\lambda > 0$, for all $L > 0$ sufficiently large, $b > 0$ sufficiently large, and $u \in [b^{1/4}, b^{3/4}]$,
\begin{equation*}
\P_{0, u, b}\Biggl(\M_{0, b}(\Upsilon_{\!b, \frkg}) \leq \lambda, \; \inf_{s \in [0, b]} B_s < - L,\; \Biggr) \lesssim \frac{u}{b} e^{-c \sqrt{L}}	\;.
\end{equation*}
\end{lemma}
\begin{proof}
For $L >0$, we let $\bar k = \bar k(L)$ be the smallest $k \in [b-1]$ such that $\lambda + \rmR_{k}(0) \geq L/4$, where $\rmR_{k}$ is defined in the statement of Lemma~\ref{lm:approxBrownianBridge}. 
We note that, by possibly taking $L > 0$ large enough, we have that $\bar k \approx e^{c \sqrt{L}}$. As usual, we can assume that we are on the event $\{\rmK_b < \bar k\}$, otherwise the conclusion follows trivially by Lemma~\ref{lm:upperBoundKk}. In particular, on the event $\{\rmK_b < \bar k\}$, by Lemma~\ref{lm:approxBrownianBridge} (see also \eqref{eq:Inc1Fin}), it suffices to estimate the probability of the following event
\begin{equation}
\label{eq:eventOfI}
\Biggl\{\bigcap_{j = 1}^{b-1} \bigl\{B_{j} \geq - \lambda - \rmR_{\bar k}(j) \bigr\}, \; \inf_{s \in [0, b]} B_s < -L \Biggr\} \;.
\end{equation}
To this end, we start by noticing that thanks to Lemma~\ref{lm:mainEntrRW}, there exists a constant $c >0$ such that,
\begin{equation*}
\P_{0, u, b}\Biggl(\bigcap_{j = 1}^{b-1} \bigl\{B_{j} \geq - \lambda - \rmR_{\bar k}(j)\bigr\}, \; \bigcup_{j = \bar k}^{b-1} \bigl\{B_{j} \leq \lambda + \rmR_{\bar k}(j) \bigr\} \Biggr) \lesssim \f{u}{b} e^{-c \sqrt{L}} \;.
\end{equation*} 
Hence, we can now assume that we are on the complement of the giant union appearing in the probability on the left-hand side of the above display, i.e., we need to estimate the probability of the following event 
\begin{equation}
\label{eq:eventOfII}
\Biggl\{\bigcap_{j = 1}^{b-1} \bigl\{B_{j} \geq - \lambda - \rmR_{\bar k}(j)\bigr\}, \; \bigcap_{j = \bar k}^{b-1} \bigl\{B_{j} > \lambda + \rmR_{\bar k}(j) \bigr\}, \; \inf_{s \in [0, b]} B_s < -L	\Biggr\} \;.
\end{equation}
For each $j \in [b-1]_0$, we introduce the process $(W^j_s)_{s\in [j, j+1]}$ given by
\begin{equation*}
W^{j}_s \eqdef \bigl((j+1)-s\bigr) B_j + \bigl(s-j\bigr)B_{j+1} - B_s \;, \qquad \forall \, s\in [j, j+1]\;.
\end{equation*}
The process $W^{j}$ has the law of standard Brownian bridge indexed by times in the interval $[j, j+1]$. Moreover, the collection $(W^j)_{j \in [b-1]_0}$ is independent of the values $(B_j)_{j \in [b]_0}$. Now, we note that by definition of $\bar k$, for all $j \in [\bar k]_0$, it holds that $\lambda + \rmR_{\bar k}(j) \leq L/2$. In particular, this fact implies that, on the first event, the last two events in \eqref{eq:eventOfII} are contained in the following union of events 
\begin{equation}
\label{eq:eventRemInf}
\Biggl\{\bigcup_{j = 0}^{\bar k}\Bigl\{\sup_{s \in [j, j+1]} \abs{W^j_s}\geq L/2\Bigr\}\Biggr\} \cup \Biggl\{\bigcup_{j = \bar k}^{b-1} \Bigl\{\sup_{s \in [j, j+1]} \abs{W^j_{s}} \geq \lambda + \rmR_{\bar k}(j) + L \Bigr\}\Biggr\} \;. 
\end{equation}
For the first event in \eqref{eq:eventRemInf}, using the independence of the collection $(W^j)_{j \in [b-1]_0}$ from the values $(B_j)_{j \in [b]_0}$, the Gaussian tails of the supremum of a standard Brownian bridge, arguing as in the proof Lemma~\ref{lm:upperBoundKk}, there exists a constant $c > 0$ such that 
\begin{align*}
\sum_{l = 0}^{\bar k}\P_{0, u, b}\Biggl(\bigcap_{j = 1}^{b-1} \bigl\{B_{j} \geq - \lambda - \rmR_{\bar k}(j)\bigr\}, \; \sup_{s \in [l, l+1]} \abs{W^l_s}\geq L/2\Biggr) \lesssim \frac{u}{b} e^{-c \sqrt{L}} \;.
\end{align*} 
Finally, regarding the second event in \eqref{eq:eventRemInf}, using again the independence mentioned above, and the Gaussian tails of the supremum of a standard Brownian bridge, we have that thanks to Lemma~\ref{lm:upperBoundKk}, there exists a constant $c > 0$ such that 
\begin{align*}
\sum_{l = \bar k}^{b-1}\P_{0, u, b}\Biggl(\bigcap_{j = 1}^{b-1} \bigl\{B_{j} \geq - \lambda - \rmR_{\bar k}(j)\bigr\}, \; \sup_{s \in [l, l+1]} \abs{W^l_{s}} \geq \lambda + \rmR_{\bar k}(l)+L\Biggr) \lesssim \frac{u}{b} e^{-c \sqrt{L}} \;.
\end{align*}  
Therefore, by combining all the bounds we have established so far, the claim follows.
\end{proof}

We now state and prove the following results which provides a decay of the field $\Upsilon_{\! b, \frkg}$ on the event that the supremum of such a field is bounded by one.
\begin{lemma}
\label{lm:repulsionShape}
For each $\lambda > 0$, $b > 0$ sufficiently large, $u \in [b^{1/4}, b^{3/4}]$, and $0 < k < b$ sufficiently large, it holds for all $k \leq k' \leq b-1$ that
\begin{align}
\P_{0, u, b}\Biggl(\bigcup_{j = k}^{k'} \Bigl\{\sup_{x \in \A_j} \Upsilon_{\! b, \frkg}(x) > - (\log j)^2\Bigr\}, \; \M_{0, b}(\Upsilon_{\! b, \frkg}) \leq \lambda\Biggr) & \lesssim \frac{u}{b} k^{-\f{1}{16}} \;, \label{eq:boundGrowthBB} \\
\P\Biggl(\bigcup_{j = k}^{k'} \Bigl\{\sup_{x \in \A_j} \Upsilon_{\! b, \frkg}(x) > - (\log j)^2\Bigr\}, \; \M_{0, b}(\Upsilon_{\! b, \frkg}) \leq \lambda\Biggr) & \lesssim \frac{1}{\sqrt{b}} k^{-\f{1}{16}} \;. \label{eq:boundGrowthBM}
\end{align}
Similarly, for all $k \geq 0$ sufficiently large, one has that   
\begin{equation}
\label{eq:boundGrowthLimiting}
\P\Biggl(\bigcup_{j = k}^{\infty} \Bigl\{\sup_{x \in \A_j} \tilde \Upsilon_{\! \lambda}(x) > - (\log j)^2\Bigr\}\Biggr) \lesssim k^{-\f{1}{16}} \;.	
\end{equation}
\end{lemma}
\begin{proof}
We only prove the bound for the Brownian bridge since the bound for the Brownian motion can be obtained in the same way. Furthermore, it suffices to prove \eqref{eq:boundGrowthBM} for $k' = b-1$.
As usual, we can assume that we are on the event $\{\rmK_b \leq k\}$, otherwise the conclusion follows trivially from Lemma~\ref{lm:upperBoundKk}. Thanks to Lemma~\ref{lm:approxBrownianBridge}, it holds that 
\begin{equation*}
\bigcup_{j = k}^{b-1}\Biggl\{\sup_{x \in \A_j} \Upsilon_{\! b, \frkg}(x) > - (\log j)^2,\; \rmK_b \leq k\} \subseteq \bigcup_{j = k}^{b-1}\bigl\{B_j < \lambda + 2\rmR_{k}(j)\bigr\} \;,
\end{equation*}
and also
\begin{equation*}
\bigl\{\M_{0, b}(\Upsilon_{\! b, \frkg})  \leq \lambda,\; \rmK_b \leq k\bigr\} \subseteq \bigcap_{j = 1}^{b-1}\bigl\{B_{j} \geq - \lambda - 2\rmR_{k}(j) \bigr\} \;.	
\end{equation*}
The claim follows since by Lemma~\ref{lm:mainEntrRW} it holds that  
\begin{equation*}
\P_{0, u, b}\Biggl(\bigcap_{j = 1}^{b-1} \bigl\{B_{j} \geq - \lambda - 2\rmR_{k}(j)\bigr\}, \; \bigcup_{j = k}^{b-1} \bigl\{B_j  < + 2\rmR_{k}(j)\bigr\}\Biggr) \lesssim \frac{u}{b} k^{-\f{1}{16}}  \;.
\end{equation*}
Finally, we note that the bound \eqref{eq:boundGrowthLimiting} follows directly from Theorems~\ref{th:cluster}~and~\ref{th:clusterProb}, and by Portmanteau's theorem, by taking the limit as $b \to \infty$ in \eqref{eq:boundGrowthBM}.
\end{proof}

\begin{remark}
\label{rm:DecayFieldsAS}
An immediate consequence of \eqref{eq:boundGrowthLimiting} is that there almost surely exists a (random) $k \geq 0$ such that, for all $j \geq k$, the supremum of the field $\tilde \Upsilon_{\! \lambda}$ on the annulus $\A_j$ is less than $-(\log j)^2$.
\end{remark}

%%%%%%%%%%%%%%%%%%%%%%%%%%%%%%%%%%%%%%%%%%%%%%
\subsection{Tail estimates for near-maximal level sets}
\label{sub:tailsD}
In this section, we establish a key result concerning the tail behaviour of the volume of ``near-maximal level sets'' $\abs{\D^{\lambda}_{0, j}(\Upsilon_{\! b, \frkg})}$, for $j \leq b$, where here we recall the notation introduced in \eqref{eq:maximalExpBalls}, as well as the definition of $\Upsilon_{\!b, \frkg}$ in \eqref{e:defUpsbg}. Throughout this section, we assume that the field $\frkg_b$ satisfies \ref{as:GG1} \dash \ref{as:GG3}. 

\begin{lemma}
\label{lm:tailsDgUpsilon}
There exists $\delta = \delta(d) \in (0, 1)$ such that for any $j_0 > 0$ sufficiently large and $b > j_0$ sufficiently large, it holds for any $\lambda > 0$, $j \in \{j_0,\ldots, b\}$, $u \in [b^{1/4}, b^{3/4}]$, and $\eta \geq 0$,
\begin{equation}
\label{eq:boundProbuOrdinary}
\P_{0, u, b}\bigl(\abs{\D^{\lambda}_{0, j}(\Upsilon_{\!b, \frkg})}^{-1} \geq \eta,\; \M_{0, b}(\Upsilon_{\!b, \frkg}) \leq \lambda\bigr) \lesssim \frac{u}{b} \bigl(1 \wedge \eta^{-(1+\delta)}\bigr) \;.
\end{equation}
Moreover, for all $\sigma \in [0, \delta)$, it holds that
\begin{equation}
\label{eq:boundExpuOrdinary}
\E_{0, u, b}\Bigl[\abs{\D^{\lambda}_{0, j}(\Upsilon_{\!b, \frkg})}^{-(1 + \sigma)} \one_{\{\M_{0, b}(\Upsilon_{\! b, \frkg}) \leq \lambda\}}\Bigr] \lesssim \f{u}{b} \;.
\end{equation}
\end{lemma}

\begin{remark}
\label{rm:TailEstBM}
As usual, the previous lemma admits a corresponding version for Brownian motions in place of Brownian bridges. Specifically, if we replace the conditional probability law $\P_{0, u, b}$ with the unconditional probability law $\P$, then the conclusions of the previous lemma remain valid, with the only difference that $u/b$ is replaced by $1/\sqrt{b}$.
\end{remark}

Before proceeding with the proof of Lemma~\ref{lm:tailsDgUpsilon}, we state and prove the following auxiliary result concerning the tail behaviour of the second derivative of the field $\rmZ_b$. 
\begin{lemma}
\label{lm:techPsi}
There exists a constant $c > 0$ such that for any $b \in \N$, $r \in [0, e^b]$, $i$, $k \in [d]$, and $\eta \geq 0$, 
\begin{equation}
\label{e:boundSupDer}
\P\Biggl(\sup_{\abs{x} \leq r} \bigl|\partial^2_{i, k} \rmZ_{b}(x)\bigr| > \eta \Biggr)
\lesssim 1 \wedge r^d e^{-c\eta^2}\;.
\end{equation}
\end{lemma}
\begin{proof}
Recall that, for $x$, $y \in \B_b$,
\begin{equation*}
\E\bigl[\partial^2_{i, k} \rmZ_{b}(x)  \partial^2_{i, k} \rmZ_{b}(y)\bigr] = \int_0^b \bigl(e^{-4s} \frkK_{i,i,k,k}(e^{-s}(x-y)) - e^{-4s}\frkK_{i,k}(e^{-s} x) \frkK_{i,k}(e^{-s} y)\bigr) ds \;, 
\end{equation*}
where, given $n \in \N$ and $(j_1, \ldots j_n) \in [d]^n$, we write $\frkK_{j_1, \ldots, j_n}$ for the $n$-th derivative of $\frkK$ along the directions $e_{j_1}, \ldots, e_{j_n}$. In particular, by the smoothness of $\frkK$ and the fact that it is supported in the unit ball, it holds
\begin{equation*}
\E\bigl[\abs{\partial^2_{i, k} \rmZ_{b}(x) - \partial^2_{i, k} \rmZ_{b}(y)}^2\bigr] \lesssim \abs{x-y}^2 \wedge 1 \;,
\end{equation*}
where the implicit constant is independent of $b$.
It follows from Fernique's majorizing criterion (Lemma~\ref{lm_Fernique}) and Borell-TIS inequality (Lemma~\ref{lm_Borell}) that, for any $\eta \geq 0$ and for some $c > 0$, one has the bound 
\begin{equation*}
\P \Biggl(\sup_{\abs{x - y} \leq 1} \bigl|\partial^2_{i, k} \rmZ_{b}(x)\bigr| > \eta \Biggr) \lesssim e^{-c\eta^2}\;,
\end{equation*}
uniformly over $y$. The requested estimate then follows from the union bound.
\end{proof}

Before proceeding with the proof of Lemma~\ref{lm:tailsDgUpsilon}, we introduce some notation. For $r \geq 0$, we set
\begin{equation}
\label{eq:defSrNearLevel}
\rmS_r \eqdef \sup_{i, k \in [d], \, x \in B(0, r)} \abs{\partial^2_{i, k} \Upsilon_{\! b, \frkg}(x)} \;,
\end{equation} 
where, to simplify the notation, we have omitted the dependence of $\rmS_r$ on the parameter $b$.
Furthermore, recalling the characterisation \eqref{eq:Phi} of the field $\Phi_{b}$, for any $i$, $k \in [d]$ and $x \in \B_b$, we have 
\begin{equation}
\label{eq:secDirecDerUps}
\partial^2_{i,k} \Upsilon_{\! b, \frkg}(x) =  \int_0^b e^{-2 s}  \frkK_{i, k}(e^{-s} x) dB_{s} + \partial^2_{i, k} \rmZ_{b}(x) - \sqrt{\smash[b]{2d}} \int_0^b e^{-2 s}  \frkK_{i, k}(e^{-s} x) ds + \partial^2_{i, k} \frkg_b(x),
\end{equation}
where we have used the same notation introduced in the proof of Lemma~\ref{lm:techPsi}.  

The reason why it is useful to look at the second derivative of the field $\Upsilon_{\! b, \frkg}$ is due to the following observation.
\begin{remark}
\label{rem:implicationSmallD}
For any $\lambda > 0$, $\eta \geq 0$, and $0 < j \leq b$, we observe that the following implication holds
\begin{equation}
\label{eq:implicationSmallD}
\abs{\D^{\lambda}_{0, j}(\Upsilon_{\!b, \frkg})}^{-1} \geq \eta \quad \implies \quad \rmS_{e^j} > c_{\lambda} \eta^{2/d} \;,
\end{equation}
for some suitable constant $c_{\lambda}$ depending on $\lambda$ and $d$. Indeed, if the condition on the right-hand side of the above display is not satisfied then there exists a ball of area at least $\eta^{-1}$ around the maximum of $\Upsilon_{\!b, \frkg}$ inside $\B_j$ where $\Upsilon_{\! b, \frkg}$ is greater than its maximal value minus $\lambda$, which contradict the fact that $\abs{\D^{\lambda}_{0, j}(\Upsilon_{\!b, \frkg})}^{-1} \geq \eta$. 
\end{remark} 

A straightforward consequence of Lemma~\ref{lm:techPsi} is the following result. 
\begin{lemma}
\label{lm:techPsiReal}
Let $A > 0$ be fixed. There exist constants $c_1$, $c_2 >0$ such that for any $\eta \geq 0$, $b \geq j > 0$, and $u > -A$, it holds that 
\begin{equation}
\label{e:betterBound}
\P_{0, u, b} \bigl(\rmS_{e^j} > \eta^{2/d}\bigr) \lesssim 1 \wedge e^{d j + c_1 \frac{u}{b} \eta^{2/d}} e^{-c_2 \eta^{4/d}}\;. 
\end{equation}
\end{lemma} 
\begin{proof}
Consider the expression \eqref{eq:secDirecDerUps} for the second directional derivative of $\Upsilon_{\! b, \frkg}$. We note that the deterministic term is bounded by a constant independent of everything else. Moreover, thanks to \ref{as:GG3}, the term $\partial^2_{i, k} \frkg_b$ has Gaussian tails uniformly over all $i$, $k \in [d]$ and $x \in \B_b$. Furthermore, we also note the field  
\begin{equation*}
\B_b \ni x \mapsto \int_0^b e^{-2 s}  \frkK_{i, k}(e^{-s} x) dB_{s} - \frac{u}{b} \int_0^b e^{-2 s}  \frkK_{i, k}(e^{-s} x) ds
\end{equation*}
has uniform Gaussian tails uniformly over $i$, $k \in [d]$, $x \in \B_b$, and over the probability laws $\P_{0, u, b}$, for all $b > 0$ and $\smash{u > - A}$. In particular, since the integral multiplying $u/b$ in the above display is bounded uniformly over $i$, $k \in [d]$ and $x \in \B_b$, using Lemma~\ref{lm:techPsi}, we get that there exist constants $c_1$, $c_2 > 0$ such that 
\begin{equation*}
\P_{0, u, b} \bigl(\rmS_{e^j} > \eta^{2/d}\bigr) \lesssim 1 \wedge e^{d j + c_1 \frac{u}{b} \eta^{2/d}} e^{-c_2 \eta^{4/d}}\;,
\end{equation*}
thus proving the claim.
\end{proof}

We are now ready for the proof of Lemma~\ref{lm:tailsDgUpsilon}. 
\begin{proof}[Proof of Lemma~\ref{lm:tailsDgUpsilon}]
For $\eta \geq 0$, we define 
\begin{equation*}
r_{\eta} \eqdef \lfloor \exp(\eta^{1/d})\rfloor\;,
\end{equation*}
so that $\log r_\eta \approx \eta^{1/d}$, and we also consider $b \geq j > 0$ and $u \in [b^{1/4}, b^{3/4}]$. We split the proof into three distinct cases. 

\textbf{Case 1:} We begin by treating the following case: 
\begin{equation*}
\log r_\eta \geq \f{b}{2} \;.
\end{equation*}
Let $c_{\lambda} > 0$ be the constant introduced in Remark~\ref{rem:implicationSmallD}. Then, thanks to \eqref{eq:implicationSmallD} and Lemma~\ref{lm:techPsiReal}, there exists a constant $c > 0$ such that
\begin{equation}
\label{eq:InclusionKeyCase1}
\P_{0, u, b}\bigl(\abs{\D^{\lambda}_{0, j}(\Upsilon_{\!b, \frkg})}^{-1} \geq \eta\bigr) \leq \P_{0, u, b} \bigl(\rmS_{e^j} > c_{\lambda} \eta^{2/d}\bigr) \lesssim e^{dj} e^{-c \eta^{4/d}} \;.
\end{equation}
which, in the regime $\smash{\log r_\eta \geq b/2}$, is trivially bounded by a constant times $(u/b) \eta^{-2}$.

\textbf{Case 2:} We now consider the following case: 
\begin{equation*}
\f{j}{2} \leq \log r_\eta < \f{b}{2} \;.
\end{equation*}
Proceeding, as above, thanks to Lemma~\ref{lm:techPsiReal}, it suffices to bound the following probability 
\begin{equation*}
\P_{0, u, b}\bigl(\rmS_{e^j} > c_{\lambda} \eta^{2/d},\; \M_{0, b}(\Upsilon_{\!b, \frkg}) \leq \lambda\bigr)\;. 
\end{equation*} 
Recalling Definition~\ref{def:ControlVar} of the control variable $\rmK_b$, on the event ${\rmK_b \geq \log r_{\eta}}$, for $b \geq 0$ sufficiently large, Lemma~\ref{lm:upperBoundKk} provides the following bound
\begin{equation*}
\P_{0, u, b}\bigl(\M_{0, b}(\Upsilon_{\! b, \frkg}) \leq \lambda, \;\rmK_b\geq \log r_\eta\bigr) \lesssim \f{u}{b} e^{-c (\log \eta)^2} \;,
\end{equation*}
for some constants $c >0$.
Hence, we can further restrict ourselves to the event $\{\rmK_b< \log r_{\eta}\}$. We recall that on this event, for each $l \in [b-1]_0$, it holds that $\Osc_l(B) \leq \Theta_{\log r_\eta}(l)$. In particular, this implies that
\begin{equation*}
\sup_{i,k \in [d], \, x \in \B_b} \Biggl\lvert\int_0^b e^{-2 s} \frkK_{i,k}(e^{-s} x) dB_{s}\Biggr\rvert \lesssim \sum_{l = 0}^{b-1} e^{-2l}  \Osc_l(B) \lesssim (\log \eta)^2\;.
\end{equation*}
Hence, noting that the deterministic term appearing in $\partial^2_{i,k} \Upsilon_{\! b, \frkg}$ is bounded by a constant independent of $b$, it holds that
\begin{equation*}
\bigl\{\rmS_{e^j} > c \eta^{2/d},\; \rmK_b< \log r_\eta\bigr\} 
\subseteq \Biggl\{\sup_{i,k \in [d], \, x \in \B_j} \abs{\partial^2_{i, k} \rmZ_{b}(x) + \partial^2_{i, k} \frkg_{b}(x)} \gtrsim \eta^{2/d}\Biggr\} \eqdef \rmE_0 \;.
\end{equation*}

Now, we observe that, for $b \geq 0$ sufficiently large, thanks again to Lemma~\ref{lm:upperBoundKk}, it holds that 
\begin{equation}
\label{e:PA0FirstB}
\P_{0, u, b}\bigl(\rmK_b < \log r_{\eta},\; \M_{0, b}(\Upsilon_{\!b, \frkg}) \leq \lambda,\, \rmE_0\bigr) \lesssim \frac{u}{b} \sqrt{\P(\rmE_0)} \;.	
\end{equation}
By Lemma~\ref{lm:techPsiReal} and \ref{as:GG3}, there exist constants $c_1$, $c_2 > 0$ such that 
\begin{equation*}
\P(\rmE_0) \lesssim e^{dj - c_1 \eta^{4/d}} \leq e^{- c_2 \eta^{4/d}} \;,
\end{equation*}
where the last inequality follows thanks to the fact that $j/2 \leq \log r_{\eta}$. Combining the above bound with \eqref{e:PA0FirstB} yields the claim in this case.

\textbf{Case 3:} Finally, we consider the following case: 
\begin{equation}
\label{e:constrRlambda}
0 \leq \log r_\eta < \f{j}{2}\;.
\end{equation}
Arguing as above, and using Lemma~\ref{lm:upperBoundKk}, by taking $b \geq 0$ sufficiently large, we can restrict ourselves to the event $\{\rmK_b < \log r_{\eta}\}$.
We observe that, by \eqref{e:constrRlambda} and Lemma~\ref{lm:approxBrownianBridge},
\begin{equation}
\label{e:implicationMax}
\bigl\{\M_{0, b}(\Upsilon_{\! b, \frkg}) \leq \lambda, \; \rmK_b< \log r_\eta\bigr\} \subseteq \Biggl\{\bigcap_{j =1}^{b-1} \bigl\{B_{j} \geq - \lambda - 2\rmR_{\rmK_b}(j)\bigr\}, \; \rmK_b< \log r_\eta\Biggr\}\;.
\end{equation}
Now, recalling \eqref{eq:defSrNearLevel}, we show that for suitable values of $r \in (0,r_\eta]$ and $M \ge (\log \eta)^2$, we can restrict to the event $\{\rmS_r \leq M\}$. Indeed, arguing exactly as in the previous case, we have that 
\begin{equation*}
\bigl\{\rmS_r > M,\; \rmK_b< \log r_\eta\bigr\} 
\subseteq \Biggl\{\sup_{i,k \in [d], \, x \in B(0, r)} \abs{\partial^2_{i, k} \rmZ_{b}(x) + \partial^2_{i, k} \frkg_{b}(x)} \gtrsim M\Biggr\} \eqdef \rmE_1 \;.
\end{equation*}
By Lemma~\ref{lm:techPsiReal} and \ref{as:GG3}, there exists a constant $a > 0$ such that 
\begin{equation}
\label{e:PA0}
\P\bigl(\rmE_1\bigr) \lesssim r^d e^{-a M^2}\;.
\end{equation}
By combining this with \eqref{eq:Inc1Fin} and \eqref{eq:tech1FinAsy} of Lemma~\ref{lm:upperBoundKk}, we find that
\begin{equation*}
\P_{0, u, b}\bigl(\M_{0, b}(\Upsilon_{\! b, \frkg}) \leq \lambda, \; \rmK_b< \log r_\eta,\;  \rmS_r > M\bigr) \lesssim \frac{u}{b} \sqrt{\P(\rmE_1)} \;. \label{eq:case22TailsNear}
\end{equation*}
Hence, recalling \eqref{e:PA0}, this implies that we can impose finitely many conditions of the type
$\{\rmS_r \le M\}$ provided that the parameters $r$ and $M$ are such that $\smash{r^{d} e^{-a M^2} \lesssim \eta^{-\beta}}$ for some $\beta > 2$.
In our particular case, we can use this to impose
\begin{equation}
\label{e:conditionS}
\bigl\{\rmS_{r_\eta} \le c_{\lambda} \eta^{\frac2d}\bigr\} \qquad \text{ and } \qquad \bigl\{\rmS_1 \le (\log \eta)^4\bigr\}\;,
\end{equation}
where $c_{\lambda} > 0$ is the same constant introduced in Remark~\ref{rem:implicationSmallD}. Therefore, it remains to show the bound
in the statement for the event
\begin{equation*}
\rmE_2 \eqdef \Bigl\{\abs{\D^{\lambda}_{0, j}(\Upsilon_{\!b, \frkg})}^{-1} \geq \eta,\; \M_{0, b}(\Upsilon_{\!b, \frkg}) \leq \lambda, \; \rmK_b< \log r_\eta, \;  \rmS_{r_\eta} \le c\eta^{\frac2d}, \; S_1 \le (\log \eta)^4\Bigr\}\;.
\end{equation*}
We note now that the first condition in \eqref{e:conditionS} implies that if the maximum is achieved within the ball of radius $r_\eta$, then there is a ball of area at least $\eta^{-1}$ around that maximum where $\smash{\Upsilon_{\!b, \frkg}}$ is greater than its maximal value minus $\lambda$. This however cannot happen if $\smash{\abs{\D^{\lambda}_{0, j}(\Upsilon_{\!b, \frkg})}^{-1} \geq \eta}$. In particular, since $\Upsilon_{\!b, \frkg}(0) = 0$, we have shown that  on the event $\rmE_2$ there exists $x \in \R^d$ with $|x| \ge r_\eta$ such  that $\smash{\Upsilon_{\!b, \frkg}(x) \geq 0}$. Since we are on the event $\smash{\{\rmK_b< \log r_\eta\}}$, thanks to Lemma~\ref{lm:approxBrownianBridge}, there must be some $\smash{l \geq \log r_\eta}$ such that $\smash{B_l \leq \lambda + 2\rmR_{\log r_\eta}(l)}$.

Similarly, the condition $\rmS_1 \leq (\log \eta)^4$ implies that,
\begin{equation*}
\rmE_2 \subseteq \Biggl\{\sup_{i \in [d]} \abs{\partial_i \Upsilon_{\! b, \frkg}(0)} \lesssim \eta^{-\f1d}(\log\eta)^4\Biggr\}\;.
\end{equation*}
Otherwise, one could find a point $z \in \R^d$ at distance of order $\eta^{-1/d}$ from the origin
such that $\Upsilon_{\! b, \frkg}(z) \gtrsim \eta^{-2/d} (\log\eta)^4$, which implies that the values of $\Upsilon_{\! b, \frkg}$ inside a ball of radius $\eta^{-1/d}$ around $z$ are all positive. But since we are on the event $\{\M_{0, b}(\Upsilon_{\!b, \frkg}) \leq \lambda\}$, these points are contained in $\D^{\lambda}_{0, j}(\Upsilon_{\!b, \frkg})$, leading again to a contradiction.

Now, we note that, thanks to \ref{hp_K1}, it holds that the gradient of $\frkK$ at the origin is zero, and so $\d_i \Upsilon_{\! b, \frkg}(0)$ is independent of $B$. Combining these considerations with \eqref{e:implicationMax}, it follows that
\begin{equation}
\label{e:secondlastBound}
\begin{alignedat}{1}
\P_{0, u, b} (\rmE_2) \leq & \P_{0, u, b}\Biggl(\sup_{i \in [d]} \abs{\partial_i \Upsilon_{\! b, \frkg}(0)} \lesssim \eta^{-\f1d}(\log\eta)^4\Biggr) \\
& \cdot \P_{0, u, b}\Biggl(\bigcap_{l = 1}^{b-1} \bigl\{B_{l} \geq - \lambda - 2\rmR_{\log r_\eta}(l) \bigr\},\; \bigcup_{l = \log r_\eta}^{b-1} \bigl\{B_l \leq \lambda + 2\rmR_{\log r_\eta}(l)\bigr\}\Biggr) \;.
\end{alignedat}
\end{equation}

We can now apply Lemma~\ref{lm:mainEntrRW} to bound the second probability, showing that for $j_0 \geq 0$ sufficiently large, $b \geq j_0$ sufficiently large, and for any $j \in [j_0, b]$, it holds that 
\begin{equation*}
\P_{0, u, b}\bigl(\rmE_2\bigr) \lesssim \f{u}{b} \eta^{-\f{1}{16 d}} \P_{0, u, b}\Biggl(\sup_{i \in [d]} \abs{\partial_i \Upsilon_{\! b, \frkg}(0)} \lesssim \eta^{-\f1d}(\log\eta)^4\Biggr) \;.\label{e:lastBound}
\end{equation*}
In order to bound the remaining factor, we note that $\d_i \Upsilon_{\! b, \frkg}(0) = \d_i \rmZ_{b}(0) + \d_i \frkg_{b}(0)$ and that, by \eqref{e:covZb},
\begin{equation}
\label{e:covGradField}
\E \bigl[\d_i \rmZ_{b}(0)\d_j \rmZ_{b}(0)\bigr]
= -\d^2_{ij} \frkK(0) \int_0^b e^{-2s}\,ds\;.
\end{equation}
Since $(\d_i \rmZ_{b}(0))_{i \in [d]}$ are jointly Gaussian, it follows from Remark~\ref{rem:strict}
and \eqref{e:covGradField} that 
\begin{equation}
\label{e:covGradFieldFinal}
\P_{0, u, b}\Biggl(\sup_{i \in [d]} \abs{\partial_i \Upsilon_{\! b, \frkg}(0)} \leq \eps\Biggr)
\lesssim \eps^d\;,\qquad \forall \, \eps \in (0,1]\;,
\end{equation}
uniformly over $b \geq 0$ and $u \in [b^{1/4}, b^{3/4}]$. Therefore, we conclude that \eqref{e:lastBound} is bounded by some constant times $(u/b) \eta^{-(1 + \delta)}$ for any $\delta \in (0,  1/(16d))$.

To conclude, we note that \eqref{eq:boundExpuOrdinary} follows immediately from \eqref{eq:boundProbuOrdinary}.
\end{proof}

We also need the following version of Lemma~\ref{lm:tailsDgUpsilon}, where the end point of the Brownian bridge is taken to be less than $b^{1/4}$.
\begin{lemma}
\label{lm:tailsDgUpsilonSmallU}
Let $A > 0$ be fixed. There exists $\delta = \delta(d) \in (0, 1)$ such that for any $b \geq 0$ sufficiently large, it holds for any $\lambda > 0$, $u \in [-A, b^{1/4}]$, and $\eta \geq 0$,
\begin{equation}
\label{eq:boundProbSmallu}
\P_{0, u, b}\Bigl(\abs{\D^{\lambda}_{0, b}(\Upsilon_{\!b, \frkg})}^{-1} \geq \eta,\; \M_{0, b}(\Upsilon_{\!b, \frkg}) \leq \lambda\Bigr) \lesssim \f{1}{\sqrt{b}} \eta^{-(1+\delta)} \;.
\end{equation}
Moreover, for all $\sigma \in [0, \delta)$, it holds that 
\begin{equation}
\label{eq:boundExpSmallu}
\E_{0, u, b}\Bigl[\abs{\D^{\lambda}_{0, b}(\Upsilon_{\!b, \frkg})}^{-(1 + \sigma)} \one_{\{\M_{0, b}(\Upsilon_{\! b, \frkg}) \leq \lambda\}}\Bigr]  \lesssim \f{1}{\sqrt{b}} \;.
\end{equation}
\end{lemma}
\begin{proof}
We do not provide a detailed proof of this lemma, as it closely follows that of Lemma~\ref{lm:tailsDgUpsilon}. Instead, we outline the necessary modifications to that proof.

First, in the regime $\smash{\log r_\eta \geq b/2}$, we can proceed exactly as in the proof of Lemma~\ref{lm:tailsDgUpsilon}. Regarding the regime $\smash{\log r_\eta < b/2}$, it suffices to use Lemma~\ref{lm:upperBoundKkSmallU} in place of Lemma~\ref{lm:upperBoundKk} to restrict to the event $\{\rmK_b < \log r_{\eta}\}$. Similarly, we obtain that the probability on the left-hand side of \eqref{eq:case22TailsNear} is bounded above by a multiple of $b^{-1/2} \sqrt{\P(\rmE_1)}$, when $b \geq 0$ is large enough.

Hence, it remains to bound $\P_{0, u, b}(\rmE_2)$ appearing on the left-hand side of the first line in \eqref{e:secondlastBound}. 
To this end, arguing as in the proof of \cite[Lemma~4.16]{BiskupLouidor} and using \eqref{eq:boundsBBstaysPos}, we observe that the probability on the right-hand side of the second line in \eqref{e:secondlastBound} can be bounded from above by
\begin{equation*}
\P_{0, u, b}\Biggl(\bigcap_{j = 1}^{b-1} \bigl\{B_{j} \geq - \lambda - 2\rmR_{\log r_\eta}(j) \bigr\}\Biggr)
\lesssim \P_{0, u, b}\Biggl(\inf_{s \in [0, b]} B_s \gtrsim -(\log b)^2\Biggr) 
\lesssim b^{-5/8} \;,
\end{equation*}
where, to derive the last inequality, we used the fact that $u \leq b^{1/4}$. Finally, since we are in the regime $\log r_\eta < b/2$, it follows that $b^{-5/8}$ is at most a multiple of $b^{-1/2} \eta^{-1/(8d)}$. Combining this with the bound in \eqref{e:covGradFieldFinal} completes the proof of \eqref{eq:boundProbSmallu}. To conclude, we note that \eqref{eq:boundExpSmallu} follows immediately from \eqref{eq:boundProbSmallu}.
\end{proof}

Finally, we also have the following result. 
\begin{lemma}
\label{lm:tailsDgUpsilonLargeU}
For $\lambda > 0$ and for $b > 0$ sufficiently large, it holds that 
\begin{equation*}
\label{eq:boundExpLargeu}
\E_{0, u, b}\bigl[\abs{\D^{\lambda}_{0, b}(\Upsilon_{\!b, \frkg})}^{-1}\bigr] \lesssim 
\begin{cases}
	b^{d^2}\;, \quad & \text{ if } u \in [b^{3/4}, b^{2d}] \;, \\
	(u/b)^{2d} \;, \quad & \text{ if } u > b^{2d} \;.
\end{cases}
\end{equation*}
\end{lemma}
\begin{proof}
We begin by recalling that, for $r > 0$, the quantity $\rmS_r$ is defined in \eqref{eq:defSrNearLevel}. Furthermore, we also recall that, thanks to Remark~\ref{rem:implicationSmallD}, for all $\eta > 0$, it holds that
\begin{equation*}
\P_{0, u, b}(\abs{\D^{\lambda}_{0, b}(\Upsilon_{\!b, \frkg})}^{-1} \geq \eta) \leq \P_{0, u, b} \bigl(\rmS_{e^b} > c_{\lambda} \eta^{2/d}\bigr) \;,
\end{equation*}
for some constant $c_{\lambda} > 0$ depending only on $\lambda$. We divide the proof into two disjoint cases.

\textbf{Case 1:} We begin by considering the case $u \in [b^{3/4}, b^{2d}]$. In this case, using Lemma~\ref{lm:techPsiReal}, there exist constants $c_1$, $c_2$, $c_3 > 0$ such that, for any $b > 0$ sufficiently large,
\begin{align*}
\E_{0, u, b}\bigl[\abs{\D^{\lambda}_{0, b}(\Upsilon_{\!b, \frkg})}^{-1}\bigr] 
& \leq b^{d^2} + \int_{b^{d^2}}^{\infty} \P_{0, u, b} \bigl(\rmS_{e^b} > c_{\lambda} \eta^{2/d}\bigr) d\eta \\
& \lesssim b^{d^2} + e^{d b}\int_{b^{d^2}}^{\infty} e^{c_1\frac{u}{b} \eta^{2/d} - c_2 \eta^{4/d}} d\eta \\
& \lesssim b^{d^2} + e^{d b}\int_{b^{d^2}}^{\infty} e^{-c_3 \eta^{4/d}} d\eta \lesssim b^{d^2} \;.
\end{align*}

\textbf{Case 2:} We now focus on the case $u > b^{2d}$. 
Proceeding similarly to the previous case, we obtain that there exist constants $c_1$, $c_2$, $c_3 > 0$ such that, for any $b > 0$ sufficiently large,
\begin{align*}
\E_{0, u, b}\bigl[\abs{\D^{\lambda}_{0, b}(\Upsilon_{\!b, \frkg})}^{-1}\bigr] 
& \lesssim (u/b)^{2d} + e^{db} \int_{(u/b)^{2d}}^{\infty} e^{c_1 \frac{u}{b} \eta^{2/d} - c_2 \eta^{4/d}} d\eta \\
& \lesssim (u/b)^{2d} + e^{db} \int_{(u/b)^{2d}}^{\infty} e^{-c_3 \eta^{4/d}} d\eta \lesssim (u/b)^{2d} \;,
\end{align*}
which completes the proof.
\end{proof}

%%%%%%%%%%%%%%%%%%%%%%%%%%%%%%%%%%%%%%%%%%%%%%
%%%%%%%%%%%%%%%%%%%%%%%%%%%%%%%%%%%%%%%%%%%%%%
\section{Resampling property and independence from the threshold}
\label{sec:resampling}
The main goal of this section is to present some results concerning the fields $\tilde \Upsilon_{\! \lambda}$ and $\Psi_{\lambda}$ introduced in Theorem~\ref{th:cluster} and Definition~\ref{def:recoverPhi}, respectively. In particular, in Section~\ref{sub:Uncountable}, we prove some key technical results related to the field $\tilde \Upsilon_{\! \lambda}$.
Then, in Section~\ref{sub:resamplingProp}, we establish the ``resampling property'' of the field $\tilde{\Upsilon}_{\! \lambda}$, as stated in Proposition~\ref{pr:resempPropUps}, from which Proposition~\ref{pr:inversionPsi} follows directly.
Finally, in Section~\ref{sub:indthreshold}, we establish the independence of certain quantities of interest from $\lambda$. In particular, we prove Proposition~\ref{pr:PsiIndepLambda}.

%%%%%%%%%%%%%%%%%%%%%%%%%%%%%%%%%%%%%%%%%%%%%%
\subsection{Some technical results}
\label{sub:Uncountable}
The main goal of this section is to establish a result concerning the limiting ``shape field'' $\tilde \Upsilon_{\! \lambda}$ introduced in Theorem~\ref{th:cluster}, which will play a crucial role in the proof of some of our main theorems. Specifically, the main goal of this section is to prove the following lemma.
\begin{lemma}
\label{lm:Countable}
There exists a countable set $\Lambda^c \subseteq \R^{+}$ such that for each $\lambda$, $\theta \in \Lambda = \R^{+} \setminus \Lambda^c$ it holds that 
\begin{equation}
\label{eq:Countable1}
\P\bigl(\M(\tilde \Upsilon_{\! \lambda}) = \theta\bigr) = 0 \;.
\end{equation}
Moreover, for all $\lambda$, $\theta \in \Lambda$, it holds that 
\begin{equation}
\label{eq:Countable2}
\P\bigl(\abs{\{x \in \R^d \, : \, \tilde \Upsilon_{\! \lambda}(x) - \M(\tilde \Upsilon_{\! \lambda}) = -\theta\}} > 0\bigr) = 0 \;.
\end{equation}  
Furthermore, for each $\lambda_1 \neq \lambda_2\in \Lambda$ such that $\lambda_2 < \lambda_1$, the field $\smash{\tilde{\Upsilon}_{\! \lambda_2}}$ has the same law of the field $\smash{\tilde{\Upsilon}_{\! \lambda_1}}$ conditioned on the event $\smash{\M(\tilde{\Upsilon}_{\! \lambda_1}) \leq \lambda_2}$. 
\end{lemma}

\begin{remark}
\label{rm:redTechSmallk}
It follows from the proof of Lemma~\ref{lm:Countable} that we may (and will) assume that, for any $\lambda$, $\theta \in \Lambda$, the identities \eqref{eq:Countable1} and \eqref{eq:Countable2} hold also with $\smash{\M(\tilde \Upsilon_{\! \lambda})}$ replaced by $\M_{0, k}(\tilde \Upsilon_{\! \lambda})$ for some fixed $k \geq 0$. 
\end{remark}

In order to prove Lemma~\ref{lm:Countable}, we need to introduce some notation. We begin by defining the following space 
\begin{align*}
%\X & \eqdef \bigl\{(b, \phi) \, : \, b \in \R_{0}^{+} \cup \{\infty\}, \; \phi \in \CC(\R^d) \bigr\} \;, \\
\X \eqdef \bigl\{(b, \phi) \, : \, b \in \R_{0}^{+} \cup \{\infty\}, \; \phi \in \CC(\R^d, \R \cup \{-\infty\}) \bigr\} \;.
\end{align*}
We equip the space $\X$ with the pseudometrics $\dd_{\X}$ and $\bar\dd_{\X}$ defined as follows
\begin{align*}
\dd_{\X}\bigl((b_1, \phi_1), (b_2, \phi_2)\bigr) & \eqdef \abs{e^{-b_1} - e^{-b_2}} + \sum_{k = 1}^{\infty} 2^{-k} \Biggl(1 \wedge \sup_{x \in \B_k \cap \B_{b_1 \wedge b_2}} \abs{\phi_1(x) - \phi_2(x)}\Biggr) \;, \\
\bar\dd_{\X}\bigl((b_1, \phi_1), (b_2, \phi_2)\bigr) & \eqdef \abs{e^{-b_1} - e^{-b_2}} + \sup_{x \in \R^d}  \abs{\chi_{b_1}(x) e^{c\phi_1(x)} -  \chi_{b_2}(x)e^{c \phi_2(x)}} \;,
\end{align*}
where $c = c(d) \eqdef \sqrt{\smash[b]{2/d}}$, and for each $b \in \R_{0}^{+} \cup \{\infty\}$, $\chi_b:\R^d \to \R$ is a smooth function such that $\chi_b(x) = 1$ for all $\abs{x} \leq e^b$ and $\chi_b(x) = 0$ for all $\abs{x} \geq e^{b}+1$. 

For each $\lambda > 0$, we let $\A_{\lambda} \subseteq \X$ be the set defined as follows
\begin{equation*}
\A_{\lambda} \eqdef \bigl\{(b, \phi) \in \X \, : \, \M_{0, b}(\phi) \leq \lambda\bigr\}\;.
\end{equation*}
Furthermore, we define the collection of measures $(\nu_{b, \lambda})_{b > 0}$ and the measure $\nu_{\infty, \lambda}$ as follows
\begin{equation*}
\nu_{b, \lambda} \eqdef c_{\star, \lambda} \sqrt{b} \Law\bigl[\bigl(b, \Upsilon_{\! b}\bigr)\bigr] |\A_{\lambda} \qquad \text{ and } \qquad \nu_{\infty, \lambda} = \Law\bigl[\bigl(\infty, \tilde{\Upsilon}_{\! \lambda}\bigr)\bigr] \;,
\end{equation*}
where $\cdot|\A_{\lambda}$ denotes the restriction to the subspace $\A_{\lambda}$, and $c_{\star, \lambda}$ is the constant introduced in Theorem~\ref{th:clusterProb}.

\begin{remark}
The reason we introduced the ``stronger'' pseudometric $\bar\dd_{\X}$ is that the boundaries of the set $\A_{\lambda}$ are \emph{not} disjoint in $(\X, \dd_{\X})$ for different values of $\lambda$. In contrast, in $(\X, \bar\dd_{\X})$, we have $\bar \partial_{\X} \A_{\lambda_1} \cap \bar\partial_{\X} \A_{\lambda_2} = \emptyset$ for any $\lambda_1, \lambda_2 > 0$ with $\lambda_1 \neq \lambda_2$. This fact will play an important role in what follows.
\end{remark}

\begin{lemma}
\label{lm:upgradeCon}
For each $\lambda > 0$, the collection of measures $(\nu_{b, \lambda})_{b > 0}$ converges weakly in $(\X, \bar\dd_{\X})$ to $\nu_{\infty, \lambda}$ as $b \to \infty$. 
\end{lemma}
\begin{proof}
Let $\smash{(\infty, \tilde \Upsilon_{\! \lambda})}$ be sampled from $\smash{\nu_{\infty, \lambda}}$, and with a slight abuse of notation, let $\smash{(b, \tilde \Upsilon_{\! b, \lambda})}$ be sampled from $\smash{\nu_{b, \lambda}}$.
Thanks to Theorem~\ref{th:cluster}, we know that $\nu_{b, \lambda}$ converges weakly to $\nu_{\infty, \lambda}$ in $(\X, \dd_{\X})$ as $b \to \infty$. This fact implies that we can find a coupling between the sequence $\smash{((b, \tilde \Upsilon_{\! b}))}_{b > 0}$ and $\smash{(\infty, \tilde \Upsilon_{\! \lambda})}$ such that 
\begin{equation}
\label{eq:convMetricLocal}
\lim_{b \to \infty} \E\bigl[\dd_{\X}\bigl((b,\tilde\Upsilon_{\! b, \lambda}), (\infty, \tilde{\Upsilon}_{\! \lambda})\bigr)\bigr] = 0 \;.
\end{equation}
Therefore, to obtain the desired result, we need to check that the same limit holds if we replace the pseudometric $\dd_{\X}$ in the above display with the pseudometric $\bar\dd_{\X}$. In particular, it suffices to show that 
\begin{equation*}
\lim_{b \to \infty} \E\Biggl[\sup_{x \in \R^d}  \abs{\chi_{b}(x) e^{c \tilde\Upsilon_{\! b, \lambda}(x)} - e^{c \tilde\Upsilon_{\! \lambda}(x)}}\Biggr] = 0
\end{equation*}
To this end, we let $0 \leq j \leq b$, and we note that  
\begin{equation*}
\sup_{x \in \R^d} \abs{\chi_{b}(x) e^{c \tilde\Upsilon_{\! b, \lambda}(x)} - e^{c \tilde\Upsilon_{\! \lambda}(x)}} \leq \sup_{x \in \B_j} \abs{e^{c \tilde\Upsilon_{\! b, \lambda}(x)} - e^{c \tilde\Upsilon_{\! \lambda}(x)}} + \sup_{x \in \R^d \setminus \B_j} \abs{\chi_b(x) e^{c \tilde\Upsilon_{\! b, \lambda}(x)} - e^{c  \tilde\Upsilon_{\! \lambda}(x)}} \;.
\end{equation*}
The fact that the limit as $b \to \infty$ of the expectation of the first term on the right-hand side of the above display goes to zero follows from \eqref{eq:convMetricLocal}. Hence, it remains to show that the expectation of the second term on the right-hand side of the above display also tends to zero as $b \to \infty$, followed by $j \to \infty$. Thanks to Lemmas~\ref{lm:upperBoundKkUpsilon}~and~\ref{lm:repulsionShape}, it holds that 
\begin{equation*}
\lim_{j \to \infty}\lim_{b \to \infty}\E\Biggl[\sup_{x \in \B_b \setminus \B_j} e^{c \tilde\Upsilon_{\! b, \lambda}(x)}\Biggr] = 0\;, \qquad \lim_{j \to \infty}\E\Biggl[\sup_{x \in \R^d \setminus \B_{j}} e^{c \tilde\Upsilon_{\! \lambda}(x)}\Biggr] = 0 \;.
\end{equation*} 
Therefore, recalling that $\chi_b(x) = 0$ for all $\abs{x} \geq e^{b}+1$, it remains to check that
\begin{equation}
\label{eq:limitWantedShell}
\lim_{b \to \infty} \E\Biggl[\sup_{\abs{x} \in [e^b, e^b+1)} e^{c \tilde{\Upsilon}_{\! b, \lambda}(x)} \Biggr] = 0 \;.
\end{equation}
In particular, thanks to Lemma~\ref{lm:upperBoundKkUpsilon}, for $b \geq 0$ large enough, by recalling the definition \eqref{eq:defUpsilonBBeg} of the field $\Upsilon_{\! b}$ and since, by \ref{hp_K2}, the function $\frkK$ is supported in $B(0, 1)$, the expectation on the left-hand side of the above display is bounded above by a multiple of
\begin{align}
\sqrt{b} \E\Biggl[\sup_{\abs{x} \in [e^b, e^b+1)} e^{c \Upsilon_{\! b}(x)} \one_{\{\M_{0, b}(\Upsilon_{\! b}) \leq \lambda\}}\Biggr] 
& \leq \sqrt{b} \E\Biggl[e^{-c (B_b + \sqrt{\smash[b]{2d}} b)} \sup_{\abs{x} \in [e^b, e^b+1)} e^{c\rmZ_b(x)}\Biggr] \nonumber\\
& = \sqrt{b} e^{c^2 b/2 - \sqrt{\smash[b]{2d}} c b} \E\Biggl[\sup_{\abs{x} \in [e^b, e^b+1)} e^{c \rmZ_b(x)}\Biggr]  \;, \label{eq:limitWantedShellBound}
\end{align}
and thus, we are left with the task of bounding the expectation in the final term of the above display. To achieve this, we start by noting that, by the Borell-TIS inequality (Lemma~\ref{lm_Borell}) and Fernique's majorizing criterion (Lemma~\ref{lm_Fernique}), for any $y \in \R^d$ such that $B(y, 1) \subseteq \R^d \setminus \B_b$, it holds that 
\begin{equation*}
\P\Biggl(\sup_{x \in B(y, 1)} \abs{\rmZ_b(x)} \geq \eta \Biggr) \lesssim e^{-\frac{\eta^2}{2b}}\;, \qquad \forall \, \eta \geq 0\;.	
\end{equation*} 
Consequently, thanks to the union bound, we get that 
\begin{equation}
\label{eq:boundInterExpShell}
\P\Biggl(\sup_{\abs{x} \in [e^b, e^b+1)} \abs{\rmZ_b(x)} \geq \eta \Biggr) \lesssim e^{b(d-1) - \frac{\eta^2}{2b}}\;, \qquad \forall \, \eta \geq 0\;.	
\end{equation} 
Next, observing that for any non-negative random variable $X$ it holds that $\E[e^{cX}] = 1 + c\int_0^{\infty} e^{cx} \P(X > x) dx$, we can write
\begin{equation*}
\E\Biggl[\sup_{\abs{x} \in [e^b, e^b+1)} e^{c \rmZ_b(x)}\Biggr] = 1 + c \int_0^{\infty} e^{c \eta} \P\Biggl(\sup_{\abs{x} \in [e^b, e^b+1)} \rmZ_b(x) \geq \eta\Biggr) d\eta \;.
\end{equation*}
In particular, for $k > 0$, using the fact that $\eta^2 \geq 2 k \eta - k^2$, we have that 
\begin{align*}
\int_0^\infty e^{c \eta} \bigl(1\wedge e^{b(d-1) - \frac{\eta^2}{2b}}\bigr) d\eta
&\leq \int_0^{k} e^{c \eta} d\eta + \int_{k}^\infty e^{c \eta + b (d-1) - \frac{k \eta}{b} + \frac{k^2}{2b}} d\eta \\
&\leq k e^{c k} + \frac{b}{k - cb} e^{c k+ b(d-1) - \frac{k^2}{2b}}\;,	
\end{align*}
provided that $k > bc$. It is natural to choose $k$ such that $k^2/(2b) = b(d-1)$, i.e., $k = b \sqrt{2(d-1)}$ which is indeed greater than $bc$ since $c = \sqrt{\smash[b]{2/d}} < \sqrt{2(d-1)}$. Therefore, recalling \eqref{eq:boundInterExpShell} and combining the previous observations, we conclude that
\begin{equation*}
\E\Biggl[\sup_{\abs{x} \in [e^b, e^b+1)} e^{c \rmZ_b(x)}\Biggr] \lesssim b e^{\sqrt{2(d-1)} c b}\;.
\end{equation*}
Therefore, by plugging the previous bound into \eqref{eq:limitWantedShellBound}, we showed that the expectation in \eqref{eq:limitWantedShell} is bounded above by a multiple of 
\begin{equation*}
\E\Biggl[\sup_{\abs{x} \in [e^b, e^b+1)} e^{c \tilde{\Upsilon}_{\! b, \lambda}(x)} \Biggr] \lesssim b^{\f32} e^{bc (c/2 + \sqrt{2(d-1)} - \sqrt{\smash[b]{2d}})}	\to 0 \quad \text{ as } \quad b \to \infty \;, 
\end{equation*}
where, recalling that $c = \sqrt{\smash[b]{2/d}}$, we have used the fact that $c/2 + \sqrt{2(d-1)} - \sqrt{\smash[b]{2d}} \leq 0$.
\end{proof}

We are now ready to prove the following intermediate result.
\begin{lemma}
\label{lm:CountableRestr}
There exists a countable set $\bar \Lambda^c \subseteq \R^{+}$ such that, letting $\bar \Lambda = \R^{+} \setminus \bar\Lambda^c$, for each $n \in \N$ and $\lambda \in \bar \Lambda \cap (0, n]$, it holds that $\nu_{\infty, \lambda} = \nu_{\infty, n}|\A_{\lambda}$. 
\end{lemma}
\begin{proof}
Thanks to Lemma~\ref{lm:upgradeCon}, we know that for each $\lambda > 0$, the sequence $(\nu_{b, \lambda})_{b > 0}$ converges weakly to $\nu_{\infty, \lambda}$ in $(\X, \bar\dd_{\X})$. 
Moreover, for all $n \in \N$, $\lambda \in (0, n]$, and $b > 0$, we have that $\nu_{b, \lambda} = \nu_{b, n}|\A_{\lambda}$.
We note that there exists a countable set $\smash{\bar\Lambda^c_n \subseteq (0, n]}$ such that for all $\smash{\lambda \in (0, n] \setminus \bar\Lambda^c_n}$, it holds that $\nu_{b, \lambda}$ converges weakly to $\nu_{\infty, n}|\A_{\lambda}$ in $(\X, \bar\dd_{\X})$. 
Indeed, this follows from the uniqueness of the weak limit and the fact that the sets $(\bar \partial_{\X} \A_{\lambda})_{\lambda \in (0, n]}$ are disjoint, implying that $\nu_{\infty, n}$ can assign positive mass to at most countable many of them. 
Therefore, the conclusion follows by setting $\smash{\bar \Lambda^c = \cup_{n \in \N} \bar \Lambda_n^c}$ and noting that a countable union of countable sets remains countable.
\end{proof}

We are now ready to prove Lemma~\ref{lm:Countable}

\begin{proof}[Proof of Lemma~\ref{lm:Countable}]
We begin by observing that an immediate consequence of Lemma~\ref{lm:CountableRestr} is that for each $\smash{\lambda_1 \neq \lambda_2\in \bar{\Lambda}}$ such that $\lambda_2 < \lambda_1$, the field $\smash{\tilde{\Upsilon}_{\! \lambda_2}}$ has the same law of the field $\smash{\tilde{\Upsilon}_{\! \lambda_1}}$ conditioned on the event $\smash{\M(\tilde{\Upsilon}_{\! \lambda_1}) \leq \lambda_2}$.

For $\theta > 0$, we introduce the set $\S_{\theta, 1} \subseteq \X$ by letting 
\begin{equation*}
\S_{\theta, 1} \eqdef \bigl\{\bigl(\infty, \phi\bigr) \in \X \, : \, \M(\phi) = \theta\bigr\} \;.
\end{equation*}
Since the sets $(\S_{\theta, 1})_{\theta > 0}$ are disjoint, for all $n \in \N$, there exists a countable set $\smash{\bar \Lambda_{1, n}^c \subseteq \bar \Lambda}$ such that for all $\smash{\theta \in \bar\Lambda \setminus \bar\Lambda_{1, n}^c}$, it holds that $\smash{\nu_{\infty, n}(\S_{\theta, 1}) = 0}$. Indeed, if that were not the case, the measure $\nu_{\infty, n}$ would assign positive measure to uncountable many disjoint sets, which is, of course, not possible. Furthermore, thanks to Lemma~\ref{lm:CountableRestr}, we know that for each $\lambda \in \bar\Lambda \cap (0, n]$, it holds that $\nu_{\infty, \lambda} = \nu_{\infty, n}|\A_{\lambda}$. In particular, this implies that for all $\smash{\lambda}$, $\smash{\theta \in \bar\Lambda \setminus \bigcup_{n \in \N}\bar\Lambda_{1, n}^c}$, it holds that $\smash{\nu_{\infty, \lambda}(\S_{\theta, 1}) = 0}$.

Next, for $\theta > 0$, we introduce the set $\S_{\theta, 2} \subseteq \X \subseteq \bar{\X}$ by letting 
\begin{equation*}
\S_{\theta, 2} \eqdef \bigl\{\bigl(\infty, \phi\bigr) \in \X \, : \, \abs{\{x \in \R^d \, : \, \phi(x) - \M(\phi) = -\theta\}} > 0\bigr\}
\end{equation*}
Consider the mapping $\Theta : \X \times \R^d \to \R$ defined as $\Theta((b, \phi), x) = \phi(x) - \M_{0, b}(\phi)$. For $n \in \N$ and a standard multivariate normal random variable $\smash{\CN(0, I_d)}$, let $\smash{\nu^{\star}_{\infty, n}}$ be the pushforward of $\smash{\nu_{\infty, n} \otimes \CN(0, I_d)}$ under $\Theta$, which is a measure on $\R_0^{-}$. We note that there exists a countable set $\bar \Lambda_{2, n}^c \subseteq \bar \Lambda$ such that for all $\theta \in \bar\Lambda \setminus \bar\Lambda_{2, n}^c$, it holds that $\smash{\nu^{\star}_{\infty, n}(\{-\theta\}) = 0}$. This is due to the fact that the measure $\smash{\nu^{\star}_{\infty, n}}$ can only have at most countable many point masses. In turn, this implies that for all $\theta \in \bar\Lambda \setminus \bar\Lambda_{2, n}^c$, it holds that $\smash{\nu_{\infty, n}(\S_{\theta, 2}) = 0}$. Now, we observe once again that thanks to Lemma~\ref{lm:CountableRestr}, for each $\smash{\lambda \in \bar\Lambda \cap (0, n]}$, it holds that $\smash{\nu_{\infty, \lambda} = \nu_{\infty, n}|\A_{\lambda}}$. Therefore, this implies that for all $\lambda$, $\smash{\theta \in \bar\Lambda \setminus \bigcup_{n \in \N} \bar\Lambda_{2, n}^c}$, it holds that $\smash{\nu_{\infty, \lambda}(\S_{\theta, 2}) = 0}$.
Finally, combining the previous discussions, the conclusion follows by setting 
$
\Lambda \eqdef \bar \Lambda \setminus \bigcup_{n \in \N} (\bar\Lambda^c_{1, n} \cup \bar\Lambda^c_{2, n})
$.
\end{proof}

We conclude this section by proving that the proportionality constant appearing in Definition~\ref{def:recoverPhi} lies in the interval $(0, \infty)$. Specifically, it suffices to establish the following result.
\begin{lemma}
\label{lm:NorContPsiOK}
For all $\lambda \in \Lambda$, it holds that $\smash{\E[\abs{\D^{\lambda}(\tilde \Upsilon_{\! \lambda})}^{-1}] \in (0, \infty)}$.
\end{lemma}
\begin{proof}
The fact that $\smash{\E[\abs{\D^{\lambda}(\tilde \Upsilon_{\! \lambda})}^{-1}] > 0}$ follows immediately from the continuity of the field $\tilde \Upsilon_{\! \lambda}$ and from its almost sure decay at infinity, as implied by Lemma~\ref{lm:repulsionShape} (see also Remark~\ref{rm:DecayFieldsAS}).
Hence, it remains to verify that $\smash{\E[\abs{\D^{\lambda}(\tilde \Upsilon_{\! \lambda})}^{-1}] < \infty}$. To this end, fix $0 < j \leq b$. Then, using Lemma~\ref{lm:tailsDgUpsilon} (see also Remark~\ref{rm:TailEstBM}) and proceeding exactly as in the proof of Lemma~\ref{lm:truncOK} below, one obtains that 
\begin{equation*}
\lim_{L \to \infty} \limsup_{b \to \infty} \sqrt{b}\Bigl|\E\bigl[\abs{\D^{\lambda}_{0, j}(\Upsilon_{\! b})}^{-1} \one_{\{\M_{0, b}(\Upsilon_{\! b}) \leq \lambda\}} \bigr]- \E\bigl[\bigl(\abs{\D^{\lambda}_{0, j}(\Upsilon_{\! b})}^{-1} \wedge L\bigr) \one_{\{\M_{0, b}(\Upsilon_{\! b}) \leq \lambda\}} \bigr] \Bigr| = 0 \;.
\end{equation*}
Since $\lambda \in \Lambda$, we have that the set of discontinuities of the mapping $\smash{\CC(\R^d) \ni \phi \mapsto \abs{\D^{\lambda}_{0, j}(\phi)}^{-1} \wedge L}$ is assigned measure zero by the law of the field $\smash{\tilde \Upsilon_{\! \lambda}}$. Therefore, we are in a position to apply Theorems~\ref{th:cluster} and~\ref{th:clusterProb} from which we can deduce that 
\begin{equation*}
\lim_{b \to \infty} \sqrt{b} \E\bigl[\abs{\D^{\lambda}_{0, j}(\Upsilon_{\! b})}^{-1} \one_{\{\M_{0, b}(\Upsilon_{\! b}) \leq \lambda\}} \bigr]	= \alpha c_{\star, \lambda} \E\bigl[\abs{\D^{\lambda}_{0, j}(\tilde \Upsilon_{\! \lambda})}^{-1}\bigr] \;.
\end{equation*}
By invoking the monotone convergence theorem and Lemma~\ref{lm:tailsDgUpsilon} (see also Remark~\ref{rm:TailEstBM}), the claim follows by taking the limit as $j \to \infty$ on both sides of the above expression.
\end{proof}

%%%%%%%%%%%%%%%%%%%%%%%%%%%%%%%%%%%%%%%%%%%%%%
\subsection{A resampling property}
\label{sub:resamplingProp}
In what follows, for $x\in\R^d$, we recall the definition \eqref{eq:shiftOperator} of the shift operator $\tau_x$.  
We begin with the proof of the ``resampling property'' stated in Proposition~\ref{pr:resempPropUps} for the field $\smash{\tilde{\Upsilon}_{\! \lambda}}$. We will use the notation introduced in \eqref{eq:maximal}, \eqref{eq:maximalExpBalls}, and \eqref{eq:maximalExpAnnulus}. Furthermore, we recall that $\Lambda \subseteq \R^{+}$ denotes the set introduced in Lemma~\ref{lm:Countable}.

\begin{proof}[Proof of Proposition~\ref{pr:resempPropUps}]
For $b \in \N$, we recall that $\Upsilon_{\! b} = \Phi_b - \sqrt{\smash[b]{2d}} \frka_b$, where the field $\Phi_b$ is such that, for all $x$, $y \in \R^d$,
\begin{equation*}
\E\bigl[\Phi_b(x) \Phi_b(y)\bigr] = \frka_b(x) + \frka_b(y) - \frka_b(x-y) \;,
\end{equation*}
and where the function $\frka_b$ is defined as in \eqref{eq:frkgb}.
In particular, the field $\Phi_{b}$ is shift invariant in the sense that $\tau_x \Phi_{b}$ has the same law as $\Phi_{b}$ for every $x \in \R^d$.
Then, by applying the Cameron--Martin theorem (see Lemma~\ref{lm_Girsanov}), we have that for all $\bfF : \CC(\R^d) \to \R$ and $y \in \R^d$, it holds that
\begin{align}
\E\bigl[\bfF\bigl(\Upsilon_{\! b}\bigr)\bigr] & = \E\bigl[\bfF\bigl(\Phi_{b} - \sqrt{\smash[b]{2d}} \frka_b\bigr)\bigr] \nonumber \\
& = \E\Bigl[\bfF\bigl(\Phi_{b} + \sqrt{\smash[b]{2d}} \E[\Phi_b(\cdot) \Phi_b(y)] - \sqrt{\smash[b]{2d}} \frka_b\bigr) e^{-\sqrt{\smash[b]{2d}}\Phi_b(y) - 2d \frka_b(y)} \Bigr] \nonumber \\
& = \E\Bigl[\bfF\bigl(\tau_{-y}\Phi_{b} + \sqrt{\smash[b]{2d}} \E[\Phi_b(\cdot) \Phi_b(y)] - \sqrt{\smash[b]{2d}} \frka_b\bigr) e^{-\sqrt{\smash[b]{2d}}\tau_{-y}\Phi_b(y) - 2d \frka_b(y)} \Bigr] \nonumber \\
& = \E\Bigl[\bfF\bigl(\tau_{-y}\Upsilon_{\! b} + \sqrt{\smash[b]{2d}} \tau_{-y} \frka_{b} + \sqrt{\smash[b]{2d}} \E[\Phi_b(\cdot) \Phi_b(y)] - \sqrt{\smash[b]{2d}} \frka_b\bigr) e^{\sqrt{\smash[b]{2d}}\Phi_b(-y) - 2d \frka_b(y)} \Bigr] \nonumber \\
& = \E\Bigl[\bfF\bigl(\tau_{-y}\Upsilon_{\! b}\bigr) e^{\sqrt{\smash[b]{2d}}\Upsilon_{\! b}(-y)}\Bigr] \;. \label{eq:GirsUps}
\end{align}

Now, for each $\lambda \in \Lambda$, $\bfF \in \CC^b_{\loc}(\CC(\R^d))$, and $0 \leq k \leq b$, we can write 
\begin{equation*}
\E\Bigl[\bfF\bigl(\Upsilon_{\! b}\bigr)\one_{\{\M_{0, b}(\Upsilon_{\! b}) \leq \lambda\}}\Bigr] 
= \E\Biggl[\frac{\int_{\B_k} \bfF(\Upsilon_{\! b}) \one_{\{\M_{0, b}(\Upsilon_{\! b}) \leq \lambda\}} e^{\sqrt{\smash[b]{2d}}(\Upsilon_{\! b}(x)- \M_{x, b}(\Upsilon_{\! b}))}\one_{\{\Upsilon_{\! b}(x) \geq \M_{x, b}(\Upsilon_{\! b}) - \lambda\}}\,dx}{\int_{\B_k} e^{\sqrt{\smash[b]{2d}}(\Upsilon_{\! b}(y) -\M_{y, b}(\Upsilon_{\! b}))}\one_{\{\Upsilon_{\! b}(y) \geq \M_{y, b}(\Upsilon_{\! b})- \lambda\}}dy}\Biggr]\;,
\end{equation*}
Therefore, using \eqref{eq:GirsUps}, the expectation on the right-hand side of the previous display is equal to
\begin{align}
& \int_{\B_k} \E\Biggl[\frac{\bfF(\tau_{-x} \Upsilon_{\! b}) \one_{\{\M_{-x, b}(\Upsilon_{\! b}) - \Upsilon_{\! b}(-x) \leq \lambda\}} e^{\sqrt{\smash[b]{2d}}(\Upsilon_{\! b}(-x) - \M_{0, b}(\Upsilon_{\! b}))}\one_{\{\M_{0, b}(\Upsilon_{\! b}) \leq \lambda\}}}{\int_{\B_k} e^{\sqrt{\smash[b]{2d}}(\Upsilon_{\! b}(y-x) -\M_{y-x, b}(\Upsilon_{\! b}))}\one_{\{\Upsilon_{\! b}(y-x) \geq \M_{y-x, b}(\Upsilon_{\! b}) - \lambda\}}dy}\Biggr] dx \nonumber \\
& \qquad = \E\Biggl[\int_{\B_k}  \frac{\bfF(\tau_{x} \Upsilon_{\! b}) e^{\sqrt{\smash[b]{2d}}(\Upsilon_{\! b}(x) - \M_{0, b}(\Upsilon_{\! b}))}\one_{\{\Upsilon_{\! b}(x) \geq \M_{x, b}(\Upsilon_{\! b}) - \lambda\}}}{\int_{\B_k(x)} e^{\sqrt{\smash[b]{2d}}(\Upsilon_{\! b}(y) -\M_{y, b}(\Upsilon_{\! b}))}\one_{\{\Upsilon_{\! b}(y) \geq \M_{y, b}(\Upsilon_{\! b}) - \lambda\}}dy} dx \one_{\{\M_{0, b}(\Upsilon_{\! b}) \leq \lambda\}}\Biggr] \;. \label{eq:startResempling}
\end{align}

We observe that the integral over $x$ inside the expectation in \eqref{eq:startResempling} is bounded above by a positive finite constant. Indeed, since the function $\bfF$ is bounded by assumption and the exponential terms are easily bounded, it suffices to verify that
\begin{equation*}
\int_{\B_k}  \frac{\one_{\{\Upsilon_{\! b}(x) \geq \M_{x, b}(\Upsilon_{\! b}) - \lambda\}}}{\int_{\B_k(x)} \one_{\{\Upsilon_{\! b}(y) \geq \M_{y, b}(\Upsilon_{\! b}) - \lambda\}}dy} dx < \infty \;.
\end{equation*}
To simplify the notation, we introduce the set $\rmA_{\lambda} \eqdef \{x \in \R^d \, : \, \Upsilon_{\! b}(x) \geq \M_{x, b}(\Upsilon_{\! b}) - \lambda\}$. We note that there exist a positive number $\sigma_d$ that depends only on $d$ and points $\{z_1, \ldots z_{\sigma_d}\} \subseteq \R^d$, such that $\B_k \subseteq \cup_{i=1}^{\sigma_d} B(z_i, e^k/2)$. Consequently, this implies that the quantity on the left-hand side of the previous expression is bounded above by
\begin{equation*}
\sum_{i = 1}^{\sigma_d} \int_{B(z, e^k/2)} \frac{\one_{\{x \in \rmA_{\lambda}\}}}{\int_{\B_k(x)} \one_{\{y \in \rmA_{\lambda}\}} dy} dx \leq \sum_{i = 1}^{\sigma_d} \int_{B(z, e^k/2)} \frac{\one_{\{x \in \rmA_{\lambda}\}}}{\int_{B(z, e^k/2)} \one_{\{y \in \rmA_{\lambda}\}} dy} dx \leq \sigma_d \;.
\end{equation*}

We now want to take the limit as $b \to \infty$ in the expectation in \eqref{eq:startResempling} by applying Theorems~\ref{th:cluster}~and~\ref{th:clusterProb}. However, the function that assigns to $\Upsilon_{\! b}$ the quantity inside the expectation in \eqref{eq:startResempling} is not a function in $\CC^b_{\loc}(\CC(\R^d))$. In order to overcome this issue, we proceed in three steps.

\textbf{Step~1:} We can choose $0 < k \leq j \leq b$ sufficiently large such that for all $x \in \B_k$ and all $y \in \B_k(x)$ the following inclusions hold 
\begin{equation*}
\B_{j} \subseteq \B_b \subseteq \B_{2b} \;, \qquad \B_{j} \subseteq \B_b(x) \subseteq \B_{2b} \;, \qquad \B_{j} \subseteq \B_b(y) \subseteq \B_{2b} \;.
\end{equation*}
In particular, this implies that on the event $\{\M_{0, 2b}(\Upsilon_{\! b}) = \M_{0, j}(\Upsilon_{\! b})\}$ it holds, for all $x \in \B_k$ and all $y \in \B_k(x)$, that 
\begin{equation*}
\M_{0, b}(\Upsilon_{\! b}) = \M_{x, b}(\Upsilon_{\! b}) = \M_{y, b}(\Upsilon_{\! b}) = \M_{0, j}(\Upsilon_{\! b}) \;.
\end{equation*}
Hence, since as we have already observed the integral over $x$ in \eqref{eq:startResempling} is bounded above, we have that the difference
\begin{align*}
& \Biggl\lvert\E\Biggl[\int_{\B_k}  \frac{\bfF(\tau_{x} \Upsilon_{\! b}) e^{\sqrt{\smash[b]{2d}}(\Upsilon_{\! b}(x) - \M_{0, b}(\Upsilon_{\! b}))}\one_{\{\Upsilon_{\! b}(x) \geq \M_{x, b}(\Upsilon_{\! b}) -  \lambda\}}}{\int_{\B_k(x)} e^{\sqrt{\smash[b]{2d}}(\Upsilon_{\! b}(y) -\M_{y, b}(\Upsilon_{\! b}))}\one_{\{\Upsilon_{\! b}(y) \geq \M_{y, b}(\Upsilon_{\! b}) - \lambda\}}dy} dx \one_{\{\M_{0, b}(\Upsilon_{\! b}) \leq \lambda\}}\Biggr]	 \\
& \hspace{20mm} - \E\Biggl[\int_{\B_k}  \frac{\bfF(\tau_{x} \Upsilon_{\! b}) e^{\sqrt{\smash[b]{2d}}(\Upsilon_{\! b}(x) - \M_{0, j}(\Upsilon_{\! b}))}\one_{\{\Upsilon_{\! b}(x) \geq \M_{0, j}(\Upsilon_{\! b}) - \lambda\}}}{\int_{\B_k(x)} e^{\sqrt{\smash[b]{2d}}(\Upsilon_{\! b}(y) - \M_{0, j}(\Upsilon_{\! b}))}\one_{\{\Upsilon_{\! b}(y) \geq \M_{0, j}(\Upsilon_{\! b}) - \lambda\}}dy} dx \one_{\{\M_{0, b}(\Upsilon_{\! b}) \leq \lambda\}}\Biggr]\Biggr\rvert 
\end{align*}
is bounded by a positive constant times the following probability 
\begin{equation*}
\P\bigl(\M_{0, 2b}(\Upsilon_{\! b}) \neq \M_{0, j}(\Upsilon_{\! b}), \; \M_{0, b}(\Upsilon_{\! b}) \leq \lambda \bigr) \;. 
\end{equation*}
By using Lemmas~\ref{lm:approxBrownianBridge},~\ref{lm:upperBoundKk} and a version of Lemma~\ref{lm:mainEntrRW} for Brownian motions (see e.g.\ \cite[Lemma~4.18]{BiskupLouidor}), we have that 
\begin{equation*}
\lim_{j \to \infty} \limsup_{b \to \infty} \sqrt{b} \, \P\bigl(\M_{0, 2b}(\Upsilon_{\! b}) \neq \M_{0, j}(\Upsilon_{\! b}), \; \M_{0, b}(\Upsilon_{\! b}) \leq \lambda\bigr) = 0 \;. 	
\end{equation*}

\textbf{Step~2:} We note that on the event $\{\M_{0, j, k/2}(\Upsilon_{\! b}) < \M_{0, j}(\Upsilon_{\! b}) - \lambda\}$, it holds that 
\begin{align*}
& \int_{\B_k}  \frac{\bfF(\tau_{x} \Upsilon_{\! b}) e^{\sqrt{\smash[b]{2d}}(\Upsilon_{\! b}(x) - \M_{0, j}(\Upsilon_{\! b}))}\one_{\{\Upsilon_{\! b}(x) \geq \M_{0, j}(\Upsilon_{\! b}) - \lambda\}}}{\int_{\B_k(x)} e^{\sqrt{\smash[b]{2d}}(\Upsilon_{\! b}(y) -\M_{0, j}(\Upsilon_{\! b}))}\one_{\{\Upsilon_{\! b}(y) \geq \M_{0, j}(\Upsilon_{\! b}) - \lambda\}}dy} dx \\
& \hspace{40mm} =\frac{ \int_{\B_{k/2}} \bfF(\tau_{x} \Upsilon_{\! b}) e^{\sqrt{\smash[b]{2d}}(\Upsilon_{\! b}(x) - \M_{0, j}(\Upsilon_{\! b}))}\one_{\{\Upsilon_{\! b}(x) \geq \M_{0, j}(\Upsilon_{\! b}) - \lambda\}} dx}{\int_{\B_{k/2}} e^{\sqrt{\smash[b]{2d}}(\Upsilon_{\! b}(y) - \M_{0, j}(\Upsilon_{\! b}))}\one_{\{\Upsilon_{\! b}(y) \geq \M_{0, j}(\Upsilon_{\! b}) - \lambda\}}dy}\;.
\end{align*}
Therefore using this fact, we have that the difference 
\begin{align*}
& \Biggl\lvert\E\Biggl[\int_{\B_k}  \frac{\bfF(\tau_{x} \Upsilon_{\! b}) e^{\sqrt{\smash[b]{2d}}(\Upsilon_{\! b}(x) - \M_{0, j}(\Upsilon_{\! b}))}\one_{\{\Upsilon_{\! b}(x) \geq \M_{0, j}(\Upsilon_{\! b}) - \lambda\}}}{\int_{\B_k(x)} e^{\sqrt{\smash[b]{2d}}(\Upsilon_{\! b}(y) -\M_{0, j}(\Upsilon_{\! b}))}\one_{\{\Upsilon_{\! b}(y) \geq \M_{0, j}(\Upsilon_{\! b}) - \lambda\}}dy} dx \one_{\{\M_{0, b}(\Upsilon_{\! b}) \leq \lambda\}}\Biggr] \\
& \hspace{20mm} - \E\Biggl[\frac{\int_{\B_{k/2}} \bfF(\tau_{x} \Upsilon_{\! b}) e^{\sqrt{\smash[b]{2d}}(\Upsilon_{\! b}(x) - \M_{0, j}(\Upsilon_{\! b}))}\one_{\{\Upsilon_{\! b}(x) \geq \M_{0, j}(\Upsilon_{\! b}) - \lambda\}}dx }{\int_{\B_{k/2}} e^{\sqrt{\smash[b]{2d}}(\Upsilon_{\! b}(y) - \M_{0, j}(\Upsilon_{\! b}))}\one_{\{\Upsilon_{\! b}(y) \geq \M_{0, j}(\Upsilon_{\! b}) - \lambda\}}dy} \one_{\{\M_{0, b}(\Upsilon_{\! b}) \leq \lambda\}}\Biggr] \Biggr\rvert 
\end{align*}
is bounded by a finite positive constant times the following probability 
\begin{equation*}
\P\bigl(\M_{0, j, k/2}(\Upsilon_{\! b}) \geq \M_{0, j}(\Upsilon_{\! b}) - \lambda, \; \M_{0, b}(\Upsilon_{\! b}) \leq \lambda\bigr) \;.	
\end{equation*}
Once again, by using Lemmas~\ref{lm:approxBrownianBridge},~\ref{lm:upperBoundKk} and \cite[Lemma~4.18]{BiskupLouidor}, we have that 
\begin{equation*}
\lim_{k \to \infty} \limsup_{b \to \infty} \sqrt{b} \, \P\bigl(\M_{0, j, k/2}(\Upsilon_{\! b}) \geq \M_{0, j}(\Upsilon_{\! b}) - \lambda, \; \M_{0, b}(\Upsilon_{\! b}) \leq \lambda\bigr) = 0 \;. 	
\end{equation*}

\textbf{Step~3:} To conclude, we observe that, since $\bfF \in \CC^b_{\loc}(\CC(\R^d))$, the mapping 
\begin{equation*}
\CC(\R^d) \ni \phi \mapsto \frac{\int_{\B_{k/2}} \bfF(\tau_{x} \phi) e^{\sqrt{\smash[b]{2d}}(\phi(x) - \M_{0, j}(\phi))}\one_{\{\phi(x) \geq \M_{0, j}(\phi) - \lambda\}}dx }{\int_{\B_{k/2}} e^{\sqrt{\smash[b]{2d}}(\phi(y) - \M_{0, j}(\phi))}\one_{\{\phi(y) \geq \M_{0, j}(\phi) - \lambda\}}dy} 
\end{equation*}
depends on the values of $\phi$ inside a compact subsets of $\R^d$. Furthermore, since $\lambda \in \Lambda$, the set of discontinuities of the above mapping is assigned measure zero by the law of the field $\tilde \Upsilon_{\! \lambda}$ in $\CC(\R^d)$ (see Remark~\ref{rm:redTechSmallk}). Hence, by combining \eqref{eq:startResempling} and Steps~1~and~2, by taking the limit as $b \to \infty$, and using Theorems~\ref{th:cluster}~and~\ref{th:clusterProb}, we obtain that  
\begin{align*}
\E\bigl[\bfF\bigl(\tilde \Upsilon_{\! \lambda} \bigr)\bigr] = \lim_{k \to \infty} \lim_{j \to \infty} \E\Biggl[\frac{\int_{\B_{k/2}} \bfF(\tau_{x} \tilde \Upsilon_{\! \lambda}) e^{\sqrt{\smash[b]{2d}}\tilde \Upsilon_{\! \lambda}(x)}\one_{\{\tilde \Upsilon_{\! \lambda}(x) \geq \M_{0, j}(\tilde \Upsilon_{\! \lambda}) - \lambda\}}dx }{\int_{\B_{k/2}} e^{\sqrt{\smash[b]{2d}} \tilde \Upsilon_{\! \lambda}(y)}\one_{\{\tilde \Upsilon_{\! \lambda}(y) \geq \M_{0, j}(\tilde \Upsilon_{\! \lambda})-\lambda\}}dy}\Biggr]  \;, 
\end{align*}
and so the claim follows from the dominated convergence theorem and the monotone convergence theorem.
\end{proof}

We are now ready to prove Proposition~\ref{pr:inversionPsi}.
\begin{proof}[Proof of Proposition~\ref{pr:inversionPsi}]
Let $\lambda \in \Lambda$ and $\bfF \in \CC^b_{\loc}(\CC(\R^d))$. By applying the characterisation \eqref{e:recoverPhi} of the field $\Psi_{\lambda}$ to the right-hand side of \eqref{e:tilt3}, we obtain that the claim follows if one can prove that the following equality holds
\begin{align*}
\E\bigl[\bfF(\tilde{\Upsilon}_{\! \lambda})\bigr] =  \E\Biggl[\frac{\int_{\R^d} \bfF(\tau_x \tilde{\Upsilon}_{\!\lambda})e^{\sqrt{\smash[b]{2d}} \tilde \Upsilon_{\!\lambda}(x)} \one_{\{\tilde\Upsilon_{\!\lambda}(x) \geq \M(\tilde \Upsilon_{\!\lambda}) - \lambda\}} dx}{\int_{\R^d} e^{\sqrt{\smash[b]{2d}}\tilde \Upsilon_{\!\lambda}(x)}\one_{\{\tilde \Upsilon_{\!\lambda}(x) \geq \M(\tilde \Upsilon_{\!\lambda}) - \lambda\}} dx}\Biggr] \;.
\end{align*}
We note that this is precisely \eqref{e:wantedUpsilon}, and thus the claim is established by Proposition~\ref{pr:resempPropUps}.
\end{proof}

%%%%%%%%%%%%%%%%%%%%%%%%%%%%%%%%%%%%%%%%%%%%%%
\subsection{Independence from the choice of threshold}
\label{sub:indthreshold}
In this section, we verify the independence of certain quantities of interest from the arbitrary threshold $\lambda$. Specifically, we begin by proving Proposition~\ref{pr:PsiIndepLambda}, which establishes the independence of the law of the field $\Psi_{\lambda}$ from $\lambda \in \Lambda$. Then, in Lemma~\ref{lm:AStarIndepLambda}, we show that the constant $a_{\star}$, defined in \eqref{eq:defAStar}, is also independent from $\lambda \in \Lambda$.

We start by proving Proposition~\ref{pr:PsiIndepLambda}.
\begin{proof}[Proof of Proposition~\ref{pr:PsiIndepLambda}]
Consider $\lambda_1$, $\lambda_2 \in \Lambda$ such that $\lambda_2 < \lambda_1$. Thanks to Lemma~\ref{lm:Countable}, we know that the field $\smash{\tilde{\Upsilon}_{\! \lambda_2}}$ has the same law as the field $\smash{\tilde{\Upsilon}_{\! \lambda_1}}$, conditioned on the event $\smash{\M(\tilde{\Upsilon}_{\! \lambda_1}) \leq \lambda_2}$.
Therefore, for any $\bfF \in \CC^b_{\loc}(\CC(\R^d))$, recalling Definition~\ref{def:recoverPhi}, we can write 
\begin{equation}
\label{eq:PsiIndepLambdaProof1}
\E\bigl[\bfF(\Psi_{\lambda_2})\bigr] \propto \E\left[\frac{\bfF\bigl(\tau_{x_{\star}} \! \tilde \Upsilon_{\! \lambda_1}\bigr) e^{\sqrt{\smash[b]{2d}} \M(\tilde{\Upsilon}_{\! \lambda_1})}}{\int_{\R^d} e^{\sqrt{\smash[b]{2d}}\tilde\Upsilon_{\! \lambda_1}(x)} \one_{\{\tilde \Upsilon_{\! \lambda_1}(x) \geq \M(\tilde{\Upsilon}_{\! \lambda_1}) - \lambda_2\}} dx} \one_{\{\M(\tilde{\Upsilon}_{\! \lambda_1}) \leq \lambda_2\}}\right] \;,
\end{equation}
where $x_{\star} \in \R^d$ denotes the point where the field $\smash{\tilde{\Upsilon}_{\! \lambda_1}}$ achieves its maximum. 
Now, one would like to apply the resampling property \eqref{e:wantedUpsilon} of the field $\smash{\tilde \Upsilon_{\! \lambda_1}}$ to the quantity on the right-hand side. However, we observe that the function that assigns to $\tilde \Upsilon_{\! \lambda_1}$ the quantity inside the expectation in \eqref{eq:PsiIndepLambdaProof1} is not a function in $\CC^b_{\loc}(\CC(\R^d))$.
%\begin{equation*}
%	\Phi \mapsto \frac{\bfF\bigl(\tau_{x_{\star}} \! \Phi\bigr) e^{\sqrt{\smash[b]{2d}}\M(\Phi)}}{\int_{\R^d} e^{\sqrt{\smash[b]{2d}}\Phi(x)} \one_{\{\Phi(x) \geq \M(\Phi) - \lambda_2\}} dx} \one_{\{\M(\Phi) \leq \lambda_2\}}
%\end{equation*}
Indeed, this map is neither bounded nor dependent on the values of the input function inside a compact set. Moreover, note that it is also not continuous. However, since $\smash{\lambda_1}$, $\smash{\lambda_2 \in \Lambda}$, its set of discontinuities has measure zero under the law of $\smash{\tilde{\Upsilon}_{\! \lambda_1}}$, so this does not pose an issue.

For $k > 0$, let $x_{\star, k} \in \B_k$ denote the point where the field $\smash{\tilde{\Upsilon}_{\! \lambda_1}}$ achieves its maximum inside $\B_k$. Then, to overcome the issues mentioned above, we claim that the following quantity
\begin{equation}
\label{eq:PsiIndepLambdaProof111}
\E\left[\frac{\bfF\bigl(\tau_{x_{\star, k}} \! \tilde \Upsilon_{\! \lambda_1}\bigr) e^{\sqrt{\smash[b]{2d}} \M_{0, k}(\tilde{\Upsilon}_{\! \lambda_1})}}{\bigl(\int_{\B_k} e^{\sqrt{\smash[b]{2d}}\tilde\Upsilon_{\! \lambda_1}(x)} \one_{\{\tilde \Upsilon_{\! \lambda_1}(x) \geq \M_{0, k}(\tilde{\Upsilon}_{\! \lambda_1}) - \lambda_2\}} dx\bigr) \vee L^{-1}} \one_{\{\M_{0, k}(\tilde{\Upsilon}_{\! \lambda_1}) \leq \lambda_2\}}\right]
\end{equation}
converges to the right-hand side of \eqref{eq:PsiIndepLambdaProof1} as $k \to \infty$ first, and then $L \to \infty$.
This follows from Theorems~\ref{th:cluster},~\ref{th:clusterProb} and Lemmas~\ref{lm:repulsionShape},~\ref{lm:tailsDgUpsilon}. For brevity, we do not provide the details here, but we note that similar computations are carried out in the proofs of Lemmas~\ref{lm:truncOK} and \ref{lm:reduceSizeOK} below.

Now, we are in a position to apply the resampling property \eqref{e:wantedUpsilon} of the field $\smash{\tilde \Upsilon_{\! \lambda_1}}$ to the quantity on the right-hand side of \eqref{eq:PsiIndepLambdaProof111} for fixed $k \geq 0$ and $L \geq 0$. Proceeding in this manner, performing some algebraic rearrangements, and then removing the cutoff parameters, it is straightforward to see that the quantity on the right-hand side of \eqref{eq:PsiIndepLambdaProof1} coincides with $\smash{\E[\bfF(\Psi_{\lambda_1})]}$, as desired.
\end{proof}

We now argue that the constant $a_{\star}$, defined in \eqref{eq:defAStar}, is independent of $\lambda \in \Lambda$. To establish this, we first need to prove the following result concerning the constant $c_{\star, \lambda}$, defined in \eqref{eq:defCStar}.
\begin{lemma}
\label{lm:repCStarLambda}
For any $\lambda_1$, $\lambda_2 \in \Lambda$ such that $\lambda_2 \leq \lambda_1$, it holds that  
\begin{equation*}
c_{\star, \lambda_2} = c_{\star, \lambda_1} \P\bigl(\M(\tilde{\Upsilon}_{\! \lambda_1}) \leq \lambda_2\bigr) \;.
\end{equation*} 
\end{lemma}
\begin{proof}
Let $\lambda_1$, $\lambda_2 \in \Lambda$ such that $\lambda_2 \leq \lambda_1$. Recalling \eqref{eq:defCStar}, we have that 
\begin{equation*}
c_{\star, \lambda_2} = \lim_{k\to \infty} \E \bigl[B_k \one_{\{B_k \in [k^{1/6}, k^{5/6}]\}} \one_{\{\M_{0, k}(\Upsilon_{\! \infty})  \leq \lambda_2\}}\bigr]\;.
\end{equation*}
Hence, for fixed $k \geq 0$, by conditioning on the value of $B_k$, the right-hand side of the above display admits the following representation, 
\begin{equation}
\label{eq:intProbIndepLambda}
\frac{1}{\sqrt{2 \pi k}} \int_{k^{1/6}}^{k^{5/6}} u e^{-\f{u^2}{2k}} \P_{0, u, k}\bigl(\M_{0, k}(\Upsilon_{\! \infty}) \leq \lambda_2, \; \M_{0, k}(\Upsilon_{\! k}) \leq \lambda_1\bigr) du \;,
\end{equation}
where here we used the fact that $\lambda_2 < \lambda_1$. Now, by employing a standard entropic repulsion argument that have been used several times in Section~\ref{sec:cluster} (see e.g.\ the proof of Lemma~\ref{lm:repulsionShape}), for all $u \in [k^{1/6}, k^{5/6}]$, one has that 
\begin{equation*}
\lim_{j \to \infty} \limsup_{k \to \infty} \frac{k}{u}\P_{0, u, k}\bigl(\M_{0, k, j}(\Upsilon_{\! k}) > \lambda_2, \; \M_{0, k}(\Upsilon_{\! k}) \leq \lambda_1\bigr) = 0 \;.
\end{equation*}
This fact implies that, for $0 \leq j \leq k $, we can just focus on the following probability  
\begin{equation*}
\P_{0, u, k}\bigl(\M_{0, j}(\Upsilon_{\! k}) \leq \lambda_2, \; \M_{0, k}(\Upsilon_{\! k}) \leq \lambda_1\bigr) \;.
\end{equation*}
In particular, since $\lambda_1$, $\lambda_2 \in \Lambda$ (see also Remark~\ref{rm:redTechSmallk}), we can apply Propositions~\ref{pr:AsyConv}~and~\ref{pr:AsyProb} with $\bfF(\Upsilon_{\! k}) = \one_{\{\M_{0, j}(\Upsilon_{\! k}) \leq \lambda_2\}}$ to obtain that 
\begin{align*}
c_{\star, \lambda_2}
& = \lim_{j \to \infty} \lim_{k \to \infty} \P_{0, u, k}\bigl(\M_{0, j}(\Upsilon_{\! k}) \leq \lambda_2, \; \M_{0, k}(\Upsilon_{\! k}) \leq \lambda_1\bigr) \frac{1}{\sqrt{2 \pi k}}\int_{k^{1/6}}^{k^{5/6}} u e^{-\f{u^2}{2k}} du \\
& = c_{\star, \lambda_1} \P\bigl(\M(\tilde{\Upsilon}_{\! \lambda_1}) \leq \lambda_2\bigr)  \;,
\end{align*}
from which the conclusion follows. 
\end{proof}

We now have all the necessary ingredients to verify that the constant $a_{\star}$, defined in \eqref{eq:defAStar}, is independent of $\lambda \in \Lambda$. More precisely, for $\lambda > 0$, we let
\begin{equation*}
a_{\star, \lambda} \eqdef \frac{\alpha \, c_{\star, \lambda}}{\gamma \, \E\bigl[\int_{\R^d} e^{\sqrt{\smash[b]{2d}}\Psi(x)}\one_{\{\Psi(x)\geq - \lambda\}}dx\bigr]} \;,
\end{equation*}
then we have the following result.
\begin{lemma}
\label{lm:AStarIndepLambda}
For any $\lambda_1$, $\lambda_2 \in \Lambda$, it holds that $a_{\star, \lambda_1} = a_{\star, \lambda_2}$.
\end{lemma}
\begin{proof}
Let $\lambda_1$, $\lambda_2 \in \Lambda$ such that $\lambda_2 \leq \lambda_1$. Then, thanks to Lemma~\ref{lm:repCStarLambda}, we have that 
\begin{equation}
\label{eq:ExpAStarLambda2}
a_{\star, \lambda_2} = \frac{\alpha \,c_{\star, \lambda_1} \P\bigl(\M(\tilde{\Upsilon}_{\! \lambda_1}) \leq \lambda_2\bigr)}{\gamma \, \E\bigl[\int_{\R^d} e^{\sqrt{\smash[b]{2d}}\Psi(x)}\one_{\{\Psi(x)\geq - \lambda_2\}}dx\bigr]} 
\end{equation}

Now, using Proposition~\ref{pr:inversionPsi}, along with the fact that, by construction, $\M(\Psi) = 0$, we have that
\begin{equation*}
\P\bigl(\M(\tilde{\Upsilon}_{\! \lambda_1}) \leq \lambda_2\bigr)= \frac{\E\bigl[\int_{\R^d} e^{\sqrt{\smash[b]{2d}}\Psi(x)}\one_{\{\Psi(x)\geq - \lambda_2\}} dx\bigr]}{\E\bigl[\int_{\R^d} e^{\sqrt{\smash[b]{2d}}\Psi(x)}\one_{\{\Psi(x)\geq - \lambda_1\}} dx\bigr]}\;.
\end{equation*}
Hence, substituting the right-hand side of the above display into the expression \eqref{eq:ExpAStarLambda2} for $a_{\star, \lambda_2}$, the conclusion follows readily.
\end{proof}

%%%%%%%%%%%%%%%%%%%%%%%%%%%%%%%%%%%%%%%%%%%%%%
%%%%%%%%%%%%%%%%%%%%%%%%%%%%%%%%%%%%%%%%%%%%%%
\section{The joint Laplace functional}
\label{sec:proofMainTh1}
The main goal of this section is to prove Proposition~\ref{pr:laplaceJoint} which is the main input for the proof of the stable convergence result Theorem~\ref{th:stableConv}. The main ingredient in the proof of this theorem is Proposition~\ref{pr:joint} below, which will be proved separately in Section~\ref{sec:Joint}. The key idea for the proof is to split the field $\rmX$ into two \emph{independent scales}, namely $\rmX_s$ and $\rmX_{s,t}$, for $0 \leq s < t$. The former field has strong local interactions, while the latter one has a strong independence structure. In particular, as we will see, the exponential of the field $\rmX_{s,t}$ macroscopically behaves like an independent stable random measure that will integrate the exponential of the field $\rmX_s$, thus generating the desired limiting measure.

This section is organised as follows. In Section~\ref{sub:setupJoint}, we introduce several definitions and state Proposition~\ref{pr:joint}, which plays a fundamental role in the proof of Proposition~\ref{pr:laplaceJoint}. In Section~\ref{sub:JointHigh}, assuming Proposition~\ref{pr:joint}, we establish the convergence of the ``small scales'' while conditioning on the ``large scales''. In Section~\ref{sub:JointLow}, we prove the ``full'' conditional convergence. Finally, in Section~\ref{sub:JointStable}, we show how Proposition~\ref{pr:laplaceJoint} follows directly from the results obtained thus far, and we also prove Theorem~\ref{th:stableConv}.

%%%%%%%%%%%%%%%%%%%%%%%%%%%%%%%%%%%%%%%%%%%%%%
\subsection{Setup and preliminary results}
\label{sub:setupJoint}
For the reminder of this section, we fix $\gamma > \sqrt{\smash[b]{2d}}$. For $n \in \N$, we consider a collection of fields $(\rmW_{i, t})_{i \in [n], t \geq 0}$ satisfying  \ref{hp:W1} \dash \ref{hp:W4}. We recall that we need to prove that for all $(\varphi, (f_i)_{i \in [n]}) \in \CC^{\infty}_c(\R^d) \times (\CC_c^{+}(\R^d))^n$, the following limit holds
\begin{equation*}
\lim_{t \to \infty} \E\Biggl[\exp\bigl(i \langle \rmX, \varphi \rangle \bigr) \prod_{i = 1}^{n} \exp\bigl(-\mu_{\gamma, t, i}(f_i)\bigr)\Biggr] = \E\Bigl[\exp\bigl(i \langle \rmX, \varphi \rangle \bigr) \exp\bigl(- \tilde a_{\star} \mu_{\gammac}\bigl(\rmT_{\gamma}[f_1, \ldots, f_n]\bigr)\bigr)\Bigr] \;,
\end{equation*}	
where we refer to the statement of Proposition~\ref{pr:laplaceJoint} for the definitions of all the quantities involved. 

To simplify the presentation, we assume that the functions $(f_i)_{i \in [n]}$ are all of the form $f_i = \theta_i \one_{[0, 1]^d}$ for some non-negative constants $(\theta_i)_{i \in [n]} \subseteq \R^{+}_{0}$. The proof developed below works in the general case where $(f_i)_{i \in [n]} \in (\CC_c^{+}(\R^d))^n$, with minor and straightforward adjustments. In particular, in the general setting, the unit box $[0, 1]^d$ can be replaced by any compact set that contains the supports of all the functions $(f_i)_{i \in [n]}$.

With a slight abuse of notation, for all $t \geq 0$, it is convenient to introduce the function $\rmF_{\gamma, t} : \R^d \to \R$ and the quantity $\rmT_{\gamma} \geq 0$ given by
\begin{equation}
\label{eq:RPsiGamma}
\rmF_{\gamma, t}(x) \eqdef \sum_{i = 1}^{n} \theta_i e^{\gamma \rmW_{i, t}(x)} \;, \qquad \rmT_{\gamma} \eqdef \E\Biggl[\Biggl(\sum_{i = 1}^{n} \theta_i \int_{\R^d} \exp\bigl(\gamma\bigl(\rmW_i(y) + \Psi(y)\bigr)\bigr) dy\Biggr)^{\f{\sqrt{\smash[b]{2d}}}{\gamma}}\Biggr] \;,
\end{equation}
where $\Psi$ denotes the field introduced in Definition~\ref{def:recoverPhi}, which, as stated in Proposition~\ref{pr:PsiIndepLambda}, has a law independent of $\lambda \in \Lambda$, where $\Lambda$ is the set introduced in Lemma~\ref{lm:Countable}. For this reason, we omit the subscript $\lambda$ when referring to the field $\Psi$.

In this way, assuming that the function $(f_i)_{i \in [n]}$ are as specified above, we need to prove that the following limit holds
\begin{equation}
\label{eq:ReductionMainth2}
\lim_{t \to \infty} \E\Bigl[\exp\bigl(i \langle \rmX, \varphi \rangle \bigr) \exp\bigl(- \mu_{\gamma, t}(\rmF_{\gamma, t})\bigr)\Bigr] = \E\Bigl[\exp\bigl(i \langle \rmX, \varphi \rangle \bigr) \exp\Bigl(- \tilde a_{\star} \rmT_{\gamma} \, \mu_{\gammac}\bigl([0, 1]^d\bigr)\Bigr)\Bigr] \;.
\end{equation}

In order to prove \eqref{eq:ReductionMainth2}, we introduce a decomposition of the unit box $[0, 1]^d$ as follows. For $\rmR \geq 1$, we consider $s \geq 0$ such that $(e^s+1) (\rmR+1)^{-1} \in \N$, and we write 
\begin{equation}
\label{eq:decoBox}
[0, 1]^d = \bigcup_{i \in [\frkI]} A_{i} \cup B_{\rmR, s}\;,
\end{equation}
where $(A_i)_{i \in [\frkI]}$ and $B_{\rmR, s}$ are as described in the caption of Figure~\ref{fig:deco}.
\begin{figure}[ht] 
\begin{center}
\begin{tikzpicture}[scale=0.95] 
\draw[fill=gray!20] (0,0) rectangle (5.85,5.85);
\foreach \x in {0,...,5}{
    \foreach \j in {0,...,5}
        \draw[fill=white!60] (\x,\j) rectangle (\x+0.85,\j+0.85);
}
\draw[<->,>=latex'] (-0.3, 0) -- (-0.3,5.85) node[midway,left] {\small $1$};
\draw[<->,>=latex'] (0,-0.3) -- (5.85,-0.3) node[midway,below] {\small $1$};

\draw[dotted] (5.85, 0) -- (11.85,1);
\draw[dotted] (3.85, 0) -- (7.85,1);
\draw[dotted] (3.85, 2) -- (7.85,5);
\draw[dotted] (5.85, 2) -- (11.85,5);

\draw[fill=gray!20] (7.85, 1) rectangle (11.85,5);
\foreach \x in {8.25,10.25}{
    \foreach \j in {1, 3}
        \draw[fill=white!60] (\x,\j) rectangle (\x+1.6,\j+1.6);
}

\draw[<->,>=latex'] (12.05, 1) -- (12.05,2.6) node[midway,right] {\small $\rmR e^{-s}$};
\draw[<->,>=latex'] (10.25,0.8) -- (11.85,0.8) node[midway,below] {\small $\rmR e^{-s}$};
\draw[<->,>=latex'] (12.05, 4.6) -- (12.05,5) node[midway,right] {\small $e^{-s}$};
\draw[<->,>=latex'] (7.85,0.8) -- (8.25,0.8) node[midway,below] {\small $e^{-s}$};
\end{tikzpicture}
\caption[short form]{\small{A diagram illustrating the decomposition of the cube $[0, 1]^d$ used in the proof of Lemma~\ref{lem:magic}, in the case $d = 2$. For $\rmR \geq 1$, we suppose that along each edge of the $d$ edges of $[0, 1]^d$, there are exactly $(e^s+1) (\rmR+1)^{-1}$ smaller closed boxes, which are drawn in white and that we enumerate by $(A_{i})_{i \in [\frkI]}$, for some $\frkI \in \N$. For each $i \in [\frkI]$, the box $A_{i}$ has side length equal to $R e^{-s}$. The shaded region is the buffer zone $B_{\rmR, s} = [0, 1]^d \setminus \cup_{i \in [\frkI]} A_{i}$. By construction, the buffer zone separates the smaller cubes by a distance of at least $e^{-s}$.}}
\label{fig:deco}
\end{center}
\end{figure}
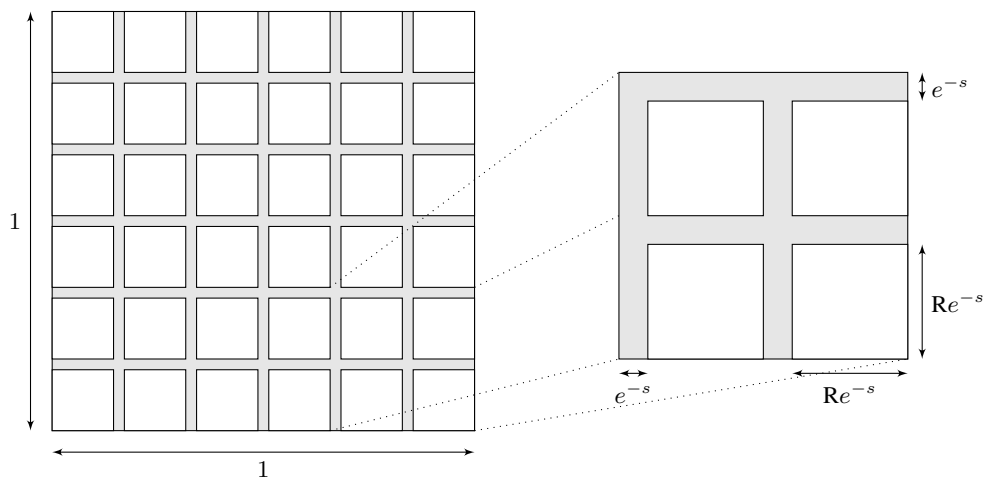

The reason why the decomposition \eqref{eq:decoBox} turns out to be useful is due to the short-range correlation properties of the fields $\rmX_t$ and $\rmW_t$. Indeed, since we are assuming that the seed covariance function $\frkK$ satisfies condition \ref{hp_K2}, as we have observed in Section~\ref{sub:propStarScale}, for $0 \leq s < t$, it holds that $\rmX_{s,t}(x)$ and $\rmX_{s,t}(y)$ are independent for $|x-y| > e^{-s}$. Similarly, thanks to \ref{hp:W3}, for $0 \leq s < t$, we have that $(\rmW_{i, t}(x))_{i \in [n]}$ and $(\rmW_{i, t}(y))_{i \in [n]}$ are independent for $|x-y| > e^{-s}$.

As mentioned earlier, the idea is to split the field $\rmX$ into two scales. Roughly speaking, we first focus on the small scales conditioned on the large scales. More precisely, we consider the sequence of measures $(\mu_{\gamma, s, t})_{0 \leq s < t}$ on $\R^d$, defined as follows,
\begin{equation}
\label{e:mu_st}
\mu_{\gamma, s, t}(dx)  \eqdef  (t-s)^{\f{3\gamma}{2\sqrt{\smash[b]{2d}}}}e^{(t - s) (\gamma/\sqrt{\smash[b]{2}} - \sqrt{\smash[b]{d}})^2} e^{\gamma \rmX_{s, t}(x) - \frac{\gamma^2}{2}(t-s)} dx \;,
\end{equation}
The role of the large scales is played by a deterministic\footnote{At this level, we are assuming that we are conditioning on the large scales. That is why, we are thinking of the large scales as being described by a deterministic function.} function $\chi : \R^d \to \R$. Given a function $\chi$, we define the measure $\rho_{\chi}$ by letting 
\begin{equation}
\label{eq:measChi}
\rho_{\chi}(dx) \eqdef \chi(x) e^{-\sqrt{\smash[b]{2d}} \chi(x)} dx \;.
\end{equation}

We now have all the necessary ingredients to state the following key proposition,
which is a generalisation of \cite[Proposition~3.2]{Glassy}. The key difference is that we allow $\rmF_{\gamma, t}$ to take the general form given in \eqref{eq:RPsiGamma}, rather than restricting it to be a constant. More importantly, we explicitly identify the multiplicative constant $C(\gamma)$ that appears in the statement of \cite[Proposition~3.2]{Glassy}.
The proof of this proposition is postponed to Section~\ref{sec:Joint}.

\begin{proposition}
\label{pr:joint}
Fix $\eps > 0$, let $\rmF_{\gamma}$ and $\rmT_\gamma$ be as in \eqref{eq:RPsiGamma}, and consider the constant $\tilde a_{\star} >0$ introduced in Proposition~\ref{pr:laplaceJoint}. Then, for all $s > 0$ large enough and for all $t \geq s$ large enough, the following holds. For any function $\chi:[0, \rmR]^d \to \R$ such that 
\begin{equation}
\label{eq:condChi}
\min_{x \in [0, \rmR]^d} \chi(x) \geq\f{1}{8 \sqrt{\smash[b]{2d}}} \log s\;, \quad  \max_{x \in [0, \rmR]^d} \chi(x) \leq \log t\;, \quad  \sup_{\substack{x, y \in [0, \rmR]^d, \\ |x-y| \leq s^{-1}}} \frac{\abs{\chi(x) - \chi(y)}}{\abs{x-y}^{1/3}} \leq 1 \;,
\end{equation}
one has that,
\begin{equation*}
\Biggl\lvert \E \Biggl[\exp\Biggl(- \int_{[0, \rmR]^d}  \rmF_{\gamma, t}(x) e^{-\gamma\chi(x)} \mu_{\gamma, t}(dx)\Biggr)\Biggr]  - 1 - \tilde a_{\star} \rmT_{\gamma} \rho_{\chi}\bigl([0, \rmR]^d\bigr) \Biggr\rvert \leq \eps \rho_{\chi}\bigl([0, \rmR]^d\bigr)\;.
\end{equation*}
\end{proposition}

Note that our definition of $\rho_{\chi}$ includes no analogue of the term 
$-\gamma^{-1}\log\theta$ appearing in \cite[Equation~3.11]{Glassy}.
This is because this term can simply be absorbed into the error term by making $s$ 
sufficiently  large.

%%%%%%%%%%%%%%%%%%%%%%%%%%%%%%%%%%%%%%%%%%%%%%
\subsection{Convergence of the small scales}
\label{sub:JointHigh}
We are now ready to state and the prove the following result in which, roughly speaking, we compute the Laplace functional of the measure $\mu_{\gamma, s, t}$ defined in \eqref{e:mu_st}. Before proceeding, we introduce the following definition.

\begin{definition}
\label{def:ChiAdm}
We say that a function $\chi:[0, 1]^d \to \R$ is \emph{scaled-admissible} if for all $\eps > 0$, there exists $\rmR_0 \geq 1$ such that $\smash{e^{ds}\rho_{\chi}(\bar{B_{\rmR, s}}) \leq \eps}$ for all $s \geq 0$ sufficiently large and for all $\rmR \geq \rmR_0$ such that $\smash{(e^s+1) (\rmR+1)^{-1}} \in \N$.
\end{definition}

We have the following key result.
\begin{lemma}
\label{lem:magic}
Let $\eps > 0$. Then, for $s \geq 0$ sufficiently large, for $t > s$ sufficiently large, and for any scaled-admissible function $\chi: [0, 1]^d \to \R$ in the sense of Definition~\ref{def:ChiAdm}, such that 
\begin{equation}
\label{eq:condChiScaled}
	\min_{x \in [0, 1]^d} \chi(x) \geq\f{1}{8 \sqrt{\smash[b]{2d}}} \log s \;, \hspace{2mm} \max_{x \in [0, 1]^d} \chi(x) \leq \log t \;, \hspace{2mm} \sup_{\substack{x, y \in [0, 1]^d, \\ |x-y| \leq e^{-s}s^{-1}}} \frac{|\chi(x) - \chi(y)|}{|x-y|^{1/3}} \leq e^{s/3} \;,
\end{equation}
one has that, 
\begin{align*}
\E\Biggl[ \exp\Biggl(- e^{ds} \int_{[0,1]^d} \rmF_{\gamma, t}(x)  e^{-\gamma \chi(x)} \mu_{\gamma, s, t}(dx)\Biggr)\Biggr] & \geq \exp \Bigl(-(\tilde a_{\star} + \eps) e^{ds} \rmT_{\gamma} \rho_{\chi}\bigl([0, 1]^d\bigr) \Bigr) -\eps \;, \\ 
\E\Biggl[ \exp\Biggl(- e^{ds} \int_{[0,1]^d} \rmF_{\gamma, t}(x)  e^{-\gamma \chi(x)} \mu_{\gamma, s, t}(dx)\Biggr)\Biggr] & \leq \exp \Bigl(-(\tilde a_{\star} - \eps) e^{ds} \rmT_{\gamma} \rho_{\chi} \bigl([0, 1]^d\bigr)\Bigr) + \eps \;,
\end{align*}
\end{lemma}
where $\tilde a_{\star} >0$ is the constant introduced in Proposition~\ref{pr:laplaceJoint}.
\begin{proof}
For $\rmR \geq 1$, $s \geq 0$ such that $(e^s+1)(\rmR+1)^{-1} \in \N$, we consider the decomposition of the cube $[0, 1]^d$ given in \eqref{eq:decoBox}. We also consider $t > s$ and a function $\chi:[0, 1]\to\R^d$ satisfying conditions \eqref{eq:condChi}. In order to lighten the notation, we define
\begin{equation}
\label{e:magic_light}
\E_{\eqref{e:magic_light}} \eqdef \E\Biggl[ \exp\Biggl(- e^{ds} \int_{[0,1]^d} \rmF_{\gamma, t}(x)  e^{-\gamma \chi(x)} \mu_{\gamma, s, t}(dx)\Biggr)\Biggr] \;.
\end{equation}
We write the integral over $[0, 1]^d$ appearing inside $\E_{\eqref{e:magic_light}}$ as the following sum
\begin{align*}
\int_{B_{\rmR, s}} \rmF_{\gamma, t}(x) e^{-\gamma \chi(x)} &  \mu_{\gamma, s, t}(dx) + \int_{\cup_{i \in [\frkI]} A_{i}} \rmF_{\gamma, t}(x) e^{-\gamma \chi(x)}  \mu_{\gamma, s, t}(dx)	 \;,
\end{align*}
and we divide the proof into three main steps. In the first step, we show that the integral over the buffer-zone can be made arbitrarily small by choosing $s > 0$ large enough and $t \geq s$ large enough. In the second step, we provide, upper and lower bound for the integral over the region given by the union of the small squares. 
Finally, in the third step, we show how to combine everything to obtain the desired result. We remark that the proof strategy follows similar lines to the proof of \cite[Lemma~3.1]{Glassy}.

\textbf{Step~1:} In this first step, we show that for any $\eps>0$, there exist $\rmR \geq 1$ large enough such that for all $s > 0$ large enough satisfying $(e^s + 1)(\rmR+1)^{-1} \in \N$, and $t \geq s$ large enough, it holds that 
\begin{equation}
\label{e:magic_step1_0}
1 - \E \Biggl[\exp \Biggl(- e^{ds}\int_{B_{\rmR, s}} \rmF_{\gamma, t}(x) e^{-\gamma \chi(x)} \mu_{\gamma, s, t}(dx) \Biggr)\Biggr] \leq \eps \;.
\end{equation}
By making the change of variables $x \mapsto e^{s} x$, using the scaling relation \eqref{e:scaling_rel}, and thanks to \ref{hp:W1} and \ref{hp:W2}, we obtain that the expectation on the left-hand side of \eqref{e:magic_step1_0} is equal to
\begin{equation*}
\label{e:magic_step1_1}
\E \Biggl[\exp \Biggl(-\int_{e^s B_{\rmR, s}}  \rmF_{\gamma, t-s}(x) e^{\gamma \chi(e^{-s}x)} \mu_{\gamma, t-s}(dx)\Biggr)\Biggr] \;.
\end{equation*}
Next, we note that there exists $\frkJ \in \N$ and a sequence of points $(x_{j})_{j \in [\frkJ]} \subset \R^d$ such that 
\begin{equation}
\label{e:magic_step1_inter_points}
e^s \bar{B_{\rmR, s}} = \bigcup_{j \in [\frkJ]} \; \bigl(x_{j} + [0, 1]^d\bigr)\;, \qquad  \bigcap_{i=1}^{d+2} \; \big(x_{j_{i}} + [0, 1]^d\big)	= \emptyset, \;\; \forall \, j_1 \neq  \ldots \neq j_{d+2} \in [\frkJ] \;.
\end{equation}
Therefore, using the inequality $\smash{1-\prod_{j \in [\frkJ]} a_{j} \leq \sum_{j \in [\frkJ]}(1-a_{j})}$, which is valid for $\smash{(a_{j})_{j \in [\frkJ]} \subset [0, 1]}$, and the translation invariance of $\rmX_{t-s}$, as well as \ref{hp:W2}, we obtain that
\begin{align}
\label{e:magic_step1_2}
1 - \E\Biggl[ \exp \Biggl(- & \int_{e^s \rmB_{\rmR, s}} \rmF_{\gamma, t-s}(x) e^{\gamma \chi(e^{-s}x)} \mu_{\gamma, t-s}(dx)\Biggr) \Biggr] \nonumber \\
& \quad \leq 1 - \E\Biggl[ \exp \Biggl(- \sum_{j \in [\frkJ]} \int_{x_j + [0, 1]^d} \rmF_{\gamma, t-s}(x) e^{\gamma \chi(e^{-s}x)} \mu_{\gamma, t-s}(dx) \Biggr) \Biggr] \nonumber\\
& \quad  = \E \Biggl[1- \prod_{j \in [\frkJ]}\exp \Biggl(-\int_{x_j+ [0, 1]^d} \rmF_{\gamma, t-s}(x) e^{\gamma \chi(e^{-s}x)} \mu_{\gamma, t-s}(dx)\Biggr)\Biggr] \nonumber\\
& \quad \leq \sum_{j \in [\frkJ]} \Biggl(1 - \E\Biggl[ \exp \Biggl(-\int_{[0, 1]^d} \rmF_{\gamma, t-s}(x) e^{\gamma \chi(x_j + e^{-s}x)} \mu_{\gamma, t-s}(dx)\Biggr)\Biggr] \Biggr)  \;.
\end{align}
Since by assumption the function $\chi$ satisfies all the conditions in \eqref{eq:condChiScaled}, one can readily check that the function $\chi(x_{j} + e^{-s} \, \cdot)$ satisfies all the conditions in \eqref{eq:condChi}, for all $j \in [\frkJ]$. Therefore, by Proposition~\ref{pr:joint}, for any $\eps>0$, for all $s > 0$ large enough satisfying $(e^s + 1)(\rmR+1)^{-1} \in \N$, and for all $t \geq s$ large enough, we have that the expectation in \eqref{e:magic_step1_2} satisfies the following inequality, for all $j \in [\frkJ]$,
\begin{equation*}
1 - \E \Biggl[\exp \Biggl(- \int_{[0, 1]^d} \rmF_{\gamma, t-s}(x) e^{\gamma \chi(x_j + e^{-s}x)} \mu_{\gamma, t-s}(dx) \Biggr) \Biggr] \leq (\tilde a_{\star} + \eps) \rmT_{\gamma} e^{ds} \rho_{\chi(e^{-s} \cdot)} \bigl(x_j + [0, 1]^d\bigr) \;.
\end{equation*}
Hence, plugging the right-hand side of the previous expression into \eqref{e:magic_step1_2}, making the change of variables $x \mapsto x_j +  e^{-s} x$, and using the assumptions \eqref{e:magic_step1_inter_points} on the sequence of points $(x_{j})_{j \in \frkJ}$, we obtain that
\begin{equation*}
1 - \E\Biggl[ \exp \Biggl(- \int_{e^s B_{\rmR, s}} \rmF_{\gamma, t-s}(x)  e^{\gamma \chi(e^{-s}x)} \mu_{\gamma, t-s}(dx)\Biggr) \Biggr]  \leq (d+2) (\tilde a_{\star} + \eps) \rmT_{\gamma}  e^{ds} \rho_{\chi}\bigl(\bar{B_{\rmR, s}}\bigr) \;,
\end{equation*}
Since, we are assuming that $\chi$ is scaled-admissible in the sense of Definition~\ref{def:ChiAdm}, the quantity $e^{ds} \rho_{\chi}(\bar{B_{\rmR, s}})$ can be made arbitrarily small by taking $s \geq 0$ and $\rmR \geq 1$ large enough.

\textbf{Step~2:} In this second step, we show that, for any $\eps >0$, there exist $\rmR \geq 1$ large enough such that for all $s > 0$ large enough satisfying $(e^s + 1)(\rmR+1)^{-1} \in \N$, and $t \geq s$ large enough, it holds that 
\begin{align*}
& \E \Biggl[\exp \Biggl(- e^{ds}\int_{\cup_{i \in [\frkI]} A_{i}} \rmF_{\gamma, t}(x) e^{-\gamma \chi(x)} \mu_{\gamma, s, t}(dx)\Biggr)\Biggr] \geq \exp \Bigl(- (\tilde a_{\star} + \eps) e^{ds} \rmT_{\gamma} \rho_{\chi}\bigl([0, 1]^d\bigr) \Bigr)\\
& \E \Biggl[\exp \Biggl(- e^{ds}\int_{\cup_{i \in [\frkI]} A_{i}} \rmF_{\gamma, t}(x) e^{-\gamma \chi(x)} \mu_{\gamma, s, t}(dx)\Biggr)\Biggr] \leq \exp \Bigl(- (\tilde a_{\star} - \eps) e^{ds} \rmT_{\gamma} \rho_{\chi}\bigl([0, 1]^d\bigr)\Bigr) + \eps \;.
\end{align*}
Since the lower bound can be obtained in a similar manner to the upper bound, we just focus on the latter. We recall that if $x$, $y \in [0,1]^d$ belong to two different squares in the decomposition described in Figure~\ref{fig:deco}, then $|x-y| \geq e^{-s}$ and so $\rmX_{s,t}(x)$ and $\rmX_{s,t}(y)$ are independent. Moreover, if $x$, $y$ are as above, then thanks to \ref{hp:W3}, we also have that $(\rmW_{i, t}(x))_{i \in [n]}$ and $(\rmW_{i, t}(y))_{i \in [n]}$ are independent.
Hence, using this fact, making the change of variables $x \mapsto e^{s}x$, and preceding in the same manner in the first step, we obtain that
\begin{align*}
\E \Biggl[\exp \Biggl(- e^{ds} \int_{\cup_{i \in [\frkI]} A_{i}} \rmF_{\gamma, t}(x) & e^{-\gamma \chi(x)} \mu_{\gamma, s, t}(dx)\Biggr) \Biggr] \\
& = \prod_{i \in [\frkI]} \E \Biggl[ \exp \Biggl(-\int_{e^s A_{i}}  \rmF_{\gamma, t-s}(x) e^{-\gamma \chi(e^{-s} x)}  \mu_{\gamma, t - s}(dx) \Biggr)\Biggr] \;.
\end{align*}

Now, since for each $i \in [\frkI]$, the function $\chi(e^{-s}\cdot)$ on $e^s A_{i}$ satisfies all the conditions stated in \eqref{eq:condChi}, by Proposition~\ref{pr:joint}, we get that for all $s > 0$ large enough satisfying $(e^s + 1)(\rmR+1)^{-1} \in \N$, and $t \geq s$ large enough, it holds that 
\begin{equation*}
\E \Biggl[\exp \Biggl(- \int_{e^s A_{i}} \rmF_{\gamma, t-s}(x) e^{-\gamma \chi(e^{-s} x)} \mu_{\gamma, t - s}(dx) \Biggr)\Biggr] \leq 1 - (\tilde a_{\star} - \eps) \rmT_{\gamma} \rho_{\chi(e^{-s} \cdot)}\bigl(e^s A_{i}\bigr) \;.
\end{equation*}
For all $i \in [\frkI]$, using the assumptions on the function $\chi$, we obtain that for sufficiently large $s \geq 0$, the following holds
\begin{equation*}
(\tilde a_{\star} - \eps) \rmT_{\gamma} \rho_{\chi(e^{-s} \cdot)}\bigl(e^s A_{i}\bigr) \leq \rmR^d (\tilde a_{\star} - \eps) \rmT_{\gamma} \Biggl(\sup_{z \in [\log s/(8 \sqrt{\smash[b]{2d}}), \infty)} z e^{-\sqrt{\smash[b]{2d}}z}\Biggr) < 1 \;.
\end{equation*}
Therefore, using the inequality $\prod_{i \in [\frkI]} (1-a_{i}) \leq \exp(-\sum_{i \in [\frkI]} a_{i})$, which is valid for $(a_{i})_{i \in [\frkI]} \subset [0, 1]$, and making the change of variables $x \mapsto e^s x$, we get that
\begin{align*}
\E \Biggl[\exp \Biggl(- e^{ds} \int_{\cup_{i \in [\frkI]} A_{i}} & \rmF_{\gamma, t}(x) e^{-\gamma \chi(x)} \mu_{\gamma, s, t}(dx)\Biggr) \Biggr] \leq \exp \Bigl(- (\tilde a_{\star} - \eps) e^{d s} \rmT_{\gamma} \rho_{\chi}\bigl(\cup_{i \in [\frkI]} A_i\bigr) \Bigr) \\
& = \exp \Bigl(- (\tilde a_{\star} - \eps) e^{d s} \rmT_{\gamma} \bigl(\rho_{\chi}\bigl([0,1]^d\bigr) - \rho_{\chi}\bigl(B_{\rmR, s}\bigr)\bigr) \Bigr) \\
& \leq \exp\Bigl(- (\tilde a_{\star} - \eps) e^{d s} \rmT_{\gamma} \rho_{\chi}\bigl([0,1]^d\bigr) \Bigr) + \eps\;,
\end{align*}
where, once again, the last inequality is due to the fact that, since by assumption $\chi$ is scaled-admissible in the sense of Definition~\ref{def:ChiAdm}, the quantity $e^{ds} \rho_{\chi}(\bar{B_{\rmR, s}})$ can be made arbitrarily small by taking $s \geq 0$ and $\rmR \geq 1$ large enough. 

\textbf{Step~3:} In this final step, we show how to combine the previous two steps to obtain the desired result. Thanks to decomposition \eqref{eq:decoBox} and to the elementary inequality $a b \geq a + b -1$ valid for $a$, $b \in [0, 1]$, we have that 
\begin{align*}
\E_{\eqref{e:magic_light}} & \geq \E\Biggl[ \exp\Biggl(- e^{ds} \int_{\cup_{i \in [\frkI]} A_i} \rmF_{\gamma, t}(x) e^{-\gamma \chi(x)} \mu_{\gamma, s, t}(dx)\Biggr)\Biggr] \\
& \hspace{40mm} + \E\Biggl[ \exp\Biggl(- e^{ds} \int_{B_{\rmR, s}} \rmF_{\gamma, t}(x) e^{-\gamma \chi(x)} \mu_{\gamma, s, t}(dx)\Biggr)\Biggr]  - 1 \;.
\end{align*}
On the other hand, we also have the trivial inequality
\begin{align*}
\E_{\eqref{e:magic_light}} \leq \E\Biggl[ \exp\Biggl(- e^{ds} \int_{\cup_{i \in [\frkI]} A_i} \rmF_{\gamma, t}(x) e^{-\gamma \chi(x)} \mu_{\gamma, s, t}(dx)\Biggr)\Biggr] \;.
\end{align*}
Hence, to conclude it suffices to combine Steps~1~and~2. 
\end{proof}

%%%%%%%%%%%%%%%%%%%%%%%%%%%%%%%%%%%%%%%%%%%%%%%%%%%%%%%%%%%%%%%%%%%%%%%%%%%%%%%%%%%%%%%%%%%%%%%%%%%%%%%%%%%%%%%%%
\subsection{Convergence of the large scales}
\label{sub:JointLow}
The main goal of this subsection is to compute the conditional Laplace functionals of the random measures $(\mu_{\gamma, t})_{t \geq 0}$ conditioned on the $\sigma$-field $\CF_s$ defined in \eqref{eq:defSigmaFieldT}. 
Heuristically speaking, we want to get the expectations appearing in the statement of Lemma~\ref{lem:magic} to be of order one, and so we 
want to absorb some normalisation factor $r(s)$ into $\smash{\mu_{\gamma, s, t}}$. Proceeding formally for the moment, given a function $\chi:[0, 1]^d \to\R$ satisfying the conditions in the statement of Lemma~\ref{lem:magic}, we consider the function $\tilde \chi : [0, 1]^d \to \R$ such that, for all $x \in [0, 1]^d$,
\begin{equation*}
e^{-\gamma \chi(x)} =  e^{-ds} r(s) e^{-\gamma \tilde \chi(x)} \;,
\end{equation*}
so that, by Lemma~\ref{lem:magic} with $n = 1$, $\rmW_{1, \cdot} = 0$, and $\theta_1 = 1$, the following approximate identity holds
\begin{equation*}
\E \Biggl[\exp \Biggl(-\int_{[0, 1]^d} r(s) e^{-\gamma \tilde \chi(x)} \mu_{\gamma, s, t}(dx)\Biggr)\Biggr] \approx \E\Bigl[\exp\Bigl(- c e^{ds} \rho_{\tilde \chi - \log(r(s))/\gamma + ds/\gamma}\bigl([0,1]^d\bigr) \Bigr)\Bigr] \;.
\end{equation*}
for some constant $c > 0$.
Now recalling the definition \eqref{eq:measChi}, it is easily seen that to get something of order one on the right-hand side of the above expression, we need to choose $r(s)$ such that 
\begin{equation*}
e^{ds-\f{\sqrt{\smash[b]{2d}}}{\gamma} d s} r(s)^{\f{\sqrt{\smash[b]{2d}}}\gamma} \bigl|\log (e^{-ds} r(s))\bigr| \approx 1
\end{equation*}
which is achieved by setting
\begin{equation*}
r(s) = e^{ds - \gamma s \sqrt{\smash[b]{d/2}}} s^{-\f\gamma{\sqrt{\smash[b]{2d}}}}\;.
\end{equation*}
In particular, this suggests to define the measure $\smash{\tilde\mu_{\gamma, s,t} \eqdef r(s) \mu_{\gamma, s,t}}$ which is given by
\begin{equation}
\label{eq:muTilde_st}
\tilde\mu_{\gamma,s,t}(dx) =	s^{-\f\gamma{\sqrt{\smash[b]{2d}}}} e^{\gamma s \f{\sqrt{\smash[b]{2d}}}{2}} (t-s)^{\f{3\gamma}{2\sqrt{\smash[b]{2d}}}}e^{t(\gamma/\sqrt{\smash[b]{2}} - \sqrt{\smash[b]{d}})^2} e^{\gamma \rmX_{s, t}(x) -  \frac{\gamma^2}{2} t} dx \;.
\end{equation} 
Given $\rmR \geq 1$, $s \geq 0$, a function $\chi:[0, \rmR]^d \to \R$, we define the measure
\begin{equation}
\label{eq:measTildeChi}
\rho_{\chi, s}(dx) \eqdef e^{-\sqrt{\smash[b]{2d}} \chi(x)} \Biggl(\sqrt{\smash[b]{d/2}}+\f{\chi(x)}{s} + \f{\log s}{\sqrt{\smash[b]{2d}} s}\Biggr) dx \;.
\end{equation}

\begin{definition}
\label{def:ChiAdmS}
We say that a function $\chi:[0, 1]^d \to \R$ is \emph{admissible} if for all $\eps > 0$, there exists $\rmR_0 \geq 1$ such that $\smash{\rho_{\chi, s}(\bar{B_{\rmR, s}}) \leq \eps}$ for all $s \geq 0$ sufficiently large and for all $\rmR \geq \rmR_0$ such that $\smash{(e^s+1) (\rmR+1)^{-1}} \in \N$.
\end{definition}

It is immediate to check that Lemma~\ref{lem:magic} for the measure $\tilde\mu_{\gamma,s,t}$ can be stated as follows.
\begin{corollary}
\label{cor:magic}
Let $\eps > 0$. Then, for $s \geq 0$ sufficiently large, for $t > s$ sufficiently large, and for any admissible function $\chi: [0, 1]^d \to \R$ in the sense of Definition~\ref{def:ChiAdmS}, such that 
\begin{equation}
\label{eq:condChiScaleds}
\begin{gathered}
\min_{x \in [0, 1]^d} \chi(x) \geq\f{\log s}{8 \sqrt{\smash[b]{2d}}} - \f{\sqrt{\smash[b]{2d}}s}{2} \;, \qquad \max_{x \in [0, 1]^d} \chi(x) \leq \log t - \f{\sqrt{\smash[b]{2d}}s }{ 2} - \f{\log s}{\sqrt{\smash[b]{2d}}}  \;,\\
\sup_{\substack{x, y \in [0, 1]^d, \\ |x-y| \leq e^{-s}s^{-1}}} \frac{|\chi(x) - \chi(y)|}{|x-y|^{1/3}} \leq e^{s/3} \;,
\end{gathered}
\end{equation}
one has that,
\begin{align}
\E \Biggl[\exp \Biggl(- \int_{[0, 1]^d} \rmF_{\gamma, t}(x) e^{-\gamma \chi(x)} \tilde \mu_{\gamma, s,t}(dx)\Biggr)\Biggr] & \geq  \exp \Bigl(-(\tilde a_{\star} + \eps) \rmT_{\gamma} \rho_{\chi, s}\bigl([0, 1]^d\bigr) \Bigr) - \eps \;, \label{e:magic_hat0}  \\
\E \Biggl[\exp \Biggl(- \int_{[0, 1]^d} \rmF_{\gamma, t}(x) e^{-\gamma \chi(x)} \tilde \mu_{\gamma, s,t}(dx)\Biggr)\Biggr] & \geq  \exp \Bigl(-(\tilde a_{\star} - \eps) \rmT_{\gamma} \rho_{\chi, s}\bigl([0, 1]^d\bigr) \Bigr) + \eps \;. \label{e:magic_hat1}
\end{align} 
where $\tilde a_{\star} >0$ is the constant introduced in Proposition~\ref{pr:laplaceJoint}.
\end{corollary}

\begin{remark}
The previous lemma entails almost directly that the measure $\smash{\tilde\mu_{\gamma, s,t}}$ converges to an integrated atomic random measure with parameter $\gamma$ and spatial intensity given by the Lebesgue measure $dx$, i.e, up to multiplicative constants and recalling the notation introduce in Definition~\ref{def:PPP}, to $\CP_{\gamma}[dx]$. To see this, it suffices to take the function $\chi$ to be of constant order. Then, taking first the limit when $t \to \infty$ and then the limit when $s \to \infty$, the ``log corrections'' appearing in \eqref{eq:measTildeChi} vanish, and one can read off the Laplace functional associated to measure described above. Since we don't need this fact in the sequel of the proof, we refrain from giving further details.
\end{remark}

In the next lemma we compute the conditional Laplace functionals of the measures $(\mu_{\gamma, t})_{t \geq 0}$.
%For $s \geq 0$, we recall once again that $\CF_{s} \eqdef \sigma(\rmX_{r} \; : \; r \in [0, s))$, and we recall that $\mu_{\gammac}$ denotes the critical GMC associated to $\rmX$ obtain through the derivative renormalisation. 
\begin{lemma} 
\label{lem:Laplace}
There exists a diverging sequence $(s_n)_{n \geq 0}$ of non-negative real numbers such that, almost surely,
\begin{equation}
\label{eq:conJointLap}
\lim_{n \to \infty} \lim_{t \to \infty}\E\Biggl[\exp\Biggl(- \int_{[0, 1]^d} \rmF_{\gamma, t}(x) \mu_{\gamma, t}(dx) \Biggr) \, \Bigg| \, \CF_{s_n}\Biggr]  =  \exp\Bigl(- \tilde a_{\star} \rmT_{\gamma} \mu_{\gammac}\bigl([0, 1]^d\bigr) \Bigr) \;,
\end{equation}
where $\tilde a_{\star} >0$ is the constant introduced in Proposition~\ref{pr:laplaceJoint}, and $\mu_{\gammac}$ is the critical GMC. 
\end{lemma}
\begin{proof}
In order to prove the result, we rely on Corollary~\ref{cor:magic}. In particular, we want to take the function $\chi:[0, 1]^d \to \R$ in that corollary in such a way that $\smash{\exp(-\gamma \chi(x)) \tilde \mu_{\gamma, s, t}(dx) = \mu_{\gamma, t}(dx)}$, which forces us to take
\begin{equation}
\label{e:defg}
\chi(\cdot) = -\rmX_s(\cdot) - \f{\log s}{\sqrt{\smash[b]{2d}}} + \f{\sqrt{\smash[b]{2d}}s}{2} \;.
\end{equation}
Now, a simple computation yields that, for $\chi$ as in \eqref{e:defg}, 
\begin{equation*}
\rho_{\chi, s}(dx) = \mu_{\gammac, s}(dx) \;,
\end{equation*}
where here we recall \eqref{e:def_crtical_GMC_Der}.
For $\eps > 0$, $s \geq 0$, $\rmR \geq 1$ such that $\smash{(e^s+1) (\rmR+1)^{-1}} \in \N$, and $\chi$ as in \eqref{e:defg}, we consider the following events
\begin{align*}
& \rmE^1_s \eqdef \Biggl\{\min_{x \in [0, 1]^d} \chi(x) \geq\f{\log s}{8 \sqrt{\smash[b]{2d}}} - \f{\sqrt{\smash[b]{2d}}s}{2}\Biggr\}\;, \qquad  &&\rmE^2_s \eqdef \Biggl\{\max_{x \in [0, 1]^d} \chi(x) \leq\f{19 \sqrt{\smash[b]{2d}}s}{2} - \f{\log s}{\sqrt{\smash[b]{2d}}}\Biggr\} \;, \\
& \rmE^3_s \eqdef \Biggl\{\sup_{\substack{x, y \in [0, 1]^d, \\ |x-y| \leq e^{-s}s^{-1}}} \frac{|\chi(x) - \chi(y)|}{|x-y|^{1/3}} \leq e^{s/3}\Biggr\}\;, \qquad && \rmE^4_{s, \rmR} \eqdef \bigl\{\rho_{\chi, s}(\bar{B_{\rmR, s}}) \leq \eps \bigr\} \;.
\end{align*}
We also define the event $\smash{\rmE_{s, \rmR} \eqdef \cap_{i=1}^3 \rmE^{i}_s \cap \rmE^4_{s, \rmR}}$ and we note that
\begin{align*}
\E\Biggl[\exp\Biggl(- \int_{[0, 1]^d} \rmF_{\gamma, t}(x) \mu_{\gamma, t}(dx) \Biggr) \Bigg| \CF_s\Biggr] & \geq \E\Biggl[\exp\Biggl(- \int_{[0, 1]^d} \rmF_{\gamma, t}(x) \mu_{\gamma, t}(dx) \Biggr) \one_{\rmE_{s, \rmR}} \Bigg| \CF_s\Biggr]  \;,  \\
\E\Biggl[\exp\Biggl(- \int_{[0, 1]^d} \rmF_{\gamma, t}(x) \mu_{\gamma, t}(dx) \Biggr) \Bigg| \CF_s\Biggr] & \leq \E\Biggl[\exp\Biggl(- \int_{[0, 1]^d} \rmF_{\gamma, t}(x) \mu_{\gamma, t}(dx) \Biggr) \one_{\rmE_{s, \rmR}} \Bigg| \CF_s\Biggr] + \one_{\rmE^c_{s, \rmR}} \;,
\end{align*} 
where the second inequality is simply due to the fact that the map $x \mapsto e^{-x}$ is bounded by one, for all $x \geq 0$. Moreover, we observe that, by definition, on the event $\rmE_{s, \rmR}$, the function $\chi$ defined in \eqref{e:defg} satisfies the conditions required in the statement of Corollary~\ref{cor:magic}. Therefore, for $\eps > 0$ as fixed above and for $s \geq 0$ sufficiently large, it holds almost surely that
\begin{align*}
\liminf_{t \to \infty} \E\Biggl[\exp\Biggl(- \int_{[0, 1]^d} \rmF_{\gamma, t}(x) \mu_{\gamma, t}(dx) \Biggr) \Bigg| \CF_s\Biggr] & \geq \Bigl(\exp \Bigl(-(\tilde a_{\star} + \eps) \rmT_{\gamma} \mu_{\gammac, s}\bigl([0, 1]^d\bigr) \Bigr) - \eps \Bigr)\one_{\rmE_{s, \rmR}} \;, \\
\limsup_{t \to \infty}  \E\Biggl[\exp\Biggl(- \int_{[0, 1]^d} \rmF_{\gamma, t}(x) \mu_{\gamma, t}(dx) \Biggr) \Bigg| \CF_s\Biggr] & \leq \Biggl(\exp \Bigl(-(\tilde a_{\star} - \eps) \rmT_{\gamma} \mu_{\gammac, s}\bigl([0, 1]^d\bigr) \Bigr) + \eps \Biggr) \one_{\rmE_{s, \rmR}} + \one_{\rmE^c_{s, \rmR}} \;.
\end{align*}
Thanks to \cite[Theorem~4]{Critical_der}, the sequence of random measure $\smash{(\mu_{\gammac, s})_{s \geq 0}}$ converges as $s \to \infty$ almost surely to $\smash{\mu_{\gammac}}$ in the topology of vague convergence. Furthermore, we can extract positive diverging sequences $(s_n)_{n \in \N}$ and $(\rmR_n)_{n \in \N}$ such that $\smash{\one_{\rmE_{s_n, \rmR_n}}}$ converges almost surely to $1$ as $n \to \infty$. This follows since, thanks to \cite[Lemmas~3.1,~3.2,~3.3]{Madaule_Max} (see also \cite[Lemma~A.2]{Glassy}), we have that $\P(\rmE_{s, \rmR})$ converges to $1$ as $s$, $\rmR \to \infty$. In particular, this shows that the following convergences hold almost surely
\begin{align*}
\lim_{n \to \infty} \liminf_{t \to \infty}  \E\Biggl[\exp\Biggl(- \int_{[0, 1]^d} \rmF_{\gamma, t}(x) \mu_{\gamma, t}(dx) \Biggr) \Bigg| \CF_{s_n}\Biggr] & \geq \exp \Bigl(-(\tilde a_{\star} + \eps) \rmT_{\gamma} \mu_{\gammac}\bigl([0, 1]^d\bigr) \Bigr) - \eps  \;, \\
\lim_{n \to \infty} \limsup_{t \to \infty} \E\Biggl[\exp\Biggl(-\int_{[0, 1]^d} \rmF_{\gamma, t}(x) \mu_{\gamma, t}(dx) \Biggr) \Bigg| \CF_{s_n}\Biggr] & \leq \exp \Bigl(-(\tilde a_{\star} - \eps) \rmT_{\gamma}\mu_{\gammac}\bigl([0, 1]^d\bigr) \Bigr) + \eps \;.
\end{align*} 
Therefore, the conclusion follows by arbitrariness of $\eps > 0$. 
\end{proof}

%%%%%%%%%%%%%%%%%%%%%%%%%%%%%%%%%%%%%%%%%%%%%%
\subsection{Stable convergence}
\label{sub:JointStable}
The goal of this section is to prove Proposition~\ref{pr:laplaceJoint}, from which Theorem~\ref{th:stableConv} follows immediately.
\begin{proof}[Proof of Proposition~\ref{pr:laplaceJoint}]
For $\smash{n \in \N}$ and $\smash{\gamma > \sqrt{\smash[b]{2d}}}$, we consider the collection of measures $\smash{(\mu_{\gamma, i})_{i \in [n]}}$ introduced in the statement of Theorem~\ref{th:stableConv}. We recall that we need to check that, for all $(\varphi, (f_i)_{i \in [n]}) \in \CC^{\infty}_c(\R^d) \times (\CC_c^{+}(\R^d))^n$, the following holds 
\begin{equation*}
\lim_{t \to \infty} \E\Biggl[\exp\bigl(i \langle \rmX, \varphi \rangle \bigr) \prod_{i =  1}^{n} \exp\bigl(-\mu_{\gamma, t, i}(f_i)\bigr)\Biggr] = \E\Bigl[\exp\bigl(i \langle \rmX, \varphi \rangle \bigr) \exp\bigl(- \tilde a_{\star} \mu_{\gammac}(\rmT_{\gamma})\bigr)\Bigr] \;,
\end{equation*}	
where here, to simplify the notation, we have omitted the dependence of $\rmT_{\gamma}$ on $(f_i)_{i \in [n]}$, and where we recall that $\tilde a_{\star} = \beta(d, \gamma) a_{\star}$ with $a_{\star}$ as defined in \eqref{eq:defAStar}. 

We consider the sequence $(s_n)_{n \in \N}$ introduced in the statement of Lemma~\ref{lem:Laplace}. For $n \in \N$, we consider $u$, $t \geq 0$ such that $u < s_n < t$, and we note that
\begin{align*}
& \Biggl\lvert \E\Biggl[\exp\bigl(i \langle \rmX, \varphi\rangle\bigr)\Biggl(\prod_{i=1}^n \exp\bigl(- \mu_{\gamma, t, i}(f_i)\bigr) - \exp\bigl(- \tilde a_{\star} \mu_{\gammac}(\rmT_{\gamma})\bigr)\Biggr)\Biggr] \Biggr\rvert \\
& \hspace{10mm} \leq \Biggl\lvert \E\Biggl[\exp\bigl(i \langle \rmX_u, \varphi\rangle\bigr) \Biggl(\prod_{i=1}^n \exp\bigl(- \mu_{\gamma, t, i}(f_i)\bigr) - \exp\bigl(- \tilde a_{\star} \mu_{\gammac}(\rmT_{\gamma})\bigr)\Biggr)\Biggr] \Biggr\rvert  \\
& \hspace{70mm}  + 2 \E\Bigl[\Bigl\lvert\exp\bigl(i \langle \rmX, \varphi\rangle\bigr) - \exp\bigl(i \langle \rmX_u, \varphi\rangle\bigr)\Bigr\rvert\Bigr] \;,
\end{align*}
where we simply used the triangle inequality and the fact that the function $x \mapsto e^{-x}$ is bounded by $1$ for $x \geq 0$. For the term appearing in the second line of the above display, we note that it is equal to
\begin{equation*}
\Biggl\lvert \E\Biggl[\exp\bigl(i \langle \rmX_u, \varphi\rangle\bigr) \E\Biggl[\prod_{i=1}^n \exp\bigl(- \mu_{\gamma, t, i}(f_i)\bigr) - \exp\bigl(- \tilde a_{\star} \mu_{\gammac}(\rmT_{\gamma})\bigr)\, \Bigg| \, \CF_{s_n}\Biggr]\Biggr] \Biggr\rvert  \;,
\end{equation*}
where we used the fact that $\rmX_{u}$ is $\smash{\CF_{s_n}}$-measurable since by assumption $u < s_n$. On the one hand, Lemma~\ref{lem:Laplace} implies that the following convergence holds almost surely
\begin{equation}
\label{e:RHSmain1}
\lim_{n \to \infty} \lim_{t \to \infty} \E\Biggl[\prod_{i=1}^n \exp\bigl(- \mu_{\gamma, t, i}(f_i)\bigr)\, \Bigg| \, \CF_{s_n}\Biggr] = \exp\bigl(- \tilde a_{\star} \mu_{\gammac}(\rmT_{\gamma})\bigr) \;.
\end{equation} 
On the other hand, the martingale convergence theorem implies that the following convergence holds almost surely 
\begin{equation}
\label{e:RHSmain2}
\lim_{n \to \infty} \E\Bigl[\exp\bigl(- \tilde a_{\star} \mu_{\gammac}(\rmT_{\gamma})\bigr)\, \Big| \, \CF_{s_n}\Bigr] 
= \E\Biggl[\exp\bigl(- \tilde a_{\star} \mu_{\gammac}(\rmT_{\gamma})\bigr) \, \Big| \, \sigma(\rmX) \Bigr] 
= \exp\bigl(- \tilde a_{\star} \mu_{\gammac}(\rmT_{\gamma})\bigr) \;.
\end{equation}
Therefore, since the right-hand side of \eqref{e:RHSmain1} coincides with the right-hand side of \eqref{e:RHSmain2}, by the dominated convergence theorem, we obtain that uniformly over $u \geq 0$, it holds that
\begin{equation*}
\lim_{n \to \infty} \lim_{t \to \infty} \Biggl\lvert \E\Biggl[\exp\bigl(i \langle \rmX_u, \varphi\rangle\bigr) \E\Biggl[\prod_{i=1}^n \exp\bigl(- \mu_{\gamma, t, i}(f_i)\bigr) - \exp\bigl(- \tilde a_{\star} \mu_{\gammac}(\rmT_{\gamma})\bigr)\, \Bigg| \, \CF_{s_n}\Biggr]\Biggr] \Biggr\rvert  = 0 \;.
\end{equation*}
Thanks to the dominated convergence theorem and the fact that $\rmX_{u}$ converges to $\rmX$ as $u \to \infty$ almost surely in $\CH^{-\kappa}(\R^d)$, one has
\begin{equation*}
	\lim_{u \to \infty} \E\Bigl[\Bigl\lvert\exp\bigl(i \langle \rmX, \varphi\rangle\bigr) - \exp\bigl(i \langle \rmX_u, \varphi\rangle\bigr)\Bigr\rvert\Bigr] = 0 \;,
\end{equation*}
and the conclusion follows. 
\end{proof} 

We now show how the proof of Theorem~\ref{th:stableConv} follows directly.
\begin{proof}[Proof of Theorem~\ref{th:stableConv}]
Thanks to Lemma~\ref{lm:stableRV}, it suffices to show that the following joint convergence in distribution holds in $\CH_{\loc}^{-\kappa}(\R^d) \times (\CM^{+}(\R^d))^n$ for some $\kappa > 0$,
\begin{equation*}
\bigl(\rmX, (\mu_{\gamma, t, i})_{i \in [n]}\bigr) \Rightarrow \bigl(\rmX, (\mu_{\gamma, i})_{i \in [n]}\bigr)	\;.
\end{equation*}
By Lemma~\ref{lm:jointConv}, it suffices to verify that for all $(\varphi, (f_i)_{i \in [n]}) \in \CC^{\infty}_c(\R^d) \times (\CC_c^{+}(\R^d))^n$, it holds that
\begin{equation*}
\lim_{t \to \infty} \E\Biggl[\exp\bigl(i \langle \rmX, \varphi\rangle\bigr) \prod_{i=1}^n \exp\bigl(-\mu_{\gamma, t, i}(f_i)\bigr)\Biggr] = \E\Biggl[\exp\bigl(i \langle \rmX, \varphi\rangle\bigr) \prod_{i=1}^n \exp\bigl(-\mu_{\gamma, i}(f_i)\bigr)\Biggr] \;.
\end{equation*}
Furthermore, by Proposition~\ref{pr:laplaceJoint}, it suffices to establish that 
\begin{equ}[e:wantedThmC]
\E\Biggl[\exp\bigl(i \langle \rmX, \varphi\rangle\bigr) \prod_{i=1}^n \exp\bigl(-\mu_{\gamma, i}(f_i)\bigr)\Biggr] = \E\Bigl[\exp\bigl(i \langle \rmX, \varphi \rangle \bigr) \exp\bigl(- \tilde a_{\star} \mu_{\gammac}(\rmT_{\gamma})\bigr)\Bigr] \;,
\end{equ}
where, as before, to simplify the notation, we have omitted the dependence of $\rmT_{\gamma}$ on $(f_i)_{i \in [n]}$. 
Recalling the definition of the collection of measures $(\mu_{\gamma, i})_{i \in [n]}$ given in the statement of Theorem~\ref{th:stableConv}, one can use formula \eqref{eq:LaplaceCompInt} to check that
\begin{equation*}
\E\Biggl[\prod_{i=1}^n \exp\bigl(-\mu_{\gamma, i}(f_i)\bigr) \, \Bigg| \, \sigma(\rmX) \Biggr] = \exp\bigl(- \tilde a_{\star} \mu_{\gammac}(\rmT_{\gamma})\bigr) \;,
\end{equation*}
thus yielding \eqref{e:wantedThmC}, which completes the proof.
\end{proof}

%%%%%%%%%%%%%%%%%%%%%%%%%%%%%%%%%%%%%%%%%%%%%%
%%%%%%%%%%%%%%%%%%%%%%%%%%%%%%%%%%%%%%%%%%%%%%
\section{Proof of Proposition~\ref{pr:joint}}
\label{sec:Joint}
The main goal of this section is to prove Proposition~\ref{pr:joint}, and it is structured as follows. In Section~\ref{subsec:setupTechJoint}, we introduce some notation and state Lemmas~\ref{lm:techMainRed1} and~\ref{lm:techMainRed2}, which are the two main technical lemmas used in the proof of Proposition~\ref{pr:joint}. We then show how the proof of Proposition~\ref{pr:joint} follows from the two aforementioned lemmas and Propositions~\ref{pr:AsyConv}~and~\ref{pr:AsyProb}. The remaining part of the section is then devoted to the proof of Lemmas~\ref{lm:techMainRed1} and~\ref{lm:techMainRed2}. In particular, in Section~\ref{subsec:redSteps}, we prove Lemma~\ref{lm:techMainRed1}, while in Section~\ref{subsec:redStepConstant}, we prove Lemma~\ref{lm:techMainRed2}.

%%%%%%%%%%%%%%%%%%%%%%%%%%%%%%%%%%%%%%%%%%%%%%
\subsection{Setup and main technical lemmas}
\label{subsec:setupTechJoint}
Throughout this section, we fix $\gamma > \sqrt{\smash[b]{2d}}$, $n \in \N$, and a collection of non-negative constants $(\theta_i)_{i \in [n]}$. For any $x \in \R$, $k>0$, and $\lambda > 0$, we define the (random) function $\bfF^{\lambda}_{k, x}:\CC(\R^d) \to \R$ by
\begin{equation}
\label{eq:defscrF}
\bfF^{\lambda}_{k, x}(\Phi) \eqdef \f{1- \exp\bigl(-e^{-\gamma x}\int_{\B_k} \rmF_{\gamma}(y) e^{\gamma \Phi(y)} dy\bigr)}{\abs{\D^{\lambda}_{k, 0}(\Phi)}} \;,
\end{equation}
where $\rmF_{\gamma} : \R^d \to \R$ is the function given by
\begin{equation}
\label{eq:RPsiGammaNoT}
\rmF_{\gamma}(y) \eqdef \sum_{i = 1}^{n} \theta_i e^{\gamma \rmW_{i}(y)} \;.
\end{equation}

We introduce here the main processes and fields that will be used for the reminder of this section:
\begin{itemize}
\item Let $B'$ be a standard Brownian motion and $R$ a three-dimensional Bessel process starting at zero. For any $z \geq 0$, define $\rmU_z$ to be a random variable uniformly distributed on the interval $[-z, 0]$, independent of all other processes. Also, let $\tau_{z} \eqdef \inf\{s \geq 0 \, : \, B'_s = \rmU_{z}\}$. We then define the process
\begin{equation}
\label{eq:processGammaz}
\Gamma_{\! s}^z \eqdef \begin{cases}
		B'_s\;, \quad &\text{ if } s \leq \tau_{z} \;,\\
		R_{s-\tau_z} + \rmU_z \;, \quad &\text{ if } s > \tau_{z} \;.
\end{cases}
\end{equation}
\item For $b> 0$ and $z \geq 0$, we let $\smash{\frkg^z_{b}}$ be the field on $\R^d$ given by
\begin{equation}
\label{eq:fieldGzb}
\frkg^z_{b}(\cdot) \eqdef - \int_0^{\infty} \bigl(1- \frkK(e^{-(s+b)} \cdot)\bigr) d\Gamma_{\! s}^{z} + \rmZ'_{\infty}(e^{-b}\cdot) - \sqrt{\smash[b]{2d}}\int_0^{\infty} \bigl(1 - \frkK(e^{-(s+b)} \cdot)\bigr) ds  \;,
\end{equation}
where $\rmZ'_{\infty}$ has the same law as the field $\rmZ_{\infty}$ defined in Definition~\ref{def:fieldsZ}, and it is independent of the process $\Gamma^{z}$.
\item For $b > 0$, we recall that $\Upsilon_{\! b}$ denotes the field on $\R^d$ introduced in \eqref{eq:defUpsilonBBeg} and given by 
\begin{equation*}
\Upsilon_{\! b}(\cdot) \eqdef - \int_0^b \bigl(1 - \frkK(e^{-s} \cdot)\bigr) dB_{s} + \rmZ_b(\cdot) - \sqrt{\smash[b]{2d}} \frka_b(\cdot) \;,
\end{equation*}
where $\rmZ_b$ is introduced in Definition~\ref{def:fieldsZ}, $B$ is an independent Brownian motion, and $\frka_b$ is the function defined in \eqref{eq:frkgb}.
\end{itemize}

\begin{remark}
We emphasise that the processes and fields introduced above are all assumed to be mutually independent. Additionally, given $x$, $y \in \R$ and $b > 0$, we write $\P_{x, y, b}$ for the probability measure under which $(B_s)_{s \in [0, b]}$ is a Brownian bridge from $x$ to $y$ in time $b$, while the other processes/fields are left unchanged. Moroever, given a function $g:\R^d \to \R$, we set $\Upsilon_{\! b, g} = \Upsilon_{\! b} + g$. 

In what follows, for $b > 0$ and $z \geq 0$, we need to consider the field given by the sum of $\Upsilon_{\! b}$ and $\smash{\frkg^z_{b}}$. In order to lighten the notation, instead of writing the field $\smash{\frkg^z_{b}}$ as a subscript of $\Upsilon_{\! b}$, we let
\begin{equation}
\label{eq:defUpsilonBZ}
\Upsilon_{\! b}^z \eqdef \Upsilon_{\! b} + \frkg^z_{b} \;.
\end{equation}
We also observe that a standard Gaussian tail bound implies that, uniformly over all $z \geq 0$, the field $\smash{\frkg^z_{b}}$ satisfies \ref{as:GG1}\dash \ref{as:GG3}.
\end{remark}

For $\lambda > 0$ and $A$, $L$, $b \geq 0$, we introduce the function $\frkF^{\lambda}_{A, L, b} : \R \times \CC(\R^d) \to \R$ given by
\begin{equation}
\label{eq:asympFALb0}
\begin{alignedat}{1}
\frkF^{\lambda}_{A, L, b}(z, g) 
& \eqdef  \frac{1}{\sqrt{2 \pi b}} \int_{0}^{A+L} e^{\sqrt{\smash[b]{2d}} (x - L) -\f{(x - z)^2}{2b}} \E_{x, z, b} \Bigl[\bfF^{\lambda}_{b, x - L}(\Upsilon_{\!b, g}) \\
& \hspace{20mm}\cdot \one_{\{\inf_{s \in [0, b]} B_{s} \geq 0\}} \one_{\{\M_{0, b + 1, b}(\Upsilon_{\! b, g}) \leq x-(A+L)\}}  \one_{\{\M_{0, b}(\Upsilon_{\! b, g}) \leq \lambda\}}  \Bigr] dx \;.
\end{alignedat}
\end{equation}
Moreover, we define the constant $\bfC^{\lambda}_{A, L, b} > 0$ by letting
\begin{equation}
\label{eq:defCALb}
\bfC^{\lambda}_{A, L, b} \eqdef \alpha \int_{0}^{\infty} z \E\bigl[\frkF^{\lambda}_{A, L, b}\bigl(z, \frkg^z_{b}\bigr)\bigr] dz \;,
\end{equation}
where we recall that $\alpha = \sqrt{\smash[b]{2/\pi}}$.

We are now ready to state the following key lemma whose proof is given in Section~\ref{subsec:redSteps}.
\begin{lemma}
\label{lm:techMainRed1}
For any $\lambda > 0$, $\rmR \geq 1$, and $\eps > 0$, there exist $0 \leq A < L$ sufficiently large such that there exists $b_0 > 0$ and $s_0 > 0$ sufficiently large, such that for any $s \geq s_0$ satisfying $(e^s + 1)/(\rmR+1) \in \N$ and any $b \geq b_0$, there exists a sufficiently large $T \geq 0$ such that for all $t \geq T$ and any function $\chi:[0,\rmR]^d \to \R$ satisfying the conditions in \eqref{eq:condChi}, it holds that
\begin{equation*}
\Biggl\lvert \E \Biggl[1-\exp\biggl(-\int_{[0, \rmR]^d} \rmF_{\gamma, t}(x) e^{-\gamma \chi(x)} \mu_{\gamma, t}(dx)\biggr)\Biggr]  - \bfC^{\lambda}_{A, L, b} \, \rho_{\chi}\bigl([0, \rmR]^d\bigr) \Biggr\rvert \leq \eps \, \rho_{\chi}\bigl([0, \rmR]^d\bigr) \;,
\end{equation*}
where we recall that the measure $\rho_{\chi}$ is defined in \eqref{eq:measChi}.
\end{lemma}

\begin{remark}
We emphasise that in the statement of Lemma~\ref{lm:techMainRed1}, the conditions \eqref{eq:condChi} on the function $\chi:[0, \rmR]^d \to \R$ depends on $s$. 
\end{remark}

Thanks to Lemma~\ref{lm:techMainRed1}, to prove Proposition~\ref{pr:joint} it suffices to derive an explicit expression for the constant $\bfC^{\lambda}_{A, L, b}$ defined in \eqref{eq:defCALb} as the cutoff parameters are taken to infinity, and then show that this expression coincides with $\rmT_{\gamma}$ defined in \eqref{eq:RPsiGamma}, up to a multiplicative constant. To this end, for $A$, $L \geq 0$ and $0 \leq k < b$, recalling that $\alpha = \sqrt{\smash[b]{2/\pi}}$, we define the constant $\smash{\bfC^{\lambda, \new}_{A, L, k, b}}$ by letting
\begin{equation}
\label{eq:defCLambdaNew}
\begin{alignedat}{1}
& \bfC^{\lambda, \new}_{A, L, k, b} = \alpha  \int_{L/2}^{A+L} e^{\sqrt{\smash[b]{2d}} (x - L)} \int_{b^{1/4}}^{b^{3/4}} \frac{e^{-\f{u^2}{2b}}}{\sqrt{2 \pi b}} (u+x) \\
& \hspace{45mm} \cdot \E_{0, u, b} \Bigl[\bfF^{\lambda}_{k, x-L, L}\bigl(\Upsilon^{u+x}_{\! b}\bigr) \one_{\{\M_{0,b}(\Upsilon_{\! b}^{ u+x}) \leq \lambda \}}\Bigr] du dx \;,
\end{alignedat}
\end{equation} 
where for any $x \in \R$, $k>0$, $L > 0$ and $\lambda > 0$, we define the (random) function $\bfF^{\lambda}_{k, x, L}:\CC(\R^d) \to \R$ by
\begin{equation}
\label{eq:defTruncationLargeL}
\bfF^{\lambda}_{k, x, L}(\Phi) \eqdef \f{1- \exp\bigl(-e^{-\gamma x}\int_{\B_k} \rmF_{\gamma}(y) e^{\gamma \Phi(y)} dy\bigr)}{\abs{\D^{\lambda}_{k, 0}(\Phi)} \vee L^{-1}} \;,
\end{equation}

where we emphasise that the only difference from $\bfF^{\lambda}_{k, x}$, as defined in \eqref{eq:defscrF}, is the presence of the maximum in the denominator. Note that, with a slight abuse of our previous notation, we have absorbed the expectation with respect to the field $\frkg^{u+x}_{b}$ into the expectation $\E_{0, u, b}$.

We can now state the second key lemma of this section whose proof is given in Section~\ref{subsec:redStepConstant}.
\begin{lemma}
\label{lm:techMainRed2}
For any $\lambda > 0$, $\eps > 0$, and $A \geq 0$, there exists $L \geq 0$ sufficiently large, such that there exists $k_0 \geq 0$ for which, for all $k \geq k_0$, there exists $b_0 \geq 0$ sufficiently large such that for all $b \geq b_0$, it holds that 
\begin{equation*}
\abs{\bfC^{\lambda}_{A, L, b} - \bfC^{\lambda, \new}_{A, L, k, b}} \leq \eps \;. 
\end{equation*}
\end{lemma}

The proof of Proposition~\ref{pr:joint} follows by combining Lemmas~\ref{lm:techMainRed1} and~\ref{lm:techMainRed2}.
\begin{proof}[Proof of Proposition~\ref{pr:joint}]
Let $\rmR \geq 1$ and $\lambda \in \Lambda$, where we recall that the set $\Lambda$ is introduced in Lemma~\ref{lm:Countable}.
For any $\eps > 0$, thanks to Lemmas~\ref{lm:techMainRed1}~and~\ref{lm:techMainRed2}, by taking the involved parameters large enough as specified in the statements of the aforementioned lemmas, we have that 
 \begin{equation*}
\Biggl\lvert \E \Biggl[1-\exp\biggl(-\int_{[0, \rmR]^d} \rmF_{\gamma, t}(x) \mu^{\chi}_{\gamma, t}(dx)\biggr)\Biggr]  - \bfC^{\lambda, \new}_{A, L, k, b} \rho_{\chi}\bigl([0, \rmR]^d\bigr) \Biggr\rvert \leq \eps \rho_{\chi}\bigl([0, \rmR]^d\bigr) \;.
\end{equation*}
Hence it suffices to show that when the cutoff parameters are taken to infinity, the limit of the constant $\smash{\bfC^{\lambda, \new}_{A, L, k, b}}$ coincides with $\rmT_{\gamma}$ defined in \eqref{eq:RPsiGamma},  up to a multiplicative constant.
To establish that this is indeed the case, we rely on Propositions~\ref{pr:AsyConv} and~\ref{pr:AsyProb}.

We observe that for any $A$, $L$, $k \geq 0$, $x \in [L/2, A+L]$, and $k \in \N$, conditional on $(\rmW_{i})_{i \in [n]}$, the function $\smash{\bfF^{\lambda}_{k, x-L, L}: \CC(\R^d) \to \R}$ is continuous, bounded, depends on the values of the input field in a bounded set, and it is such that its set of discontinuities is assigned measure zero by the law of the field $\smash{\tilde \Upsilon_{\! \lambda}}$ on $\CC(\R^d)$. This last fact follows since we fixed $\lambda \in \Lambda$ and from Remark~\ref{rm:redTechSmallk}.
Furthermore, as can be easily verified using the Gaussian tail bound, the field $\smash{\frkg^z_{b}}$, introduced in \eqref{eq:fieldGzb}, satisfies \ref{as:GG1}\dash \ref{as:GG3} uniformly over all $z \geq 0$.
Therefore, we are in a position to apply Propositions~\ref{pr:AsyConv} and~\ref{pr:AsyProb}, from which we deduce that for any $\eps > 0$, $k \geq 0$, $b \geq 0$ sufficiently large, and $u \in [b^{1/4}, b^{3/4}]$, it holds that
\begin{equation}
\label{eq:AbsValStartProp}
\Biggl\lvert \E_{0, u, b} \Bigl[\bfF^{\lambda}_{k, x-L, L}\bigl(\Upsilon_{\! b}^{ u+x}\bigr) \one_{\{\M_{0,b}(\Upsilon_{\! b}^{ u+x}) \leq \lambda\}} \Bigr] - 2 c_{\star, \lambda} \frac{u}{b} \E\bigl[\bfF^{\lambda}_{k, x-L, L}\bigl(\tilde\Upsilon_{\! \lambda}\bigr)\bigr] \Biggr\rvert \leq \eps \frac{u}{b} \;.
\end{equation}
Furthermore, thanks to Proposition~\ref{pr:inversionPsi}, it holds that 
\begin{equation}
\E\bigl[\bfF^{\lambda}_{k, x-L, L}(\tilde{\Upsilon}_{\!\lambda})\bigr] = \frac{\E\bigl[\int_{\R^d} \bfF^{\lambda}_{k, x-L, L}(\tau_z \Psi)e^{\sqrt{\smash[b]{2d}}\Psi(z)}\one_{\{\Psi(z)\geq -\lambda\}} dz\bigr]}{\E\bigl[\int_{\R^d} e^{\sqrt{\smash[b]{2d}}\Psi(z)}\one_{\{\Psi(z)\geq -\lambda\}}dz\bigr]}\;,
\end{equation}
where, once again, we recall that we omit the subscript $\lambda$ when writing the field $\Psi$, since, thanks to Proposition~\ref{pr:PsiIndepLambda}, its law does not depend on $\lambda \in \Lambda$.
Therefore, recalling the definition \eqref{eq:defAStar} of $a_{\star}$ and performing a change of variables in the integral over $x$ in the definition \eqref{eq:defCLambdaNew} of $\smash{\bfC^{\lambda, \new}_{A, L, k, b}}$, we obtain that 
\begin{equation}
\label{eq:asympFALbFinalFinal}
\lim_{b \to \infty} \bfC^{\lambda, \new}_{A, L, k, b} = a_{\star}  \gamma \int_{-L/2}^{A} e^{\sqrt{\smash[b]{2d}} x} \E\Biggl[\int_{\R^d} \bfF^{\lambda}_{k, x, L}(\tau_z \Psi)e^{\sqrt{\smash[b]{2d}}\Psi(z)}\one_{\{\Psi(z)\geq -\lambda\}} dz\Biggr]  dx \;.
\end{equation}
Recalling the definition of the function $\bfF^{\lambda}_{k, x}$ in \eqref{eq:defscrF} and that of $\bfF^{\lambda}_{k, x, L}$ in \eqref{eq:defTruncationLargeL}, and using the dominated convergence theorem along with the monotone convergence theorem, we observe that taking first the limit as $k \to \infty$, then as $L \to \infty$ and finally as $A \to \infty$ of the expression on the right-hand side of \eqref{eq:asympFALbFinalFinal}, we 
conclude that $\lim_{A, L, k, b \to \infty} \bfC^{\lambda, \new}_{A, L, k, b}$ equals 
\begin{equation*}
a_{\star} \gamma \E\Biggl[\int_{\R^d} \int_{-\infty}^{\infty} e^{\sqrt{\smash[b]{2d}} x} \Biggl(1- \exp\Biggl(-e^{-\gamma x} \int_{\R^d} \rmF_{\gamma}(y) e^{\gamma  \tau_z \Psi(y)} dy\Biggr)\Biggr) \f{e^{\sqrt{\smash[b]{2d}}\Psi(z)} \one_{\{\Psi(z)\geq - \lambda\}}}{\abs{\D^{\lambda}(\tau_z \Psi)}} dx dz \Biggr] \;.
\end{equation*}
Now, recalling the definition \eqref{eq:defBetaDG} of the constant $\beta(d, \gamma)$, we note that, for all $c > 0$ and any $\gamma > \sqrt{\smash[b]{2d}}$, it holds that 
\begin{equation*}
\int_{-\infty}^\infty e^{\sqrt{\smash[b]{2d}} x} (1-e^{- c e^{-\gamma x}}) dx = \frac{c^{\frac{\sqrt{\smash[b]{2d}}}{\gamma}}\beta(d, \gamma)}{\gamma} \;.
\end{equation*}
Therefore, by collecting the previous considerations, using the above identity, and leveraging the independence between the fields $(\rmW_i)_{i \in [n]}$ and $\Psi$, as well as the stationarity of the fields $(\rmW_i)_{i \in [n]}$ implied by \ref{hp:W2}, we obtain that
\begin{align*}
\lim_{A, L, k, b \to \infty} \bfC^{\lambda, \new}_{A, L, k, b} & = \tilde a_{\star} \,\E\Biggl[\int_{\R^d} \Biggl(\int_{\R^d} \rmF_{\gamma}(y) e^{\gamma \tau_z \Psi(y)} dy\Biggr)^{\frac{\sqrt{\smash[b]{2d}}}{\gamma}} \f{e^{\sqrt{\smash[b]{2d}}\Psi(z)} \one_{\{\Psi(z)\geq -\lambda\}}}{\abs{\D^{\lambda}(\tau_z \Psi)}}  dz \Biggr] \\
& = \tilde a_{\star}  \, \E\Biggl[\Biggl(\int_{\R^d} \rmF_{\gamma}(y) e^{\gamma \Psi(y)} dy\Biggr)^{\frac{\sqrt{\smash[b]{2d}}}{\gamma}}\Biggr] \\
& = \tilde a_{\star} \, \rmT_{\gamma} \;,
\end{align*} 
where we recall that $\rmT_{\gamma}$ is defined in \eqref{eq:RPsiGamma}, and $\tilde a_{\star} = \beta(d, \gamma) a_{\star}$. Hence, the desired result follows. 
\end{proof}

%%%%%%%%%%%%%%%%%%%%%%%%%%%%%%%%%%%%%%%%%%%%%
\subsection{Proof of Lemma~\ref{lm:techMainRed1}}
\label{subsec:redSteps}
In this subsection, we prove Lemma~\ref{lm:techMainRed1} by reducing the joint Laplace transform of $\mu_{\gamma, t}$ to a more manageable quantity through a series of reduction steps, following the strategy developed in \cite{Glassy} with some non-trivial modifications. To streamline this process, we introduce some shorthand notations. We fix for the reminder of this section $\rmR \geq 1$. For a function $\chi: [0, \rmR]^d \to \R$, we introduce the sequence of measures $(\mu_{\gamma, t}^{\chi})_{t \geq 0}$ on $\R^d$ defined as follows,
\begin{equation*}
	\mu_{\gamma, t}^{\chi}(dx) \eqdef e^{-\gamma \chi(x)} \mu_{\gamma, t}(dx) \;,
\end{equation*}
where we recall that $\mu_{\gamma, t}$ is the regularised and normalised supercritical GMC measure as defined in \eqref{e:norm_super}. For $t > 0$ and a function $\chi:\R^d \to \R$, we introduce the fields $\rmY_t$ and $\rmY^{\chi}_t$ on $\R^d$, as well as the constant $\frkd_t$ by setting 
\begin{equation}
\label{eq:notMainTech1frkd}
\rmY_t(x) \eqdef \rmX_t(x) - \sqrt{\smash[b]{2d}}t\;, \qquad \rmY^{\chi}_t(x) \eqdef \rmY_t(x) - \chi(x)\;, \qquad \frkd_t \eqdef - \frac{3}{2 \sqrt{\smash[b]{2d}}} \log t \;,
\end{equation}
so that the measure $\mu_{\gamma, t}^{\chi}$ introduced above can be written as
\begin{equation*}
	\mu_{\gamma, t}^{\chi}(dx) = e^{\gamma(\rmY^{\chi}_t(x) - \frkd_t) + dt} dx \;.
\end{equation*}
For $\lambda > 0$, recalling the definition \eqref{eq:RPsiGamma} of the random function $\rmF_{\gamma, t}$ and the notation introduced in \eqref{eq:maximal}, only for this section, we let
\begin{equation}
\label{eq:DefGMDR}
\G_{\rmR} \eqdef \exp\Biggl(-\int_{[0, \rmR]^d} \rmF_{\gamma, t}(x) \mu_{\gamma, t}^{\chi}(dx) \Biggr) \;, \qquad \M_{\rmR} \eqdef \M_{\rmR}(\rmY^{\chi}_t)\;, \qquad \D^{\lambda}_{\rmR} = \D^{\lambda}_{\rmR}(\rmY^{\chi}_t) \;,
\end{equation}
and for a subset $\rmD \subseteq [0, \rmR]^d$, we use $\G_{\rmD}$ to denote the same quantity as above, but with $[0, \rmR]^d$ replaced by $\rmD$. With this notation in hand, we note that we can write
\begin{equation}
\label{eq:III0}
\E \bigl[1-\G_{\rmR}\bigr] = \E\Biggl[\int_{[0, \rmR]^d} \frac{\one_{\{m \in \D^{\lambda}_{\rmR}\}}}{\abs{\D^{\lambda}_{\rmR}}} \bigl(1-\G_{\rmR}\bigr) dm\Biggr] \;,
\end{equation}
where here we used the fact that $\abs{\D^{\lambda}_{\rmR}}$ is almost surely positive.

\subsubsection{High value constraint}
We start with the following lemma, which essentially states that only the points where the field attains sufficiently high values contribute to the integral on the right-hand side of \eqref{eq:III0}.
\begin{lemma}
\label{lm:IdentA}
For any $\lambda > 0$ and $\eps > 0$, there exists a constant $A \geq 0$ sufficiently large such that for any $s \geq 0$ sufficiently large satisfying $(e^s + 1)(\rmR+1)^{-1} \in \N$, there exists $T \geq 0$ sufficiently large such that for any $t \geq T$ and $\chi:[0, \rmR]^d \to \R$ satisfying the conditions in \eqref{eq:condChi}, it holds that  
\begin{equation}
\label{eq:III0A}
\E\Biggl[\int_{[0, \rmR]^d} \frac{\one_{\{m \in \D^{\lambda}_{\rmR}\}} \one_{\{\rmY^{\chi}_t(m) - \frkd_t \leq -A\}}}{\abs{\D^{\lambda}_{\rmR}}} \bigl(1-\G_{\rmR}\bigr) dm\Biggr] \leq \eps \rho_{\chi}\bigl([0, \rmR]^d\bigr)  \;. 
\end{equation}
\end{lemma}
\begin{proof}
We start by observing that on the event $\smash{\{m \in \D^{\lambda}_{\rmR}, \; \rmY^{\chi}_t(m) - \frkd_t \leq -A\}}$, it holds that 
\begin{equation*}
\rmY^{\chi}_t(x) - \frkd_t \leq -A + \lambda \;, \qquad \forall \, x \in [0, \rmR]^d \;.
\end{equation*} 
Therefore, using this fact, we obtain that
\begin{align*}
& \E\Biggl[\int_{[0, \rmR]^d} \frac{\one_{\{m \in \D^{\lambda}_{\rmR}\}} \one_{\{\rmY^{\chi}_t(m) - \frkd_t \leq -A\}}}{\abs{\D^{\lambda}_{\rmR}}} \bigl(1-\G_{\rmR}\bigr) dm\Biggr] \\
& \qquad \leq \E\Biggl[\int_{[0, \rmR]^d} \frac{\one_{\{m \in \D^{\lambda}_{\rmR}\}}}{\abs{\D^{\lambda}_{\rmR}}} \Biggl(1-\exp\Biggl(-\int_{[0, \rmR]^d} \rmF_{\gamma, t}(x) \one_{\{\rmY^{\chi}_t(x) - \frkd_t \leq -A + \lambda\}} \mu_{\gamma, t}^{\chi}(dx) \Biggr)\Biggr) dm\Biggr] \\
& \qquad = \E\Biggl[1-\exp\Biggl(-\int_{[0, \rmR]^d} \rmF_{\gamma, t}(x) \one_{\{\rmY^{\chi}_t(x) - \frkd_t \leq -A + \lambda\}} \mu_{\gamma, t}^{\chi}(dx)\Biggr)\Biggr] \;.
\end{align*}
If $\rmF_{\gamma, t}(\cdot)$ were a deterministic, $t$-independent function, then the conclusion would follow by a direct application of \cite[Proposition~4.1]{Glassy}. However, by following the proof of \cite[Proposition~4.1]{Glassy}, the same conclusion holds also in our more general setting by using \ref{hp:W1} and \ref{hp:W4}.
\end{proof}

For $A \geq 0$, we introduce the quantity
\begin{equation}
\label{eq:III0AA}
\E_{\eqref{eq:III0AA}} \eqdef \E\Biggl[\int_{[0, \rmR]^d} \frac{\one_{\{m \in \D^{\lambda}_{\rmR}\}} \one_{\{\rmY^{\chi}_t(m) - \frkd_t \geq -A\}}}{\abs{\D^{\lambda}_{\rmR}}} \bigl(1-\G_{\rmR}\bigr) dm\Biggr] \;.
\end{equation}
The upshot of Lemma~\ref{lm:IdentA} is that for any $\eps > 0$, we can find $A \geq 0$ sufficiently large such that for any $s \geq 0$ sufficiently large satisfying $(e^s + 1)(\rmR+1)^{-1} \in \N$, there exists $T \geq 0$ sufficiently large such that for any $t \geq T$ and $\chi:[0, \rmR]^d \to \R$ satisfying the conditions in \eqref{eq:condChi},
\begin{equation*}
\left\lvert\E \Biggl[1 - \exp\biggl(-\int_{[0, \rmR]^d} \rmF_{\gamma, t}(x) \mu^{\chi}_{\gamma, t}(dx)\biggr)\Biggr]  - \E_{\eqref{eq:III0AA}}\right\rvert \leq \eps \rho_{\chi}\bigl([0, \rmR]^d\bigr) \;.
\end{equation*}
Therefore, in what follows we we can just focus on the expectation $\E_{\eqref{eq:III0AA}}$ for a fixed $A \geq 0$. 

\subsubsection{Path constraint}
We now want to exclude the points $m \in [0, \rmR]^d$ such that $\rmY^{\chi}_t(m) - \frkd_t \geq -A$ with an unlikely path $[0, t] \ni s \mapsto \rmY_{s}(m)$. To this end, for $A$, $L$, $z \geq 0$, we consider the set of functions
\begin{equation*}
\S_{t}^{z, A, L} \eqdef \Bigl\{\phi : \R^+_0 \to \R \, : \, \sup_{s \in [0, t]} \phi(s) \leq z, \; \sup_{s \in [t/2, t]} \phi(s) \leq z + \frkd_t + L, \; \phi(t) \geq z + \frkd_t - A \Bigr\}\;.
\end{equation*}
\begin{figure}[ht]
\centering
\pgfmathsetseed{50}
\def\dh{0.025}
\def\dv{0.017}
\def\dsh{0.17}
\begin{tikzpicture}
\fill[color = gray!20] (0,0.6) -- (0, 1.2) -- (10, 1.2) -- (10,-1.2) -- (5,-1.2) -- (5, 0.6);
\draw[->,>=latex] (0,0) -- (10.8,0);
\draw[->,>=latex] (0,1.2) -- (0,-4.2);
\draw (0,0) node[left]{$0$};
\draw (5,0) node[below left]{$\f{t}{2}$};
\draw (10,0.07) -- (10,-0.07); \draw (10,0) node[below right]{$t$};
\draw (0,0)						
	\foreach \x in {1,...,200}		
	{   -- ++(\dh,\dsh*rand - \dv)}
	\foreach \x in {1,...,200}		
	{   -- ++(\dh,\dsh*rand + 0.6*\dv)}; 

\draw (0,0.6) -- (5,0.6) -- (5,-1.2) -- (10,-1.2);
\draw (0,0.6) node[left]{$z$};
\draw[dashed] (5,-1.2) -- (0,-1.2) node[left]{$z + \frkd_t + L$};
\draw[dashed] (10,-3) -- (0,-3) node[left]{$z + \frkd_t - A$};
\draw[very thick] (10,-4) -- (10,-3);
\end{tikzpicture}
\caption[short form]{\small{A typical path in $\S_{t}^{z, A, L}$.}}
\vspace{4mm}
\label{figure A_n}
\end{figure}
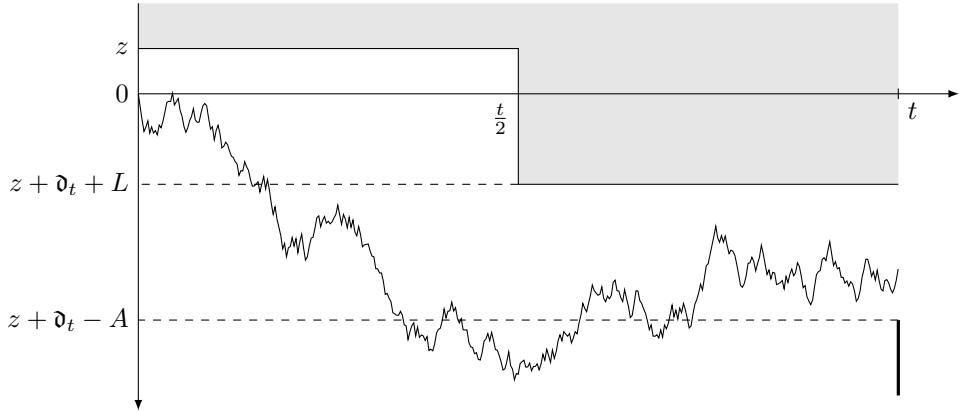

\begin{lemma}[{\cite[Lemma~5.1]{Glassy}}]
\label{lm:path}
For any $\eps > 0$ and $A \geq 0$, there exists $L \geq 0$ sufficiently large such that for any $s \geq 0$ sufficiently large satisfying $(e^s + 1)(\rmR+1)^{-1} \in \N$, there exists $T \geq 0$ sufficiently large such that for any $t \geq T$ and $\chi:[0, \rmR]^d \to \R$ satisfying the conditions in \eqref{eq:condChi}, it holds that   
\begin{equation*}
	\P\bigl(\exists m \in [0, \rmR]^d \; \text { such that } \; \rmY^{\chi}_t(m) - \frkd_t \geq -A \; \text{ and } \; \rmY_{\cdot}(m) \not \in \S_{t}^{\chi(m), A, L} \bigr) \leq \eps \rho_{\chi}\bigl([0, \rmR]^d\bigr)\;.
\end{equation*}
Furthermore, it also holds that 
\begin{equation*}
	\P\bigl(\exists m \in [0, \rmR]^d \setminus S_{\rmR, t} \; \text { such that } \; \rmY^{\chi}_t(m) - \frkd_t  \geq - A \bigr) \leq \eps \rho_{\chi}\bigl([0, \rmR]^d\bigr)\;,
\end{equation*}
where $\smash{S_{\rmR, t} \eqdef \bigl[e^{-t/2}, R-e^{-t/2}\bigr]^d}$. 
\end{lemma}

Now, for $\lambda > 0$, we define the following quantities,  
\begin{align}
& \E_{\eqref{eq:pathI}} \eqdef \E\Biggl[\int_{S_{\rmR, t}} \frac{\one_{\{m \in \D^{\lambda}_{\rmR}\}} \one_{\{\rmY_{\cdot}(m) \in \S_{t}^{\chi(m), A, L}\}}}{\abs{\D^{\lambda}_{\rmR}}} \bigl(1-\G_{\rmR}\bigr) dm\Biggr] \label{eq:pathI} \;, \\
& \E_{\eqref{eq:pathII}} \eqdef \E\Biggl[\int_{S_{\rmR, t}} \frac{\one_{\{m \in \D^{\lambda}_{\rmR}\}} \one_{\{\rmY^{\chi}_t(m) - \frkd_t \geq -A, \, \rmY_{\cdot}(m) \not \in \S_{t}^{\chi(m), A, L}\}}}{\abs{\D^{\lambda}_{\rmR}}} \bigl(1-\G_{\rmR}\bigr) dm\Biggr] \label{eq:pathII} \;, \\
& \E_{\eqref{eq:pathIII}} \eqdef \E\Biggl[\int_{[0, \rmR]^d \setminus S_{\rmR, t}} \frac{\one_{\{m \in \D^{\lambda}_{\rmR}\}} \one_{\{\rmY^{\chi}_t(m) - \frkd_t \geq -A\}}}{\abs{\D^{\lambda}_{\rmR}}} \bigl(1-\G_{\rmR}\bigr) dm\Biggr] \label{eq:pathIII} \;,
\end{align}
so that, using the fact that $\smash{\{\rmY_{\cdot}(m) \in \S_{t}^{\chi(m), A, L}\} \subseteq \{\rmY^{\chi}_t(m) - \frkd_t \geq -A\}}$, we can write
\begin{equation*}
\E_{\eqref{eq:III0AA}}  = \E_{\eqref{eq:pathI}} + \E_{\eqref{eq:pathII}} +\E_{\eqref{eq:pathIII}} \;.
\end{equation*}
Using Lemma~\ref{lm:path}, one can easily verify that  
%(see \cite[p.~1399]{Madaule_Max} for further details) 
that for any $\eps > 0$ and $A \geq 0$, there exists $L \geq 0$ sufficiently large such that for any $s \geq 0$ sufficiently large satisfying $(e^s + 1)(\rmR+1)^{-1} \in \N$, there exists $T \geq 0$ sufficiently large such that for any $t \geq T$ and $\chi:[0, \rmR]^d \to \R$ satisfying the conditions in \eqref{eq:condChi},
\begin{equation*}
	\E_{\eqref{eq:pathII}} +\E_{\eqref{eq:pathIII}} \leq \eps \rho_{\chi}\bigl([0, \rmR]^d\bigr) \;.
\end{equation*}
Therefore, in what follows we can just focus on the expectation $\E_{\eqref{eq:pathI}}$ for fixed $A$, $L \geq 0$. 

\subsubsection{Localisation near the maximum}
\label{sec:localisation}

For $\lambda > 0$, $0 \leq b < t/2$, and for any $m \in [0, \rmR]^d$, recalling the notation introduced in \eqref{eq:maximalExpBalls} and \eqref{eq:maximalExpAnnulus}, we define the localised versions of the quantities in \eqref{eq:DefGMDR} by setting
\begin{equation}
\label{eq:DefGDMbtm}
\begin{gathered}	
\G_{m, b_t} \eqdef \exp\Biggl(-\int_{\B_{b_t}(m)} \rmF_{\gamma, t}(x) \mu_{\gamma, t}^{\chi}(dx) \Biggr) \;, \qquad \M_{m, b_t} \eqdef \M_{m, b_t}(\rmY^{\chi}_t) \;, \\
\D^{\lambda}_{m, b_t}\eqdef \D^{\lambda}_{m, b_t}(\rmY^{\chi}_t) \;, 
\end{gathered} 
\end{equation} 
where $b_t \eqdef b-t$. We also introduce the following quantity  
\begin{equation*}
\M_{m, b_t + 1, b_t} \eqdef \M_{m, b_t + 1, b_t}(\rmY^{\chi}_t) \;.
\end{equation*}
Furthermore, for $A \geq 0$, we define the following event
\begin{equation*}
\L_{m, b_t}^{A} \eqdef \bigl\{\exists y \in [0, \rmR]^d \setminus \B_{b_t}(m) \; \text{ such that } \; \rmY^{\chi}_t(y) - \frkd_t \geq - A - \lambda \bigr\} \;.
\end{equation*} 
Roughly speaking, on the complement of the event $\smash{\L_{m, b_t}^{A}}$, everything happens inside $\B_{b_t}(m)$. More precisely, on the event $\smash{\{\rmY^{\chi}_t(m) - \frkd_t \geq -A\}}$ and on the complement of $\smash{\L_{m, b_t}^{A}}$, the maximum of $\rmY^{\chi}_t$ over $[0, \rmR]^d$ must be attained inside $\B_{b_t}(m)$. Furthermore, the values of the field $\smash{\rmY_t^{\chi}}$ at points in $[0, \rmR]^d \setminus \B_{b_t}(m)$, are more than $\lambda$ away from the supremum of $\rmY_t^{\chi}$. Consequently, we have that 
\begin{align*}
& \one_{\{(\L_{m, b_t}^{A})^c\}} \frac{\one_{\{m \in \D^{\lambda}_{\rmR}\}} \one_{\{\rmY^{\chi}_t(m) - \frkd_t \geq -A\}}}{\abs{\D^{\lambda}_{\rmR}}} \\
& \hspace{40mm} = \one_{\{(\L_{m, b_t}^{A})^c\}}  \frac{\one_{\{m \in \D^{\lambda}_{m, b_t}\}} \one_{\{\M_{m, b_t + 1, b_t} -\frkd_t < -A \}} \one_{\{\rmY^{\chi}_t(m) - \frkd_t \geq -A\}}}{\abs{\D^{\lambda}_{m, b_t}}} \;.
\end{align*}
\begin{remark}
\label{rem:RedInd}
The reason why, on the right-hand side of the previous display, we included the seemingly redundant indicator function of the event $\{\M_{m, b_t + 1, b_t} -\frkd_t < -A\}$ is due to a technicality in the proof of Lemma~\ref{lm:cluster} below. This will be better explained during the course of the proof of that lemma. 
\end{remark}

Now, by introducing the following quantities,
\begin{align}
& \E_{\eqref{eq:pathbI}} \eqdef \E\Biggl[\int_{S_{\rmR, t}} \frac{\one_{\{m \in \D^{\lambda}_{m, b_t}\}} \one_{\{\M_{m, b_t + 1, b_t} -\frkd_t < -A \}} \one_{\{\rmY_{\cdot}(m) \in \S_{t}^{\chi(m),A, L}\}}}{\abs{\D^{\lambda}_{m, b_t}}} \bigl(1-\G_{m, b_t}\bigr) dm\Biggr] \;, \label{eq:pathbI} \\
& \E_{\eqref{eq:pathbII}} \eqdef \E\Biggl[\int_{S_{\rmR, t}} \one_{\{{\L_{m, b_t}^{A}}\}} \frac{\one_{\{m \in \D^{\lambda}_{\rmR}\}}\one_{\{\rmY_{\cdot}(m) \in \S_{t}^{\chi(m), A, L}\}}}{\abs{\D^{\lambda}_{\rmR}}} \bigl(1-\G_{\rmR}\bigr) dm\Biggr] \;, \label{eq:pathbII}\\
& \E_{\eqref{eq:pathbIII}} \eqdef \E\Biggl[\int_{S_{\rmR, t}} \one_{\{{\L_{m, b_t}^{A}}\}} \frac{\one_{\{m \in \D^{\lambda}_{m, b_t}\}} \one_{\{\M_{m, b_t + 1, b_t} -\frkd_t < -A \}} \one_{\{\rmY_{\cdot}(m) \in \S_{t}^{\chi(m), A, L}\}}}{\abs{\D^{\lambda}_{m, b_t}}} \nonumber \\
& \hspace{97mm} \cdot\bigl(1-\G_{m, b_t}\bigr) dm\Biggr] \;, \label{eq:pathbIII} \\
& \E_{\eqref{eq:pathbIV}} \eqdef \E\Biggl[\int_{S_{\rmR, t}} \one_{\{(\L_{m, b_t}^{A})^c\}} \frac{\one_{\{m \in \D^{\lambda}_{\rmR}\}} \one_{\{\rmY_{\cdot}(m) \in \S_{t}^{\chi(m),A, L}\}}}{\abs{\D^{\lambda}_{\rmR}}} \bigl(1-\G_{\rmR}\bigr) dm\Biggr] \;, \label{eq:pathbIV} \\
& \E_{\eqref{eq:pathbV}} \eqdef \E\Biggl[\int_{S_{\rmR, t}} \one_{\{(\L_{m, b_t}^{A})^c\}}\frac{\one_{\{m \in \D^{\lambda}_{\rmR}\}} \one_{\{\rmY_{\cdot}(m) \in \S_{t}^{\chi(m),A, L}\}}}{\abs{\D^{\lambda}_{\rmR}}}  \bigl(1-\G_{m, b_t}\bigr) dm\Biggr] \;, \label{eq:pathbV} 
\end{align}
one can easily check that,
\begin{equation*}
	\E_{\eqref{eq:pathI}} = \E_{\eqref{eq:pathbI}} + \E_{\eqref{eq:pathbII}} - \E_{\eqref{eq:pathbIII}} + \E_{\eqref{eq:pathbIV}} - \E_{\eqref{eq:pathbV}} \;.
\end{equation*}
We will show that the fist term is dominant by separately bounding the sum of the second and third term (since both
terms are positive, this dominates their difference), 
as well as the difference between the two last terms.

\begin{lemma}
\label{lm:redExp}
For any $\lambda > 0$ and $\eps > 0$, there exists a constant $A \geq 0$ sufficiently large such that for any $s \geq 0$ sufficiently large satisfying $(e^s + 1)(\rmR+1)^{-1} \in \N$, there exists $T \geq 0$ sufficiently large such that for any $t \geq T$, $b \geq 0$, $L \geq 0$, and $\chi:[0, \rmR]^d \to \R$ satisfying the conditions in \eqref{eq:condChi}, it holds that
\begin{equation*}
\abs{\E_{\eqref{eq:pathbIV}} - \E_{\eqref{eq:pathbV}}} \leq \eps \rho_{\chi}\bigl([0, \rmR]^d\bigr) \;.
\end{equation*}
\end{lemma}
\begin{proof}
Using the elementary inequality $1-e^{-(u_1+u_2)} \leq (1 - e^{-u_1}) + (1 - e^{-u_2})$ which is valid for $u_1$, $u_2 \geq 0$, we note that for all $m \in S_{\rmR, t}$, it holds that
\begin{equation*}
1-\G_{\rmR} \leq \bigl(1- \G_{[0, \rmR]^d \setminus \B_{b_t}(m)}\bigr) + \bigl(1-\G_{m, b_t}\bigr)\;.
\end{equation*}
Hence the claim is proved if we can show that there exists a constant $A \geq 0$ sufficiently large such that for any $s \geq 0$ sufficiently large satisfying $(e^s + 1)(\rmR+1)^{-1} \in \N$, there exists $T \geq 0$ sufficiently large such that for any $t \geq T$, $b \geq 0$, $L \geq 0$, and $\chi:[0, \rmR]^d \to \R$ satisfying the conditions in \eqref{eq:condChi}, it holds that
\begin{equation*}
\E\Biggl[\int_{S_{\rmR, t}} \one_{\{(\L^{A}_{m, b_t})^c\}}\frac{\one_{\{m \in \D^{\lambda}_{\rmR}\}} \one_{\{\rmY_{\cdot}(m) \in \S_{t}^{\chi(m),A, L}\}}}{\abs{\D^{\lambda}_{\rmR}}} \bigl(1-\G_{[0, \rmR]^d\setminus \B_{b_t}(m)}\bigr) dm\Biggr] \leq \eps \rho_{\chi}\bigl([0, \rmR]^d\bigr) \;.
\end{equation*}
By bringing the indicator function of the complement of the event $\L_{m, b_t}^A$ inside the exponential, the above expectation can be bounded from above by
\begin{equation*}
\E\Biggl[1-\exp\Biggl(-\int_{[0, \rmR]^d} \rmF_{\gamma, t}(x) \one_{\{\rmY^{\chi}_t(x) - \frkd_t \leq -A - \lambda\}} \mu_{\gamma, t}^{\chi}(dx)\Biggr)\Biggr] \;,
\end{equation*}
and so the conclusion follows by choosing $A \geq 0$ as in Lemma~\ref{lm:IdentA}. 
\end{proof}

\begin{lemma}
\label{lm:cluster}
For any $\lambda > 0$, $\eps > 0$, and $A$, $L \geq 0$, there exist $b_0$, $s_0 \geq 0$ large enough such that for any $s \geq s_0$ satisfying $(e^s + 1)(\rmR+1)^{-1} \in \N$ and $b \geq b_0$, there exists $T \geq 0$ large enough such that for any $t \geq T$, and $\chi:[0, \rmR]^d \to \R$ satisfying the conditions in \eqref{eq:condChi}, it holds that 
\begin{equation*}
\abs{\E_{\eqref{eq:pathbII}} + \E_{\eqref{eq:pathbIII}}} \leq \eps \rho_{\chi}\bigl([0, \rmR]^d\bigr) \;.
\end{equation*}
\end{lemma}
\begin{proof}
The proof of this result is similar to the proof of \cite[Lemmas~5.1,~5.2]{Madaule_Max}. We only consider the quantity $\E_{\eqref{eq:pathbIII}}$, since $\E_{\eqref{eq:pathbII}}$ can be bounded by following a similar, and in fact simpler, strategy. We introduce the lattice $\Lambda_{b_t}$ by letting 
\begin{equation*}
\Lambda_{b_t} \eqdef \frac{1}{4 \sqrt{d}}e^{b_t} \Z^d \cap [0, \rmR]^d \;.
\end{equation*}
Using the fact that $S_{\rmR, t} \subseteq \cup_{x \in \Lambda_{t, b}} \B_{b_t-\log4}(x)$, we note that the quantity inside the expectation $\E_{\eqref{eq:pathbIII}}$ can be bounded from above by   
\begin{align}
& \int_{S_{\rmR, t}} \one_{\{\L^{A}_{m, b_t}\}} \frac{\one_{\{m \in \D^{\lambda}_{m, b_t}\}} \one_{\{\M_{m, b_t + 1, b_t} - \frkd_t < -A\}} \one_{\{\rmY_{\cdot}(m) \in \S_{t}^{\chi(m), A, L}\}}}{\abs{\D^{\lambda}_{m, b_t}}} dm \nonumber \\
& \leq \sum_{x \in \Lambda_{b_t}} \int_{\B_{b_t-\log4}(x)} \one_{\{\L^{A}_{m, b_t}\}} \frac{\one_{\{m \in \D^{\lambda}_{m, b_t}\}} \one_{\{\M_{m, b_t + 1, b_t} - \frkd_t < -A\}} \one_{\{\rmY_{\cdot}(m) \in \S_{t}^{\chi(m), A, L}\}}}{\abs{\D^{\lambda}_{m, b_t}}} dm \nonumber \\
& \leq \sum_{x \in \Lambda_{b_t}} \one_{\{\L^{A}_{x, b_t-\log(4/3)}\}} \one_{\{\exists m \in \B_{b_t-\log4}(x) \text{ such that } \rmY_{\cdot}(m) \in \S_{t}^{\chi(m), A, L}\}} \;. \label{eq:bridge05152}
\end{align}
In order to get the last inequality, one can note that for $m \in \B_{b_t-\log4}(x)$, it holds that 
\begin{equation*}
\one_{\{\M_{m, b_t + 1, b_t} - \frkd_t < -A\}} \one_{\{\rmY_t^{\chi}(m) - \frkd_t \geq -A\}} \leq \one_{\{\M_{m, b_t} = \M_{x, b_t+\log(5/4)}\}} \;,
\end{equation*}
where we refer to Figure~\ref{fig:decoComplex} for a diagram illustrating the sets involved. 
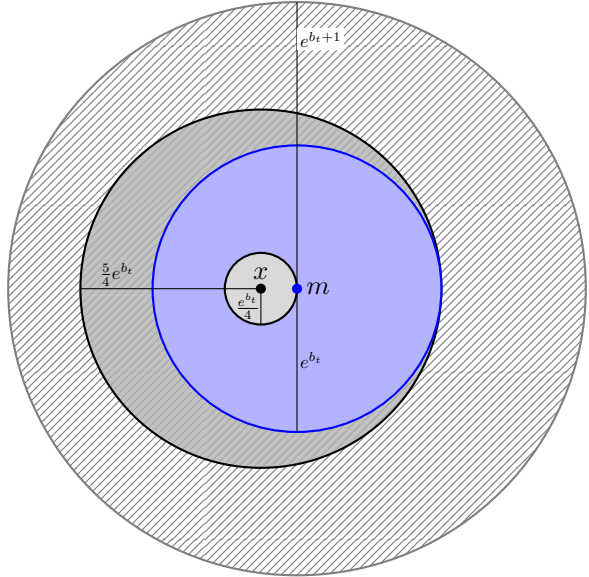
\begin{figure}[ht] 
\begin{center}
\begin{tikzpicture}[scale=0.95]
	\fill[pattern=north east lines, pattern color=gray] (0.5,0) circle (4);
	\draw[gray, thick] (0.5,0) circle (4);
	
	\fill[black!30, opacity = 0.8] (0,0) circle (2.5);
	\draw[black, thick] (0,0) circle (2.5);

	\fill[blue!30, opacity = 1] (0.5,0) circle (2);
	\draw[blue, thick] (0.5,0) circle (2);
	
	\fill[gray!30, opacity = 1] (0,0) circle (0.5);
	\draw[black, thick] (0,0) circle (0.5);
	\draw[black] (0,0) -- (-2.5,0) node[midway, above, inner sep=1pt, xshift=-20pt] {\scalebox{0.7}{$\f54 e^{b_t}$}};
	\draw[black] (0,0) -- (0,-0.5) node[midway, left, inner sep=0pt] {\scalebox{0.7}{$\f{e^{b_t}}{4}$}};
	\fill[black] (0,0) circle (2pt);
	\node[above] at (0,0) {$x$};
	
	\draw[black] (0.5,0) -- (0.5,-2) node[midway, right, inner sep=1pt] {{\color{black}\scalebox{0.7}{$e^{b_t}$}}};
	\draw[black] (0.5,0) -- (0.5,4) node[midway, right, fill, white, inner sep=1pt, yshift=40pt] {{\color{black}\scalebox{0.7}{$e^{b_t + 1}$}}};
		\fill[blue] (0.5, 0) circle (2pt);
	\node[right] at (0.5,0) {$m$};
\end{tikzpicture}
\caption[short form]{\small{The point $m$ is chosen inside the small grey ball $\B_{b_t - \log 4}(x)$, and the blue ball represents  $\B_{b_t}(m)$. Since we are on the event $\{\rmY_t^{\chi}(m) - \frkd_t \geq -A\}$, on the event $\{\M_{m, b_t + 1, b_t} -\frkd_t < -A\}$, the supremum of the field $\rmY_t^{\chi}$ inside the large striped ball must coincide with the supremum of the field $\rmY_t^{\chi}$ inside the shaded black ball $\B_{b_t + \log(5/4)}(x)$.}}
\label{fig:decoComplex}
\end{center}
\end{figure}
Thus, by letting
 \begin{equation*}
\bar{\D}^{\lambda}_{x, b_t} \eqdef \bigl\{y \in \B_{b_t-\log4}(x) \, : \, \rmY^{\chi}_t(y) \geq \M_{x, b_t+\log(5/4)} - \lambda\bigr\} \;,
\end{equation*}
we have that\footnote{As we mentioned in Remark~\ref{rem:RedInd}, the reason we included the indicator function of the event $\smash{\{\M_{m, b_t + 1, b_t} - \frkd_t < -A \}}$ is due to a technical reason. Indeed,  without this indicator function, we would not be able to bound the integral by one. It is not difficult to construct a ``pathological example'' for a possible realisation of the field $\rmY_t$ for which the integral becomes arbitrarily large.}
\begin{align*}
& \int_{\B_{b_t-\log4}(x)}  \frac{\one_{\{\M_{m, b_t + 1, b_t} - \frkd_t < -A \}} \one_{\{\rmY_t^{\chi}(m) - \frkd_t \geq -A\}} \one_{\{m \in \D^{\lambda}_{m, b_t}\}}}{\abs{\D^{\lambda}_{m, b_t}}} dm \\
& \hspace{80mm}\leq \int_{\B_{b_t-\log4}(x)} \frac{\one_{\{m \in  \bar{\D}^{\lambda}_{x, b_t}\}}}{\abs{\bar{\D}^{\lambda}_{x, b_t}}} dm = 1 \;.
\end{align*}
From this point, it suffices to follow the proof of \cite[Lemma~5.1]{Madaule_Max} to conclude.
\end{proof}

%%%%%%%%%%%%%%%%%%%%%%%%%%%%%%%%%%%%%%%%%%%%%
\subsubsection{A renewal result}
Combining Lemmas~\ref{lm:IdentA},~\ref{lm:path},~\ref{lm:redExp},~and~\ref{lm:cluster}, we have shown in the previous subsection that for any $\lambda > 0$ and $\eps > 0$, there exist $0 \leq A < L$ sufficiently large such that there exists $b_0 > 0$ and $s_0 > 0$ sufficiently large, such that for any $s \geq s_0$ satisfying $(e^s + 1)/(\rmR+1) \in \N$ and any $b \geq b_0$, there exists a sufficiently large $T \geq 0$ such that for all $t \geq T$ and any function $\chi:[0,\rmR]^d \to \R$ satisfying the conditions in \eqref{eq:condChi}, it holds that
\begin{equation*}
\left\lvert\E \Biggl[1-\exp\biggl(-\int_{[0, \rmR]^d} \rmF_{\gamma, t}(x) \mu^{\chi}_{\gamma, t}(dx)\biggr)\Biggr]  - \E_{\eqref{eq:pathbI}}\right\rvert \leq \eps \rho_{\chi}\bigl([0, \rmR]^d\bigr) \;,
\end{equation*}
Therefore, in what follows, we fix $\lambda > 0$ and $A$, $L$, $b \geq 0$, and we study the quantity $\E_{\eqref{eq:pathbI}}$. Given a function $g:\R^d \to \R$, we introduce the fields $\bar \Upsilon_{\!b}$ and $\bar \Upsilon_{\!b, g}$ on $\R^d$ by letting
\begin{equation}
\label{eq:defRealUps}
\bar \Upsilon_{\!b}(x) \eqdef \rmY_b(e^{-b} x)-\rmY_b(0)\;, \qquad \bar \Upsilon_{\!b, g}(x) \eqdef \bar \Upsilon_{\!b}(x) + g(x) \;.
\end{equation}
Thanks to translation invariance, the scaling relation \eqref{e:scaling_rel}, the Markov property at time $\smash{t_b \eqdef t-b}$ of the process $\smash{(\rmY_t(\cdot))_{t \geq 0}}$, the assumptions \ref{hp:W1} and \ref{hp:W2}, and, for $m \in S_{\rmR, t}$, by applying the Cameron--Martin theorem (see Lemma~\ref{lm_Girsanov}) with density $\smash{\exp(\sqrt{\smash[b]{2d}}\rmY_{t_b}(m) + d (t-b))}$ to the process $\smash{(\rmY_s(m))_{s \in [0, t_b]}}$, we can rewrite $\E{\eqref{eq:pathbI}}$ as follows
\begin{align*}
& \int_{S_{\rmR, t}} e^{-\sqrt{\smash[b]{2d}} \chi(m)} t^{3/2} \E_{-\chi(m)}\Bigl[\one_{\{\sup_{s \in [0, t_b]} B_s \leq 0, \; \sup_{s\in [t/2, t_b]} B_s \leq \frkd_t + L\}} \smash{\bar{\frkF}}^{\lambda}_{A, L, b}\bigl(B_{t_b} - \frkd_t - L, \frkg^{\chi, m}_{t, b}\bigr)\Bigr] dm \;,
\end{align*}
where, under $\P_{- \chi(m)}$ the Brownian motion $(B_s)_{s \geq 0}$ is started from $-\chi(m)$, and recalling the definition \eqref{eq:defscrF} of the function $\bfF^{\lambda}_{k, x}$, the map $\smash{\smash{\bar{\frkF}}^{\lambda}_{A, L, b} : \R \times \CC(\R^d) \to \R}$ is given by
\begin{equation}
\label{eq:defbarFALb}
\begin{alignedat}{1}
&\smash{\bar{\frkF}}^{\lambda}_{A, L, b}(z, g) \eqdef e^{-\sqrt{\smash[b]{2d}} (z + L) + db} \E_{z} \Bigl[\one_{\{\rmY_b(0) \geq - A - L \}} \one_{\{\sup_{s \in [0, b]} \rmY_s(0) \leq 0\}} \\
& \hspace{30mm} \cdot \one_{\{\M_{0, b + 1, b}(\bar \Upsilon_{\! b, g}) + \rmY_b(0) < -A-L\}} \one_{\{\M_{0, b}(\bar \Upsilon_{\! b, g})  \leq \lambda\}} \bfF^{\lambda}_{b, -\rmY_b(0) - L}\bigl(\bar \Upsilon_{\! b, g}\bigr)\Bigr] \;.
\end{alignedat}
\end{equation}
where, under $\P_z$ the field $(\rmY_s(x))_{s \geq 0, x \in \R^d}$ has the same law of $(\rmY_s(x) + z)_{s \geq 0, x \in \R^d}$ under $\P$, and $(\frkg^{\chi, m}_{t, b}(x))_{x \in \R^d}$ is an independent continuous random field given by
\begin{equation*}
\frkg^{\chi, m}_{t, b}(x) \eqdef \int_0^{t_b} (1-\frkK(e^{s-t}x)) dB_{s} + \rmZ_{t_b}(e^{-b} x) - \sqrt{\smash[b]{2d}}\int_0^{t_b} (1 - \frkK(e^{s-t}x)) ds - \bigl(\chi(m+e^{-t}x) - \chi(m)\bigr) \;,
\end{equation*}
where we recall once again that $\rmZ_{t_b}$ is the field introduced in Definition~\ref{def:fieldsZ}. 

\begin{lemma}
\label{lm:bRegular}
For any $\lambda > 0$ and $A$, $L$, $b \geq 0$, consider the function $\smash{\smash{\bar{\frkF}}^{\lambda}_{A, L, b}: \R \times \CC(\R^d) \to \R}$ defined in \eqref{eq:defbarFALb}. Then, there exist two functions $\frkh: \R \to \R^{+}$ and $\smash{\frkF_{*}: \CC(\R^d) \to \R^{+}}$, possibly depending on the parameters $\lambda$, $A$, $L$, $b$, such that:
\begin{enumerate}[start=1,label={{{(\arabic*})}}]
\item \label{eq:bReg1} It holds that 
\begin{equation*}
\sup_{z \in \R} \frkh(z) < \infty \qquad \text{ and } \qquad \frkh(z) = O(e^z) \quad \text{ as } z \to -\infty \;.
\end{equation*}
\item \label{eq:bReg2} For any $z \in \R$ and $g \in C(\R^d)$ it holds that
\begin{equation*}
\smash{\bar{\frkF}}^{\lambda}_{A, L, b}(z, g) \leq \frkh(z) \frkF_{*}(g) \;.
\end{equation*}
\item \label{eq:bReg3} There exists $c > 0$ such that for any $\delta \in (0, 1)$ and $g \in \CC(\R^d)$ satisfying  
\begin{equation*}
\sup_{x, y \in \R^d,\, \abs{x - y} \leq e^b\delta}\abs{g(x) - g(y)} \leq 1/4 \;,
\end{equation*} 
it holds that 
\begin{equation*}
\frkF_{*}(g) \leq c \delta^{-10} \;.
\end{equation*}
\item \label{eq:bReg4} There exists $c > 0$ such that for any $g_1$, $g_2 \in \CC(\R^d)$ satisfying  
\begin{equation*}
\|g_1 - g_2\|_{\infty} \eqdef \sup_{x \in \R^d}\abs{g_1(x) - g_2(x)} \leq 1/8 \;,
\end{equation*} 
it holds that 
\begin{equation*}
\bigl|\smash{\bar{\frkF}}^{\lambda}_{A, L, b}(z, g_1) - \smash{\bar{\frkF}}^{\lambda}_{A, L, b}(z, g_2)\bigr| \leq c \|g_1 - g_2\|_{\infty}^{1/8} \frkh(z) \frkF_{*}(g_1) \;.
\end{equation*}
\end{enumerate}
\end{lemma}
\begin{proof}
	The proof can be found in Appendix~\ref{ap:proofbReg}.
\end{proof}

We are now in position to conclude the proof of Lemma~\ref{lm:techMainRed1}, which as we see below follows by \cite[Theorem~5.6]{Madaule_Max}.

\begin{proof}[Proof of Lemma~\ref{lm:techMainRed1}]
Thanks to Lemma~\ref{lm:bRegular}, we have that for any $\lambda > 0$ and $A$, $L$, $b \geq 0$, the function $\smash{\smash{\bar{\frkF}}^{\lambda}_{A, L, b}: \R \times \CC(\R^d) \to \R}$ defined in \eqref{eq:defbarFALb} is ``$b$-regular'' in the sense of \cite{Madaule_Max}. Hence, thanks to \cite[Theorem~5.6]{Madaule_Max}, we have that for any $\eps >0$, for any $s \geq 0$ sufficiently large satisfying $(e^s + 1)(\rmR+1)^{-1} \in \N$, there exists $T \geq 0$ sufficiently large such that for any $t \geq T$ and $\chi:[0, \rmR]^d \to \R$ satisfying the conditions in \eqref{eq:condChi}, it holds that
\begin{equation*}
\lvert\E_{\eqref{eq:pathbI}}  - \bar{\bfC}^{\lambda}_{A, L, b} \, \rho_{\chi}\bigl([0, \rmR]^d\bigr) \rvert \leq \eps \rho_{\chi}\bigl([0, \rmR]^d\bigr) \;,
\end{equation*}
with the constant $\bar{\bfC}^{\lambda}_{A, L, b}$ defined as
\begin{equation*}
\bar{\bfC}^{\lambda}_{A, L, b} \eqdef \alpha \int_{0}^{\infty} z \E\bigl[\smash{\bar{\frkF}}^{\lambda}_{A, L, b}\bigl(-z, \frkg^z_{b}\bigr)\bigr] dz \;,
\end{equation*}
where we recall once again that $\alpha = \sqrt{\smash[b]{2/\pi}}$ and the field $\frkg^z_{b}$ is defined in \eqref{eq:fieldGzb}. 
Hence, in order to conclude, it suffices to show that for all $z \geq 0$ and $g \in \CC(\R^d)$ it holds that  $\smash{\smash{\bar{\frkF}}^{\lambda}_{A, L, b}(-z, g) = \frkF^{\lambda}_{A, L, b}(z, g)}$. 
To this end, we begin by observing that, thanks to Lemma~\ref{lm:decoPointsStar}, and recalling Definition~\ref{def:fieldsZ} as well as \eqref{eq:frkgb}, for all $y \in \R^d$, it holds that
\begin{align*}
\bar \Upsilon_{\! b}(y) 
& = - \int_0^b \bigl(1 - \frkK(e^{s-b} y)\bigr) d \rmY_{s}(0) + \bar\rmZ^0_b(e^{-b} y) - \sqrt{\smash[b]{2d}}  \int_0^b \bigl(1 - \frkK(e^{s-b} y)\bigr) ds \\
& \eqlaw - \int_0^b \bigl(1 - \frkK(e^{-s} y)\bigr) d \rmY_{b-s}(0) + \rmZ_b(y) - \sqrt{\smash[b]{2d}}  \frka_b(y) \;.
\end{align*}
Hence, by applying the Cameron--Martin theorem (see Lemma~\ref{lm_Girsanov}) with density $\smash{\exp(\sqrt{\smash[b]{2d}}\rmY_b(0) + db)}$ to the process $\smash{(\rmY_s(0))_{s \in [0,b]}}$, using the equality in law $\smash{(B_s)_{s \in [0,b]} \eqlaw (-B_s)_{s \in [0, b]}}$, using the fact that $\P_z(B_b \in dx) = (2 \pi b)^{-1/2} e^{-(x-z)^2/(2b)} dx$, and finally the fact that a time reversed Brownian bridge is still a Brownian bridge but with starting and final point swapped, we obtain that 
\begin{equation*}
\begin{alignedat}{1}
\smash{\bar{\frkF}}^{\lambda}_{A, L, b}(-z, g) 
&  = \frac{1}{\sqrt{2 \pi b}} \int_{0}^{A+L} e^{\sqrt{\smash[b]{2d}}(x-L)} e^{-\frac{(x-z)^2}{2b}}  \E_{x, z, b} \Bigl[\bfF^{\lambda}_{b, x - L}\bigl(\Upsilon_{\! b, g}\bigr) \\
& \hspace{20mm} \cdot \one_{\{\inf_{s \in [0, b]} B_s \geq 0\}} \one_{\{\M_{0, b + 1, b}(\Upsilon_{\! b, g})  < x- (A+L)\}} \one_{\{\M_{0, b}(\Upsilon_{\! b, g})  \leq \lambda\}} \Bigr] dx \;.
\end{alignedat}
\end{equation*}
Therefore, the conclusion follows since the right-hand side of the previous display coincides with $\smash{\frkF^{\lambda}_{A, L, b}(z, g)}$ as desired. 
\end{proof}

%%%%%%%%%%%%%%%%%%%%%%%%%%%%%%%%%%%%%%%%%%%%%
\subsection{Proof of Lemma~\ref{lm:techMainRed2}}
\label{subsec:redStepConstant}
The main goal of this section is to prove Lemma~\ref{lm:techMainRed2}. As in the previous case, the proof follows a sequence of reduction steps, allowing us to transition from $\smash{\bfC^{\lambda}_{A, L, b}}$ to $\smash{\bfC^{\lambda, \new}_{A, L, k, b}}$. Each of these reduction steps forms a lemma within this section. Recalling the definition \eqref{eq:defCALb} of the constant $\smash{\bfC^{\lambda}_{A, L, b}}$, it is convenient to introduce the function $\smash{\bfG^{\lambda}_{A, L, b}: \CC(\R^d) \to \R}$ given by
\begin{equation}
\label{eq:defF'ALb}
\bfG^{\lambda}_{A, L, b}(g) \eqdef \alpha \int_{0}^{\infty} z \frkF^{\lambda}_{A, L, b}\bigl(z, g\bigr) dz \;.
\end{equation}
By plugging the expression \eqref{eq:asympFALb0} for the function $\smash{\frkF^{\lambda}_{A, L, b}}$ into the right-hand side of \eqref{eq:defF'ALb}, and doing a change of variables, we get that the expectation of $\smash{\bfG^{\lambda}_{A, L, b}(\frkg^z_{b})}$ is equal to
\begin{equation} 
\label{eq:expGALb}
\begin{alignedat}{1}
\bfG^{\lambda}_{\eqref{eq:expGALb}} & \eqdef \alpha \int_{0}^{A+L} e^{\sqrt{\smash[b]{2d}} (x - L)} \int_{-x}^{\infty} \frac{e^{-\f{u^2}{2b}}}{\sqrt{2 \pi b}} (u+x) \E_{0, u, b} \Bigl[\bfF^{\lambda}_{b, x - L}(\Upsilon_{\! b}^{ u+x})   \\
& \hspace{10mm}\cdot \one_{\{\M_{0, b + 1, b}(\Upsilon_{\! b}^{ u+x}) \leq x-(A+L)\}} \one_{\{\inf_{s \in [0, b]} B_{s} \geq - x\}} \one_{\{\M_{0, b}(\Upsilon_{\! b}^{ u+x}) \leq \lambda\}}\Bigr] du dx \;, 
\end{alignedat}
\end{equation}
where we recall that $\smash{\Upsilon_{\! b}^{z} = \Upsilon_{\! b} + \frkg^z_{b}}$, where $\smash{\frkg^z_{b}}$ is the field introduced in \eqref{eq:fieldGzb}. 
As before, with a slight abuse of our previous notation, we have absorbed the expectation with respect to the field $\frkg^{u+x}_{b}$ into the expectation $\E_{0, u, b}$.
In what follows, we will implicitly use the fact \dash previously noted in the proof of Proposition~\ref{pr:joint} \dash that the field $\smash{\frkg^z_{b}}$ satisfies \ref{as:GG1} \dash \ref{as:GG3} uniformly over all $z \geq 0$. In particular, all the results derived in Section~\ref{sec:cluster} remain valid for the field $\Upsilon_{\! b}^{z}$ uniformly over all $z \geq 0$.

The goal is now to simplify the quantity $\smash{\bfG^{\lambda}_{\eqref{eq:expGALb}}}$ into a more manageable form through a sequence of reduction steps.
We begin with Lemma~\ref{lm:ExtremeIntBB}, where we show that the integral over $u$ is concentrated around $\sqrt{b}$. 
Next, in Lemma~\ref{lm:truncOK}, we establish that $\smash{\bfF^{\lambda}_{b, x-L}}$ can be replaced by its truncation at a large values. 
Afterward, in Lemma~\ref{lm:removIndOK}, we show that the two first two indicator functions on the second line in \eqref{eq:expGALb} can be removed. 
Finally, in Lemma~\ref{lm:reduceSizeOK}, we show that $\smash{\bfF^{\lambda}_{b, x-L}}$ can be replaced by $\smash{\bfF^{\lambda}_{k, x-L}}$ for some $k \ll b$.

\subsubsection{Reduction steps}
\label{sub:prooftechLemmaId}
We can now begin reducing $\smash{\bfG^{\lambda}_{\eqref{eq:expGALb}}}$ to a more manageable quantity. To this end, we start by defining
\begin{equation}
\label{eq:expGALb1}
\begin{alignedat}{1}
\bfG^{\lambda}_{\eqref{eq:expGALb1}} & \eqdef \alpha \int_{0}^{A+L} e^{\sqrt{\smash[b]{2d}} (x - L)} \int_{b^{1/4}}^{b^{3/4}} \frac{e^{-\f{u^2}{2b}}}{\sqrt{2 \pi b}} (u+x) \E_{0, u, b} \Bigl[\bfF^{\lambda}_{b, x-L}\bigl(\Upsilon_{\! b}^{ u+x}\bigr)  \\
& \hspace{10mm}\cdot \one_{\{\M_{0, b + 1, b}(\Upsilon_{\! b}^{ u+x}) \leq x-(A+L)\}} \one_{\{\inf_{s \in [0, b]} B_{s} \geq -x\}} \one_{\{\M_{0, b}(\Upsilon_{\! b}^{ u+x}) \leq \lambda\}}\Bigr] du dx \;,
\end{alignedat}
\end{equation}
where we emphasise that the only difference between $\bfG^{\lambda}_{\eqref{eq:expGALb1}}$ and $\bfG^{\lambda}_{\eqref{eq:expGALb}}$ is the domain of integration in the $u$ variable. 
\begin{lemma}
\label{lm:ExtremeIntBB}
For any $\lambda > 0$, $\eps > 0$, and $A$, $L \geq 0$, there exists $b \geq 0$ sufficiently large such that 
\begin{equation*}
\abs{\bfG^{\lambda}_{\eqref{eq:expGALb1}} - \bfG^{\lambda}_{\eqref{eq:expGALb}}} \leq \eps \;.
\end{equation*}
\end{lemma}
\begin{proof}
We begin by observing that the quantity $\smash{\abs{\bfG^{\lambda}_{\eqref{eq:expGALb1}} - \bfG^{\lambda}_{\eqref{eq:expGALb}}}}$ is bounded above by a multiple of
\begin{align}
& \int_{0}^{A+L} e^{\sqrt{\smash[b]{2d}} (x - L)} \int_{-x}^{b^{1/4}}\frac{e^{-\f{u^2}{2b}}}{\sqrt{2 \pi b}} (u+x) \E_{0, u, b} \Bigl[\bfF^{\lambda}_{b, x-L}\bigl(\Upsilon_{\! b}^{ u+x}\bigr) \one_{\{\M_{0, b}(\Upsilon_{\! b}^{ u+x}) \leq \lambda\}}\Bigr] du dx \label{eq:techIdTring1} \\	
& \hspace{1mm} + \int_{0}^{A+L} e^{\sqrt{\smash[b]{2d}} (x - L)} \int_{b^{3/4}}^{\infty}\frac{e^{-\f{u^2}{2b}}}{\sqrt{2 \pi b}} (u+x) \E_{0, u, b} \Bigl[\bfF^{\lambda}_{b, x-L}\bigl(\Upsilon_{\! b}^{ u+x}\bigr) \one_{\{\M_{0, b}(\Upsilon_{\! b}^{ u+x}) \leq \lambda\}}\Bigr] du dx \label{eq:techIdTring2}\;.
\end{align}
Recalling the definition \eqref{eq:defscrF}, we have that
\begin{equation*}
\bfF^{\lambda}_{b, x-L}\bigl(\Upsilon_{\! b}^{ u+x}\bigr) \leq \abs{\D^{\lambda}_{b, 0}(\Upsilon_{\! b}^{ u+x})}^{-1} \;.
\end{equation*}
By using the above estimate, we proceed to bound \eqref{eq:techIdTring1} and \eqref{eq:techIdTring2} separately.
Starting from \eqref{eq:techIdTring1}, using Lemma~\ref{lm:tailsDgUpsilonSmallU}, we obtain that 
\begin{equation*}
\eqref{eq:techIdTring1} \lesssim b^{-\f12} \int_0^{A+L} e^{\sqrt{\smash[b]{2d}} (x - L)} (b^{1/4} + x)  dx \;,
\end{equation*}
and the quantity on the right-hand side can be made arbitrarily small by taking $b \geq 0$ sufficiently large.
Regarding \eqref{eq:techIdTring2}, thanks to Lemma~\ref{lm:tailsDgUpsilonLargeU}, we obtain that  
\begin{equation*}
\eqref{eq:techIdTring2} \lesssim \int_0^{A+L}e^{\sqrt{\smash[b]{2d}} (x - L)} \int_{b^{3/4}}^{\infty} \frac{e^{-\f{u^2}{2b}}}{\sqrt{2 \pi b}} (u+x) (b^{d^2} + (u/b)^{2d}) du dx \;,
\end{equation*}
and, as before, the quantity on the right-hand side can be made arbitrarily small by taking $b \geq 0$ sufficiently large. Hence, the claim follows.
\end{proof}

Thanks to the previous lemma, from now on, we can focus on $\smash{\bfG^{\lambda}_{\eqref{eq:expGALb1}}}$ instead of $\smash{\bfG^{\lambda}_{\eqref{eq:expGALb}}}$. We now address the fact that the quantity $\smash{\bfF^{\lambda}_{b, x-L}(\Upsilon_{\! b}^{ u+x})}$ is not a priori bounded.
The approach to overcoming this issue is quite straightforward. Specifically, recalling \eqref{eq:defTruncationLargeL}, we show that replacing $\smash{\bfF^{\lambda}_{b, x-L}(\Upsilon_{\! b}^{ u+x})}$ in $\smash{\bfG^{\lambda}_{\eqref{eq:expGALb1}}}$ with $\smash{\bfF^{\lambda}_{b, x-L, L}(\Upsilon_{\! b}^{ u+x})}$ introduces only a negligible error, provided that $L$ is sufficiently large. 
More precisely, we consider
\begin{equation}
\label{eq:expGALb3}
\begin{alignedat}{1}
\bfG^{\lambda}_{\eqref{eq:expGALb3}} & \eqdef \alpha \int_{L/2}^{A+L} e^{\sqrt{\smash[b]{2d}} (x - L)} \int_{b^{1/4}}^{b^{3/4}} \frac{e^{-\f{u^2}{2b}}}{\sqrt{2 \pi b}} (u+x) \E_{0, u, b} \Bigl[\bfF^{\lambda}_{b, x-L, L}(\Upsilon_{\! b}^{ u+x}) \\
& \hspace{10mm}\cdot \one_{\{\M_{0, b + 1, b}(\Upsilon_{\! b}^{ u+x}) \leq x-(A+L)\}} \one_{\{\inf_{s \in [0, b]} B_{s} \geq - x\}} \one_{\{\M_{0, b}(\Upsilon_{\! b}^{ u+x}) \leq \lambda\}}\Bigr] du dx \;,
\end{alignedat}
\end{equation} 
and we claim the following result.
\begin{lemma}
\label{lm:truncOK}
For any $\lambda > 0$, $\eps > 0$ and $A \geq 0$, there exist $L \geq 0$ sufficiently large, and $b \geq 0$ sufficiently large such that 
\begin{equation*}
\abs{\bfG^{\lambda}_{\eqref{eq:expGALb3}}  - \bfG^{\lambda}_{\eqref{eq:expGALb1}}} \leq \eps \;.
\end{equation*}
\end{lemma}
\begin{proof}
We start by addressing the fact that the integral over $x$ from $0$ to $A+L$ in $\smash{\bfG^{\lambda}_{\eqref{eq:expGALb1}}}$ can be reduced, up to a negligible error, to an integral from $L/2$ to $A+L$. To this end, for $x \in [0, A+L]$ and $u \in [b^{1/4}, b^{3/4}]$, we recall that thanks to Lemma~\ref{lm:tailsDgUpsilon}, it holds that  
\begin{equation}
\label{eq:boundL/2}
\E_{0, u, b}\Bigl[\bfF^{\lambda}_{b, x-L}(\Upsilon_{\! b}^{ u+x}) \one_{\{\M_{0, b}(\Upsilon_{\! b}^{ u+x}) \leq \lambda\}} \Bigr] \lesssim \frac{u}{b}  \;.
\end{equation}
Using the above bound, one can easily see that the difference between $\smash{\bfG^{\lambda}_{\eqref{eq:expGALb1}}}$ and the same quantity with the integral over $x$ from $0$ to $A+L$ replaced by the integral over $x$ from $L/2$ to $A+L$ can be made arbitrarily small by taking $L \geq 0$ sufficiently large uniformly over all $b \geq 0$.

Now, for $x \in [L/2, A+L]$ and $u \in [b^{1/4}, b^{3/4}]$, we need to bound from above the following difference,  
\begin{equation}
\label{eq:boundTrunc}
\E_{0, u, b}\Bigl[\abs{\bfF^{\lambda}_{b, x - L}(\Upsilon_{\! b}^{ u+x}) - \bfF^{\lambda}_{b, x-L, L}(\Upsilon_{\! b}^{ u+x})} \one_{\{\M_{0, b}(\Upsilon_{\! b}^{ u+x}) \leq \lambda\}} \Bigr] \;.
\end{equation}
To this end, recalling the definition of the function $\bfF^{\lambda}_{k, x}$ in \eqref{eq:defscrF} and that of $\bfF^{\lambda}_{k, x, L}$ in \eqref{eq:defTruncationLargeL}, we note that 
\begin{equation*}
\abs{\bfF^{\lambda}_{b, x-L}(\Upsilon_{\! b}^{ u+x}) - \bfF^{\lambda}_{b, x-L, L}(\Upsilon_{\! b}^{ u+x})}
\leq \bigl(\abs{\D^{\lambda}_{0,b}(\Upsilon_{\! b}^{ u+x})}^{-1} - L\bigr)\one_{\{\abs{\D^{\lambda}_{0,b}(\Upsilon_{\! b}^{ u+x})}^{-1} \geq L\}} \;.
\end{equation*}
Therefore, plugging this into \eqref{eq:boundTrunc} and using Lemma~\ref{lm:tailsDgUpsilon}, there exists $\delta = \delta(d) \in (0, 1)$ such that 
\begin{align*}
\E_{0, u, b}\Bigl[\abs{\bfF^{\lambda}_{b, x-L} & (\Upsilon_{\! b}^{ u+x}) - \bfF^{\lambda}_{b, x-L, L}(\Upsilon_{\! b}^{ u+x})} \one_{\{\M_{0, b}(\Upsilon_{\! b}^{ u+x}) \leq \lambda \}} \Bigr]	\\
& \leq \int_{L}^{\infty} \P_{0, u, b}\bigl(\abs{\D^{\lambda}_{0,b}(\Upsilon_{\! b}^{ u+x})}^{-1} \geq \eta,\; \M_{0, b}(\Upsilon_{\! b}^{ u+x}) \leq \lambda \bigr) d\eta \\
& \lesssim \frac{u}{b} \int_{L}^{\infty} \eta^{-(1+\delta)} d\eta \lesssim \frac{u}{b} L^{-\delta} \;.
\end{align*}
Therefore, the conclusion follows readily from the above bound.
\end{proof}

We now need to show how we can remove the first two indicator functions appearing on the second line of $\smash{\bfG^{\lambda}_{\eqref{eq:expGALb3}}}$. To this end, we consider
\begin{equation}
\label{eq:expGALb4}
\begin{alignedat}{1}
\bfG^{\lambda}_{\eqref{eq:expGALb4}} & \eqdef \alpha \int_{L/2}^{A+L} e^{\sqrt{\smash[b]{2d}} (x - L)} \int_{b^{1/4}}^{b^{3/4}} \frac{e^{-\f{u^2}{2b}}}{\sqrt{2 \pi b}}  (u+x) \\
& \hspace{40mm} \cdot \E_{0, u, b} \Bigl[\bfF^{\lambda}_{b, x-L, L}(\Upsilon_{\! b}^{ u+x}) \one_{\{\M_{0, b}(\Upsilon_{\! b}^{ u+x}) \leq \lambda\}}\Bigr]  du dx \;,
\end{alignedat}
\end{equation} 
and we claim the following lemma.
\begin{lemma}
\label{lm:removIndOK}
For any $\lambda > 0$, $\eps > 0$ and $A \geq 0$, there exist $L \geq 0 $ sufficiently large, and $b \geq 0$ sufficiently large such that 
\begin{equation*}
\abs{\bfG^{\lambda}_{\eqref{eq:expGALb4}} -\bfG^{\lambda}_{\eqref{eq:expGALb3}}} \leq \eps \;.
\end{equation*}
\end{lemma}
\begin{proof}
We divide the proof in two steps. First, we prove that the indicator function involving the Brownian motion can be removed. Then, we show that the indicator function concerning the quantity $\M_{0, b + 1, b}(\Upsilon_{\! b}^{u+x})$ can also be disregarded.

\textbf{Step 1:} We start by showing how we can remove the indicator function regarding the Brownian motion. 
To this end, we fix $x \in [L/2, A+L]$ and $u \in [b^{1/4}, b^{3/4}]$, and we proceed to bound from above the following quantity,
\begin{equation}
\label{eq:toBoudnRemInf}
\E_{0, u, b}\Bigl[\bfF^{\lambda}_{b, x-L, L}(\Upsilon_{\! b}^{ u+x}) \one_{\{\inf_{s \in [0, b]} B_{s} < -x\}} \one_{\{\M_{0, b}(\Upsilon_{\! b}^{ u+x})\leq \lambda\}}\Bigr]\;,  
\end{equation} 
which arises when considering the difference $\smash{\bfG^{\lambda}_{\eqref{eq:expGALb4}} -\bfG^{\lambda}_{\eqref{eq:expGALb3}}}$. Since the quantity $\bfF^{\lambda}_{b, x-L, L}(\Upsilon_{\! b}^{ u+x})$ is by definition bounded above by $L$, we can proceed to estimate the following probability 
\begin{equation*}
\P_{0, u, b}\Biggl(\inf_{s \in [0, b]} B_{s} < - x, \; \M_{0, b}(\Upsilon_{\! b}^{ u+x}) \leq \lambda\Biggr)\;,
\end{equation*}
for $x$ and $u$ as specified above. By monotonicity, since $x > L/2$, the probability in the previous display is bounded above by 
\begin{equation*}
\P_{0, u, b}\Biggl(\inf_{s \in [0, b]} B_{s} < - L/2, \; \M_{0, b}(\Upsilon_{\! b}^{ u+x}) \leq \lambda\Biggr) \;.
\end{equation*}
This probability can be bounded by using Lemma~\ref{lm:useCon1}, from which we deduce that is less than a constant times $e^{-c \sqrt{L}}u/b$, for some $c > 0$. 

\textbf{Step 2:} We now show how we can remove the indicator function regarding the quantity $\M_{0, b + 1, b}(\Upsilon_{\! b}^{ u+x})$. Fix $x \in [L/2, A+L]$ and $u \in [b^{1/4}, b^{3/4}]$. Using again the fact that the quantity $\smash{\bfF^{\lambda}_{b, x-L, L}(\Upsilon_{\! b}^{ u+x})}$ is by definition bounded above by $L$, we need to estimate the following probability
\begin{align*}
\P_{0, u, b}\bigl(\M_{0, b + 1, b}(\Upsilon_{\! b}^{ u+x}) > x-(A+L)\bigr) \;.
\end{align*}
Thanks to \ref{hp_K2}, the seed covariance function $\frkK$ is supported in $B(0, 1)$ and so, recalling the definition \eqref{eq:Phi} of the field $\Phi_b$, we have that, for all $y \in \R^d$ with $\abs{y} > e^b$, 
\begin{equation*}
	\Upsilon_{\! b}^{ u+x}(y) = -B_b + \rmZ_{b}(y) -\sqrt{\smash[b]{2d}} b + \frkg_{b, u+x}(y) \;, 
\end{equation*}
where we recall that $\rmZ_{b}$ is the centred Gaussian field on $\R^d$ introduced in Definition~\ref{def:fieldsZ}. In particular, leveraging again on the fact that $\frkK$ is supported in $B(0, 1)$, we have that for all $x$, $y \in \R^d$ with $\abs{x}$, $\abs{y} > e^b$,  
\begin{equation*}
\E\bigl[\rmZ_{b}(x) \rmZ_{b}(y)\bigr] = \int_0^b \frkK(e^{-s} (x-y)) ds \;.
\end{equation*}
In particular, this implies that the scaled field $\bar \rmZ_b(\cdot) \eqdef \rmZ_{b}(e^b \cdot)$ restricted to the annulus $\B_1 \setminus \B_0$ has the same covariance of the martingale approximation at level $b$ of a $\star$-scale invariant field with seed covariance function $\frkK$. Therefore, by rescaling space, the event $\{\M_{0, b + 1, b}(\Upsilon_{\!b, g}) > x-(A+L)\}$ is equivalent to the event that 
\begin{equation*}
\bigl\{\M_{0, 1, 0}(\bar \rmZ_{b} + \frkg_{b, u+x}(e^b \cdot)) -\sqrt{\smash[b]{2d}} b  > B_b +  x - (A+L) \bigr\} \;.
\end{equation*}	
Therefore, on the events $\{\M_{0, 1, 0}(\frkg_{b, u+x}(e^b \cdot)) < b^{1/4}/2\}$ and $\{B_b > b^{1/4}\}$, using the same argument as in the proof of Lemma~\ref{lm:controlX'}, we observe that, there exists a constant $c >0$ such that, for any $x \in [0, A+L]$ and $u \in [b^{1/4}, b^{3/4}]$,
\begin{align*}
\P\bigl(\M_{0, 1, 0}(\bar \rmZ_{b}) -\sqrt{\smash[b]{2d}} b  > b^{1/4}/2 + x - (A+L)\bigr)  \lesssim e^{-c (b^{1/4} - (A+L))} \;.
\end{align*}
On the other hand, to treat the event $\{\M_{0, 1, 0}(\frkg_{b, u+x}(e^b \cdot)) \geq b^{1/4}/2\}$, we can use the fact that the supremum over $\B_1 \setminus \B_0$ of $\frkg_{b, u+x}(e^b \cdot)$ has uniform Gaussian tails, in order to obtain a similar (stronger) bound.  

Finally, the conclusion follows readily by combining the bounds provided in Steps~1~and~2.
\end{proof}

Finally, we need to prove that the function $\smash{\bfF^{\lambda}_{b, x-L, L}}$ in $\smash{\bfG^{\lambda}_{\eqref{eq:expGALb4}}}$ can be replaced by $\bfF^{\lambda}_{k, x-L, L}$ for some $0 \leq k < b$. To this end, for $0 \leq k < b$, we define 
\begin{equation}
\label{eq:expGALb5}
\begin{alignedat}{1}
\bfG^{\lambda}_{\eqref{eq:expGALb5}} & \eqdef \alpha \int_{L/2}^{A+L} e^{\sqrt{\smash[b]{2d}} (x - L)} \int_{b^{1/4}}^{b^{3/4}} \frac{e^{-\f{u^2}{2b}}}{\sqrt{2 \pi b}} (u+x) \\
& \hspace{40mm} \cdot \E_{0, u, b} \Bigl[\bfF^{\lambda}_{k, x-L, L}(\Upsilon_{\! b}^{ u+x}) \one_{\{\M_{0, b}(\Upsilon_{\! b}^{ u+x}) \leq \lambda\}} \Bigr] du dx \;,
\end{alignedat}
\end{equation} 
and we claim the following lemma.
\begin{lemma}
\label{lm:reduceSizeOK}
For any $\lambda > 0$, $\eps > 0$ and $A$, $L \geq 0$, there exist $k \geq 0$ sufficiently large, and $b \geq k$ sufficiently large such that 
\begin{equation*}
\abs{\bfG^{\lambda}_{\eqref{eq:expGALb5}} -\bfG^{\lambda}_{\eqref{eq:expGALb4}}} \leq \eps \;.
\end{equation*}
\end{lemma}
\begin{proof}
For $0 \leq k < b$, we define the following event 
\begin{equation*}
	\rmE_{k, b} \eqdef \bigcup_{j = k}^{b-1} \Biggl\{\sup_{y \in \A_j} \Upsilon_{\! b}^{ u+x}(y)\geq - a \log(j)^2\Biggr\} \;,
\end{equation*}
for some constant $a > 0$.  
By arguing as in the proof of Lemma~\ref{lm:repulsionShape}, and using the fact that the function $\smash{\bfF^{\lambda}_{b, x-L, L}}$ is bounded above by $L$, we note that for any $x \in [L/2, A+L]$, $b \geq 0$ sufficiently large, and $u \in [b^{1/4}, b^{3/4}]$, it holds that 
\begin{equation*}
\E_{0, u, b}\Bigl[\bfF^{\lambda}_{b, x-L, L}(\Upsilon_{\! b}^{ u+x}) \one_{\{\rmE_{k, b}\}} \one_{\{\M_{0, b}(\Upsilon_{\! b}^{ u+x}) \leq \lambda\}} \Bigr] \lesssim L k^{-\f{1}{16}} \f{u}{b} \;.
\end{equation*}
Obviously, the same bound also holds with $\smash{\bfF^{\lambda}_{b, x-L, L}(\Upsilon_{\! b}^{ u+x})}$ replaced by $\smash{\bfF^{\lambda}_{k, x-L, L}(\Upsilon_{\! b}^{ u+x})}$. Therefore, we can focus our attention on the complement of the event $\rmE_{k, b}$. Specifically, we need to estimate the following expectation 
\begin{equation}
\label{eq:redcetoKBoundMes}
\E_{0, u, b}\Bigl[\abs{\bfF^{\lambda}_{b, x-L, L}(\Upsilon_{\! b}^{ u+x}) - \bfF^{\lambda}_{k, x-L, L}(\Upsilon_{\! b}^{ u+x})} \one_{\{\rmE^c_{k, b}\}} \one_{\{\M_{0, b}(\Upsilon_{\! b}^{ u+x}) \leq \lambda\}} \Bigr] \;. 
\end{equation}
By recalling the Definition~\ref{def:ControlVar} of the control variable $\rmK_b$, we can further restrict our attention to the event $\{\rmK_b \leq k\}$, as the bound on the complement of this event can be obtained by using Lemma~\ref{lm:upperBoundKk}. 
We recall that $\Upsilon_{\! b}^{ u+x}(0) = 0$, and so on the complement of the event $\rmE_{k, b}$, it holds that $\smash{\abs{\D^{\lambda}_{b, 0}(\Upsilon_{\! b}^{ u+x})} = \abs{\D^{\lambda}_{k, 0}(\Upsilon_{\! b}^{ u+x})}}$. Therefore, we have that 
\begin{align*}
\abs{\bfF^{\lambda}_{b, x-L, L}(\Upsilon_{\! b}^{ u+x})&  - \bfF^{\lambda}_{k, x-L, L}(\Upsilon_{\! b}^{ u+x})} \\
& \hspace{25mm} \leq L\Biggl(1 - \exp\Biggl(- e^{- \gamma (x - L)} \int_{\B_b \setminus \B_k} \rmF_{\gamma}(y) e^{\gamma \Upsilon_{\! b}^{ u+x}(y)} dy\Biggr)\Biggr) \;.
\end{align*}
Thanks to Lemma~\ref{lm:approxBrownianBridge}, and by choosing $a > 0$ sufficiently large, we observe that there exists a constant $c_1 = c_1(a) > 0$ such that 
\begin{equation*}
\rmE^c_{k, b} \cap \bigl\{\rmK_b \leq k\bigr\} \subseteq \bigcap_{j = k}^{b-1} \bigl\{B_j \geq c_1 (\log j)^2 \bigr\} \cap \bigl\{\rmK_b \leq k\bigr\} \;. 
\end{equation*}
Therefore, recalling once again the Definition~\ref{def:ControlVar} of the control variable $\rmK_b$, on the event $\rmE_{k, b}^c \cap \{\rmK_b \leq k\}$, by using the decomposition \eqref{eq:decoRestrictedShape} of the field $\Upsilon_{\! b}$ and by choosing the constant $a > 0$ sufficiently large, we obtain that there exist constants $c_2$, $c_3 > 0$ such that
\begin{align*}
\int_{\B_b \setminus \B_k} \rmF_{\gamma}(y) e^{\gamma\Upsilon_{\! b}^{ u+x}(y)} dy 
& \leq \sum_{j = k}^{b-1} \int_{\A_j} \rmF_{\gamma}(y)  e^{\gamma (- c_2 (\log j)^2 + \rmZ_{j}(y) - \sqrt{\smash[b]{2d}}j)} dy \\
& \lesssim \sum_{j = k}^{b-1} e^{-\gamma c_3 (\log j)^2} \int_{\B_1 \setminus \B_0} \rmF_{\gamma}(e^j y) \underbrace{e^{\gamma(\rmZ_{j}(e^{j} y) - \sqrt{\smash[b]{2d}} j - \frkd_j) + dj} dy}_{\bar{\mu}_{\gamma, j}(dy)} \;,
\end{align*}
where we recall that $\frkd_j$ is defined as in \eqref{eq:notMainTech1frkd}, and we use the notation introduced in \eqref{eq:RPsiGamma}. Furthermore, since thanks to \ref{hp_K2}, the function $\frkK$ is supported in $B(0, 1)$, we observe that on $\B_1 \setminus \B_0$, the measure $\bar{\mu}_{\gamma, j}$ has the same law as the measure introduced in \eqref{e:norm_super} (with $t$ replaced by $j$), i.e., it is a regularised and normalised supercritical GMC measure.
Hence, applying the previous estimate, we deduce that the quantity in \eqref{eq:redcetoKBoundMes} is bounded above by
\begin{equation*}
\label{eq:Bound1minusExp}
L \E\Biggl[\Biggl(1 - \exp\Biggl(- e^{- \gamma (x - L)} \sum_{j = k}^{\infty} e^{-\gamma c_3 (\log j)^2} \int_{\B_1 \setminus \B_0} \rmF_{\gamma}(e^j y) \bar{\mu}_j(dy)\Biggr)\Biggr) \one_{\{\M_{0, b}(\Upsilon_{\! b}^{ u+x}) \leq \lambda,\, \rmK_b\leq k\}}\Biggr].
\end{equation*}
We observe that the quantity inside the parenthesis in the previous display is independent of $(B_s)_{s \leq b}$. Therefore, using \eqref{eq:Inc1Fin} and arguing as in the proof of Lemma~\ref{lm:upperBoundKk}, for any $k \geq 0$ sufficiently large and for $b \geq 0$ sufficiently large, the quantity in the previous display is bounded from above by a multiple of 
\begin{align*}
& \frac{u}{b} L \E\Biggl[1 - \exp\Biggl(- \sum_{j = k}^{\infty} e^{-\gamma (x - L + c_3(\log j)^2)} \int_{\B_1 \setminus \B_0} \rmF_{\gamma}(e^j y)\bar{\mu}_{\gamma, j}(dy)\Biggr)\Biggr]^{\f12} \\
& \qquad \leq  \frac{u}{b} L \Biggl(\sum_{j = k}^{\infty}  \E\Biggl[1 - \exp\Biggl(- e^{-\gamma \log j} \int_{\B_1 \setminus \B_0} \rmF_{\gamma}(e^j y)\bar{\mu}_{\gamma, j}(dy)\Biggr)\Biggr]\Biggr)^{\f12} \\
& \qquad \lesssim \frac{u}{b} L \Biggl(\sum_{j = k}^{\infty} j^{-\sqrt{\smash[b]{2d}}} \log j\Biggr)^{\f12}\;,
\end{align*} 
where here we used the elementary inequality $\smash{1 -  e^{-\sum_{j = k}^{\infty}u_j} \leq \sum_{j = k}^{\infty} (1 - e^{-u_j})}$ valid for $(u_j)_{j \geq k} \subset \R^{+}_0$, and the last inequality follows from \cite[Proposition~4.2]{Glassy} applied with $\chi = \log j$. 
Strictly speaking, we cannot directly apply \cite[Proposition~4.2]{Glassy}, as the function $\rmF_{\gamma}$ is random. However, by following the proof of that proposition, we obtain the same conclusion in our more general setting by leveraging assumptions \ref{hp:W1} and \ref{hp:W4}.
Thus, the conclusion follows, as the series in the last line of the above display converges to zero as $k \to \infty$.
\end{proof}

We are now ready to prove Lemma~\ref{lm:techMainRed2}.
\begin{proof}[Proof of Lemma~\ref{lm:techMainRed2}]
The claim follows by combining Lemmas~\ref{lm:ExtremeIntBB},~\ref{lm:truncOK},~\ref{lm:removIndOK},~and~\ref{lm:reduceSizeOK}. 
\end{proof}

%%%%%%%%%%%%%%%%%%%%%%%%%%%%%%%%%%%%%%%%%%%%%%
%%%%%%%%%%%%%%%%%%%%%%%%%%%%%%%%%%%%%%%%%%%%%%
\appendix

%%%%%%%%%%%%%%%%%%%%%%%%%%%%%%%%%%%%%%%%%%%%%%
%%%%%%%%%%%%%%%%%%%%%%%%%%%%%%%%%%%%%%%%%%%%%%
\section{Brownian bridge above a curve} 
\label{ap:BBEstimates}
In this appendix, we collect some estimates for the probability of a Brownian bridge staying above a positive/negative curve, and some ``entropic repulsion'' estimates. In particular, we are interested in results analogous to those stated in \cite[Section~4.2]{BiskupLouidor}. We emphasise that in \cite[Section~4.2]{BiskupLouidor}, there are actually estimates for Brownian bridges but are only limited to the case where the underlying curve is symmetric about the midpoint of the lifespan of the Brownian bridge. Hence, their techniques for transferring statements about Brownian motion to statements about the Brownian bridge cannot be applied to our setting. Generally speaking, the proofs of our results are relatively simple adaptations of the proofs of the results for the Brownian motion in \cite[Section~4.2]{BiskupLouidor}.

%%%%%%%%%%%%%%%%%%%%%%%%%%%%%%%%%%%%%%%%%%%%%%
\subsection{Some preliminary results}
We collet in this subsection some simple results that are used several times throughout the remaining part of this appendix. For a continuous function $\zeta: \R^{+}_0 \to \R$, we define the stopping time 
\begin{equation*}
	\tau_{\zeta} \eqdef \inf\bigl\{s \geq 0 \, : \, B_s = \zeta(s)\bigr\} \;.
\end{equation*}
We recall that for $x$, $u \in \R$ and $b > 0$, the law of $(B_s)_{s \in [0, b]}$ under $\P_{x, u, b}$ is that of a Brownian bridge from $x$ to $u$ in time $b$. We start with some basic facts about Brownian bridges.
\begin{lemma}
\label{lm:basicEstBB}
For $b > 0$ and $x$, $u > 0$, it holds that
\begin{equation}
\label{eq:boundsBBstaysPosExact}
\P_{x, u, b}\bigl(\tau_{0} > b\bigr) = 1 - e^{-\frac{2xu}{b}} \;,
\end{equation}
and in particular, we have the following upper and lower bounds,
\begin{equation}
\label{eq:boundsBBstaysPos}
\f{2xu}{b}\Biggl(1-\f{x u}{b}\Biggr) \leq \P_{x, u, b}\bigl(\tau_{0} > b\bigr) \leq \f{2xu}{b} \;. 
\end{equation}
Furthermore, it holds that
\begin{equation}
\label{eq:boundsBBstaysPosDens}
\P_{x, u, b}\bigl(\tau_0 \in ds, \; \tau_0  \leq b \bigr) = \frac{b x e^{-\frac{((b-s)x + su)^2}{2 b s (b-s)}}}{s^{3/2} \sqrt{2 \pi b (b-s)}} \one_{\{s \in [0,b]\}}ds \;.
\end{equation}
\end{lemma}
\begin{proof}
The equality \eqref{eq:boundsBBstaysPosExact} is standard and follows by the reflection principle. The bounds in \eqref{eq:boundsBBstaysPos} follows by the elementary inequalities $x-x^2/2 \leq 1 -e^{-x} \leq x$ which are valid for $x \geq 0$. To get \eqref{eq:boundsBBstaysPosDens}, it suffices to note that, for $s \in [0, b]$,
\begin{equation*}
\P_{x, u, b}\bigl(\tau_0 > s, \; \tau_0 \leq b \bigr) = \int_{0}^{\infty} \P_{x, y, s}\bigl(\tau_0 > s\bigr) \P_{x, u, b}\bigl(B_s \in dy\bigr)\;.
\end{equation*}
Since the density $\P_{x, u, b}\bigl(B_s \in dy\bigr)$ is the same as that of a normal random variable with mean $x + s(u-x)/b$ and variance $s(b-s)/b$ and by using the exact identity $\eqref{eq:boundsBBstaysPosExact}$, the result follows by differentiation. 
\end{proof}

\begin{lemma}
\label{lm:infSupBoundBB}
For $b >0 $, consider $M_b \eqdef \inf_{r \in [0, b]} B_r$ and $T_b \eqdef \sup\{r \in [0, b] \, : \, B_r = M_b\}$. Then, for $u > 0$, it holds that 
\begin{equation}
\label{eq:dendInf}
\P_{0, u, b}(T_b \in ds, \: M_b \in dz) = \sqrt{\f{2}{\pi}} \frac{\sqrt{b} (-z) (u-z) e^{-\f{(b z- u s)^2}{2b s(b-s)}}}{(s(b-s))^{3/2}} \one_{\{s \in [0, b]\}} \one_{\{z < 0\}} ds dz \;.
\end{equation}
\end{lemma}
\begin{proof}
The proof is inspired by the proof of \cite[Lemma~A.1]{BiskupLouidor}.
Using the path continuity and the strong Markov property of the Brownian bridge, along with the exact identity \eqref{eq:boundsBBstaysPosExact}, we note that for any $s \in [0, b]$ and $z \leq 0$, it holds that 
\begin{align*}
\P_{0, u, b}\bigl(T_b \leq s, \; M_b \leq z\bigr) 
& = \lim_{\eps \to 0} \sum_{k \in \N_0} \int_0^s \P_{z-k\eps, u, b-r}\bigl(\tau_{z-(k+1)\eps} > b-r\bigr) \P_{0, u, b}\bigl(\tau_{z-k\eps} \in dr\bigr) \\
& = \lim_{\eps \to 0} \sum_{k \in \N_0} \int_0^s \P_{\eps, u -(z-(k+1)\eps), b-r}\bigl(\tau_{0} > b-r\bigr) \P_{0, u, b}\bigl(\tau_{z-k\eps} \in dr\bigr) \\
& = 2 \int_{-\infty}^{z} \int_{0}^{s} \frac{u-w}{b-r} \P_{0, u, b}\bigl(\tau_{w} \in dr\bigr) dw\;,
\end{align*}
and so the claim follows by using \eqref{eq:boundsBBstaysPosDens} and by differentiation.
\end{proof}

\begin{lemma}
Let $b \geq 1$ and $x \geq 1$. For all $u \in [0, b^{3/4}]$ and $s \in [1, b]$, it holds that
\begin{equation}
\label{eq:boundUse}
\P_{x, u, b}\bigl(\tau_{0} > s \bigr) \leq \frac{2 x^2}{s} + \frac{8 x}{s^{1/4}} \;.
\end{equation}	
\end{lemma}
\begin{proof}
Fix $b \geq 1$, $x \geq 1$, and $u \in [0, b^{3/4}]$. As we have observed in the proof of Lemma~\ref{lm:basicEstBB}, for any $s \in [1, b]$, by conditioning on the value of the Brownian bridge at times $s$, we can write
\begin{equation*}
\P_{x, u, b}\bigl(\tau_{0} > s \bigr) = \int_{0}^{\infty} \P_{x, y, s}\bigl(\tau_{0} > s \bigr) \P_{x, u, b}\bigl(B_s \in d y \bigr) \;,
\end{equation*}
where, as one can easily check, the density $\P_{x, u, b}\bigl(B_s \in dy\bigr)$ is that of a normal random variable with mean $x + s(u-x)/b$ and variance $s(b-s)/b$, i.e., 
\begin{equation*}
\P_{x, u, b}\bigl(B_s \in d y \bigr) = \sqrt{\f{b}{2 \pi s (b - s)}} \exp\Biggl(-\f{b}{2 s (b - s)}\Biggl(y - \Biggl(x + s\f{u-x}{b}\Biggr)\Biggr)^2\Biggr) dy \;.
\end{equation*}
Therefore, for $s \in [1, b]$, by letting
\begin{equation*}
r = r(s, u, b, x) \eqdef x+ \frac{s(u-x)}{b} + \sqrt{\frac{s(b-s)}{b}}\log(e+s)\;,
\end{equation*}
we have that 
\begin{equation}
\label{eq:intSplitPosBB}
\P_{x, u, b}\bigl(\tau_{0} > s \bigr) 
\leq \int_0^{r} \Bigl(1-e^{-\frac{2xy}{s}}\Bigr)\P_{x, u, b}\bigl(B_s \in d y \bigr) + \int_{r}^{\infty}\P_{x, u, b}\bigl(B_s \in d y \bigr) \;,
\end{equation}
where we simply bounded the integrand by one on the second half of the integration interval. By using the elementary inequality $1-e^{-x} < x$ valid for $x \geq 0$, we have the following bound for the first integral on the right-hand side of \eqref{eq:intSplitPosBB},
\begin{equation*}
\int_0^{r} \Bigl(1-e^{-\frac{2xy}{s}}\Bigr)\P_{x, u, b}\bigl(B_s \in d y \bigr)  \leq \frac{2 x r}{s} \leq \frac{2 x^2}{s} + \frac{6 x}{s^{1/4}}  \;.
\end{equation*}
where here we used that $u <b^{3/4}$ to bound $u/b$ by $1/s^{1/4}$, and $s>1$ to bound $\log(e+s)$ by $2 s^{1/4}$.
Regarding the second integral on the right-hand side of \eqref{eq:intSplitPosBB}, we note that it is equivalent to
\begin{equation*}
\P\Biggl(\CN\Biggl(x+ \frac{s(u-x)}{b}, \frac{s(b-s)}{b}\Biggr) > r \Biggr) = \P\bigl(\CN(0,1) > \log(e+s)\bigr) \leq \frac{2}{s^{1/4}} \;.
\end{equation*}
Therefore, thanks to the fact that $x \geq 1$, the claim follows by collecting the previous estimates. 
\end{proof}

%%%%%%%%%%%%%%%%%%%%%%%%%%%%%%%%%%%%%%%%%%%%%%
\subsection{Brownian bridge above a positive curve}
\begin{proposition}
\label{pr:BoundBBAbovePos}
For $\iota \in (0, 1/8)$, let $\zeta: \R^{+}_0 \to \R^{+}_0$ be a continuous and non-decreasing function such that $\zeta(0)> 1$ and $\zeta(s) = o(s^{\iota})$ as $s \to \infty$, and let $x > \zeta(0)$. Then, for all $b \geq 1$ and $u \in [b^{\iota}, b^{3/4}]$, it holds that 
\begin{equation*}
\P_{x, u, b}\bigl(\tau_{\zeta} > b\bigr) \geq (1-\delta) \frac{2xu}{b} \;,
\end{equation*}
where 
\begin{equation*}
	\delta = 2 \Biggl(\frac{xu}{2 b} + \frac{\zeta(b)}{u} + \frac{\rho(x)}{x}\Biggr)	 \;,
\end{equation*}
where the function $\rho:\R^{+}_0 \to \R^{+}_0$ is defined as follows
\begin{equation}
\label{eq:defRho}
	\rho(x) = \zeta(x^4) + 2 x^2 \int_{x^4}^{\infty} \frac{\zeta(s)}{s^2} ds + 2 x \int_{x^4}^{\infty} \frac{\zeta(s)}{s^{5/4}} ds \;.
\end{equation}
\end{proposition}
\begin{proof}
We start by noticing that, thanks to \eqref{eq:boundsBBstaysPos}, it holds that 
\begin{align}
\P_{x, u, b}\bigl(\tau_{\zeta} > b\bigr) 
& = \P_{x, u, b}\bigl(\tau_{0} > b\bigr) - \P_{x, u, b}\bigl(\tau_{\zeta} \leq b < \tau_{0}\bigr) \nonumber \\
& \geq \f{2xu}{b}\Biggl(1-\f{x u}{b}\Biggr)  - \P_{x, u, b}\bigl(\tau_{\zeta} \leq b < \tau_{0}\bigr) \;, \label{eq:firstBoundPos}
\end{align} 
and so we can just focus on finding a suitable upper bound for the probability in \eqref{eq:firstBoundPos}. Using the strong Markov property of the Brownian bridge at the stopping time $\tau_{\zeta}$ and again \eqref{eq:boundsBBstaysPos}, we have that 
\begin{align*}
\P_{x, u, b} & \bigl(\tau_{\zeta} \leq b < \tau_{0}\bigr) \\
& \leq \int_0^{b/2} \P_{x, u, b}\bigl(\tau_{\zeta} \in ds \bigr) \P_{\zeta(s), u, b-s}\bigl(\tau_{0} > b-s\bigr) + \int_{b/2}^{b} \P_{x, u, b}\bigl(\tau_{\zeta} \in ds \bigr) \P_{x, \zeta(s), s}\bigl(\tau_{0} > s\bigr) \\
& \leq \int_0^{b/2} \P_{x, u, b}\bigl(\tau_{\zeta} \in ds \bigr) \P_{\zeta(s), u, b-s}\bigl(\tau_{0} > b-s\bigr) + \int_{b/2}^{b} \P_{x, u, b}\bigl(\tau_{\zeta} \in ds \bigr) \frac{2 x \zeta(s)}{s} \\
& \leq \frac{4 u}{b} \int_0^{b/2} \zeta(s) \P_{x, u, b}\bigl(\tau_{\zeta} \in ds \bigr) + \frac{4x\zeta(b)}{b} \;.
\end{align*}
We now focus on the integral in the last line of the above display. By integrating by parts, we note that 
\begin{align*}
\int_0^{b/2} \zeta(s) \P_{x, u, b}\bigl(\tau_{\zeta} \in ds \bigr)
& \leq \zeta(0) + \int_0^{b/2} \zeta'(s) \P_{x, u, b}\bigl(\tau_{\zeta} \geq s \bigr) ds  \\
& \leq \zeta(x^4) + \int_{x^4}^{b/2} \zeta'(s) \P_{x, u, b}\bigl(\tau_{0} \geq s \bigr) ds  \,.
\end{align*}
Now, using the bound \eqref{eq:boundUse} and the fact that $\zeta(s) = o(s^{1/4})$ as $s \to \infty$ (since $\iota \in (0, 1/8)$), we obtain that
\begin{align*}
\int_{x^4}^{b/2} \zeta'(s) \P_{x, u, b}\bigl(\tau_{0} \geq s \bigr) ds  & 
\leq x \int_{x^4}^{\infty} \zeta'(s) \Biggl(\frac{2x}{s} + \frac{8}{s^{1/4}}\Biggr) ds \\
& \leq 2 x\int_{x^4}^{\infty} \zeta(s)\Biggl(\frac{x}{s^2}+\frac{1}{s^{5/4}}\Biggr) ds \;.
\end{align*}
Therefore, the conclusion follows by combining all the previous estimates.
\end{proof}

%%%%%%%%%%%%%%%%%%%%%%%%%%%%%%%%%%%%%%%%%%%%%%
\subsection{Brownian bridge above a negative curve}
\begin{proposition}
\label{pr:BoundBBAboveNeg}
For $\iota \in (0, 1/8)$, let $\zeta: \R^{+}_0 \to \R^{+}_0$ be a continuous and non-decreasing function such that $\zeta(0)> 1$ and $\zeta(s) = o(s^{\iota})$ as $s \to \infty$, and let $x > \zeta(0)$.  Then, for all $b \geq 1$ sufficiently large and $u \in [b^{\iota}, b^{3/4}]$, it holds that   
\begin{equation*}
\P_{x, u, b}\bigl(\tau_{-\zeta} > b\bigr) \leq (1+ \tilde \delta) \frac{2xu}{b} \;,
\end{equation*}
where 
\begin{equation}
\label{eq:ErrorTermNeg}
\tilde \delta = 4 \Biggl(\frac{x}{u} + \frac{4 \zeta(b)^2}{x u} + \frac{4 \tilde\rho(x)}{x}\Biggr)  \;.
\end{equation}
where, recalling the definition \eqref{eq:defRho} of the function $\rho$, the function $\tilde \rho:\R^{+}_0 \to \R^{+}_0$ is defined as follows
\begin{equation}
\label{eq:defRhoTilde}
\tilde \rho(x) = \rho(x) + 2 \frac{\zeta(x^2)^2}{x} + \int_{x^2}^{\infty} \frac{\zeta(s)^2}{s^{3/2}} ds \;.
\end{equation}
\end{proposition}

The proof of Proposition~\ref{pr:BoundBBAboveNeg} is based on the following lemma.
\begin{lemma}
\label{lm:badProbMax}
For $\iota \in (0, 1/8)$, let $\zeta: \R^{+}_0 \to \R^{+}_0$ be a continuously differentiable and non-decreasing function such that $\zeta(0)> 1$ and $\zeta(s) = o(s^{\iota})$ as $s \to \infty$. Then, for all $b \geq 1$ sufficiently large and $u \in [b^{\iota}, b^{3/4}]$, it holds that  
\begin{equation*}
\P_{0, u, b}\bigl(\tau_{-\zeta} > b\bigr) \leq 16 \frac{u}{b}\Biggl(\zeta(0) + \int_{0}^{b} \frac{\zeta(z) \zeta'(z)}{\sqrt{z}} dz + \frac{\zeta(b)^2}{u}\Biggr)  \;.
\end{equation*}
\end{lemma}
\begin{proof}
Recalling the notation introduced in the statement of Lemma~\ref{lm:infSupBoundBB}, we begin by noting that 
\begin{equation}
\label{eq:badFirstSplit}
\P_{0, u, b}\bigl(\tau_{-\zeta} > b\bigr) \leq 	\P_{0, u, b}\bigl(\tau_{-\zeta(0)} > b\bigr) + \P_{0, u, b}\bigl(-\zeta(T_b) < M_b < -\zeta(0)\bigr) \;.
\end{equation}
The first probability on the right-hand side can be bounded by using \eqref{eq:boundsBBstaysPos} as follows
\begin{equation*}
\P_{0, u, b}\bigl(\tau_{-\zeta(0)} > b\bigr) \leq 2\frac{\zeta(0)(\zeta(0) + u)}{b} \leq 4 \frac{u}{b} \zeta(0)\;,
\end{equation*}
where here we assumed that $b \geq  1$ is taken large enough so that $\zeta(0)/u < 1$ uniformly over all $u \in [b^{\iota}, b^{3/4}]$.
Concerning the second probability on the right-hand side of \eqref{eq:badFirstSplit}, using \eqref{eq:dendInf} of Lemma~\ref{lm:infSupBoundBB}, we have that
\begin{align}
\P_{0, u, b}\bigl(-\zeta(T_b) < M_b < -\zeta(0)\bigr) 
& = \sqrt{\f{2}{\pi}} \int_0^b \int_{-\zeta(s)}^{-\zeta(0)} \frac{\sqrt{b} (-z) (u-z) e^{-\f{(-bz + u s)^2}{2bs(b-s)}}}{(s(b-s))^{3/2}} dz ds \nonumber\\
& = \sqrt{\f{2}{\pi}} \int_0^b \int_{\zeta(0)}^{\zeta(s)} \frac{\sqrt{b} z (u+z) e^{-\f{(bz + u s)^2}{2bs(b-s)}}}{(s(b-s))^{3/2}} dz ds \nonumber\\
& = \sqrt{\f{2}{\pi}} \int_{\zeta(0)}^{\zeta(b)} \int_{\zeta^{-1}(z)}^b \frac{\sqrt{b} z (u+z) e^{-\f{(bz + u s)^2}{2bs(b-s)}}}{(s(b-s))^{3/2}} ds dz \;, \label{eq:innerInt}
\end{align}
where here, we assumed that $\zeta$ is invertible. This is not restrictive since the general case can be obtained by a standard approximation argument. 
We now bound the inner integral in \eqref{eq:innerInt} by splitting the interval of integration around $\zeta^{-1}(z) \vee b/2$. We note that, 
\begin{align}
\int_{\zeta^{-1}(z)}^{\zeta^{-1}(z) \vee b/2} \frac{\sqrt{b} z (u+z) e^{-\f{(bz + u s)^2}{2bs(b-s)}}}{(s(b-s))^{3/2}} ds 
& \leq  \frac{2^{3/2} (u+z)z}{b} \int_{\zeta^{-1}(z)}^{\infty} \frac{e^{-\f{z^2}{2s}}}{s^{3/2}} ds \nonumber \\ 
& \leq \frac{8 (u+z)}{b} \frac{z}{\sqrt{\zeta^{-1}(z)}} \nonumber \\
& \leq \frac{16 u}{b} \frac{z}{\sqrt{\zeta^{-1}(z)}} \;, \label{eq:case1Zetab}
\end{align}
where in order to get the last inequality we used that fact that for $b \geq 1$ sufficiently large, it holds that $z/u < 1$, uniformly over $u \in [b^{\iota}, b^{3/4}]$ and $z \in [\zeta(0), \zeta(b)]$. Similarly, again by taking $b \geq 1$ sufficiently large, we have that  
\begin{align}
\int_{\zeta^{-1}(z) \vee b/2}^{b} \frac{\sqrt{b} z (u+z) e^{-\f{(bz + u s)^2}{2bs(b-s)}}}{(s(b-s))^{3/2}} ds 
&\leq \int_{b/2}^{b} \frac{\sqrt{b} z (u+z) e^{-\f{(bz + u s)^2}{2bs(b-s)}}}{(s(b-s))^{3/2}} ds  \nonumber \\
& \leq \frac{2^{3/2} (u+z) z}{b} \int_{0}^{\infty} \frac{e^{-\f{(z+u)^2}{16 s}}}{s^{3/2}} ds \nonumber \\
& \leq 2^{7/2} \sqrt{\pi} \frac{z}{b} \;.\label{eq:case2Zetab}
\end{align} 

Therefore, plugging the estimates \eqref{eq:case1Zetab} and \eqref{eq:case2Zetab} into \eqref{eq:innerInt}, for $b \geq 1$ sufficiently large, it holds that
\begin{align*}
\P_{0, u, b}\bigl(-\zeta(T_b) < M_b < -\zeta(0)\bigr) 
& \leq \sqrt{\frac{2}{\pi}} \frac{u}{b} \int_{\zeta(0)}^{\zeta(b)} \Biggl(\frac{16 z}{\sqrt{\zeta^{-1}(z)}} + 2^{7/2} \sqrt{\pi} \frac{z}{u}\Biggr)dz  \\
& \leq 16 \frac{u}{b}\Biggl(\int_{0}^{b} \frac{\zeta(z) \zeta'(z)}{\sqrt{z}} dz + \frac{\zeta(b)^2}{u}\Biggr)\;.
\end{align*}
where to go from the first to the second line, we also performed the change of variables $z \mapsto \zeta^{-1}(z)$.
\end{proof}

\begin{proof}[Proof of Proposition~\ref{pr:BoundBBAboveNeg}]
We begin by noting that, thanks to \eqref{eq:boundsBBstaysPos}, it holds that  
\begin{equation*}
	\P_{x, u, b}\bigl(\tau_{-\zeta} > b\bigr) \leq \P_{x,u,b}\bigl(\tau_0 > b\bigr) + \P_{x, u, b}\bigl(\tau_0 \leq b < \tau_{-\zeta}\bigr) \leq \frac{2 x u}{b} + \P_{x, u, b}\bigl(\tau_0 \leq b < \tau_{-\zeta}\bigr) \;.
\end{equation*}
Therefore, our task is now to find a suitable upper bound for the probability on the right-hand side of the above display. To this end, using again \eqref{eq:boundsBBstaysPos}, we note that 
\begin{align*}
\P_{x, u, b}\bigl(\tau_0 \leq b < \tau_{-\zeta}\bigr) 
& \leq \P_{x, u, b}\bigl(\tau_0 \leq b/2, \; \tau_{-\zeta} > b\bigr) + \P_{x, 0, b/2}\bigl(\tau_{-\zeta(0)}>b/2\bigr) \\
& \leq \P_{x, u, b}\bigl(\tau_0 \leq b/2, \; \tau_{-\zeta} > b\bigr) + \frac{4 (x+\zeta(0))\zeta(0)}{b} \\
& \leq \P_{x, u, b}\bigl(\tau_0 \leq b/2, \; \tau_{-\zeta} > b\bigr) + \frac{8 x^2}{b} \;,
\end{align*}
where to obtain the bound in the first line, we observed that the probability that the first time the Brownian bridge from $x$ to $u$ in time $b$ hits $0$ is after time $b/2$ is bounded above by the probability that the Brownian bridge from $x$ to $0$ in time $b/2$ stays above height $-\zeta(0)$. 
Therefore, from now on, we can just focus on the probability in the last line of the above display. Using the strong Markov property of the Brownian bridge, we note that 
\begin{align}
\label{eq:decomStrongMarkovNeg}
\P_{x, u, b}\bigl(\tau_0 \leq b/2, \; \tau_{-\zeta} > b\bigr) = \int_{0}^{b/2} \P_{0, u, b-s}(\tau_{-\zeta(s+\cdot)} > b-s) \P_{x, u, b}\bigl(\tau_0 \in ds) \;.
\end{align}
Now, thanks to Lemma~\ref{lm:badProbMax}, we have that, for all $s \in [0, b/2]$, it holds that 
\begin{equation*}
\P_{0, u, b-s}\bigl(\tau_{-\zeta(s+\cdot)} > b-s\bigr) \leq 32 \frac{u}{b} \Biggl(\zeta(s) + \int_{0}^{b} \frac{\zeta(z+s) \zeta'(z+s)}{\sqrt{z}} dz + \frac{\zeta(b)^2}{u}\Biggr)\;,
\end{equation*}
and so plugging this estimate into the right-hand side of \eqref{eq:decomStrongMarkovNeg}, we note that  
\begin{align}
\P_{x, u, b}\bigl(& \tau_0 \leq b/2, \; \tau_{-\zeta} > b\bigr) \leq 32 \frac{u}{b}\Biggl(\frac{\zeta(b)^2}{u}  \nonumber \\
& +\int_{0}^{b/2} \zeta(s) \P_{x, u, b}\bigl(\tau_0 \in ds) + \int_{0}^{b/2}\int_{0}^{b} \frac{\zeta(z+s) \zeta'(z+s)}{\sqrt{z}} dz \P_{x, u, b}\bigl(\tau_0 \in ds)\Biggr) \;.\label{eq:decomStrongMarkovNegExtra}
\end{align} 
Regarding the first integral in \eqref{eq:decomStrongMarkovNegExtra}, we note that by proceeding in the same exact way as in the proof of Proposition~\ref{pr:BoundBBAbovePos}, we get that
\begin{equation*}
\int_{0}^{b/2} \zeta(s) \P_{x, u, b}\bigl(\tau_0 \in ds)  \leq \rho(x) \;,
\end{equation*} 
where we recall that the function $\rho$ is defined in \eqref{eq:defRho}. Regarding the second integral in \eqref{eq:decomStrongMarkovNegExtra}, by using \eqref{eq:boundsBBstaysPosDens}, we note that 
\begin{align*}
\int_{0}^{b} \int_{0}^{b/2} \frac{\zeta(z+s) \zeta'(z+s)}{\sqrt{z}} \frac{b xe^{-\frac{(b x+s (u-x))^2}{2 b s (b-s)}}}{\sqrt{2 \pi} s \sqrt{b s (b-s)}}ds dz 
& \leq \frac{1}{\sqrt{\pi}} \int_{0}^{\infty} \int_{0}^{\infty} \frac{\zeta(z+s) \zeta'(z+s)}{\sqrt{z}} \frac{xe^{-\frac{x^2}{2s}}}{s^{3/2}}ds dz \\
& \leq 2 \frac{\zeta(x^2)^2}{x} + \int_{x^2}^{\infty} \frac{\zeta(s)^2}{s^{3/2}} ds \;.
\end{align*}
where to obtain the last inequality, it is sufficient to follow the calculations in the proof of \cite[Lemma~A.8]{BiskupLouidor}. Hence, the claim follows by gathering toghther all the previous estimates.
\end{proof}

Let $a$, $k > 1$, and consider the function $\zeta_{a, k}:\R^{+}_0 \to \R^{+}_0$ given by
\begin{equation}
\label{eq:zetaAK}
\zeta_{a, k}(s) \eqdef a\bigl(1+\bigr[\log(1+k+s)\bigr]^2\bigl) \;. 
\end{equation}
\begin{lemma}
\label{lm:techRhoAK}
Let $a$, $k > 1$ and consider the functions $\rho_{a, k}$, $\tilde \rho_{a, k}:[0, \infty] \to \R^{+}_0$ defined in \eqref{eq:defRho} and \eqref{eq:defRhoTilde} with $\zeta = \zeta_{a, k}$ as in \eqref{eq:zetaAK}. Then, there exist constants $c_1$, $c_2$, $\tilde c_1$, $\tilde c_2$, $\tilde c_3 >0$, depending on $a$ but not on $k$, such that  
\begin{equation*}
\rho_{a, k}(x) \leq c_1 \zeta_{a, k}(0)  + c_2 \bigl[\log(e+x^4)\bigr]^2 \;, \qquad \forall \, x \geq 1\;,
\end{equation*}
and 
\begin{equation*}
\tilde \rho_{a, k}(x) \leq \tilde c_1 \zeta_{a, k}(0) + \tilde c_2 \frac{\zeta_{a, k}(0)^2}{x}  + \tilde c_3 \bigl[\log(e+x^4)\bigr]^4\;, \qquad \forall \,  x \geq 1\;.
\end{equation*}
\end{lemma}
\begin{proof}
For $a$, $k>1$, we have that
\begin{equation*}
\zeta'_{a, k}(s) = \frac{2a\log(1+ k+s)}{1+k+s} \;, \qquad \forall s \geq 0 \;.
\end{equation*}
and so using the elementary inequality $\log(1+s)/(1+s) \leq e \log(e+s)/(e+s)$ and the fact that the expression on the right decreases for $s \geq 0$, we get that 
\begin{align*}
\zeta_{a, k}(s) = \int_0	^{k+s} \zeta'_{a, 0}(r) dr + a  & = 2a\int_0^{k} \frac{\log(1+r)}{1+r} dr + 2a\int_{k}^{k+s} \frac{\log(1+r)}{1+r} dr + a\\
& \leq \zeta_{a, k}(0) + 2 a e \int_0^{s} \frac{\log(e+r)}{e+r} dr  \\
& \leq \zeta_{a, k}(0) + e a \bigl[\log(e+s)\bigr]^2 \;.
\end{align*}
By using the above bound on $\zeta_{a, k}$ and recalling the exact expressions of the functions $\rho_{a, k}$ and $\tilde \rho_{a,k}$, the claims follow by some elementary calculations.
\end{proof}

%%%%%%%%%%%%%%%%%%%%%%%%%%%%%%%%%%%%%%%%%%%%%%
\subsection{Entropic repulsion}
\begin{lemma}
\label{lm:entrBasic}
For $a$, $k > 1$, let $\zeta_{a, k}:\R^{+}_0 \to \R^{+}_0$ be the function defined in \eqref{eq:zetaAK}. There exists a constant $c = c(a) > 0$ such that for all $\iota \in (0, 1/8)$, $b \geq 1$ sufficiently large, $u \in [b^{\iota}, b^{3/4}]$, all $s \in (0, b/2]$, and all $x \geq \zeta_{a, k}(s)$, 
\begin{equation*}
\P_{0, u, b}\Biggl(\inf_{r \in [0, b]}\bigl(B_r + \zeta_{a, k}(r)\bigr) > 0, \; B_s \leq x\Biggr) \leq c \frac{u  x^2}{b \sqrt{s}} \;. 
\end{equation*}
\end{lemma}
 
Intuitively, Lemma~\ref{lm:entrBasic} states that, when conditioned to stay above a negative curve, the Brownian motion is repelled upward.

\begin{proof}[Proof of Lemma~\ref{lm:entrBasic}]
For $a$, $k > 1$ and $s \in (0, b/2]$, we let
\begin{equation*}
A_s \eqdef \Biggl\{\inf_{r \in [s, b]}\bigl(B_r + \zeta_{a, k}(r)\bigr) > 0\Biggr\} \;,
\end{equation*}
so that the event in the lemma statement can be written as $A_0 \cap \{B_s \leq x\}$. Using the Markov property of the Brownian bridge, we note that 
\begin{equation}
\label{eq:boundEntr101}
\P_{0, u, b}\bigl(A_0,\; B_s \leq x\bigr) \leq \E_{0, u, b}\Bigl[\one_{\{-\zeta_{a, k}(s) < B_s\leq x\}} \P_{0, u, b}\bigl(A_s \big| \sigma(B_s)\bigr)\Bigr] \;.
\end{equation}
On the event $\{B_s \in (-\zeta_{a, k}(s), x]\}$, the above conditional probability is maximised when $\{B_s = x\}$. Now, since $s \in (0, b/2]$ and $x \geq \zeta_{a, k}(s)$, Proposition~\ref{pr:BoundBBAboveNeg} and Lemma~\ref{lm:techRhoAK} ensure that there exists a constant $c_1 = c_1(a) >0$ such that for all $b \geq 1$ sufficiently large and $u \in [b^{\iota}, b^{3/4}]$, it holds that
\begin{equation}
\label{eq:boundEntr11}
\P_{0, u, b}\bigl(A_s \big|B_s = x\bigr) = \P_{x, u, b-s}\Biggl(\inf_{r \in [0, b-s]}\bigl(B_r + \zeta_{a, k+s}(r)\bigr) > 0\Biggr) \leq c_1 \frac{x u}{b} \;.
\end{equation}
To conclude, it remains to estimate the probability of the event in the indicator function on the right-hand side of \eqref{eq:boundEntr101}. Since $x \geq \zeta_{a, k}(s)$, we observe that there exists a constant $c_2 > 0$ such that 
\begin{equation*}
\P_{0, u, b} \bigl(-\zeta_{a, k}(s) < B_s\leq x\bigr) \leq \P_{0, u, b} \bigl(\abs{B_s} \leq x\bigr) \leq c_2 \frac{x}{\sqrt{s}}\;.
\end{equation*}
Finally, putting everything together, we showed that there exist a constant $c_3 = c_3(a) > 0$ such that 
\begin{equation*}
\P_{0, u, b}\bigl(A_0,\; B_s \leq x\bigr) \leq c_3 \frac{u x^2}{b \sqrt{s}} \;,
\end{equation*}
and so the claim follows.
\end{proof}

\begin{proposition}
\label{pr:BoundBBEntropic}
For $a$, $k > 1$, let $\zeta_{a, k}:\R^{+}_0 \to \R^{+}_0$ be the function defined in \eqref{eq:zetaAK}. There exist constants $c_1 = c_1(a)$, $c_2 = c_2(a) > 0$ such that for all $\iota \in (0, 1/8)$, $b \geq 1$ sufficiently large, $u \in [b^{\iota}, b^{3/4}]$, and all $s \in (c_2, b/2]$, 
\begin{equation*}
	\P_{0, u, b}\Biggl(\inf_{r \in [0, b]}\bigl(B_r + \zeta_{a, k}(r)\bigr) > 0, \;\inf_{r \in [s, b]}\bigl(B_r - \zeta_{a, k}(r)\bigr) < 0\Biggr) \leq c_1 \frac{u}{b} \frac{(k+s)^{7/16}}{\sqrt s}  \;.
\end{equation*}
\end{proposition}
\begin{proof}
For $a$, $k> 1$ and $s \geq 0$, we define the events 
\begin{equation*}
A^{\pm}_s \eqdef \Biggl\{\inf_{r \in [s, b]}\bigl(B_r \pm \zeta_{a, k}(r)\bigr) > 0\Biggr\} \;,	
\end{equation*}
so that the probability in the statement can be written as $\P_{0, u, b}(A^{+}_0 \setminus A^{-}_s)$. For $x \geq \zeta_{a, k}(s)$, we have that 
\begin{equation}
\label{eq:startEntropic}
\frac{\P_{0, u, b}(A^{+}_0 \setminus A^{-}_s)}{\P_{0, u, b}(A_0^{+})}  = 1-\frac{\P_{0, u, b}(A_0^{+} \cap A_{s}^{-})}{\P_{0, u, b}(A_0^{+})} \leq 1-\frac{\P_{0, u, b}(A_0^{+} \cap A_{s}^{-} \cap \{B_s \geq x\})}{\P_{0, u, b}(A_0^{+})} \;.
\end{equation}
Our goal now is to find a suitable lower bound for the probability in the numerator on the right hand-side of \eqref{eq:startEntropic}. To this end, we define the event 
\begin{equation*}
A_{0, s}^{+} \eqdef \Biggl\{\inf_{r \in [0, s]} \bigl(B_r + \zeta_{a, k}(r)\bigr) >0\Biggr\}\;,
\end{equation*}
and we note that 
\begin{equation*}
\P_{0, u, b}\bigl(A_{0}^{+} \cap A_{s}^{-} \cap \{B_s \geq x\}\bigr) = \E_{0, u, b}\Bigl[\one_{\{A_{0, s}^{+} \cap \{B_s \geq x\}\}} \P_{0, u, b}\bigl(A_{s}^{-} \, \big| \, \sigma((B_r)_{r \leq s})\bigr)\Bigr] \;.
\end{equation*}
For any $s \in (0, b/2]$, on the event $\{B_s \geq x\}$, thanks to Proposition~\ref{pr:BoundBBAbovePos} and Lemma~\ref{lm:techRhoAK}, we have that there exists a constant $c_1 = c_1(a) >0$ such that
\begin{equation*}
\P_{0, u, b}\bigl({A_{s}^{-}} \, \big | \, \sigma((B_r)_{r \leq s})\bigr) \geq \frac{2 B_s u}{b-s}\Biggl(1-c_1 \Biggl(\frac{\zeta_{a, k+s}(0)}{x} + \frac{[\log(e+x^4)]^2}{x} + o_b(1)\Biggr)\Biggr)\;.
\end{equation*}
Similarly, for any $s \in (0, b/2]$, on the event $\{B_s \geq x\}$, thanks to Proposition~\ref{pr:BoundBBAboveNeg} and Lemma~\ref{lm:techRhoAK}, we have that there exists a constant $c_2 = c_2(a) >0$ such that 
\begin{equation*}
\P_{0, u, b}\bigl({A_{s}^{+}} \, \big | \, \sigma((B_r)_{r \leq s})\bigr) \leq \frac{2 B_s u}{b-s}\Biggl(1+c_2 \Biggl(\frac{\zeta_{a, k+s}(0)}{x} + \frac{\zeta_{a, k+s}(0)^2}{x^2}+ \frac{[\log(e+x^4)]^4}{x} + o_b(1)\Biggr)\Biggr)\;.
\end{equation*}
Therefore, denoting by $\delta$ and $\tilde \delta$ the two error terms in the previous two displays, we have that 
\begin{equation*}
\P_{0, u, b}\bigl(A_{0}^{+} \cap A_{s}^{-} \cap \{B_s \geq x\}\bigr) \geq \frac{1-\delta}{1+\tilde\delta} \P_{0, u, b}\bigl(A_{0}^{+} \cap \{B_s \geq x\}\bigr) \;,
\end{equation*} 
and so 
\begin{align}
\frac{\P_{0, u, b}\bigl(A_{0}^{+} \cap A_{s}^{-} \cap \{B_s \geq x\}\bigr)}{\P_{0, u, b}(A_0^{+})} 
& \geq \frac{1-\delta}{1+\tilde \delta} \frac{\P_{0, u, b}\bigl(A_{0}^{+} \cap \{B_s \geq x\}\bigr)}{{\P_{0, u, b}(A_0^{+})}} \nonumber\\
& = \frac{1-\delta}{1+\tilde \delta} \Biggl(1 - \frac{\P_{0, u, b}\bigl(A_{0}^{+} \cap \{B_s \leq x\}\bigr)}{{\P_{0, u, b}(A_0^{+})}}\Biggr) \;.\label{eq:entrMainInter} 
\end{align} 
Now, by using \eqref{eq:boundsBBstaysPos}, we note that, by possibly taking $b > 0$ large enough depending on $k$ and $a$, it holds that  
\begin{align*}
\P_{0, u, b}(A^{+}_0) \geq \P_{0, u, b}(\tau_{-\zeta_{a, k}(0)} > b) 
& \geq \frac{2 \zeta_{a, k}(0)(u+\zeta_{a, k}(0))}{b}\Biggl(1-\f{\zeta_{a, k}(0)(u+\zeta_{a, k}(0))}{b}\Biggr)  \\
& \geq \frac{\zeta_{a, k}(0)(u+\zeta_{a, k}(0))}{b} \;,
\end{align*}
By using this estimate along with Lemma~\ref{lm:entrBasic}, we obtain that there exists a constant $c_3 = c_3(a)$ such that, for all $x \geq \zeta_{a, k}(s)$, 
\begin{equation*}
\frac{\P_{0, u, b}(A_0^{+} \cap \{B_s \leq x\})}{\P_{0, u, b}(A_0^{+})} \leq c_3 \frac{u x^2}{\zeta_{a, k}(0)(u + \zeta_{a, k}(0)) \sqrt{s}} \;.
\end{equation*}  
Hence, plugging this estimate back into \eqref{eq:entrMainInter} and recalling \eqref{eq:startEntropic}, we get that 
\begin{equation*}
\frac{\P_{0, u, b}(A^{+}_0 \setminus A^{-}_s)}{\P_{0, u, b}(A_0^{+})} \leq 1 - \frac{1-\delta}{1+\tilde \delta} \Biggl(1 - c_3 \frac{u x^2}{\zeta_{a, k}(0)(u + \zeta_{a, k}(0)) \sqrt{s}} \Biggr) \leq \tilde \delta + \delta + c_3 \frac{u x^2}{\zeta_{a, k}(0)(u + \zeta_{a, k}(0)) \sqrt{s}} \;.
\end{equation*}
Thanks to Proposition~\ref{pr:BoundBBAboveNeg}, Lemma~\ref{lm:techRhoAK}, and by using the fact that $\zeta_{a, k}(s) \leq \zeta_{a, 0}(s) + \zeta_{a, k}(0)$, we have that there exist a constant $c_4 = c_4(a)$ such that
\begin{equation*}
\P_{0, u, b}(A_0^{+}) \leq c_4 \frac{\zeta_{a, k}(0)(u + \zeta_{a, k}(0))}{b} \;,
\end{equation*}
and so, putting everything together, we obtain that there exist constants $c_5=c_5(a)$, $c_6 = c_6(a) >0$ such that 
\begin{equation*}
\P_{0, u, b}(A^{+}_0 \setminus A^{-}_s) \leq c_5 \frac{u x^2}{b \sqrt{s}} + c_6 \frac{\zeta_{a, k}(0)(u + \zeta_{a, k}(0))}{b} (\delta + \tilde \delta)
\end{equation*}
If we choose $x = c_7 (k+s)^{7/32}$, for some constant constant $c_7 = c_7(a) > 0$ for which $x > \zeta_{a, k}(s)$, one can check that the first term on the right-hand side of the above expression dominates the others as soon as $s$ is larger than some constant depending on $a$, and so we have that 
\begin{equation*}
\P_{0, u, b}(A^{+}_0 \setminus A^{-}_s) \leq c_8 \frac{u}{b} \frac{(k+s)^{7/16}}{\sqrt s}\;,
\end{equation*}
for some constant $c_8 = c_8(a)$. 
\end{proof}

%%%%%%%%%%%%%%%%%%%%%%%%%%%%%%%%%%%%%%%%%%%%%%
\subsection{Random walk estimates}
We collect in this subsection some results that allows us to transfer the statements for Brownian bridge we obtained in the previous subsection to the case of the random walk.  

\begin{lemma}
\label{lm:transferBBRW}
Let $\zeta: \R^{+}_0 \to \R^{+}_0$	be a non-decreasing concave function. For each $x \in \R$, all $b \geq k \geq 1$ it holds that 
\begin{align*}
\P_{x, u, b}\Biggl(\bigcap_{j = k}^{b} \bigl\{B_j & > - \zeta(j) \bigr\}   \, \Bigg| \, B_{k} = z\Biggr) \\
& \leq \P_{z, u, b-k}\Biggl(\inf_{s \in [0, b-k]} \bigl(B_s + 2\zeta(k+s)\bigr) > 0\Biggr) \prod_{j=k}^{b}\Bigl(1-e^{-2\zeta(j)^2}\Bigr)^{-2}\;.
\end{align*}
Similarly, for all $x \in \R$ and all $b \geq k \geq 1$, it holds that 
\begin{equation*}
\P_{x, u, b}\Biggl(\bigcap_{j = k}^b \bigl\{B_j > - \zeta(j) \bigr\} \, \Bigg| \, B_{k} = z\Biggr) \geq \P_{z, u, b-k}\Biggl(\inf_{s \in [0, b-k]} \bigl(B_s + \zeta(k+s)\bigr) > 0\Biggr) \;.
\end{equation*}
\end{lemma}
\begin{proof}
The proof follows exactly the same lines as that of \cite[Lemma 4.15]{BiskupLouidor}.
\end{proof}

\begin{lemma}
\label{lm:mainEntrRW}
For $a$, $k > 1$, let $\zeta_{a, k}:\R^{+}_0 \to \R^{+}_0$ be the function defined in \eqref{eq:zetaAK}. There exist constants $c_1 = c_1(a)$, $c_2 = c_2(a) > 0$ such that for all $\iota \in (0, 1/8)$, $b \geq 1$ sufficiently large, $u \in [b^{\iota}, b^{3/4}]$, and all $k \in (c_2, b/2]$, 
\begin{equation*}
\P_{0, u, b}\Biggl(\bigcap_{j = 1}^{b-1} \bigl\{B_j  > - \zeta_{a,k}(j)\bigr\}, \; \bigcup_{j = k}^{b-1} \bigl\{B_j < \zeta_{a,k}(j)\bigr\}\Biggr) \leq c_1 \frac{u}{b} k^{-1/16} \;.
\end{equation*}
\end{lemma}
\begin{proof}
Thanks to Lemma~\ref{lm:transferBBRW}, we have that 
\begin{align*}
\P_{0, u, b}\Biggl(\bigcap_{j = 1}^{b-1} \bigl\{B_j  & > - \zeta_{a,k}(j) \bigr\}, \; \bigcup_{j = k}^{b-1} \bigl\{B_j < \zeta_{a,k}(j)\bigr\}\Biggr)   \prod_{j=1}^{b}\Bigl(1-e^{-2\zeta_{a, k}(j)^2}\Bigr)^{2} \\
 & \leq \P_{0, u, b}\Biggl(\inf_{r \in [0, b]}\bigl(B_r + \zeta_{2 a, k}(r)\bigr) > 0, \;\inf_{r \in [k, b]}\bigl(B_r - \zeta_{2 a, k}(r)\bigr) < 0\Biggr) 	\;.
\end{align*}
By observing that the product on the first line is bounded away from zero uniformly in $k$ and $b$, the conclusion follows from Proposition~\ref{pr:BoundBBEntropic} with $s=k$. 
\end{proof}

%%%%%%%%%%%%%%%%%%%%%%%%%%%%%%%%%%%%%%%%%%%%%%
%%%%%%%%%%%%%%%%%%%%%%%%%%%%%%%%%%%%%%%%%%%%%%
\section{Proof of Lemma~\ref{lm:bRegular}}
\label{ap:proofbReg}
In this appendix, we verify that, for any $\lambda > 0$ and $A$, $L$, $b \geq 0$, the function $\smash{\smash{\bar{\frkF}}^{\lambda}_{A, L, b}: \R \times \CC(\R^d) \to \R}$ defined in \eqref{eq:defbarFALb} satisfies conditions \ref{eq:bReg1}\dash \ref{eq:bReg4} of Lemma~\ref{lm:bRegular}. To simplify the notation, we prove the case $\lambda = 1$ and note that the general case, where $\lambda > 0$, follows with straightforward modifications.

To streamline the presentation, only in this appendix, for any $g \in \CC(\R^d)$ and $\eta \geq 0$, we let
\begin{align*}
& \D_{g}(\eta) \eqdef \bigl\{y \in \B_b \; : \; \bar \Upsilon_{\! b, g} (y) \geq \M_{0, b}(\bar \Upsilon_{\! b, g}) - \eta\bigr\} \;.
\end{align*}
We note that if $0 \leq \eta_1 \leq \eta_2$, then $\D_g(\eta_1) \subseteq \D_g(\eta_2)$. 

\begin{proof}[Proof of Lemma~\ref{lm:bRegular}]
Fix $\lambda = 1$. We begin by defining the functions $\frkh: \R \to \R^{+}$ and $\frkF_{*}: \CC(\R^d) \to \R^{+}$ as follows
\begin{equation*}
\frkh(z) \eqdef e^{-\sqrt{\smash[b]{2d}} (z + L)}  \P_{z+A+L}(\rmY_b(0) \geq 0)^{1/2} \;, \qquad \frkF_{*}(g) \eqdef \E\Biggl[\frac{\one_{\{\M_{0, b}(\bar \Upsilon_{\! b, g}) \leq  2 \}}}{e^{-8db}\abs{\D_{g}(1/2)}^8 \wedge 1}\Biggr]^{1/4} \;,
\end{equation*}
We observe that \ref{eq:bReg1} follows immediately from the definition of $\frkh$, while \ref{eq:bReg2} is an immediate consequence of Cauchy--Schwarz's inequality. Hence, it remains to verify that \ref{eq:bReg3} and \ref{eq:bReg4} are satisfied.

Regarding condition \ref{eq:bReg3}, we fix $\delta \in (0,1)$ and $g \in C(\R^d)$ such that 
\begin{equation}
\label{eq:bRegAssumption3}
\sup_{x, y \in \R^d,\, \abs{x-y} \leq \delta}|g(x) - g(y)| \leq 1/4 \;.
\end{equation}
We define the random radius $r \geq 0$ as the largest radius for which all of the following conditions hold:
\begin{enumerate}
	\item There exists $x \in \B_b$ for which $B(x, r) \subseteq \B_b$ and $\bar \Upsilon_{\! b, g} (x) = \bar \M_{0, b}(\Upsilon_{\! b, g})$.
	\item For all $y \in B(x, r)$, it holds that $\bar \Upsilon_{\! b, g} (y) \geq \bar \M_{0, b}(\Upsilon_{\! b, g}) - 1/2$.
\end{enumerate}
For $\delta > 0$ and letting $S$ be the volume of the unit ball, we observe that
\begin{align*}
\frkF_{*}(g)^4 \leq S^{-8} \lceil 1/\delta \rceil^{8d} + \sum_{k = \lceil 1/\delta \rceil}^{\infty} S^{-8} (k+1)^{8d}\E\bigl[\one_{\{S/(k+1)^d \leq e^{-db}\abs{\D_g(1/2)} < S/k^d\}}\bigr] \;.
\end{align*}
For $k \in \N$ such that $\smash{k \geq \lceil 1/\delta \rceil}$, we note that $e^{-db}\abs{\D_g(1/2)} \leq S/k^d$ implies that $r \leq e^b/k < e^b \delta$. On this event, we know that there exists $z \in \B_b$ with $\abs{x - z} \leq e^b/k$ such that $\bar \Upsilon_{\! b, g}(z) < \bar \M_{0, b}(\Upsilon_{\! b, g}) - 1/2$. In turn, this implies that  
\begin{equation*}
\abs{\bar\Upsilon_{\!b, g}(x) - \bar\Upsilon_{\!b, g}(z)} \geq \frac12 \implies \abs{\rmY_b(e^{-b} x) - \rmY_b(e^{-b} z)} \geq \f12 - \abs{g(x) - g(z)}\;. 
\end{equation*}
In particular, recalling that we are assuming \eqref{eq:bRegAssumption3}, for all $k \in \N$ such that $\smash{k \geq \lceil 1/\delta \rceil}$, it holds that
\begin{equation*}
\P\Biggl(e^{-db} \abs{\D_g(1/2)} \leq \f{S}{k^d}\Biggr) \leq \P\Biggl(r\leq \f{e^{b}}{k}\Biggr) \leq \P\Biggl(\sup_{x, y \in \B_b, \, |x-y|\leq e^b/k}\bigl|\rmY_b(e^{-b}x) - \rmY_b(e^{-b}y)\bigr| \geq \f14\Biggr)\;.
\end{equation*}
By \cite[Equation~(3.10)]{Madaule_Max}, we know that there exist constants $c_1$, $c_2 > 0$ such that the probability on the right-hand side in the previous display is bounded from above by $c_1 e^{-c_2 e^{-b}k}$. Hence, we have that
\begin{align*}
\frkF_{*}(g)^4 & \leq S^{-8} \lceil 1/\delta \rceil^{8d} + \sum_{k = \lceil 1/\delta \rceil}^{\infty} S^{-8} (k+1)^{8d} \P\Biggl(\sup_{x, y \in \B_b, \, |x-y|\leq e^b/k}\bigl|\rmY_b(e^{-b}x) - \rmY_b(e^{-b}y)\bigr| \geq \f14\Biggr) \\
& \leq S^{-8} \lceil 1/\delta \rceil^{8d} + c_1 \sum_{k = \lceil 1/\delta \rceil}^{\infty} S^{-8} (k+1)^{8d} e^{-c_2 e^{-b}k} \lesssim \delta^{-10 d}\;,
\end{align*}
where the implicit constant depends on $b$.

Finally, regarding condition \ref{eq:bReg4}, we fix $g_1$, $g_2 \in \CC(\R^d)$ such that $\sup_{x \in \B_{b+2}}\abs{g_1(x) - g_2(x)} \leq 1/8 $. Then, thanks to the triangle inequality and Cauchy--Schwarz's inequality, we obtain that 
\begin{align}
& \abs{\smash{\bar{\frkF}}^{1}_{A, L, b}(z, g_1) - \smash{\bar{\frkF}}^{1}_{A, L, b}(z, g_2)} \nonumber \\ 
& \qquad \leq e^{-\sqrt{\smash[b]{2d}} (z + L) + db} \E_z\Biggl[\one_{\{\rmY_b(0) \geq -A-L\}} \frac{\abs{\one_{\{\M_{0,b}(\bar \Upsilon_{\!b, g_1}) \leq 1\}}  - \one_{\{\M_{0,b}(\bar \Upsilon_{\!b, g_2}) \leq 1\}}}}{\abs{\D_{g_1}(1)}}\Biggr] \label{eq:bRegGiant1} \\
& \qquad + e^{-\sqrt{\smash[b]{2d}} (z + L) + db} \E_z\Biggl[\one_{\{\rmY_b(0) \geq -A-L\}} \one_{\{\M_{0,b}(\bar \Upsilon_{\!b, g_1}) \leq 1\}} \frac{\abs{\one_{\{\rmE_{g_1}\}} - \one_{\{\rmE_{g_2}\}}}}{\abs{\D_{g_1}(1)}} \Biggr] \label{eq:bRegGiant2} \\
& \qquad + e^{-\sqrt{\smash[b]{2d}} (z + L) + db} \E_z\Biggl[\one_{\{\rmY_b(0) \geq -A-L\}} \one_{\{\M_{0,b}(\bar \Upsilon_{\!b, g_2})\leq 1\}} \Biggl|\frac{1}{\abs{\D_{g_1}(1)}} - \frac{1}{\abs{\D_{g_2}(1)}} \Biggr|\Biggr] \label{eq:bRegGiant3} \\
& \qquad + \frkh(z) \frkF_{*}(g_1) \E_z\bigl[\Delta(g_1, g_2)^4\bigr]^{1/4} \label{eq:bRegGiant4} \;,
\end{align}
where, for $g \in \CC(\R^d)$, we set
\begin{equation*}
\rmE_{g} \eqdef \bigl\{\M_{0, b + 1, b}(\bar \Upsilon_{\! b, g}) + \rmY_b(0) < -A-L\} \;,
\end{equation*}
and also, for $g_1$, $g_2 \in \CC(\R^d)$, we let 
\begin{equation*}
	\Delta(g_1, g_2) \eqdef \exp\Biggl(-\int_{\B_b} \rmF_{\gamma}(y) e^{\gamma(\bar \Upsilon_{\! b, g_1}(y) + \rmY_b(0)+L)} dy\Biggr) - \exp\Biggl(-\int_{\B_b} \rmF_{\gamma}(y) e^{\gamma(\bar \Upsilon_{\! b, g_2}(y) + \rmY_b(0)+L)} dy\Biggr) \;.
\end{equation*}
We proceed to bound separately \eqref{eq:bRegGiant1}, \eqref{eq:bRegGiant2}, \eqref{eq:bRegGiant3}, and \eqref{eq:bRegGiant4}. We will use repeatedly the fact that, letting $\delta = \|g_1 - g_2\|_{\infty}$, it holds that $|\M_{0,b}(\bar \Upsilon_{\!b, g_1}) - \M_{0,b}(\bar \Upsilon_{\!b, g_2})| \leq \delta$ and $| \M_{0, b + 2, b, g_1}(\bar \Upsilon_{\!b, g_1}) - \bar \M_{0, b + 2, b, g_2}(\bar \Upsilon_{\!b, g_1})| \leq \delta$. 

Regarding \eqref{eq:bRegGiant1}, we have the following bound
\begin{equation*}
\abs{\one_{\{\M_{0,b}(\bar \Upsilon_{\!b, g_1}) \leq 1\}} - \one_{\{\M_{0,b}(\bar \Upsilon_{\!b, g_2}) \leq 1\}}} \leq \one_{\{\M_{0,b}(\bar \Upsilon_{\!b, g_1})\in [1-\delta, 1+\delta]\}}\;.
\end{equation*}
Therefore, using \cite[Theorem~3.1]{SupDens} as in the proof of Lemma~\ref{lm:entroDelta}, thanks to the Cauchy--Schwarz's inequality, one obtains that  
\begin{equation*}
\eqref{eq:bRegGiant1} \lesssim \|g_1-g_2\|_{\infty}^{1/8} \frkh(z) \frkF_{*}(g_1)\;.
\end{equation*}
Regarding \eqref{eq:bRegGiant2}, we note that 
\begin{equation*}
\abs{\one_{\{\rmE_{g_1}\}} - \one_{\{\rmE_{g_2}\}}} \leq \one_{\{\M_{0, b + 1, b}(\bar \Upsilon_{\! b, g_1})  + \rmY_b(0) + A+L \in [-\delta,\delta]\}} \;.
\end{equation*}
Therefore, thanks again to \cite[Theorem~3.1]{SupDens} and using the Cauchy--Schwarz's inequality, we have that
\begin{equation*}
\eqref{eq:bRegGiant2} \lesssim \|g_1-g_2\|_{\infty}^{1/8} \frkh(z) \frkF_{*}(g_1)\;.
\end{equation*}
Regarding \eqref{eq:bRegGiant3}, we note that 
\begin{align*}
& \one_{\{\M_{0,b}(\bar \Upsilon_{\! b, g_2}) \leq 1\}} \Biggl|\frac{1}{\abs{\D_{g_1}(1)}} - \frac{1}{\abs{\D_{g_2}(1)}} \Biggr| \\
& \hspace{40mm} \leq \int_{\B_b} \frac{\one_{\{\M_{0,b}(\bar \Upsilon_{\!b, g_1}) \leq 1-\delta\}}}{\abs{\D_{g_1}(1/2)}^2} \one_{\{\M_{0,b}(\bar \Upsilon_{\!b, g_1}) - \bar \Upsilon_{\! b, g_1}(y) \in [1-\delta, 1+\delta]\}} dy\;,
\end{align*}
where we also used the fact that $\min\{\abs{\D_{g_1}(1)}, \abs{\D_{g_2}(1)}\} \geq \abs{\D_{g_1}(1/2)}$. Therefore, proceeding similarly to the above, we obtain that
\begin{equation*}
\eqref{eq:bRegGiant3} \lesssim \|g_1-g_2\|_{\infty}^{1/8} \frkh(z) \frkF_{*}(g_1) \;. 
\end{equation*}

To conclude, it remains to handle \eqref{eq:bRegGiant4}. For this term, it suffices to show that $\E_z[\Delta(g_1, g_2)^4]^{1/4} \lesssim \|g_1 - g_2\|_{\infty}^{1/8}$, which follows from some elementary computations.
\end{proof}

%%%%%%%%%%%%%%%%%%%%%%%%%%%%%%%%%%%%%%%%%%%%%%
%%%%%%%%%%%%%%%%%%%%%%%%%%%%%%%%%%%%%%%%%%%%%%
\section{Gaussian toolbox}
\label{sec:GaussianTool}
We collect here some well known results on Gaussian fields. In all the subsequent lemmas, we assume $\rmD \subseteq \R^d$ to be a bounded domain. We begin with the well known 
Cameron--Martin theorem \cite{CameronMartin}.
\begin{lemma}
\label{lm_Girsanov}
Consider an almost surely continuous centred Gaussian field $(\rmX(x))_{x \in \rmD}$ and a real-valued random variable $\rmZ$ belonging to the $L^2$ closure of the vector space spanned by $\{\rmX(x)\}_{x \in \rmD}$. Let $\bfF: \CC(\rmD) \to \R$ be a bounded measurable functional. Then the following equality holds  
\begin{equation*}
\E\Bigl[e^{\rmZ-\frac{\E[\rmZ^2]}{2}} \bfF\left(\rmX(\cdot)\right)\Bigr] = \E\bigl[\bfF\bigl(\rmX(\cdot)+\E[\rmX(\cdot)\rmZ]\bigr)\bigr] \,.
\end{equation*}
\end{lemma}

We have the following standard concentration inequality for Gaussian fields which is known as Borell-TIS inequality. We refer to \cite[Theorem~2.1.1]{Adler} for a proof.
\begin{lemma}
\label{lm_Borell}
Consider an almost surely continuous centred Gaussian field $(\rmX(x))_{x \in \rmD}$. Then it holds that
\begin{equation*}
\P\Biggl(\Biggl|\sup_{x \in \rmD} \rmX(x) - \E\Biggl[\sup_{x \in \rmD} \rmX(x)\Biggr]\Biggr| > t \Biggr) \leq 2 e^{-\frac{t^2}{2 \sigma_\rmD^2}} \,,
\end{equation*}
for all $t \geq 0$, where $\sigma_\rmD^2 := \sup_{x \in \rmD} \E[\rmX(x)^2]$.
\end{lemma}

We now state Fernique's majorizing criterion, and we refer to {\cite[Theorem~1.3.3]{Adler}} for a proof.
\begin{lemma}
\label{lm_Fernique}
Consider an almost surely continuous centred Gaussian field $(\rmX(x))_{x \in \rmD}$. Consider the pseudometric on $\rmD$ defined as follows
\begin{equation*}
\dd_X(x, y) := \sqrt{\E[\abs{\rmX(x) - \rmX(y)}^2]}, \qquad \forall \, x, y \in \rmD \;.
\end{equation*}
Then there exists a universal constant $C$ such that for any probability measure $\sigma$ on $\rmD$, 
\begin{equation*}
\E\Biggl[\sup_{x \in \rmD} \rmX(x)\Biggr] < C \sup_{x \in \rmD}\int_{0}^{\infty} \sqrt{\smash[b]{-\log \sigma(B_\rmX(x, r))}} dr \,,
\end{equation*}
where $B_\rmX(x, r) \eqdef \{y \in \rmD \, : \, \dd_\rmX(x, y) < r\}$. 
\end{lemma}
\endappendix

%%%%%%%%%%%%%%%%%%%%%%%%%%%%%%%%%%%%%%%%%%%%%%
%%%%%%%%%%%%%%%%%%%%%%%%%%%%%%%%%%%%%%%%%%%%%%
\small
\bibliographystyle{Martin}
\bibliography{./refs}{}
%%%%%%%%%%%%%%%%%%%%%%%%%%%%%%%%%%%%%%%%%%%%%%
%%%%%%%%%%%%%%%%%%%%%%%%%%%%%%%%%%%%%%%%%%%%%%

\end{document}